\documentclass[a4paper,12pt]{article}
\usepackage[affil-it]{authblk} %Affiliations in italics

\usepackage{soul} %enables to cross out text by \st{Hello World}

\usepackage{enumitem}
\setlist[enumerate]{label=(\roman*)} %An option of enumitem package, which globally changes the style of enumerated lists from (1) to (i)
\usepackage[mathscr]{euscript}
\usepackage{dsfont}

\usepackage{mathrsfs,mathtools}
\usepackage{fullpage,color}

\usepackage{tabularx} %tables, e.g. in list of symbols
\usepackage{amssymb}
\usepackage{amsmath}
\usepackage{amsthm}
\usepackage{latexsym}
\usepackage{amsfonts}
\usepackage{geometry}
\usepackage{hyperref} %clicable links to websites and references to literature, equations, theorems

\usepackage{graphicx} %graphics put exactly where I want
\usepackage[T1]{fontenc} %polish letters
\usepackage[latin2]{inputenc} %polish letters
\usepackage{authblk} %adding affiliations

\usepackage{bbm} %for writing indicator functions as doule 1 with \mathbbm{1}\te
\newcommand\dela[1]{}
\newcommand\dele[1]{}
\newcommand\delas[1]{}
\newcommand\deln[1]{}
\newcommand{\nun}{\nu}
\newcommand{\Mbar}{\overline{M}}
\newcommand{\Nbar}{\overline{N}}
\newcommand{\Pbar}{\overline{\mathbb{P}}}
\newcommand{\Fbar}{\overline{\mathscr{F}}}
\newcommand{\Ebar}{\overline{\mathbb{E}}}
\newcommand{\loc}{\mathrm{loc}}
\newcommand{\Aoneab}{\mathscr{A}_1(\alpha,\beta)}
\newcommand{\Atwoa}{\mathscr{A}_{2}(\alpha)}

\newcommand{\Omegabar}{\overline{\Omega}}
\usepackage{xcolor}

\usepackage[toc,page]{appendix}
\usepackage{verbatim}%% to use \begin{comment}
\usepackage{fouridx}

\usepackage{todonotes}

\usepackage{soul}

\usepackage{cancel}
\usepackage[normalem]{ulem}

\RequirePackage{fix-cm}

\newcommand{\power}{2}

\newcommand{\sinOT}{s\in [0,T]}
\newcommand{\tinOT}{t\in [0,T]}

\newcommand{\te}{\overline{\mathbb{E}}}
\newcommand{\unk}{{u}_{{n}_{k}}}
\newcommand{\cadlag}{{c\`{a}dl\`{a}g}\,}

\newcommand{\eps}{\varepsilon }
\newcommand\todown{\searrow}
\newcommand{\s}{s}

\newcommand{\acal}{\mathscr{A}}%{\mathcal{A}}
\newcommand{\ccal}{\mathscr{C}}%{\mathcal{C}}
\newcommand{\fcal}{\mathscr{F}}%{\mathcal{F}}
\newcommand{\lcal}{\mathcal{L}}

\newcommand{\ocal}{\mathcal{O}}

\newcommand{\scal}{\mathcal{S}}
\newcommand{\tcal}{\mathcal{T}}
\newcommand{\vcal}{\mathscr{V}}%{\mathcal{V}}
\newcommand{\xcal}{\mathscr{X}}%{\mathcal{X}}
\newcommand{\ycal}{\mathscr{Y}}%{\mathcal{Y}}
\newcommand{\zcal}{\mathscr{Z}}%{\mathcal{Z}}

\newcommand{\emm}{k}

\newcommand{\fmath}{\mathbb{F}}

\newcommand{\ind}[1]{{1\! \!  1 }_{#1}  }
\newcommand{\ip}[3]{{\bigl( #1 , #2 \bigr) }_{#3}}
\newcommand{\Ip}[3]{\Bigl( #1 , #2 {\Bigr)}_{#3}}

\newcommand{\dual}[3]{\fourIdx{}{}{}{#3}{\big\langle #1, #2 \big\rangle}}

\newcommand{\duali}[3]{\fourIdx{}{U^\prime}{#3}{U}{\langle #1, #2 \rangle}}

\newcommand{\normb}[3]{\fourIdx{}{}{#3}{#2}{\vert #1 \vert}}

\newcommand\soutok[1]{}
\newcommand\xcancelok[1]{}
\newcommand\footnoteok[1]{}

\newcommand{\ilsk}[3]{{(#1 ,#2 )}_{#3}}

\newcommand{\duality}[2]{\fourIdx{}{U^\prime}{}{U}{\bigl\langle #1, #2 \bigr\rangle}} 
\newcommand{\dualf}[4]{\fourIdx{}{#4}{}{#3}{\big\langle #1, #2 \big\rangle}}

\newcommand{\lb}[1]{\fourIdx{}{U^\prime}{}{}{\bigl\langle #1 }} 

\newcommand{\rbb}{\fourIdx{}{}{}{U}{ \bigr\rangle}}

\newcommand{\dualtrzy}[3]{\fourIdx{}{#1^\prime}{}{#1}{\bigl\langle #2, #3 \bigr\rangle}}

\newcommand{\dualityc}[2]{\fourIdx{}{U^\prime}{\mathrm{c}}{U}{\langle #1, #2 \rangle}}

\newcommand{\Mbarphi}[1]{\fourIdx{}{}{#1}{\phi}{\overline{M}}}

\newcommand{\tomega}{\overline{\omega}}
\newcommand{\tP}{\overline{\mathbb{P}}}
\newcommand{\tF}{\overline{\mathbb{F}}}

\newcommand{\tun}{{\overline{u}}_{n}}

\newcommand{\tunk}{{\overline{u}}_{{n}_{k}}}
\newcommand{\tu}{\overline{u}}
\newcommand{\ubar}{\overline{u}}

\newcommand{\tfcal}{\overline{\fcal}}

\newcommand{\rM}{\mathrm{M}}
\newcommand{\rH}{H}%{\mathrm{H}}
\newcommand{\rV}{V}%{\mathrm{V}}
\newcommand{\rB}{B}%{\mathrm{B}}
\newcommand{\rU}{U}%{\mathrm{U}}
\newcommand{\rY}{Y}%{\mathrm{Y}}
\newcommand{\bX}{\mathbf{X}}
\newcommand{\bH}{\mathbb{H}}
\newcommand{\bL}{\mathbb{L}}

\newcommand{\HR}{H_R}%{\rH_{B_R}}
\newcommand{\VR}{V_R}%{\rV_{B_R}}

\newcommand{\id}{\mathrm{id}}
\newcommand{\dom}{\mathrm{Dom}}

\newcommand\toup{\nearrow}
\newcommand\embed{\hookrightarrow}
\newcommand{\ddual}[4]{\fourIdx{}{#1}{}{#4}{\big\langle #2, #3 \big\rangle}}

\numberwithin{equation}{section} %numbering equations with a section number, requires amsmath

\theoremstyle{plain}
\newtheorem{theorem}{Theorem}[section]
\newtheorem{lemma}[theorem]{Lemma}
\newtheorem{assum-A}{Assumption A.$\!$}
\newtheorem{assum-P}{Assumption P.}
\newtheorem{assum-F}{Assumption F.}
\newtheorem{corollary}[theorem]{Corollary}
\newtheorem{proposition}[theorem]{Proposition}
\newtheorem*{Theorem*}{Theorem}
\newtheorem*{Lemma*}{Lemma}
\newtheorem*{Corollary*}{Corollary}
\newtheorem*{Proposition*}{Proposition}

\theoremstyle{definition}
\newtheorem{definition}[theorem]{Definition}
\newtheorem{remark}[theorem]{Remark}

\newtheorem{condition}[theorem]{Condition}
\newtheorem{example}[theorem]{Example}
\newtheorem{exercise}[theorem]{Exercise}

\newtheorem*{Definition*}{Definition}
\newtheorem*{Notation*}{Notation}
\newtheorem*{Fact*}{Fact}
\newtheorem*{Example*}{Example}
\newtheorem*{Observation*}{Observation}
\newtheorem*{Remark*}{Remark}
\newtheorem*{Exercise*}{Exercise}
\newtheorem*{Question*}{Question}
\newtheorem*{Problem*}{Problem}
\newtheorem*{Assumption*}{Assumption}
\newtheorem*{Idea*}{Idea}
\DeclarePairedDelimiter\abs{\lvert}{\rvert}%
\DeclarePairedDelimiter\norm{\lVert}{\rVert}%

\makeatletter
\let\oldabs\abs
\def\abs{\@ifstar{\oldabs}{\oldabs*}}
\let\oldnorm\norm
\def\norm{\@ifstar{\oldnorm}{\oldnorm*}}
\makeatother
\usepackage{mathrsfs}

\newcommand{\R}{\mathbb{R}}

\newcommand{\F}{\mathbb{F}}
\newcommand{\N}{\mathbb{N}}

\newcommand{\ud}{\mathrm{d}}

\newcommand{\dx}{\, \mathrm{d}x}
\newcommand{\dr}{\, \mathrm{d}r}
\newcommand{\ds}{\, \mathrm{d}s}
\newcommand{\dt}{\, \mathrm{d}t}

\newcommand{\E}{\mathbb{E}}

\newcommand{\1}{\mathds{1}}% requires package dsfont
\DeclareMathOperator{\supp}{\mathrm{supp}}

\DeclareMathOperator{\divergence}{\mathrm{div}}
\DeclareMathOperator{\diver}{\mathrm{div}\,}

\DeclareMathOperator{\Span}{\mathrm{Span}}

\DeclareMathOperator{\Id}{\mathrm{Id}}
\DeclareMathOperator{\Leb}{\mathrm{Leb}}

\usepackage{letltxmacro}

\LetLtxMacro\Oldfootnote\footnote

\makeatletter
\def\namedlabel#1#2{\begingroup
    #2%
    \def\@currentlabel{#2}%
    \phantomsection\label{#1}\endgroup
}
\makeatother

\newcommand{\D}{\mathrm{D}}

\usepackage[open,openlevel=2]{bookmark} %to have a pdf with the sections, subsections etc. displayed on the left-hand side

\begin{document}

\title{Weak solutions of Navier-Stokes Equation with purely discontinuous L\'evy Noise}
\author[1]{Zdzis\l{}aw Brze\'zniak}
\author[2]{Tomasz Kosmala}
\author[3]{El{\. z}bieta Motyl}
\author[4]{Paul Razafimandimby}
\affil[1]{Department of Mathematics, University of York
Heslington, York, YO10 5DD, UK;
\href{mailto:zdzislaw.brzezniak@york.ac.uk}{zdzislaw.brzezniak@york.ac.uk}}
\affil[2]{Department of Mathematics, Imperial College London, 180 Queen's Gate, London SW7 2AZ, UK; \href{mailto:t.kosmala@imperial.ac.uk}{t.kosmala@imperial.ac.uk}}
\affil[3]{Faculty of Mathematics and Computer Science, University of \L\'{o}d\'{z},  Poland;
\href{mailto:elzbieta.motyl@wmii.uni.lodz.pl}{elzbieta.motyl@wmii.uni.lodz.pl}}
\affil[4]{School of Mathematical Sciences, Dublin City University, Collins Avenue Dublin 9, Ireland; \href{mailto:paul.razafimandimby@dcu.ie}{paul.razafimandimby@dcu.ie}}
\maketitle

\begin{abstract}
In this paper we prove the existence of weak martingale solutions to the stochastic Navier-Stokes Equations driven by pure jump L\'evy processes. Our proof consists of two  parts. In the first one, mostly classical,  we recall a priori estimates, from the paper by the third named author,  for solutions to suitable constructed Galerkin approximations  and  we use the Jakubowski-Skorokhod Theorem to find a sequence of processes on a new probability space convergent point-wise to a limit process. In the second one, we show that the limit process is a weak martingale solution to the 
SNSEs by using an approach of Kallianpur and Xiong. The core of this method consists of a proof that  a certain natural process on the new probability space is a purely discontinuous martingale and then to  use a suitable representation theorem. In this way we propose a method of proving solutions to stochastic PDEs which is different from the method used recently by the first and fourth named author in their joint paper with E.\ Hausenblas, see \cite{Brz+Haus+Raza_2018_reaction_diffusion}. 
\end{abstract}

\section{Introduction}
\label{sec_introduction}

In this paper we consider the following Navier-Stokes Equations (NSEs), in the whole space $\R^d$, with $d=2$ or $d=3$, driven by a  Poisson Random Measure
\begin{equation}
\label{eqn-reduced_Navier_Stokes}
\begin{cases}
\partial_t u -  \Delta u + \langle u, \nabla \rangle u + \nabla p = f(t) \dt + \int_{\R^d} F(t,u,y) \,\tilde{\eta}(\ud t,\ud y) \\
\divergence u = 0,\\
u(0)=u_0,
\end{cases}
\end{equation}
where 
\begin{align*}
    &u\colon\mathbb{R} \times \mathbb{R}^d \to \mathbb{R}^d , 
    \\
    &p\colon  \mathbb{R} \times  \mathbb{R}^d \to \mathbb{R}
    \end{align*}
    are respectively the velocity and pressure of the fluid. The symbol $\tilde{\eta}$ denotes a compensated time homogeneous Poisson random measure 
 on  $(\R_+ \times \rY, \mathscr{B}(\R_+) \otimes \ycal )$ over a filtered probability space  $(\Omega , \fcal , \fmath , \mathbb{P})$  intensity measure  of the form $\Leb \otimes \nu$, where  $\Leb $ is the Lebesgue measure on  the $\sigma$-field $ \mathscr{B}(\R_+)$ and $\nu$ is a  $\sigma$-finite measure on the $\sigma$-field $\ycal $, see  Definition \ref{def_PRM_filtration}.
 We omit here the dependence of the sample element $\omega \in \Omega$. 

Before we  describe the field of Stochastic Navier Stokes Equations (SNSEs) driven by jump processes, let us spend some time on describing the literature concerned with the existence theory of the Stochastic Navier Stokes Equations driven by Gaussian  processes. To the best of our knowledge, the first paper on this topic was written by Bensoussan and Temam \cite{Bensoussan+Temam_1973} who considered the SNSEs driven by the additive Wiener process. This line of research was continued by Vishik and Fursikov, see e.g. their monograph \cite{Vishik+Fursikov_1988}. The first paper which consider the SNSEs driven by the so called transport noise was the paper \cite{BCF_1993}, by the first named author, Flandoli and Capi{\'n}ski. This direction was later taken up by Flandoli and Gatarek \cite{Fland+Gat_1995}, who by the way, also treated other types of gaussian noises.  All these papers and others, which we have not cited, have dealt with NSEs in bounded domains. The first paper which dealt with the exietnce of solutions to the SNSEs (as well as Euler Equations) in  unbounded domains was the paper \cite{Brz+Peszat_2001} which treated multiplicative but not of transport type  gaussian noise. After that Rozovski and Mikulevicius  published the paper \cite{Mikulevicius+Rozovskii_2005} which studied the SNSEs in the whole space.
The method used by these authors was different and used weak topolopogies instead of weighted spaces used in the former paper. The latter paper was subsequently generalized by the first and third named authors in \cite{Brzezniak+Motyl_2013}, who proved the existence of weak martingale solutions to SNSEs (in both dimensions $d=2$ and $d=3 $ for general unbounded domains. One important improvement was the use of the Jakukowski's generalization \cite{Jakubowski_1998} to a class of non-metric spaces of the Skorokhod representation theorem \cite{Skorohod_1956}. We have not mentioned papers related to qualitative behaviours of SNSEs as for instance the existence and properties of attractors, the existence, uniqueness and properties invariant measures and of stationary solutions.  We want to finish this part of our Introduction about very recent developments. In particular about the question of regularization by noise, see 
e.g. \cite{Flandoli+Luo_2021} and references therein, the non-uniqueness of very weak solutions in the sense of convex integration, $L^p$-theory of SNSEs, including  transport noise, see e.g. \cite{Brz+Peszat_2000}, \cite{Agresti+Veraar_2024} and references therein, as well the  geometric approach to transport noise, see e.g. \cite{Crisan+Holm+Leahy+Nilssen_2022} and references therein.

There exists  also a  vast amount of results related to the stochastic Navier-Stokes equations driven by the  L\'evy processes called also jump noise. The existence of solutions, the long-time behaviour of these solutions, in particular the existence and uniqueness of invariant measure associated to the solutions, has been the subjects of extensive mathematical analysis, which have generated several important results, see \cite{ZB+EH+JH_2013,Dong+Li+Zhai_2012,Dong+Xie_2009,Dong+Xu+Zhang_2011,Dong+Zhai_2011,Motyl_2013, Motyl_2014,Sritharan_2012,Zhang-Tusheg_2009}  among others. 

To the best of our knowledge, the first papers about the Stochastic Navier-Stokes Equations driven by L\'evy type noise were articles by Dong Zhao and his collaborators. In particular, Dong Zhao and his collaborators proved  the existence of weak solutions of 2D and 3D Navier-Stokes with L\'evy noise in   \cite{Dong+Xie_2009}, \cite{Dong+Li+Zhai_2012}, established the Large Deviation Principles of L\'evy noise driven 2D Navier-Stokes equations  in    \cite{Zhang-Tusheg_2009}, and show the existence of invariant measure of the 2D Navier-Stokes equatiosn driven by $\alpha$-stable L\'evy noise in \cite{Dong+Xu+Zhang_2011}. Then, there were two papers \cite{Brz+Haus+Zhu_2013} and \cite{Sritharan_2012} by, respectively,  Brze{\'z}niak et al and Striharan, which used a similar approach, called Barbu method adapted to the jump case. These two papers and the ones by Dong Zhao and his collaborators do not use arguments based on  tightness or compactness theorems.

These papers were followed by two fundamental papers \cite{Motyl_2013} and \cite{Motyl_2014}, who extended her own works \cite{Brzezniak+Motyl_2011} and \cite{Brzezniak+Motyl_2013} from the Wiener processes to the full L\'evy  case.  Recently, 
Brze{\'z}niak,  Peng and Zhai proved in \cite{Brz+X+Zhai_2023} the existence of weak and strong solutions to the 2D stochastic Navier Stokes Equations driven by pure jump L\'evy processes for the initial data from the solenoidal versions of either the Lebesgue space $L^2$ or the Sobolev space $H^1$. They used  an 
ingenuous  proof based on the classical Banach Fixed Point Theorem. 
Most importantly, they also established  and fully proved the Large Deviations Principle (what was the main motivation of that paper).   This paper was generalized in \cite{Peng+Yang+Zhai_2022} by Peng,  Yang and Zhai by removing some "atypical" assumptions.  The recent literature on the topic of SPDEs driven by L\'evy processes of jump type  includes also (as examples) the following aricles: 
\cite{Chen+Wang+Wang_2019} by Chen, Wang and Wang, \cite{Nguyen+Tawri+Temam_2021} by Nguyen, Tawri and Temam, \cite{Kumar+Mohan_2024} by Kumar and Mohan, and  \cite{Ngana+Deugoeu+Medjo_2023} Ndongmo  Ngana, Deugoue and Tachim Medjo.

Despite this vast amount of papers, there are still some questions, which are well-understood for the Wiener process driven Navier-Stokes equations, remain unanswered for the Navier-Stokes driven by pure jump noise, e.g.\ uniqueness of invariant measure for 2D Navier-Stokes driven by a noise of the form $\sum_{k=1}^\infty \alpha_k L_k $, where $(L_k)_{k\in \mathbb{N}}$ is sequence of i.i.d.\ L\'evy processes and only finitely many of the $\alpha_k\in \mathbb{R}$ are non-zero, and the existence of weak martingale solutions via a martingale representation theorem. The main reason for this is that many of the mathematical tools which are useful for the analysis of Wiener processes driven SPDEs might not be applicable to SPDEs with L\'evy noise. Also, there are some important results which are true for SPDEs driven by specific classes of L\'evy noise but cannot be applied to SPDEs driven by a different type of L\'evy noise, for instance the earlier cited paper \cite{Dong+Xu+Zhang_2011}, who studied the existence of invariant measures of stochastic 2D Navier-Stokes equations driven by $\alpha$-stable processes.

Our main (and modest) goal in this paper is to establish the existence of a global weak solution of  \eqref{eqn-reduced_Navier_Stokes}. 
In the construction of a weak solution in the probabilistic sense, an important r\^{o}le is play an appropriate representation theorems. While considering SPDEs equations driven by Gaussian processes, the well-known theorem on the representation of the square integrable continuous martingales as the It\^{o} stochastic integrals with respect to Wiener processes is used, see e.g. \cite{Fland+Gat_1995}, \cite{Brzezniak+Motyl_2013}. 
Following this classical idea, in our construction of a  weak solution to problem \eqref{eqn-reduced_Navier_Stokes}, we use the representation theorem for square integrable purely discontinuous martingales as the stochastic integrals with respect to the Poisson random measures, see 
\cite{Kallianpur_Xiong}, \cite[Section II.\,7]{Ikeda+Watanabe_1989} and reference therein. This approach enables us, in particular, to indicate explicitly appropriate Poisson random measure being a part of the weak solution to problem \eqref{eqn-reduced_Navier_Stokes}, see Definition \ref{def-solution martingale}. This is the main novelty of the present paper.

We will use the Skorokhod Theorem in the Jakubowski version \cite{Jakubowski_1998},
but the proof that the limiting process is a solution to equation \eqref{eqn-reduced_Navier_Stokes}  is completely different from the one used in the previously mentioned paper \cite{Brz+Haus+Raza_2018_reaction_diffusion}. 
Our method is based on representation of pure jump martingales as stochastic integrals with respect to Poisson Random measure. A byproduct of our work is that we have prepared ground for treating other SPDEs driven by L\'evy processes. 
The reason  we wanted to find a proof which does not rely on  the method used in the paper \cite{Brz+Haus+Raza_2018_reaction_diffusion} by the first and fourth named authors is that we are no longer convinced that part (iii)  of Theorem C.1 therein is correct, as it was pointed out to us by M. Ondrej\'{a}t \cite{Ondrejat-2000}.
We are grateful  to  M. Ondrej\'{a}t for his personal information and to A. Debussche  for related discussion.

Let us describe the content of the present paper in more detail. We prove, under fairly general assumption on the noise coefficients and the measure $\nu$ that the system \eqref{eqn-reduced_Navier_Stokes} has a global weak solution. The term \textit{weak solution} is understood in the sense of both the Stochastic Analysis  and the Partial Differential Equations. That is we look for both the velocity $u$, a filtered probability space $(\Omega, \mathcal{F}, \mathbb{F}, \mathbb{P})$, where $\mathbb{F}=(\mathcal{F}_t)_{t\ge 0}$, satisfying the usual condition and a Poisson Random Measure $\eta$ with the prescribed L\'evy measure $\nu$. The velocity $u$ is required to satisfies the variational formulation of \eqref{eqn-reduced_Navier_Stokes}. 
A precise formulation of the definition of a solution, we use the term "martingale solution", is given in Definition \ref{def-solution martingale} and the precise formulation of the main result is given in Theorem \ref{thm-main-existence}.
The presentation from sections \ref{sec-statement} and \ref{sec-existence-I} follows mostly the  papers \cite{Motyl_2013} and \cite{Motyl_2014} by the third named author. However, because of some small but important changes and for the sake of the completeness, we include as much details as possible.
Moreover, we  prove our main result by using a theorem on representation of pure jump martingales by stochastic integrals driven by compensated Poisson Random Measure. While this kind of representation theorem is well-known and understood for martingales with continuous paths, it is known to hold only for purely discontinuous processes with values in a co-nuclear space, see \cite[Th.\ 3.4.7]{Kallianpur_Xiong}. Our second main  contribution, which is interesting in itself, is the proof of a  Representation Theorem for purely discontinuous martingales for Hilbert-space valued processes. Although we closely follow the presentation  in \cite{Kallianpur_Xiong}, the proof is non-trivial, see Appendix \ref{App:Appendix-RepThm}.

The layout of this paper can be summarized as follows. In Section \ref{sec-preliminaries} we recall and present some preliminary notions and results that are needed in the subsequent sections. Section 3 is devoted to the statement of the standing assumptions and the main result of the paper.
We start the proof of our main result in Section \ref{sec-existence-I} where we introduce the Galerkin approximation of the problem \eqref{eqn-reduced_Navier_Stokes}. Still in Section \ref{sec-existence-I},  we  derive several important uniform estimates for the solutions of the Galerkin approximation, which  will be used to show that the family of the approximate solutions' laws are tight on a space $\mathscr{Z}_T$ introduced in \eqref{eqn-Z_cadlag_NS}. This tightness result along with the Skorokhod-Jakubowski Theorem enables us to construct new probability space on which almost sure convergence of a sequence of random variables (with the same laws as  the  solutions to the Galerkin approximation problems)   to the candidate solution $\tu$ holds.
 Section 5 is devoted to one of the crucial and technical steps of the proof. There we prove that the
 processes 
$$ \begin{aligned}
{\Mbar}  (t)
   &= \tu (t) \,  -    \tu (0)
   - \int_{0}^{t}  f(s) \ds
   \\ &+ \int_{0}^{t}  \acal \tu (s) \ds
  + \int_{0}^{t} {\rB}  \bigl( \tu (s)  \bigr) \ds,  \quad
t \in [0,T].
\end{aligned}
$$ 
 and 
$$
\begin{aligned}
  \Nbar_\phi(t) &=\duali{\tu (t)}{\phi}{2} - \duali{u_0}{\phi}{2} \\
  &+2  \int_0^t \duality{\acal \tu (s)+\rB(\tu (s))-f(s)}{\phi}\;\duality{\tu (s)}{\phi} \ds
  \\
  &-\int_0^t \int_{Y} \ip{F(s,\tu (s),y)}{\phi}{\rH}^2 \, \nun(\ud y) \ds, \;\; t\in [0,T], \, \phi \in\rU 
  \end{aligned}
$$
are martingales with respect to standard augmentation of the filtration generated by $\tu$.

In Section \ref{sec-existence-III}  we use these results along with the Martingale Representation Theorem for purely discontinuous martingales \ref{thm-martingale_rep}, due to Kallianpur and Xiong, see \cite[Theorem\ 3.4.7]{Kallianpur_Xiong}, in order to construct the desired martingale solutions. The analyses in Sections \ref{sec-existence-II} and \ref{sec-existence-III} are nontrivial and require long, tedious, and technical arguments. They are also the main difference between the current paper and previous work on global weak solutions of stochastic Navier-Stokes with L'evy noise. In the appendices, we state and prove several results, some of which are well-known and some less known, that are necessary for our analysis. In particular, we establish in Appendix \ref{App:Appendix-RepThm} a Martingale Representation Theorem for purely discontinuous martingales. This theorem is interesting in itself and will be useful in the proof of existence of martingale solutions of other problems.

\subsection*{Notation.}
Throughout this paper we set $\R_+=[0,\infty)$ and $\R_+^\ast=(0,\infty)$.  We denote by $\overline{\R}_+$ and by  $\R^\ast$ the sets  $[0,\infty]$ and $\mathbb{R} \setminus \{0\}$, respectively.  We also denote by $\mathbb{N}$ and $\overline{\mathbb{N}}$ the sets $\{0,1,2,3,\ldots\}$ and  $\mathbb{N} \cup \{+\infty\}$, respectively.
For a nonempty subset $C \subset [0,\infty]$ we set
\begin{equation*}
\label{eqn-C^2_+}
C^2_{+}\coloneqq\bigl\{ (s,t)\in C^2: s\leq t\bigr\}.
\end{equation*}

\ 
Let $X$ be a topological Hausdorff  space, 	 $T \in (0,\infty)$
 and $f\colon [0,T]\to X$  a  c\`adl\`ag function, that is $f$ is  right continuous and has left limit at each point $t\in [0,T]$. We put
\begin{align*}
f_{-}(t)=f(t-)&\coloneqq \begin{cases}
\lim_{s\toup t} f(s),  & \mbox{ if } t \in (0,T], \\
f(0) & \mbox{ if } t=0,
\end{cases}
\end{align*}
Moreover, if $X$ is a vector space, then we put
\begin{align}\label{eqn-f-jumps}
\Delta f(t) =f(t)-f(t-), \qquad t \in [0,T],
\end{align}
By $\mathbb{D}([0,T];X)$ we will denote the space of all c\`adl\`ag functions $f\colon [0,T]\to X$. In the special case of $X=\R$ we will use a simple notation
\[\mathbb{D}([0,T]) \coloneqq \mathbb{D}([0,T];\R).
\]
We also introduce the following useful notation. For $t \in [0,T]$ we put
\begin{equation}\label{eqn-D(0,T)-continuous points}
\mathbb{D}_{t}^{c}([0,T];X)\coloneqq\begin{cases}
\mathbb{D}([0,T]; X), &\mbox{ if } t \in \{0,T\},\\
\bigl\{ x \in \mathbb{D}([0,T];X):\; \Delta x(t)=0 \}  &\mbox{ if } t \in (0,T).
\end{cases}
\end{equation}

By \cite[Cor.\ I.1.25]{Jacod_Shiryaev}, if $\xi$ is a c\`adl\`ag adapted process with values in a separable Banach space $X$, then both processes
$\xi_{-}$ and $\Delta \xi$ are  optional.

By an \emph{evanescent} set we mean a random set $A\subset [0,T]\times \Omega$, whose projection  on $\Omega$, i.e.\ the set  $\bigl\{ \omega \in \Omega: \exists t \in [0,T]: (t,\omega) \in A \bigr\}$, is a $\mathbb{P}$-null set, see \cite[I.1.10]{Jacod_Shiryaev}.

To avoid any ambiguity, let us  emphasise that given a property $R$,
whenever  we write
\begin{equation}
    \label{eqn-R property-1}
    \mathbb{P}\mbox{-a.s. } R  \mbox{ holds for every  }t\in [0,T],
\end{equation}
or
\begin{equation}
    \label{eqn-R property-2}
\mathbb{P}\mbox{-a.s., for every  } t\in [0,T], \;  \mbox{ property  $R$ holds } 
\end{equation}
we mean that the set 
\begin{equation}
    \label{eqn-Omega_R}
\Omega_R\coloneqq\bigl\{ \omega \in \Omega: \mbox{  for every } t \in [0,T], \mbox{  property } R \mbox{  holds for } (t,\omega) \bigr\}
\end{equation}
is  $\mathbb{P}$-full,   i.e.  the random set 
\begin{equation}
    \label{eqn-A-evanescent}
A\coloneqq\bigl\{ (t,\omega) \in [0,T] \times \Omega :  \mbox{  property } R \mbox{  does not  hold for } (t,\omega) \bigr\}
\end{equation}
is an evanescent set.

On the other hand, when we write 
\begin{equation}
    \label{eqn-R property-3}
    \mbox{ for every  $t\in [0,T],$  $R$ holds  $\mathbb{P}$-a.s.} 
    \end{equation}
we mean that  for every  $t\in [0,T]$, the set 
\[
\bigl\{ \omega \in \Omega: \mbox{  property } R \mbox{  holds for } (t,\omega) \bigr\} \mbox{ is  a $\mathbb{P}$-full set.}\]

For two normed vector  spaces $X$ and $Y$, $\mathscr{L}(X,Y)$ denotes the space of bounded linear operators from $X$ to $Y$ and the operator norm is denoted by $\norm{\cdot}_{\mathscr{L}(X,Y)}$. When $Y=X$ we will simply write 
$\mathscr{L}(X)$ and  $\norm{\cdot}_{\mathscr{L}(X)}$.
By $X^{\prime }$ we will denote dual space of $X$, i.e.  $X^{\prime }=\mathscr{L}(X,\mathbb{R})$, 
and by $\dualtrzy{X}{\cdot}{\cdot}$ we will denote the duality between ${X}^{\prime }$ and $X$.

\section{Preliminaries}\label{sec-preliminaries}
In this section we introduce few definitions related to Poisson Random Measures and introduce technical terms and definitions related to the theory of Navier-Stokes equations.

\subsection{Poisson Random Measures}\label{seubsec-PRM}
As mentioned in the introduction of this paper we are interested in the analysis of the Navier-Stokes equation with L\'evy noise. Similarly to the papers by Brze\'zniak and Motyl \cite{Brzezniak+Motyl_2013} and by Motyl  \cite{Motyl_2013, Motyl_2014} we consider a mutliplicative noise described by means of  Poisson Random Measures. In this section we recall two concepts, namely that  of a Poisson Random Measure (PRM)  by closely following \cite{Ikeda+Watanabe_1989} and its extension to a PRM with respect to a given filtration.
To this aim, let $M_{\overline{\mathbb{N}}}(E)$ denote the set of all $\overline{\mathbb{N}}$-valued measures on a measurable space $(E,\mathscr{E})$. The space $M_{\overline{\mathbb{N}}}(E)$ has the natural $\sigma$-algebra $\mathscr{M}_{\overline{\mathbb{N}}}(E)$ generated by the following family os maps
\begin{equation*}\mbox{
$M_{\overline{\mathbb{N}}}(E) \ni \mu \mapsto \mu(A) \in \overline{\mathbb{N}}$
}, \;\; A \in \mathscr{E}.
\end{equation*}

We start by recalling the following definition of a Poisson random measure. Our expositure is based on   \cite[Def.\ 8.1]{Ikeda+Watanabe_1989}.

\begin{definition}
\label{def_PRM}
A measurable function  
\begin{equation}
    \eta\colon \Omega \to M_{\overline{\mathbb{N}}}(E)
    \end{equation}
    is a Poisson random measure if and only if 
\begin{enumerate}
\item for any $A\in \mathscr{E}$, the function
\begin{equation*}
\eta(A)\colon \Omega \ni \omega \mapsto \eta(\omega)(A) \in \overline{\mathbb{N}}
\end{equation*}
 is a Poisson random variable with parameter $\E \left[ \eta(A) \right]$ if $\E \left[ \eta(A) \right] < \infty$, and $\eta(A)=\infty $ $\mathbb{P}$-a.s.\ otherwise.
\item[(ii)] For every finite family $A_1,A_2,\ldots A_n \in \mathscr{E}$ of pairwise  disjoint sets, the random variables $\eta(A_1),\ldots, \eta(A_n)$ are independent.
\end{enumerate}
By the \emph{intensity measure} of a PRM $\eta$  we understand   the measure
\begin{equation}\label{eqn-intensity measure}
\mu\colon \mathscr{E} \ni A \mapsto   \E \left[ \eta(A) \right] \in \overline{\mathbb{N}}_+.
\end{equation}
\end{definition}

We will always assume that the filtered probability space
 $(\Omega , \fcal , \mathbb{P},\mathbb{F})$ satisfies the so-called \textit{usual hypothesis}, i.e.
 \begin{trivlist}
\item[(i)] the filtration $\mathbb{F}$ is right-continuous, that is 
 \[\mathscr{F}_t=\bigcap\limits_{u>t} \mathscr{F}_u,\qquad \text{for all } t\in[0,\infty)\,,\]
\item[(ii)]
and  every $\mathbb{P}$-null, i.e.\ $\mathbb{P}$-negligible,  element  $A \in \mathscr{F}$,  belongs to $\mathscr{F}_0$.
 \end{trivlist}

Next, similarly to the concepts of a Wiener or a L\'evy process with respect to a filtration, we introduce a notion  of a time-dependent Poisson random measure, see e.g.
\cite[Def.\ 2.3]{Brzezniak_Hausenblas_Razafimandiby}.

\begin{definition}
\label{def_PRM_filtration}
Let $(\rY,\ycal)$ be a measurable space and let 
\begin{equation}
    E = \R_+ \times Y\mbox{ and }\mathscr{E} = \mathscr{B}(\R_+) \otimes \ycal.
    \end{equation}
    Let $(\Omega , \fcal , \mathbb{P})$ be a probability space and $\mathbb{F}=(\mathscr{F}_t)_{t\geq 0}$ be a filtration on $(\Omega,\mathscr{F}, \mathbb{P})$ satisfying the usual conditions.
A measurable function
\begin{equation}\label{eqn-PRM-02}
\eta\colon \Omega \to M_{\overline{\mathbb{N}}}(E)
\end{equation}
is called a Poisson random measure over
$(\Omega , \fcal , \fmath , \mathbb{P})$  if and only  it satisfies the conditions  (i) and (ii) in Definition \ref{def_PRM} as well as  the following two  additional  conditions:
\begin{enumerate}
\item[(iii)] For all $A\in \mathscr{Y}$ and $t \geq 0$ the
function
\begin{equation*}
\eta((0,t] \times A)\colon \Omega \ni \omega \mapsto \eta(\omega)((0,t] \times A) \in \overline{\mathbb{N}}
\end{equation*}
 is $\mathscr{F}_t$-measurable.
\item[(iv)] For all $A\in \mathscr{Y}$ and $(s,t)\in [0,\infty)^2_+$ the random variable  $\eta((s,t] \times A)$ is independent of the $\sigma$-field $\mathscr{F}_s$.
\end{enumerate}

A Poisson random measure $\eta$  as in \eqref{eqn-PRM-02} will be called a time homogeneous Poisson random measure on  $(\R_+ \times \rY, \mathscr{B}(\R_+) \otimes \ycal )$ over $(\Omega , \fcal , \fmath , \mathbb{P})$ if and only if  its intensity measure $\mu$ is of the form $\Leb \otimes \nu$, where  $\Leb $ is the Lebesgue measure on 
the $\sigma$-field $ \mathscr{B}(\R_+)$ and $\nu$ is a  $\sigma$-finite measure on the $\sigma$-field $\ycal $.
\end{definition}

For more background information about the integration with respect to point processes and random measures, see the Ikeda-Watanabe monograph \cite[p.\ 61-63]{Ikeda+Watanabe_1989},
where random measures $\eta_p$ are associated to point processes $p$. We summarise that background in Appendix \ref{sec-Point processes}. Here we present integration with respect to Poisson random measures and define random measures of class (QL).

Given a Poisson point processes $p$ with its domain denoted by $\dom(p)$ and the associated Poisson random measure $\eta_p$, 
we define the following three classes of random functions, see \cite{Ikeda+Watanabe_1989}, 
\begin{equation}
    \label{eqn-F_p}
\begin{aligned}
\mathbf{F}_p&\coloneqq \Bigl\{ f\colon [0,\infty)\times Y \times \Omega \to \mathbb{R}: f \mbox{ is } \mathbb{F}-\mbox{predictable and for  every } t>0, \notag
\\
&\hspace{4truecm}
\int_0^{t} \int_Y \vert f(s,y,\omega)\vert \, \eta_p(\ud s, \ud y)< \infty \mbox{ a.s.}
\Bigr\},
\end{aligned}
\end{equation}
where we put 
\begin{equation}\label{eqn-IW-3.6-3.7}
\int_0^{t} \int_Y |f(s,y,\omega)|  \, \eta_p(\ud s, \ud y) \coloneqq \sum_{ s \in (0,t] \cap \dom(p)} \vert f(s,p(s),\omega) \vert,\;\; t \geq 0.
\end{equation}
For $k=1,2$, with $\mu$ being the intensity measure of $\eta_p$, we define
\begin{equation}
    \label{eqn-F_p^i}
\begin{aligned}
\mathbf{F}_p^k&\coloneqq \Bigl\{ f\colon [0,\infty)\times Y \times \Omega \to \mathbb{R}: f \mbox{ is } \mathbb{F}-\mbox{predictable}, \notag \\
&\hspace{4truecm}
\mathbb{E}\int_0^{t} \int_Y \vert f(s,y,\omega)\vert^k \, \mu(\ud s, \ud y) < \infty, \;\;  t>0
\Bigr\}.
\end{aligned}
\end{equation}

For $f \in \mathbf{F}_p$ the integral on the LHS of \eqref{eqn-H.5} is well-defined: 
\begin{equation}\label{eqn-H.5}
\int_0^{t} \int_Y f(s,y,\omega) \, \eta_p(\ud s, \ud y)
= \sum_{ s \in (0,t] \cap \dom(p)} f(s,p(s),\omega),\;\; t \geq 0.
\end{equation}

One can show that 
\begin{equation}\label{eqn-F_p^1 subset F_p-1}
 \mathbb{E} \int_0^{t} \int_Y |f(s,y,\omega)|  \, \eta_p(\ud s, \ud y)     =
 \mathbb{E} \int_0^{t} \int_Y |f(s,y,\omega)|  \, \mu(\ud s, \ud y) .
\end{equation}

In order to define the integral with respect to the compensated Poisson random measure defined by  
\begin{equation}\label{eqn-CPRM-02}
\tilde{\eta}_p:=\eta_p-\mu,
\end{equation}
we first define it for a random function $f \in \mathbf{F}_p^1 \cap \mathbf{F}_p^2$ by
\begin{equation}\label{eqn-H.6}
\begin{aligned}
\int_0^{t} \int_Y f(s,y,\omega) \, \tilde{\eta}_p(\ud s, \ud y)
&:= \int_0^{t} \int_Y f(s,y,\omega) \, \eta_p(\ud s, \ud y) \\
&- \int_0^{t} \int_Y f(s,y,\omega) \, \mu(\ud s, \ud y),\;\; t \geq 0.
\end{aligned}
\end{equation}
Let us point out that in view of \eqref{eqn-F_p^1 subset F_p-1}, $\mathbf{F}_p^1 \subset \mathbf{F}_p$ and so both  integrals on the RHS of \eqref{eqn-H.6} above exist $\mathbb{P}$- a.s.
The integral defined by \eqref{eqn-H.6} can be naturally extended to the whole class  $\mathbf{F}_p^2$, see \cite{Ikeda+Watanabe_1989}, section II.3 and the resulting integral $\int_0^{t} \int_Y f(s,y,\omega) \, \tilde{\eta}_p(\ud s, \ud y)$ is  a square-integrable martingale. In particular, 
\begin{equation}\label{eqn-Ito isometry}
    \mathbb{E} \bigg\vert \int_0^{t} \int_Y f(s,y,\omega) \, \tilde{\eta}_p(\ud s, \ud y) \bigg\vert^2=
    \mathbb{E} \int_0^{t}   \int_Y \vert f(s,y,\omega) \vert^2\, \mu(\ud s, \ud y) , \;\; t\geq 0.
\end{equation}

We will also use  a more general object, namely random measures of class (QL), see Definition \ref{def-QL calss random measure} for more details.

\begin{definition}\label{def-class-QL}
Assume that  $(\rY,\ycal)$ be a measurable space.  A random measure \[\eta\colon \Omega \to M_{\overline{\mathbb{N}}}(\R_+ \times Y)\] is of (QL) class if and only if
\begin{enumerate}
\item[(i)] it is $\sigma$-finite i.e. there exists a sequence  $(U_n)_{n=1}^\infty$ of elements of the family  $\Gamma_\eta$ defined by 
 \begin{equation}
\Gamma_\eta := \{ A \in \mathscr{E} : \E \,\abs{\eta([0,t] \times A)} < \infty \mbox{ for all } t \geq 0 \},
\end{equation}
such that   $U_n \nearrow E$,
\item[(ii)]
there exists a $\sigma$-finite random measure 
\[\widehat{\eta}: \Omega \times \bigl( \mathscr{B}(\R_+) \otimes \ycal\bigr) \to \R,
\]
such that $\tilde{\eta} = \eta - \widehat{\eta}$ is a martingale random measure,
\item[(iii)] for any $A \in \Gamma_\eta$  the process 
the process 
\[ \R_+ \times \Omega  \ni (t,\omega) \mapsto \widehat{\eta}(\omega;[0,t]\times A)
\]
is continuous.
\end{enumerate}
\end{definition}

\subsection{Useful Sobolev spaces on  \texorpdfstring{$\mathbb{R}^d$}{Rd}} 
\label{subsec-preliminaries}

In this subsection we introduce several definitions and concepts that are related to the mathematical analysis of the Navier-Stokes equations.
Our results, with some little extra work, can be shown to be valid in every   domain with sufficiently smooth boundary. However, we have decided to avoid unnecessary complications and present all results
in the special case of the Euclidean spaces   $\mathbb{R}^d$.

By the symbol ${H}^{1}(\R^d ,\mathbb{R}^d )=\bH^{1}(\R^d)$
we will denote   the standard Sobolev space of all functions $u \in \mathbb{L}^2(\R^d)\coloneqq{L}^{2}(\R^d ,\mathbb{R}^d)$  whose weak derivatives $D_iu=\frac{\partial u}{\partial {x}_{i}}$ exist and belong to $ \mathbb{L}^2(\R^d)$, $i=1,\ldots,d$.
For $u,v \in  \bH^{1}(\R^d)$, we put
\begin{equation*} \label{eqn-semi inner product in H^1}
\ip{\nabla u}{\nabla v}{\bL^2}
\coloneqq\sum_{i=1}^{d}\int_{\R^d}\frac{\partial u}{\partial {x}_{i}} \cdot \frac{\partial v}{\partial {x}_{i}} \dx
=\sum_{i,j=1}^{d}\int_{\R^d}\frac{\partial u_j}{\partial {x}_{i}} \frac{\partial v_j}{\partial {x}_{i}} \dx,
\end{equation*}
where the dot $\cdot$ in the middle integral denotes the scalar product in $\R^d$.
The space $\bH^{1}(\R^d)$ endowed with  the following scalar product
\begin{equation*} \label{eqn-inner product in H^1}
   \ip{u}{v}{\bH^1} \coloneqq \ip{u}{v}{\bL^2} +  \ip{\nabla u}{\nabla v}{\bL^2}, \qquad u,v \in
   \bH^{1}(\R^d),
\end{equation*}
is a Hilbert space. 

The fractional order Sobolev spaces  $\mathbb{H}^\theta(\R^d)\coloneqq{H}^{\theta}(\R^d , \mathbb{R}^d )$, for any $\theta\in \mathbb{R}$, are defined by means of the Fourier transform, see \cite{Rudin_functional,Lions+Magenes_1972,McLean_2000}.
Note that for $\theta=1$ the two definitions yield the same Hilbert space.

Next, we denote by  ${\ccal }^{\infty }_{c} (\R^d, \mathbb{R}^d )$  the space of all ${\ccal }^{\infty }$-class  $\mathbb{R}^d $-valued functions  with compact support. We will also use the following spaces
\begin{align}\label{eqn-Vcal}
 \vcal &\coloneqq \{ u \in {\ccal }^{\infty }_{c} (\R^d, \mathbb{R}^d ) : \, \, \diver u= 0 \}  ,  \\
  \label{eqn-H}
  \rH &\coloneqq \mbox{the closure of $\vcal $ in $\mathbb{L}^2(\R^d)$} ,  \\
  \label{eqn-V}
  \rV &\coloneqq \mbox{the closure of $\vcal $ in ${\bH}^{1,2}(\R^d)$} .
\end{align}
Before we continue let us formulate the following fundamental property of the space $\vcal$, a property that  is the basis for the Helmholtz decomposition, i.e. 
\begin{equation}\label{eqn-Helmoholtz}
    \int_{\R^d} u(x)\cdot \nabla p(x) \dx =0 \mbox{ for all } u \in  \vcal, \; p\in {\ccal }^{\infty }_{c} (\R^d, \mathbb{R}).
\end{equation}

We endow  the space $\rH$ with the scalar product, and hence the norm, inherited from the space $\mathbb{L}^2(\R^d)$. We
denote them  respectively by $\ip{\cdot }{\cdot }{\rH}$ and $\normb{\cdot}{\rH}{}$, i.e.
\begin{align*}
\ip{u}{v}{\rH}\coloneqq \ip{u}{v}{{\mathbb{L}}^{2}} , \qquad
\normb{u}{\rH}{}  \coloneqq \normb{u}{{\mathbb{L}}^{2}}{} , \qquad u, v \in \rH.
\end{align*}
One can show that the above spaces can be characterised in a different way, see \cite[Theorem 1.1/1.4 and 1.1/1.6]{Temam_2001}:
\begin{align*}
  \rH &= \{ u \in \mathbb{L}^2: \divergence u=0\},
\\
  \rV &= \{ u \in \mathbb{H}^1: \divergence u=0\}.
\end{align*}

We endow the space $V$ with  the scalar product
inherited from ${\bH}^{1}(\R^d)$, i.e.
\begin{align*}
  \ip{u}{v}{\rV} \coloneqq \ip{u}{v}{\rH} + \ip{\nabla u}{\nabla v}{\bL^2} ,
\end{align*}
We will denote the  norm  induced by the scalar product $\ip{\cdot }{\cdot }{\rV}$  by $\normb{\cdot}{\rV}{}$, i.e.
\begin{equation}\label{eqn-norm_V}
  \normb{u}{\rV}{2} \coloneqq  \normb{u}{\rH}{2} + \normb{\nabla u}{\bL^2}{2}.
\end{equation}

We also  define the following  scale of Hilbert spaces
\begin{equation}\label{eqn-V_s}
  {\rV}_{\theta} \coloneqq \mbox{the closure of $\vcal $ in $\mathbb{H}^\theta(\mathbb{R}^d)$},  \;\;\theta\in [0,\infty).
\end{equation}
In particular, $\rV_0=\rH$. 
Using the method of proof of   \cite[Theorem\ 1.1]{Temam_2001} one can show that
\begin{align*}
  \rV_\theta &=  \mathbb{H}^\theta (\mathbb{R}^d) \cap \rH, \;\; \;\;\theta\in [0,\infty).
\end{align*}

From now on we choose and fix a real number $s$ such that
\begin{align}\label{eqn-number s}
s > \frac{d}{2} +1.
  \end{align}
   It  follows from the definitions \eqref{eqn-V} and \eqref{eqn-V_s} of the space $\rV$ and   the space $\rV_s$, respectively,  that
$\rV_s \subset \rV \subset \rH $ and 
\begin{align}\label{eqn-V_s embed V}
\mbox{ the embedding } \rV_s &\embed \rV \mbox{ is continuous},
  \\
  \label{eqn-V_s embed H}
\mbox{ the embedding } \rV_s &\embed \rH \mbox{ is continuous}.
  \end{align}

  Now, let  ${\ccal }_{b}(\R^d, \mathbb{R}^d )$ be the space of continuous and bounded $\mathbb{R}^d $-valued functions.  Then, since $s>\frac{d}{2} +1$ by the Sobolev Embedding Theorem we have
\begin{align}\label{eqn-H^s-1 to L^infty}
    {\bH}^{s-1,2}(\R^d) \embed  {\ccal }_{b}(\R^d, \mathbb{R}^d )
   \embed {\bL}^{\infty } (\R^d ).
   \end{align}
Next, it is well known, see for instance   \cite[Lemma C.1 ]{Brzezniak+Motyl_2013},  that
there exists a Hilbert space\label{page-U space} $\rU$, with the inner product denoted by $\ip{\cdot}{\cdot}{\rU}$ such that
\begin{equation} \label{eqn-U-V_s densely}
\rU \subset  {\rV}_{s} \mbox{   densely}
\end{equation}
and 
\begin{equation} \label{eqn-U_comp_V_s}
 \mbox{the natural embedding }\iota _{s}\colon \rU \embed  {\rV}_{s} \mbox{ is compact}.
\end{equation}

We will denote by
 \begin{equation} \label{eqn-i-U to H}
\iota \colon \rU \embed \rV     \mbox{ and } i\colon \rU \embed \rH
 \end{equation}
two other  natural embeddings. Hence, in view of \eqref{eqn-V_s embed V} we deduce that
\begin{equation} \label{eqn-U embed V}
\rU \embed \rV \embed  \rH \mbox{ continuously }
\end{equation}
with the the first embedding  being   compact.
Moreover, without loss of generality, we can suppose that  the embedding 
\begin{equation} \label{eqn-U to H is contraction}
\rU \embed   \rH \mbox{ is a contraction}.
\end{equation}
By \cite[Section\ 2.1]{Lions+Magenes_1972}, there exists an unbounded positive operator $L$ in $\rH$, with domain $D(L) \subset \rH$,  such that
\begin{equation}\label{eqn-DL^12}
D(L^{1/2})= \rU \mbox{ isometrically}.
\end{equation}
In particular, the following identities hold
\begin{align} \label{eqn-op_L_ip}
    \ip{Lu}{w}{\rH} &= \ip{u}{w}{\rU}, \qquad u \in D(L), \quad w \in \rU,
\\
\label{norm_in_U_and_H_by_L^1/2}
|u|_{\rU} &= |L^{1/2}u|_{\rH}, \qquad u \in \rU.
\end{align}

Let us observe that  for each element $u \in H $ there exists a unique functional ${u}^\ast \in \rU^{\prime }$ such that \begin{equation}\label{eqn-U'-H}
    \duality{{u}^{\ast }}{v}\coloneqq \ip{u}{v}{\rH} , \qquad v \in\rU.
\end{equation}
As usual, we identify the dual ${\rH}^\prime$ with ${\rH}$ and
in what follows we will identify $u\in H$ with $u^\ast\in \rU^\prime$ and therefore we  have ${\rH} \embed\rU^{\prime }$ as well as the following equality
\begin{equation}\label{eqn-U'-H_2}
    \duality{u}{v}= \ip{u}{v}{\rH} , \qquad u \in H, \;\;v \in\rU.
\end{equation}
Moreover, the dual space ${\rV}^\prime$ can be identified with a subspace of ${\rU}^\prime$ and hence we have the following generalized Gelfand triple
\begin{equation}\label{eqn-embeddings}
\rU \embed {\rV}_{s} \embed V \embed H \cong H^{\prime}
\embed  V^{\prime }  \embed {\rV}_{s}^{\prime} \embed U^{\prime}
\end{equation}
as well as  the following identity linking the duality paring $\ddual{\rV^\prime }{\cdot}{\cdot}{\rV}{}$ between spaces $\rV^\prime$ and $\rV$ with the
duality paring $\duality{\cdot}{\cdot}$ between spaces $\rU^\prime$ and $\rU$,
\begin{equation}\label{eqn-V'-U'}
    \ddual{\rV^\prime }{u}{v}{\rV}{}=\duality{u}{v}, \qquad u \in \rV, \;\;v \in\rU.
\end{equation}

In view of \eqref{eqn-V_s embed H}, we infer that 
the natural embedding $\rU \embed  \rH$ is compact and thus, by duality we deduce we that 
\begin{equation} \label{eqn-h-to-U'-compact}
 \mbox{the  embedding } \rH \embed  \rU^\prime \mbox{ is compact}.
\end{equation}

\begin{remark}\label{rem-h-to-U'-compact}
    It is the last property of the space $\rU$, as well as \eqref{eqn-comp_VR_HR},  that we need  in our paper, see the proof of Lemma \ref{lem-comp_Galerkin}.
\end{remark}

Finally, let us also observe that since  the embedding $i:\rU \embed  \rH$ is a contraction,  so is the map $i^\prime: {\rH} \embed\rU^{\prime }$.

We conclude this subsection with the introduction of two new spaces which will be needed in defining  solutions
to the stochastic NSEs.
\begin{definition}
    Let $B_R$ be the open ball in $\R^d$ centered at $0$ with radius $R>0$. We define the following spaces
\begin{equation}  \label{eqn-HR_VR}
\HR \coloneqq \{ u \in \rH: \supp u \subset B_R  \}
 \qquad \VR \coloneqq V \cap   \HR,
\end{equation}
with the inner products and norms inherited respectively from the spaces $\rH$ and $\rV$.
\end{definition}

Since obviously $\HR \subset \rH$ and $\VR \subset \rV$ are closed subspaces and 
\begin{align*}
  \HR &
  \subset \bigl\{ u \in \mathbb{L}^2(B_R): \divergence u =0  \bigr\}, \;\;\;
  \VR 
  \subset\bigl\{ u \in \mathbb{H}^1(B_R): \divergence u =0  \bigr\}.
\end{align*}
Hence, since the sets $B_R$ are bounded, for every $R>0$, by the Rellich-Kondratchov Theorem we infer that
\begin{equation} \label{eqn-comp_VR_HR}
\mbox{ the embedding $\VR \embed  \HR$ is compact. \rm }
\end{equation}

The symbols  $\HR^{\prime }$ and $\VR^{\prime }$ will stand for the corresponding dual spaces.

\subsection{The weak Stokes operator}
\label{subsec-Stokes}

In this subsection, we will introduce the (weak) Stokes operator $\acal$  and  the bilinear operator $B$.

Since by the  definition  \eqref{eqn-norm_V} of the norm in $\rV$, 
\begin{equation}
\label{estimate_by_norm_in_V}
  |\ip{\nabla u}{\nabla v}{\bL^2}| \le \normb{\nabla u}{\bL^2}{} \normb{\nabla v}{\bL^2}{}
   \leq
 \normb{\nabla u}{\bL^2}{} \normb{v}{\rV}{}, \quad u,v \in \rV,
\end{equation}
 by the Riesz Lemma    there exists  a unique linear bounded   operator  \[\acal: \rV \to \rV^\prime\] such that
\begin{equation}  \label{eqn-Acal_ilsk_Dir}
   \ddual{\rV^\prime }{\acal u}{v}{\rV}{} = \ip{\nabla u}{\nabla v}{\bL^2} , \qquad u,v \in \rV.
\end{equation}
One can also prove  that the operator $\acal$, called here for obvious reasons the weak Stokes operator,  is accretive.

\subsection{The non-linear map}
\label{subsec-non-linear map}

Next, let us consider the tri-linear form $b \colon V \times V \times V \to \R$ defined by
\begin{equation}\label{eqn-form_b}
b(u,w,v)
\coloneqq \int_{\R^d} \bigl( u \cdot \nabla w \bigr) \cdot v \dx
= \sum_{i,j=1}^d \int_{\R^d} u_i \frac{\partial w_j}{\partial x_i} v_j \dx.
\end{equation}
Because $\rV \subset \mathbb{L}^6$ by the Sobolev Embedding Theorem, see \cite{Adams},  the Lebesgue integral on the right-hand side of \eqref{eqn-form_b} exists and thus, by   the H\"{o}lder  and the Gagliardo-Ladyzhenskaya-Nirenberg inequalities, see  \cite[Lemma \ III.3.3 and III.3.5]{Temam_2001} we deduce  that there exists  a  generic positive constant $c$  such that the  following estimates hold
\begin{equation}
     \label{eqn-b_estimate_V}
\begin{aligned}
    |b(u,w,v )|
 \leq  \normb{u}{\bL^4}{} \normb{\nabla v}{\bL^2}{}\normb{v}{\bL^4}{}
&\leq 2 \normb{u}{\bL^2}{\frac14} \normb{\nabla u}{\bL^2}{\frac34}  \normb{\nabla w}{\bL^2}{}  \normb{v}{\bL^2}{\frac14} \normb{\nabla v}{\bL^2}{\frac34}
\\ &\leq  c \normb{u}{\rV}{} \normb{w}{\rV}{} \normb{v}{\rV}{} , \qquad u,w,v \in \rV  .
\end{aligned}
\end{equation}

In particular, we infer that  the  form $b$ is continuous on $\rV$.

Let us also recall the following fundamental properties of the form $b$, see \cite[Lemma \ II.1.3]{Temam_2001}, 
\begin{equation*}  \label{eqn-antisymmetry_b}
b(u,w, v ) =  - b(u,v ,w), \qquad u,w,v \in \rV
\end{equation*}
and
\begin{equation*}  \label{eqn-wirowosc_b}
b(u,v,v) =0   \qquad u,v \in \rV.
\end{equation*}
We also observe that  by inequality \eqref{eqn-b_estimate_V}  there exists a unique bounded, i.e. continuous,   bilinear map
\[
\rB\colon\rV\times \rV \to \rV^\prime\]
such that
\[    \ddual{\rV^\prime }{B(u,w)}{v}{\rV}{}\coloneqq b(u,w, v ), \mbox{ for all } u,w, v\in \rV.
\]
 With a slight abuse of notation, we will also denote  by $\rB$   the corresponding quadratic map
\begin{equation*}  \label{eqn-B quadratic}
\rB\colon \rV \ni u \mapsto \rB(u,u)  \in  \rV^{\prime } \in \mathbb{R}.
\end{equation*}

By \eqref{eqn-H^s-1 to L^infty},  there exists a  constant $c >0 $ such that for all $u,w \in V$, $v\in V_s$, 
\begin{align}\label{eqn-b-inequality}
 |b(u,w,v)|  =  |b(u,v,w)|
   &\leq   \normb{u}{H}{} \normb{w}{H}{} \normb{\nabla v}{{L}^{\infty }}{}
 \le  {c}_{} \normb{u}{H}{} \normb{w}{H}{} \normb{v}{{\rV}_{s}}{},
 \end{align}
and,  for all $u,v \in V$, $w\in V_s$, 
\begin{equation}\label{eqn-b-inequality-2}
    |b(u,w,v)|
   \leq    {c}_{} \normb{u}{H}{} \normb{w}{{\rV}_{s}}{} \normb{v}{H}{}.
\end{equation}
   
 Hence, the trilinear form  $b$ can be uniquely extended to a tri-linear forms, denoted by the same letter,
\begin{align*}
   b \colon \rH \times \rH \times {\rV}_{s} \to \mathbb{R} &\mbox{ and }
   b \colon \rH   \times {\rV}_{s} \times \rH \to \mathbb{R}
\end{align*}
satisfying \eqref{eqn-b-inequality}  and \eqref{eqn-b-inequality-2} and  the operator $\rB$ can be uniquely extended
to  bounded bilinear operators
\begin{align}\label{eqn-B-2}
    \rB \colon \rH \times \rH \to {\rV}_{s}^{\prime },
   &\mbox{ and }
\rB \colon \rH \times {\rV}_{s} \to  \rH ,
    \end{align}
i.e. satisfying  the following estimate
\begin{align} \label{eqn-estimate_B_ext}
 |\rB(u,w) {|}_{{\rV}_{s}^{\prime }} &\le c {|u|}_{\rH}  {|w|}_{\rH} ,\qquad u,w \in \rH,
 \\
 |\rB(u,w) {|}_{\rH} &\le c {|u|}_{\rH}  {|w|}_{{\rV}_{s}} ,\qquad u \in \rH, w \in {\rV}_{s}.
\end{align}

We conclude this subsection in recalling properties of the form $b$ with respect to the spaces
$\HR$ introduced at the end of the previous subsection.

Analogously to \eqref{eqn-b-inequality} we  can show that there exists $c>0$ such that

\begin{equation}\label{eqn-b_R}
|b(u,w,v)| \le  {c}_{} \normb{u}{\HR}{} \normb{w}{\HR}{} \normb{v}{{\rV}_{s}}{}, \mbox{ for } u,w \in {\rH} \mbox{ and } v \in {\rV}_{s}: \supp(v) \subset B_R.
\end{equation}
Note that the above inequality follows from  inequality \eqref{eqn-b-inequality}. It is essential in the identification of the limit.

\subsection{Approximation of  maps}
\label{subsec-approcximation of maps}
Let us recall, see  Section \ref{subsec-preliminaries},  that $L$ is a positive self-adjoint unbounded operator  in $\rH$ such that   ${L}^{-1}$ is compact and
 $ D(L^{1/2})=\rU$ isometrically, so that  in particular 
 \begin{align}\label{eqn-U emebed V}
 \rU \subset  \rV \mbox{ continuously and densely, }
 \end{align}
 and the  identity \eqref{eqn-op_L_ip} holds.
  Therefore, there exists
an  orthonormal basis  $\{ {e}_{i} {\} }_{i =1}^{\infty  }$  of ${\rH}$ consisting  of eigenvectors of the operator $L$. The corresponding eigenvalues will be denoted by $\bigl(\lambda_i\bigr)_{i=1}^\infty$.
 Notice that ${e}_{i} \in\rU $, $i \in \mathbb{N} $, because $D(L) \subset D(L^{\frac12}) =\rU$.
For any  $n \in \mathbb{N} $  let
\[\rU_n\coloneqq\Span \{ {e}_{1},..., {e}_{n} \}
\]
be the subspace of $\rU$ spanned by the vectors ${e}_{1},...,{e}_{n}$.
We endow the space $\rU_n$ with the norm from $\rU$. Let ${\rH}_{n}$  be the space $U_{n}$ endowed with the  norm inherited from ${\rH}$. Obviously,  ${\rH}_{n}$ is a subspace of $\rU$, and hence of $\rV$.

Let ${P}_{n} $ be an   operator  defined by
\begin{equation} \label{eqn-P_n in U'}
  {P}_{n} : \rU^{\prime } \ni  {u}^\ast \mapsto  \sum_{i=1}^{n}  \ddual{\rU^{\prime }}{{u}^{\ast }}{{e}_{i}}{\rU} {e}_{i} \in \rU.
\end{equation}
Let us observe that the range of ${P}_{n}$ is  equal to $\rU_n$ and,
in view of \eqref{eqn-U'-H_2}, that
\begin{equation} \label{eqn-P_n in H}
  {P}_{n} u = \sum_{i=1}^{n} \ip{ u}{ {e}_{i}}{\rH} {e}_{i}, \qquad u \in \rH .
\end{equation}
 Therefore   the  restriction of the operator ${P}_{n} $ to the space ${\rH}$,
   is equal to the  ${\rH}$-orthogonal projection from ${\rH}$
onto ${\rH}_n$.  For the sake of simplicity, this restriction of ${P}_{n} $  will also be denoted by ${P}_{n} $. In particular, all the operators $(P_n)_{n=1}^\infty$ have unit norm, i.e. 
\begin{equation}\label{eqn-P_n contraction on H}
\Vert {P}_n\Vert_{\mathscr{L}(\rH)} = 1, \qquad n \in \mathbb{N}.
\end{equation}
The  restriction  of the operator ${P}_{n} $ to the space ${\rU}$ will be denoted by
\[\tilde{P}_n\coloneqq P_n\big\vert_\rU\colon  \, U\to U.  \]

Let us present a list of additional properties of the operators $P_n$.

\begin{lemma}\label{lem_P_n} Assume that
$n \in \mathbb{N}$. Then
\begin{trivlist}
\item[(i)] The map $\tilde{P}_n$ is an   ${\rU}$-orthogonal projection in  ${\rU}$ (with range equal to ${\rU}_n$); in particular $\Vert\tilde{P}_n\Vert_{\mathscr{L}(U)}=1$.
\item[(ii)] For all $ u,v  \in \rU$ and $ {u}^\ast \in\rU^{\prime }$
\begin{align}  
  \duality{P_n u^\ast }{ v} &= \ip{{P}_{n} {u}^\ast}{v}{\rH} =\duality{u^\ast }{\tilde{P}_n v} ,
\notag
\end{align}
\item[(iii)]  $\Vert {P}_n\Vert_{\mathscr{L}(\rU^\prime)}=1$.
\end{trivlist}
Moreover,
\begin{trivlist}
\item[(iv)] For all
 $ u  \in \rU$ and $ {u}^\ast \in\rU^{\prime }$,
 \begin{align}
 &\lim_{n \to \infty }  \left\vert{P}_{n} u^\ast -u^\ast \right\vert_{\rU^\prime}=0,
\notag  \\
  \label{eqn-P_n_convergence}
 &\lim_{n \to \infty }  \left\vert\tilde{P}_{n} u -u\right\vert_\rU =0.
\end{align}
\end{trivlist}
\end{lemma}

\begin{proof}
Part (i).  Put $\tilde{e}_i\coloneqq\frac{ {e}_{i}}{\sqrt{\lambda_i}}=L^{-\frac12}{e}_{i}$, $i \in \mathbb{N}$. Observe that since $\rU=D(L^{\frac12})$ isometrically the sequence $(\tilde{e}_i)$
is an  orthonormal basis of ${\rU}$, cf.\ \eqref{norm_in_U_and_H_by_L^1/2}.
Thus, by using \eqref{eqn-op_L_ip} we have for  u $\in \rU$
\begin{equation*}
 \ip{ u}{ {e}_{i}}{\rH} {e}_{i}   = \Ip{ u}{ \lambda_i \frac{ {e}_{i}}{\sqrt{\lambda_i}}   }{\rH} \frac{ {e}_{i}}{\sqrt{\lambda_i}}
  = \Ip{ u}{ L \frac{ {e}_{i}}{\sqrt{\lambda_i}}  }{\rH} \frac{ {e}_{i}}{\sqrt{\lambda_i}}   = \Ip{ u}{ \frac{ {e}_{i}}{\sqrt{\lambda_i}}  }{\rU} \frac{ {e}_{i}}{\sqrt{\lambda_i}},
\end{equation*}
which along with  \eqref{eqn-P_n in H} imply
\begin{equation*} \label{eqn-P_n in U}
  \tilde{P}_{n} u = \sum_{i=1}^{n} \ip{ u}{ \tilde{e}_i }{\rU} \tilde{e}_i, \qquad u \in \rU .
\end{equation*}
This last equation easily implies that $\tilde{P}_n$ is  an   ${\rU}$-orthogonal projection in  ${\rU}$  and that
$ \sup_{n}\Vert \tilde{P}_n\Vert_{\mathscr{L}(\rU)}\leq 1$. Thus, the proof of  part (i) is complete.

Part (ii).  This equality follows from identity \eqref{eqn-U'-H}.

Part (iii).  According to part (i), $\tilde{P}_n$ is an orthogonal projection in $\rU$ and hence its operator norm is equal to $1$. Then the dual map $\tilde{P}_n^\prime\colon  \rU^\prime \to \rU^\prime$ also has norm $1$ in $\mathscr{L}(U')$. Therefore, it suffices to prove that
the map $P_n$ considered as a map from $\rU^\prime$ to $\rU^\prime$, denoted for the time being by $Q_n$,  is equal to $\tilde{P}_n^\prime$. That is, we need to show that
\begin{align*}
\duality{Q_n\phi}{x}=\duality{\phi}{\tilde{P}_nx}, \;\;\mbox{ for all }\phi \in \rU^\prime,\, x \in \rU.
\end{align*}
For this purpose we first observe that
\begin{equation}\label{eqn-Q_n}
Q_n=i^\prime \circ \mathrm{R} \circ i \circ P_n,
\end{equation}
where $P_n\colon \rU^\prime \to \rU$ is defined by \eqref{eqn-P_n in U'}, $i\colon \rU \embed \rH$ is the natural embedding introduced in \eqref{eqn-i-U to H}, $\mathrm{R}\colon \rH\to \rH^\prime$ is the isometric isomorphism derived from the Riesz Lemma and $i^\prime\colon \rH^\prime \to \rU^\prime$ is the dual map of $i$.

Next, let us choose and fix $\phi \in \rU^\prime$ and $ x \in \rU$. Then  we have the following chain of equalities
\begin{align*}
\duality{\phi}{\tilde{P}_nx}
&=\lb{\phi},\sum_{j=1}^{n} \ip{ x}{ \tilde{e}_j }{\rU} \tilde{e}_j \rbb \\
&=\lb{\phi},\sum_{j=1}^{n} \ip{ x}{ {e}_j }{\rH} {e}_j \rbb
=\sum_{j=1}^{n} \ip{ x}{ {e}_j }{\rH} \;\duality{\phi}{ {e}_j} \\
&= \sum_{j=1}^{n} \duality{\phi}{ {e}_j}\; \ip{ i{e}_j }{ ix}{\rH}
= \bigg( i\Big( \sum_{j=1}^{n} \duality{\phi}{ {e}_j} {e}_j \Big) , ix \bigg)_{\rH}
\\ & =
\ip{ i(P_n \phi)}{ix}{\rH}=\ddual{\rH^\prime}{\mathrm{R}i(P_n \phi)}{ix}{\rH}\\
& =\duality{i^\prime \mathrm{R}i(P_n \phi)}{x}
=\duality{Q_n \phi}{x}.
\end{align*}
The last line completes the proof of the part (iii).

Part (iv). The second convergence   was proved in  \cite[Lemma\ 2.3 and 2.4]{Brzezniak+Motyl_2013}. In order to prove the first one
let us choose and fix $u^\ast\in U^\prime$. By the density of the embedding $\rU \subset U^\prime$ there exists $(u_k)_{k\in \mathbb{N}}\subset \rU$ such that $u_k \to u^\ast$ in $U^\prime$. From part (iii) we see that
\begin{align*}
 \lvert P_n u^\ast - u^\ast\rvert_{U^\prime} &\leq \lVert P_n\rVert_{\mathcal{L}(U^\prime)} \lvert u^\ast- u_k \rvert_{U^\prime} + \lvert P_n u_k  -  u_k \rvert_{\rU^\prime} + \lvert u_k -u^\ast\rvert_{U^\prime}\\
 &\le  2 \lvert u_k -u^\ast\rvert_{U^\prime} + C\lvert P_nu_k -u_k \rvert_{\rU}.
\end{align*}
From assertion \eqref{eqn-P_n_convergence} we infer that $\lim_{n \to \infty }  \left\vert {P}_{n} u^\ast -u^\ast \right\vert_{\rU^\prime}=0$. This completes the proof of part (iv) and the lemma.
\end{proof}

\begin{remark}\label{rem-P_n}
\begin{itemize}
\item[(i)] 
Although by Lemma \ref{lem_P_n} the operators $(P_n)_{n=1}^\infty$ are contractions and hence uniformly bounded in  $\rU$, $\rH$  and  $\rU^\prime$,
we do not know whether these operators  are uniformly bounded in the spaces  $\rV$or $\rV_s$. In particular, we do not know whether the following  are true:
\begin{align*}
\lim_{n \to \infty }  \normb{{P}_{n} u -u }{\rV}{} =0, \;\; u\in \rV,
\\
\lim_{n \to \infty }  \normb{{P}_{n} u -u }{\rV_s}{} =0, \;\; u\in \rV_s.
\end{align*}
Of course, since both $\rV \cup \rV_s \subset \rU^\prime$, the expressions of the LHSs above make sense.\\
\item[(ii)]  On the other hand, it follows from part (iv) of Lemma \ref{lem_P_n}  and \eqref{eqn-U embed V} that  $ u  \in \rU$,  
$\tilde{P}_{n} u \to u$ in $\rV$.
\end{itemize}
\end{remark}

Now, let us choose and fix   a sequence  $({\theta }_{n })_{n=1}^\infty$ of ${\ccal }^{\infty }$-class functions ${\theta }_{n } \colon  \mathbb{R}_+ \to [0,1]$,    such that
\begin{equation}\label{eqn-theta_n}
\1_{[0,n]} \leq {\theta }_{n} \leq  \1_{[0,n+1]},\;\; n \in \mathbb{N}.
\end{equation}

Next, for every $n\in \mathbb{N}$ we define the following nonlinear functions
\begin{align}\label{eqn-B_n}
\rB_n \colon  {\rH}_{n} \times {\rH}_{n} \ni (u,v)& \mapsto {\theta }_{n}(|u {|}_{\rU^\prime})  {P}_{n} \bigl( \rB(
u,v)\bigr) \in  {\rH}_{n}
\end{align}
where  ${P}_{n}\colon \rH \to \rH_n$. In view of \eqref{eqn-B-2} and since $\rH_n \subset \rV_s$, the map $B_n$ is of $\ccal^\infty$-class. Note that contrary to the map $\rB$, the function  $\rB_n$ is no longer bilinear.

With an intentional abuse of notation we will denote by $\rB_n$ the map
\begin{equation}\label{eqn-B_n_quadratic}
B_n\colon {\rH}_n \ni u \mapsto {\rB}_{n} (u,u) \in {\rH}_n.
\end{equation}
Obviously, this map is also of $\ccal^\infty$-class and Lipschitz on balls.

We also define the following linear maps
\begin{align}
\label{eqn-A_n}
\acal_n\colon {\rH}_n \ni u   &\mapsto  ({P}_{n} \circ \iota^\prime)   \acal u  \in {\rH}_n,
\end{align}
where  now ${P}_{n}\colon U^\prime \to H_n$,
 $\iota^\prime: \rV^\prime \embed \rU^\prime$ is the dual of the map
 $\iota: \rU \embed \rV$ introduced in \eqref{eqn-i-U to H}.
Hence,  $\acal_{n}$ is   well defined and  the map $\acal_n\colon {\rH}_n \to {\rH}_n$ is linear and bounded.

Let us also note the following useful identities involving the map operators ${\rB}_{n}$ and $\acal_{n}$ that will be used later in the paper. If $u,v,w \in H_{n}$, then

\begin{align}
\label{eqn-B_n_b}
\ip{\rB_n(u,v)}{w}{H_n}
&=\ip{\rB_n(u,v)}{w}{\rH}
=\ddual{\rV^\prime }{\rB_n(u,v)}{w}{\rV}{}=\duality{\rB_n(u,v)}{w}
\\
&={\theta }_{n}(|u {|}_{\rU^\prime}) \duality{P_n \rB(
u,v)}{w}={\theta }_{n}(|u {|}_{\rU^\prime}) \duality{ \rB(
u,v)}{\tilde{P}_n w}
\nonumber\\
&={\theta }_{n}(|u {|}_{\rU^\prime}) \duality{ \rB( 
u,v)}{ w}
={\theta }_{n}(|u {|}_{\rU^\prime})b(u,v,w).
\nonumber
\end{align}

Similarly, we have  the following useful identities involving the linear maps  $\acal_{n}$.  If $u,w \in H_{n}$, then

\begin{align}
\nonumber
    \ip{\acal_n u}{w}{H_n}&=\ip{\acal_n u}{w}{\rH}= \ddual{\rV^\prime }{\acal_n u}{w}{\rV}{}=\duality{ \acal_n u}{w}{}=\duality{ P_n \circ i^\prime \acal u}{w}{}
\\
= &\duality{  i^\prime(\acal u)}{\tilde{P}_n w}{}=\duality{  \acal u}{w}{} =\ip{\nabla u}{\nabla w}{\bL^2}.
\label{eqn-A_n_duality}
\end{align}

\section{Statement of the problem} \label{sec-statement}

With the notation introduced in the previous section we can rewrite  Problem \eqref{eqn-reduced_Navier_Stokes} into the following abstract stochastic evolution equations
\begin{equation}
\begin{aligned}
\ud u(t) & +\bigl( \acal u(t) + \rB \bigl( u(t) \bigr) \bigr) \dt
\\&= f(t) \dt 
+ \int_{Y} F(t,u(t),y) \, \tilde{\eta } (\ud t,\ud y), \;\; t \in [0,T]  \\
u(0) &= {u}_{0} .
\end{aligned}
\label{eqn-SNSEs}
\end{equation}
Before we give a precise definition of a solution to equation \eqref{eqn-SNSEs} let us emphasize 
that Equation \eqref{eqn-SNSEs} is derived from \eqref{eqn-reduced_Navier_Stokes} by applying the Leray-Helmholtz projection $P$ to the latter. Since 
the gradient vector fields belong to the kernel of $P$, the  term  $\nabla p$ disappears. Let us present an alternative, but based on the same principles, 
explanation how  one can deduce equation \eqref{eqn-SNSEs} from \eqref{eqn-reduced_Navier_Stokes}. Multiplying the latter equation  by an arbitrary test function $\phi \in \mathscr{V}$ and employing the identities \eqref{eqn-Acal_ilsk_Dir}, \eqref{eqn-form_b}  and  \eqref{eqn-Helmoholtz}, we arrive at the following It\^o integral equation 
\begin{equation}
\label{eqn-SNSEs-weak}
\begin{aligned}
\duality{u(t)}{\phi} & + \int_0^t  \duality{\acal u(s)}{\phi}\ds    + b( u(s), u(s), \phi) \ds
\duality{u_0}{\phi}
\\& + \int_0^t  \duality{f(s)}{\phi} \ds 
+ \int_{Y} \duality{F(s,u(s),y)}{\phi} \, \tilde{\eta } (\ud s,\ud y), \;\; t \geq 0. 
\end{aligned}
\end{equation}

The above is the weak form of equation \eqref{eqn-SNSEs}.

We are going now to present the space in which the solution will live and a list of assumptions.

\subsection{Assumptions} \label{subsec-assumptions}

The following spaces have been introduced in \cite{Brzezniak+Motyl_2013}, \cite{Motyl_2013}. They are
similar  to some spaces  considered in \cite{Mikulevicius+Rozovskii_2005}.

\begin{itemize}
\item ${L}_{w}^{2}(0,T;V)$ is the space ${L}^{2}(0,T;V)$ endowed with the weak topology which will be denoted by $\tcal_1$.
\item ${L}^{2}(0,T;{\rH}_{\loc})$ is the space of (equivalence classes of) Borel measurable functions $u\colon[0,T] \to H$ such that for every $R \in \mathbb{N}$,
\begin{align*}
{q}_{T,R}^{}(u)
&\coloneqq\normb{u}{{L}^{2}(0,T;\HR)}{}
\coloneqq \biggl( \int_{0}^{T}\int_{B_R}|u(t,x){|}^{2}\dx\dt {\biggr) }^{\frac{1}{2}} \\
&= \biggl( \int_{0}^{T} \normb{u(t)}{{L}^{2}(B_R)}{2} \dt {\biggr) }^{\frac{1}{2}} <\infty,
\end{align*}
where  the spaces  $\HR$ have been defined in  \eqref{eqn-HR_VR}.
The topology generated by the seminorm $({q}_{T,R}{)}_{R\in \mathbb{N}}$ will be denoted by $\tcal_2 $. Let us note ${L}^{2}(0,T;{\rH}_{\loc})$ is not complete, and hence neither  a  Fr\'{e}chet, space but  ${L}^{2}(0,T;{\rH}_{\loc}) \cap {L}_{w}^{2}(0,T;H)$ is a   Fr\'{e}chet space.
\item $\mathbb{D} ([0,T];\rU^{\prime })$ is the space of c\`adl\`ag functions $u\colon[0,T] \to\rU^{\prime }$ endowed with the Skorohod topology which will be denoted by $\tcal_3$.
\item $\mathbb{D} ([0,T];{\rH}_{w})$ is the space of c\`adl\`ag functions $u\colon[0,T] \to {\rH}_{w}$, where the space ${\rH}_{w}$ is equal to the space ${\rH}$ but endowed with its weak topology. The Skorokhod topology on $\mathbb{D} ([0,T];{\rH}_{w})$ is denoted by $\tcal_4$.
\end{itemize}
Let us note, as it can be easily proved,  that ${L}^{2}(0,T;{\rH}_{\loc})$ is a separable Fr{\'e}chet space, see \cite[Def.\ 1.8]{Rudin_functional}.

We now introduce
the following topological vector space
\begin{equation} \label{eqn-Z_cadlag_NS}
 \zcal_T  : = {L}_{w}^{2}(0,T;V) \cap {L}^{2}(0,T;{\rH}_{\loc}) \cap \mathbb{D} ([0,T];U^\prime)
\cap \mathbb{D} ([0,T];{\rH}_{w})
\end{equation}
whose topology $\tcal$ is the topology generated by the four topologies $\tcal_i$, $i=1,2,3,4$.

\begin{assum-P}
\label{item-P_Y_is_standard} $(Y, \ycal )$ is a standard measurable space, that is, it is Borel-isomorphic to $\N$ or $\{0,1\}^\N$,  see \cite[Definition\ I.3.3]{Ikeda+Watanabe_1989},  and $\nu$ is a $\sigma$-finite measure on $(Y, \ycal )$.
\end{assum-P}
\begin{remark}\label{rem-Blackwell space}
We believe that Assumption P.\ref{item-P_Y_is_standard} can be relaxed by assuming that $(Y, \ycal )$ is a Blackwell space, see \cite[Section\ II.1]{Jacod_Shiryaev} or \cite[Section\ 6.10(iv)]{Bogachev_II} for the definition and properties.
\end{remark}

\begin{assum-F}\label{assum-F1}
The map $F\colon [0,T]\times \rU^\prime  \times Y \to \rU^\prime $
is  a measurable function.
\end{assum-F}

\begin{assum-F}\label{assum-F2}
The   restriction
$F\colon [0,T]\times H \times Y \to H $ of the function $F$
is a well defined and   a measurable function.
There exist  constants $L>0$ and  $C_2>0$ such that
\begin{align}
    \int_{Y} \normb{F(t,{u}_{1},y)- F(t,{u}_{2},y)}{\rH}{2} \, \nu (\ud y) &\leq  L \normb{u_{1}-{u}_{2}}{\rH}{2}
   , \; {u}_{1}, \! {u}_{2} \in H, \;t \in [0,T],  \label{eqn-F_Lipschitz_cond}
\\
\label{eqn-F_linear_growth}
    \int_{Y} \normb{F(t,u,y) }{\rH}{2} \, \nu (\ud y) &\leq  {C}_{2} (1 + \normb{u }{\rH}{2}), \;  u \in H , \; t \in [0,T],
\end{align}
\end{assum-F}

\begin{assum-F}\label{assum-F3}
Moreover, for every $\phi \in \vcal $ the following Nemytski map ${\tilde{F}}_{\phi}$ associated to $F$ and  defined by
\begin{equation*} \label{eqn-F**}
\begin{aligned}
{\tilde{F}}_{\phi} \colon    \zcal_T    \ni u &\mapsto \bigl\{  [0,T] \times Y \ni (t,y) \mapsto  \ip{F(t,u({t}),y)}{\phi}{\rH} \bigr\}
  \\ &\in {L}^{2}([0,T]\times Y, \Leb\otimes \nu ; \mathbb{R} )
     \end{aligned}
\end{equation*}
is continuous.
\end{assum-F}

\begin{assum-F}\label{assum-F4}
The following Nemytski map $\textbf{F}$ associated to map $F$  defined by
\begin{align*}
\textbf{F}\colon  \zcal_T
\ni u \mapsto  \bigl\{  [0,T] \times Y \ni (t,y) & \mapsto      F(t,u({t}),y) \in\rU^\prime \bigr\} \\
&\in {L}^{2}([0,T]\times Y,\Leb\otimes \nu, U^\prime )
      \quad
      \nonumber
\end{align*}
is continuous.
\end{assum-F}

\begin{assum-F}\label{assum-F5}
There exist  a number  $\gamma >0$  and   a constant ${C}_{\gamma}>0$ such that
\begin{align}
\label{eqn-F_linear_growth-gamma}
    \int_{Y} \normb{F(t,u,y) }{\rH}{8+2\gamma} \, \nu (\ud y) &\leq  {C}_{\gamma} (1 + \normb{u }{\rH}{8+2\gamma}), \;  u \in H , \; t \in [0,T],
\end{align}
\end{assum-F}

\begin{assum-A}
\label{assum-A.1} ${u}_{0} \in {\rH}$, 
$f \in {L}^{4+\gamma }([0,T];\rV^\prime)$,  
where $\gamma  $ is the number from Assumption F\ref{assum-F5}. 
\end{assum-A}

\begin{remark} If $Y$ is a separable Banach space and $\nu$ is a L\'evy measure on $Y$, then   $\nu(\{0\})=0$, and therefore
\[\int_{Y} \1_{\{ 0\}} (F(t,x,y))  \, \nu (\ud y) =0  \mbox{ for all }x \in H  \mbox{ and }t \in [0,T].\]
Hence, in this case, the corresponding part of Assumption F.\ref{assum-F2} in \cite{Motyl_2013} is satisfied. \\
  Let us observe that by the H{\"o}lder inequality,  conditions \eqref{eqn-F_linear_growth}  and  \eqref{eqn-F_linear_growth-gamma} imply that
for every $p \in [2,8+2\gamma]$ there exists a constant ${C}_{p}$ such that
\begin{align}
\label{eqn-F_linear_growth-p}
    \int_{Y} \normb{F(t,u,y) }{\rH}{p} \, \nu (\ud y) &\leq  {C}_{p} (1 + \normb{u }{\rH}{p}), \;  u \in H , \; t \in [0,T].
\end{align}
To see this, note that for any $p\in (2, 8+2\gamma)$ there exist  $C_p>0$ and $\theta\in (0,1)$ such that $1/p= (1-\theta)/2+ \theta/(8+2\gamma)$ and
\begin{equation}
    \begin{aligned}
     \int_{Y} \normb{F(t,u,y) }{\rH}{p} \, \nu (\ud y) &\leq  C_p  \left( \int_{Y} \normb{F(t,u,y) }{\rH}{2} \, \nu (\ud y) \right)^{\frac{p(1-\theta)}{2}} 
     \\
    & \left( \int_{Y} \normb{F(t,u,y) }{\rH}{8+2\gamma} \, \nu (\ud y) \right)^{\frac{p\theta}{8+2\gamma}}.
\end{aligned}
\end{equation}
By applying \eqref{eqn-F_linear_growth}, \eqref{eqn-F_linear_growth-gamma} and the following classical inequality \begin{equation*}  a_1^r+a_2^r \le (  a_1+ a_2)^r, a_1, a_2\ge 0, r\ge 1, \end{equation*}  we infer  that
\begin{align*}
      \int_{Y} \normb{F(t,u,y) }{\rH}{p} \, \nu (\ud y) &\leq  C_p (1+ \lvert u\rvert^2_\rH)^{\frac{p(1-\theta)}{2}}  (1+ \lvert u\rvert^{8+2\gamma}_\rH)^{\frac{p\theta}{8+2\gamma}} \\
      & \le C_p (1+ \lvert u\rvert_\rH)^{p(1-\theta)}   (1+ \lvert u\rvert_\rH)^{p\theta} 
       \le C_p (1+ \lvert u \rvert^p_{\rH}).
\end{align*}
\end{remark}
We now state an  important remark  on the Assumptions F.\ref{assum-F3} and F.\ref{assum-F4}.
\begin{remark}
Assumptions F.\ref{assum-F3} and F.\ref{assum-F4} look very similar. However, let us point out that
\begin{trivlist}
  \item[(i)]  Assumption F.\ref{assum-F3} is the continuity of some map with values in $H_w$;
  \item[(ii)] Assumption F.\ref{assum-F4} is the continuity of some map with values in $U^\prime$;
\item[(iii)]  $H_w$ may not be continuously embedded into $U^\prime$ and therefore Assumption  F.\ref{assum-F3} does not  in general imply F.\ref{assum-F4}.
\end{trivlist}
\end{remark}

Before we continue  let us formulate a result that will be useful in examples of maps  satisfying Assumptions F.\ref{assum-F1}-F.\ref{assum-F5}.

\begin{lemma}\label{lem-new auxiliary} If $u_n \to u$ in $\mathbb{D} ([0,T];\rU^\prime) $ then
$\int_0^T  \vert u_n(t)-u(t) \vert_{\rU^\prime}^2 \dt \to 0$.
\end{lemma}
\begin{proof}
Note that if $u_n \to u$ in $\mathbb{D} ([0,T];\rU^\prime) $, then $u_n(t) \to u(t)$ in $\rU^\prime$ for  Leb-a.a. $t \in [0,T]$. Observe also that by \cite[Lemma\ 2]{Brzezniak+Hornung+Manna}
the assumption implies that
\begin{equation*}
	\sup_{n \in \mathbb{N}} \sup_{\tinOT} \lvert u_n (t)\rvert_{\rU^\prime}^2<\infty.
\end{equation*}
Hence, by the Lebesgue Dominated Convergence Theorem we infer that
\begin{equation*}\int_0^T  \vert u_n(t)-u_n(t) \vert_{\rU^\prime}^2 \dt \to 0.\end{equation*}
\end{proof}

\begin{example}\label{example-F is linear}
A particularly important example of the noise coefficient $F$ is when it is linear with respect to $u$. We can then write
\begin{equation}
\label{eqn-F is linear}
F(t,u,y)=G(t,y)u, \mbox{ for appropriate } (t,u,y).
\end{equation}
In this framework, our Assumptions take the following simpler form.
\begin{trivlist}
\item[(G1-2)]For every
$(t,y) \in [0,T] \times Y$,
$G(t,y)$ is a bounded linear map from $U^\prime$ to $U^\prime$ and from $H$ to $H$. Moreover,  the associated maps, denoted by the same symbols,
\[
 G \colon [0,T]\times  Y \to  \mathscr{L}(U^\prime) \cap  \mathscr{L}(H)
\]
are strongly  measurable, i.e. for all $u^\prime \in U^\prime$ and  $u \in H$, the maps
\[
G(\cdot,\cdot)u^\prime\colon [0,T]\times  Y \ni (t,y) \mapsto  G(t,y)u^\prime \in U^\prime
\]
and
\[
G(\cdot,\cdot)u \colon [0,T]\times  Y \ni (t,y) \mapsto  G(t,y)u \in H
\]
are measurable.
\item[(G3)]
There exists  a constant $C_3>0$ such that
\begin{align}
\label{eqn-G_linear_growth}
    \int_{Y} \normb{G(t,y)u }{\rH}{2} \, \nu (\ud y) &\leq  {C}_{3}  \normb{u }{\rH}{2}, \;  u \in H , \; t \in [0,T].
\end{align}
\item[(G4)] For every  $\phi \in \vcal $ the associated Nemytski map ${\tilde{G}}_{\phi}$ defined by
\begin{equation*} \label{eqn-G**}
\begin{aligned}
{\tilde{G}}_{\phi}& \colon    \zcal_T   \ni u \mapsto \bigl\{  [0,T] \times Y \ni (t,y) \mapsto  \ip{G(t,y)u(t)}{\phi}{\rH} \bigr\}
  \\ &\hspace{2truecm} \in {L}^{2}([0,T]\times Y, \Leb\otimes \nu ; \mathbb{R} )
     \end{aligned}
\end{equation*}
is continuous.
\item[(G5)] The linear map
\begin{align*}
\textbf{G}\colon  \zcal_T
\ni u \mapsto  \bigl\{  [0,T] \times Y \ni (t,y) & \mapsto      G(t,y)u(t) \in\rU^\prime \bigr\} \\
&\in {L}^{2}([0,T]\times Y,\Leb\otimes \nu, U^\prime )
      \quad
      \nonumber
\end{align*}
is continuous.
\item[(G6)] There exist  a number  $\gamma >0$  and   a constant ${C}_{\gamma}>0$ such that
\begin{align}
\label{eqn-G_linear_growth-gamma}
    \int_{Y} \normb{G(t,y)u }{\rH}{8+2\gamma} \, \nu (\ud y) &\leq  {C}_{\gamma}  \normb{u }{\rH}{8+2\gamma}, \;  u \in H , \; t \in [0,T].
\end{align}
\end{trivlist}
\begin{remark}\label{rem-G4-G5 follow from G1-2?}A natural question arises whether Assumptions (G4) and (G5) follow from (G1-2).
The simplest example of function $G$ is the following one
\[
G(t,y)= \mathbf{G}, \mbox{ for all } (t,y) \in [0,T] \times Y,
\]
where $\mathbf{G}\in \mathscr{L}(U^\prime) \cap  \mathscr{L}(H)$.\\
If additionally   $\mathbf{G}$ maps $\vcal$ into itself, then Assumption (G4) follows from Assumptions (G1-2). Indeed, we have
\[
\ip{G(t,y)u(t)}{\phi}{\rH} =\ip{u(t)}{G(t,y)^\ast\phi}{\rH}, \mbox{ for all } (t,y) \in [0,T] \times Y,\;\; \phi \in  \vcal.
\]
\end{remark}

We also have the following result which implies that condition \eqref{eqn-G_linear_growth-U'} is sufficient for assumption (\textbf{G}.5) to hold.
\begin{lemma}\label{lem-G-aux}
Assume that the map $G$ satisfies assumptions (G1-2) and the following modification of (G3):
\begin{trivlist}
\item[(G3)']
There exists  a constant $C_3^\prime>0$ such that
\begin{align}
\label{eqn-G_linear_growth-U'}
    \int_{Y} \normb{G(t,y)u }{\rU^\prime}{2} \, \nu (\ud y) &\leq  {C}_{3}^\prime  \normb{u }{\rU^\prime}{2}, \;  u \in \rU^\prime , \; t \in [0,T].
\end{align}
\end{trivlist}
Then
\begin{trivlist}
\item[(G5)'] the linear map
\begin{align*}
\textbf{G}\colon  \zcal_T
\ni u \mapsto  \bigl\{  [0,T] \times Y \ni (t,y) & \mapsto      G(t,y)u(t) \in\rU^\prime \bigr\} \\
&\in {L}^{2}([0,T]\times Y,\Leb\otimes \nu, U^\prime )
      \quad
      \nonumber
\end{align*}
is continuous with respect to the $ \mathbb{D} ([0,T];U^\prime)$ topology.
\end{trivlist}
\end{lemma}
\begin{proof}[Proof of Lemma \ref{lem-G-aux}]
Let us assume that  $(u_n)_{n\in \mathbb{N}}$ is an $\zcal_T$-valued sequence  such that $u_n \to u$ in the topology of $\mathbb{D} ([0,T];U^\prime)$.
 Then, we have the following chain of  inequalities 
\begin{align*}
&\int_0^T  \int_{Y} \vert G(t,y)u_n(t)- G(t,y)u(t) \vert_{\rU^\prime}^2 \,  \nu(\ud y)\dt 
\\&=
\int_0^T  \int_{Y} \vert G(t,y) \left( u_n(t)-u(t) \right) \vert_{\rU^\prime}^2 \,  \nu(\ud y)\dt 
 \leq {C}_{3}^\prime \int_0^T   \vert u_n(t)-u(t) \vert_{\rU^\prime}^2 \dt
\end{align*}
Thus, by Lemma \ref{lem-new auxiliary}   we conclude the proof of the Lemma.
    \end{proof}
\end{example}

Before we proceed further let us give another example of a map $F$ satisfying Assumptions F.\ref{assum-F1}-F.\ref{assum-F5}.
This example is related to Example \ref{example-F is linear}.

\begin{example}\label{exa:F_satisfying_assumptions}
Consider a one dimensional L\'evy process $L$ and the  corresponding PRM $\eta$. Then, we have $Y=\mathbb{R}$ which is a standard measurable space. Furthermore, we assume that the L\'evy measure $\nu$ of $\eta$ satisfies $\nu(\{0\})=0$ and that there exists  a number  $\gamma >0$  such that
for every $p \in \{ 2,8+2\gamma \} $
\begin{equation*} \int_{\mathbb{R}\setminus(-1,1)} \vert y\vert ^p \, \nu(\ud y) <\infty. \end{equation*}
Let $F$ be the  bilinear multiplicative map defined by
\[
F(t,u,y)\coloneqq uy, \;\; (t,u,y)\in [0,T]\times\rU^\prime \times \mathbb{R}.
\]
It is clear that $F$ satisfies Assumption F.\ref{assum-F1}. Since $\nu(\{0\})=0$ and $\int_{\mathbb{R}} \vert y\vert ^p\, \nu(\ud y) <\infty$ it is not difficult to show that $F$ satisfies the Assumptions F.\ref{assum-F2}-F.\ref{assum-F3} and Assumption F.\ref{assum-F5}.
What only remains to do is to verify whether  $F$ satisfies Assumption F.\ref{assum-F4}.

For this aim, let $(u_n)_{n\in \mathbb{N}}\subset \zcal_T$ such that $u_n \to u$ in the topology of $\mathbb{D} ([0,T];U^\prime) $.
 Then, we have the following chain of  inequalities
\begin{align*}
\int_0^T  \int_{Y} \vert F(t,u_n(t),y)- F(t,u(t),y) \vert_{\rU^\prime}^2 \,  \nu(\ud y)\dt  &=
\int_0^T  \int_{\mathbb{R}} \vert u_n(t)y-u(t)y \vert_{\rU^\prime}^2 \,  \nu(\ud y)\dt \\
=\int_0^T  \int_{\mathbb{R}} \vert u_n(t)-u(t) \vert_{\rU^\prime}^2 y^2 \, \nu(\ud y)\dt
&= \int_{\mathbb{R}} y^2 \, \nu(\ud y)  \int_0^T  \vert u_n(t)-u(t) \vert_{\rU^\prime}^2 \dt.
\end{align*}

Lemma \ref{lem-new auxiliary}  implies that
\begin{equation*}\int_0^T  \int_{Y} \vert F(t,u_n(t),y)- F(t,u(t),y) \vert_{\rU^\prime}^2 \,  \nu(\ud y)\dt \to 0,\end{equation*} as $n\to \infty$
because $\int_{\mathbb{R}} y^2 \, \nu(\ud y) <\infty$ and $\int_0^T  \vert u_n(t)-u_n(t) \vert_{\rU^\prime}^2 \dt \to 0$ since
$u_n \to u$ in $\mathbb{D} ([0,T];U^\prime) $. 
\end{example}

\subsection{Definitions and the main result} \label{subsec-definitions}

We now continue with the concept of solutions of the Problem \eqref{eqn-SNSEs} and the statement of the main result of this paper.
\begin{definition}  \rm  \label{def-solution martingale}
 \bf A martingale solution \rm of  equation (\ref{eqn-SNSEs})
is a system
\begin{equation*}\bigl( \Omega, \fcal,  \fmath ,\mathbb{P},u , \eta\bigr),\end{equation*}
where
\begin{itemize}
\item[$\bullet $]  $\bigl( \Omega, \fcal,  \fmath, \mathbb{P}  \bigr) $ is a complete filtered probability space with a filtration $\fmath = \left( \fcal_{t} \right)_{t \ge 0} $, satisfying the usual assumption,
\item[$\bullet $]  $\eta $  is a time homogeneous Poisson random measure on $(\R_+ \times \rY, \mathscr{B}(\R_+) \otimes \ycal )$ over
$(\Omega , \fcal , \fmath , \mathbb{P})$ with $\sigma$-finite  intensity measure $\Leb \otimes \nu$ and the corresponding compensated Poisson random measure $\tilde{\eta}$.

\item[$\bullet $] $u\colon  [0,T] \times \Omega \to {\rH}$ is a predictable process with $\mathbb{P}$-a.s.\ paths
\begin{align*}
  u(\cdot , \omega ) \in \mathbb{D} \bigl( [0,T]; {\rH}_{w} \bigr)
   \cap {L}^{2}(0,T;V )
\end{align*}
such that for all $ t \in [0,T] $  the following identity holds $\mathbb{P}$-a.s. for all $v \in V $
\begin{equation*}
    \begin{aligned}
\ip{u(t)}{v}{\rH} & + \int_{0}^{t} \ddual{\rV^\prime }{\acal u(s)}{v}{\rV}{}\ds
+ \int_{0}^{t} \ddual{\rV^\prime }{\rB(u(s))}{v}{\rV}{}\ds
   =\ip{{u}_{0}}{v}{\rH} 
   \\
   &+ \int_{0}^{t} \ddual{\rV^\prime }{f(s)}{v}{\rV}{}\ds
 + \int_{0}^{t} \int_{Y} \ip{F(s,u(s-),y)}{v}{\rH} \,  \tilde{\eta} (\ud s,\ud y).
\end{aligned}
\end{equation*}
\end{itemize}
\end{definition}
\begin{remark}
According to our proof the $\mathbb{P}$-full set does not depend on $v\in V$.
\end{remark}
Note that the spaces $U$ and $U^\prime$ are purely auxiliary and do no appear in the definition above.

\begin{theorem} \label{thm-main-existence}
 There exists a martingale solution of the problem (\ref{eqn-SNSEs}) provided assumptions A.\ref{assum-A.1}, P.\ref{item-P_Y_is_standard}  and  F.\ref{assum-F1}-F.\ref{assum-F5}
 are satisfied. \label{assum-F5-used-1}
\end{theorem}

The next two sections contain numerous intermediate results needed in the proof of this theorem which  will be given in Section \ref{sec-existence-III}.

\section{Galerkin approximation} \label{sec-existence-I}
In this section we construct a system of $n$ stochastic differential equations with globally Lipschitz coefficients on a finite dimensional space $\rH_n \subset \rH$ for arbitrary $n\in \mathbb{N}$. We will also derive uniform estimates for the soltuions of these SDEs.
In all our results in this section we require Assumption F.\ref{assum-F2}.
Whenever we require Assumption F.\ref{assum-F5} we mention this explicitly. \label{assum-F5-used-2}

\subsection{Faedo-Galerkin approximation}
\label{subsec-Galerkin}

Thanks to \cite[Theorem\ I.8.1]{Ikeda+Watanabe_1989} we find a time homogeneous PRM $\eta$  on $(\mathbb{R}_{+}\times Y, \mathscr{B}(\mathbb{R}_{+})\times \ycal)$ such that its intensity measure is $\Leb \otimes \nu$ and constructed on a probability space $(\Omega, \mathscr{F}, \mathbb{P})$. Denote by $\mathscr{N}$ the collection of all subsets of the null sets in $\mathscr{F}$.
The compensated PRM corresponding to $\eta$ is denoted by $\tilde{\eta}$. We equipped the probability space with the following filtration
\begin{equation*}\label{eqn-filtration_tF}
\begin{aligned}
    {\mathbb{F}} &\coloneqq({{\fcal }}_{t})_{\tinOT}, 
    \\
    \mbox{ where } &\\
    &{{\fcal }}_{t} = \sigma\bigl(  \mathscr{H}_{t+} \cup \mathscr{N}\bigr)   \\
\mbox{  and } &\\
\mathscr{H}_{t+}&\coloneqq \bigcap_{s>t} \mathscr{H}_{s} \mbox{ with } \mathscr{H}_{s}\coloneqq  \sigma \{ \eta([0, r]\times A): \, \, r \le s,\, A\in \ycal \}. \\
\end{aligned}
\end{equation*}
Note that the filtration $ {\mathbb{F}}$ satisfies the usual conditions and $\eta$ is $\mathbb{F}$-adapted.

Now, we consider the following Faedo-Galerkin approximation of the Problem \eqref{eqn-SNSEs}
  in the space $ {\rH}_{n}$
\begin{eqnarray}\label{eqn-Galerkin}
&    {u}_{n} (t) & =  {P}_{n} {u}_{0} - \int_{0}^{t}\bigl( {P}_{n} \acal {u}_{n} (s)  + {\rB}_{n}  \bigl({u}_{n} (s) \bigr) - {P}_{n} f (s)  \bigr) \ds  \nonumber   \\
&  &+  \int_{0}^{t} \int_{Y} {P}_{n} F(s,{u}_{n} ({s}{-}),y) \, \tilde{\eta } (\ud s,\ud y) ,
\quad t \in [0,T].
\end{eqnarray}

It is not difficult to prove that the map  ${\rB}_{n}\colon {\rH}_n \to {\rH}_n $ is globally Lipschitz continuous.
Thus, thanks to the linearity of $\acal_n$ and \eqref{eqn-F_Lipschitz_cond} in Assumption F.\ref{assum-F2}\label{assum-F2-used1} we see that  \eqref{eqn-Galerkin} is  a system of stochastic differential equations with globally Lipschitz coefficients. Thus, by applying \cite[Theorem 9.1]{Ikeda+Watanabe_1989} we obtain the following existence result. 

\begin{lemma}[\cite{Motyl_2013}, Lemma 3] \label{lem-Galerkin_existence}
For each $n \in \mathbb{N} $,   equation (\ref{eqn-Galerkin}) has  a unique  global $\fmath $-adapted,  c\`{a}dl\`{a}g ${\rH}_{n}$ valued solution  ${u}_{n}$ .
\end{lemma}
Now,
by employing equalities \eqref{eqn-B_n_b} and \eqref{eqn-A_n_duality}
the  following three lemmata can be established with exactly the same proofs as in \cite[Lemma\ 3 and 4]{Motyl_2013} which hold provided the noise coefficient $F$ satisfies  condition \eqref{eqn-F_linear_growth} in Assumption \ref{assum-F2} and  Assumption F.\ref{assum-F5}. \label{assum-F5-used-3}.

Lemma \ref{lem-Galerkin_estimates}  is proved by using the finite dimensional It\^{o} formula, see \cite{Ikeda+Watanabe_1989} or \cite{Metivier},
and the Burkholder-Davis-Gundy (BDG) inequality, see \cite[Theorem\ 26.12]{Kallenberg_2002}.

\begin{lemma}[\cite{Motyl_2013}, Lemma 4]
\label{lem-Galerkin_estimates}
The processes $({u}_{n} {)}_{n \in \mathbb{N} }$ from Lemma \ref{lem-Galerkin_existence} satisfy the following conditions. 
\begin{trivlist} 
\item[(i)] There exists a positive constant ${C}_{2}>0$ 
\begin{align}
  &   \mathbb{E} \left[ \int_{0}^{T} \normb{ {u}_{n} (s)}{\rV}{2} \ds \right] \le {C}_{2},\;\; \mbox{ for every } n \in \mathbb{N},
  \label{eqn-V_estimate}.
\end{align}
\item[(ii)] There exists a positive constant ${C}_{1}$  such that every    $p\in [1,4+\gamma]$ 
\begin{equation} \label{eqn-H_estimate-p>2}
  \mathbb{E} \left[ \sup_{0 \le s \le T } \normb{{u}_{n} (s) }{\rH}{p} \right] \le {C}_{1}^{\frac{p}{4+\gamma}}
 ,\;\; \mbox{ for every } n \in \mathbb{N}.
\end{equation}
\end{trivlist}
\end{lemma}

Note that it is sufficient to prove assertion  \eqref{eqn-H_estimate-p>2} only  for  $p=4+\gamma$. The case $p< 4+\gamma$ follows from the former by the Jensen inequality.
We recall the proof in Appendix \ref{app-proof_of_apriori_est}.

\subsection{Tightness}
\label{subsect-tightness}

As is Section \ref{sec-preliminaries} we chose and fix   $\s > \frac{d}{2}+1$. Let us recall that the space
 ${\rV}_{\s}$ defined in  identity \eqref{eqn-V_s}
 is dense in $\rV$ and the embedding ${j}_{\s}\colon  {\rV}_{\s} \embed V$ is continuous. Let us also recall that the  Hilbert space $\rU$   introduced at the end of the Section \ref{sec-preliminaries} is a dense subset of ${\rV}_{\s}$ and as mentioned in \eqref{eqn-U_comp_V_s} the natural embedding ${\iota }_{\s}\colon \rU \embed  {\rV}_{\s}$ is compact.

With this in mind, we will apply a standard tightness condition \cite[Corollary\ 1]{Motyl_2013} with $q=2$.

For each $n \in \mathbb{N} $, the solution ${u}_{n} $ of the Galerkin equation defines a measure
$\lcal ({u}_{n})$ on $(\zcal_T , \tcal )$, where the space $\zcal_T$ has been defined in \eqref{eqn-Z_cadlag_NS}.
Using \cite[Corollary\ 1]{Motyl_2013}, with $q=2$,  one can  prove that the set of measures
$\bigl\{ \lcal ({u}_{n} ) : n \in \mathbb{N}  \bigr\} $ is tight on $(\zcal_T , \tcal )$. The inequalities \eqref{eqn-V_estimate} and \eqref{eqn-H_estimate-p>2} with $p=2$ in Lemma \ref{lem-Galerkin_estimates} are of crucial importance.

\begin{lemma} \label{lem-comp_Galerkin}
The set of measures $\bigl\{ \lcal ({u}_{n} ) : n \in \mathbb{N}  \bigr\} $ is tight on $(\zcal_T , \tcal )$.
\end{lemma}
\begin{proof}
This lemma can be  proved with the same approach as used in \cite[Lem.\ 5]{Motyl_2013}, thus we omit the proof. Let us point out that as it is explicitely visible from the proof  of Theorem 2 therein, one needs assumption \eqref{eqn-h-to-U'-compact}, i.e. that the embedding $H \embed U^\prime$ is compact.
\end{proof}

\subsection{Application of the Skorohod-Jakubowski Theorem }\label{subsec-SJ}

We will use the following Jakubowski's version of the Skorokhod Theorem in the form given by
Brze\'{z}niak and Ondrej\'{a}t  \cite{Brz+Ondrejat_2011}, see also Jakubowski \cite{Jakubowski_1998}. The version below does not use almost sure convergence of
\cite{Brz+Ondrejat_2011} but a stronger pointwise convergence  as in \cite[Theorem 2]{Jakubowski_1998}.

\begin{theorem} \rm (Theorem A.1 in \cite{Brz+Ondrejat_2011}) \it
\label{thm-Jakubowski_Skorokhod_Ondrejat}
Let $\xcal $ be a topological space such that there exists a sequence $\{ {f}_{m}\} $ of continuous functions ${f}_{m}\colon \xcal  \to \mathbb{R} $ that separates points of $\xcal $. Let us denote by $\scal $ the $\sigma $-algebra generated by the maps $\{ {f}_{m}\} $.
Then
\begin{description}
\item[(j1) ] every compact subset of $\xcal $ is metrizable;
\item[(j2) ] if $({\mu }_{m})$ is a tight sequence of probability measures on $(\xcal ,\scal )$, then there exists a subsequence $({m}_{k})$, a probability space $(\Omega , \fcal ,\mathbb{P})$ with $\xcal $-valued Borel measurable variables ${\xi }_{k}, \xi $ such that  ${\mu }_{{m}_{k}}$ is the law of ${\xi }_{k}$
and ${\xi }_{k}$ converges to $\xi $ on $\Omega $.
Moreover, the law of $\xi $ is a Radon measure.
\end{description}
\end{theorem}

Using  Theorem \ref{thm-Jakubowski_Skorokhod_Ondrejat}, we obtain the following corollary which we will apply to construct a martingale solution to the Navier-Stokes equations.

\begin{corollary}  \label{cor-Skorokhod_Z}
Let $({y }_{n}{)}_{n \in \mathbb{N} }$ be a sequence of $\zcal_T $-valued random variables such that their laws $\lcal ({y }_{n})$ on $(\zcal_T ,\tcal )$ form a tight sequence of probability measures.
Then there exists a subsequence $({n}_{k})$, a probability space $(\Omegabar , \tfcal ,\Pbar )$ and $\zcal_T $-valued random variables $\tilde{y}, {\tilde{y}}_{n_k} $, $k \in \mathbb{N} $ such that the variables ${y }_{n_k}$
and ${\tilde{y}}_{n_k}$ have the same laws on $\zcal_T $
and ${\tilde{y}}_{n_k}$ converges to $\tilde{y}$ on $\Omegabar $.
\end{corollary}

\begin{proof}
It is proven in \cite[Remark\ 2, p.\ 905]{Motyl_2013} that the space $\zcal_T$ satisfies the assumptions of Theorem \ref{thm-Jakubowski_Skorokhod_Ondrejat}, thus the assertions of the Corollary follows from Theorem \ref{thm-Jakubowski_Skorokhod_Ondrejat}-(j2).
\end{proof}

Now, owing to Lemma \ref{lem-comp_Galerkin} and Corollary \ref{cor-Skorokhod_Z} we can find
a probability space
\begin{equation}\label{eqn-new_probability_space}
\bigl( \Omegabar ,\tfcal ,\Pbar  \bigr),
\end{equation}
 an increasing   sequence $({n}_{k}{)}_{k}$ of natural numbers and
$ \zcal_T $-valued random variables $\tu$, $\tunk$, $k \in \mathbb{N} $
defined on $\bigl( \Omegabar ,\tfcal ,\Pbar  \bigr)$, such that $\tunk$ has the same law as $\unk$  on $\zcal_T$ and
\begin{equation}
\tunk  \to \tu  \mbox{ in } \zcal_T
    \quad  {\mbox{ on } \Omegabar.} \label{eqn-Skorokhod_appl_convergence}
\end{equation}

Without loss of generality, we can  and will denote the subsequence $(\tunk)_{k}$  by $(\tun)_{n}$ or simply by $(\tun)$. In view of Definition \ref{eqn-Z_cadlag_NS} of the space $\zcal_T$,  the above assertion  \eqref{eqn-Skorokhod_appl_convergence} is equivalent to the following assertions:
\begin{align}\label{eqn-conv_L^2_w(0,T;V)}
\tun \to \tu \mbox{ in } {L}_{w}^{2}(0,T;V)       \quad  {\mbox{ on } \Omegabar},\\
\tun \to \tu \mbox{ in } {L}^{2}(0,T;{\rH}_{\loc})     \quad  {\mbox{ on } \Omegabar},
\label{eqn-conv_L^2(0,T;H_loc)}
\\
\tun \to \tu \mbox{ in }  \mathbb{D} ([0,T];{\rH}_{w}) \quad  {\mbox{ on } \Omegabar},
\label{eqn-conv_D(0,T;H_w)}
\\
\tun \to \tu \mbox{ in }  \mathbb{D} ([0,T];\rU^\prime) \quad  {\mbox{ on } \Omegabar}.
\label{eqn-conv_D(0,T;U')}
\end{align}
To close this subsection we state the following important remark.
\begin{remark}\label{rem-process ubar} It follows 
from the definition \eqref{eqn-Z_cadlag_NS} of the space  $ \zcal_T={L}_{w}^{2}(0,T;V) \cap {L}^{2}(0,T;{\rH}_{\loc}) \cap \mathbb{D} ([0,T];U^\prime)
\cap \mathbb{D} ([0,T];{\rH}_{w}) $
that
the  $ \zcal_T $-valued random variables $\tu$ defined on $\bigl( \Omegabar ,\tfcal ,\Pbar  \bigr) $ induces
a  measurable process (denoted by the same letter)
\[
\tu\colon [0,T]\times \Omegabar \to V \cap {\rH}_{\loc},
\]
whose trajectories belong  to the space
$\zcal_T$, see e.g. \cite[Proposition B.1]{Brzezniak+Rana_2021}.
\end{remark}

\subsection{A priori estimates for the processes on the new probability space}
\label{subsec_a_priori}

In the earlier parts of the paper  we viewed  $\tu$ as a $ \zcal_T $-valued random variable, but from now we treat it as an $\rU^\prime$-valued process, which can be done as in Remark \ref{rem-process ubar}, see also \cite[Proposition\ B.1]{Brzezniak+Rana_2021}. Moreover, in view of the Kuratowski Theorem, we can view it as an $\rH$-valued process as well and
as an  important consequence of \eqref{eqn-conv_D(0,T;H_w)} and \cite[Lemma 2]{Brzezniak+Hornung+Manna}
the trajectories of the  sequence of processes $(\tun)$ are uniformly bounded in $\rH$, i.e.\ they satisfy the following condition
		\begin{align}\label{eqn-boundedness_H}
	\sup_{n\in\N}\sup_{\tinOT}\normb{\tun(t)}{\rH}{}+\sup_{\tinOT}\normb{\tu(t)}{\rH}{}<\infty, \;\; \mbox{ on } \Omegabar.
	\end{align}
This observation will be crucial in our argument below. Since every weakly convergent sequence is bounded we infer  from \eqref{eqn-conv_L^2_w(0,T;V)} that
		\begin{align}\label{eqn-boundedness_L^2(0,T,V)}
	\sup_{n\in\N}  \normb{\tun}{{L}^{2}(0,T;V)}{} + \normb{\tu}{{L}^{2}(0,T;V)}{}<\infty,\;\; \mbox{ on } \Omegabar.
	\end{align}

We finish these preliminary comments on the conclusions from the Skorohod Theorem by the following result.

\begin{lemma}\label{lem_conv_L^2_w(0,T;H)} The following two assertions hold {on  $\Omegabar$}
\begin{align}\label{eqn-conv_L^2_w(0,T;H)}
\tun \to \tu \mbox{ \it in } {L}_{w}^{2}(0,T;{\rH})
\end{align}
and,  for every $v \in L^2(0,T;V)$,
\begin{align}\label{eqn-conv_L^2_w(0,T;V-dot)}
\int_0^T \ip{\nabla \tun (s)}{\nabla v(s)}{\bL^2}  \ds \to \int_0^T  \ip{\nabla \overline{u}(s)}{\nabla v(s)}{\bL^2} \ds.
\end{align}
\end{lemma}

\begin{proof}[Proof of Lemma \ref{lem_conv_L^2_w(0,T;H)}]
Let us choose and fix an arbitrary $\omega \in \Omegabar$. Then, by   \eqref{eqn-conv_L^2_w(0,T;V)}, $\tun=\tun(\omega) \to \tu=\tu(\omega)$ weakly in $L^2(0,T;V)$.
Since the natural embedding $i\colon  L^2(0,T;V) \embed L^2(0,T;{\rH})$ is linear and bounded,  its adjoint $i^\ast\colon  L^2(0,T;{\rH}) \to L^2(0,T;V)$ is linear and bounded as well. Hence  we infer that $\tun \to \tu$ weakly in $L^2(0,T;{\rH})$.\\
The second claim follows by observing that for every  $v \in L^2(0,T;V)$,
\begin{align*}
\int_0^T \ip{\nabla \overline{u}_n(s)}{\nabla v(s)}{\bL^2} \ds
= \int_0^T\ip{\overline{u}_n(s)}{v(s)}{V} - \ip{\overline{u}_n(s)}{v(s)}{H} \ds,
\end{align*}
from which along with \eqref{eqn-conv_L^2_w(0,T;V)} and \eqref{eqn-conv_L^2_w(0,T;H)} we see  that
\begin{align*}
\int_0^T \ip{\nabla \overline{u}_n(s)}{\nabla v(s)}{\bL^2} \ds &\to \int_0^T\ip{\overline{u}(s)}{v(s)}{V} - \ip{\overline{u}(s)}{v(s)}{H} \ds \\
&\qquad = \int_0^T \ip{\nabla \overline{u}(s)}{\nabla v(s)}{\bL^2} \ds.
\end{align*}
\end{proof}

We also have the following estimates.
\begin{lemma}\label{lem-Galerkin_estimates'}
\begin{enumerate}
\item $\overline{u}_n$ is almost surely in $\mathbb{D}([0,T];P_nH)$ and for any $p\in [2, 4+\gamma)$ there exists a constant $C_1(p)>0$
\begin{equation} \label{eqn-H_estimate'}
 \sup_{n\in \mathbb{N} } \overline{\mathbb{E}}\left[ \sup_{ s\in [0,T] } \bigl| \tun (s){\bigr| }_{\rH}^{p}\right] \le {C}_{1}(p),
\end{equation}
\begin{equation} \label{eqn-V_estimate'}
  \sup_{n\in \mathbb{N} } \overline{\mathbb{E}}\left[ \int_{0}^{T} {\bigl\| \tun (s)\bigr\| }_{\rV}^{2} \ds \right] \le {C}_{2}.
\end{equation}
\item $\overline{u}$ satisfies the following estimates
\begin{equation} \label{eqn-tu_V_estimate}
\overline{\mathbb{E}} \left[ \int_{0}^{T} \lVert \tu (s)\rVert_{\rV}^{2}\ds \right] < \infty,
\end{equation}
\begin{equation} \label{eqn-tu_H_estimate}
\overline{\mathbb{E}}\left[ \sup_{\sinOT } \bigl| \tu (s){\bigr| }_{\rH}^{2}\right] < \infty.
\end{equation}
\end{enumerate}
\end{lemma}

\begin{proof}
(i) Let us  observe that by the Kuratowski Theorem, see \cite[Theorem\ I.3.9]{Parthasarathy_1967},
$\mathbb{D} ([0,T];{P}_{n} H)$ is a Borel subset of $\mathbb{D} ([0,T];\rU^{\prime }) \cap {L}^{2}(0,T;{\rH}_{\loc})$. Indeed, both spaces $\mathbb{D} ([0,T];{P}_{n} H)$  and $\mathbb{D} ([0,T];\rU^{\prime }) \cap {L}^{2}(0,T;{\rH}_{\loc})$ are Polish and the natural embedding of the former into the latter is continuous (and hence Borel measurable) and obviously injective.
Since ${u}_{n}\in \mathbb{D} ([0,T];{P}_{n} H)$ $\mathbb{P}$-a.s.\ and $\tun $ and ${u}_{n}$ have the same laws, we infer that
\begin{align*}
  \lcal (\tun ) \bigl( \mathbb{D} ([0,T]; {P}_{n} H )  \bigr) =1 , \qquad n \ge 1 .
\end{align*}

Firstly, note that the function 
\begin{equation}\label{measurable_on_D}
\mathscr{Z}_T \ni u \mapsto \sup_{ s\in [0,T] } \bigl| u(s){\bigr| }_{\rH}^{p}
\end{equation}
is measurable.
Indeed, similarly to \cite[Appendix E]{Brzezniak_Ferrario_Zanella}, it is enough to show that the function \eqref{measurable_on_D} is lower semicontinuous.
This follows from the fact that the set
\[\{u \in D([0, T],H_w) : \displaystyle \sup_{t\in [0,T]} \abs{u(t)}_H \leq r\}\]

is closed in $D([0, T],H_w)$, see \cite[Proposition 1.6(v)]{Jakubowski_1986}.

The function
\begin{equation}
    \mathscr{Z}_T \ni u \mapsto \int_{0}^{T} {\bigl\| u(s)\bigr\| }_{\rV}^{2} \ds
\end{equation}
is also measurable since  it is $\mathscr{B}(L_w^2(0,T;V))/\mathscr{B}(\mathbb{R})$ measurable. Indeed, it is lower-semicontinuous in view of Corolaries to \cite[Theorem 3.12]{Rudin_functional}.

Since $\tun $ and ${u}_{n}$ have the same laws, and $\mathbb{D} ([0,T];{P}_{n} H)$ is a Borel subset of
$\mathbb{D} ([0,T];\rU^{\prime }) \cap {L}^{2}(0,T;{\rH}_{\loc})$, thus  by (\ref{eqn-V_estimate})  and (\ref{eqn-H_estimate-p>2}) with $p=2$  we obtain the estimates \eqref{eqn-H_estimate'} and \eqref{eqn-V_estimate'}.

(ii) By the Banach-Alaoglu Theorem and  inequality (\ref{eqn-V_estimate'}) we infer that the sequence $(\tun )$ contains subsequence, still denoted by $(\tun )$, which is  weakly   convergent in the space ${L}^{2}([0,T]\times \Omegabar ; V )$.
Since by (\ref{eqn-Skorokhod_appl_convergence}) $\Pbar $-a.s.\ $\tun \to \tu $ in $\zcal_T $, we conclude that $\tun  \to \tu$  weakly in  ${L}^{2}([0,T]\times \Omegabar ; V )$ and that
$\tu \in {L}^{2}([0,T]\times \Omegabar ; V )$, i.e.\ \eqref{eqn-tu_V_estimate} holds.

Similarly, by the Banach-Alaoglu Theorem and inequality (\ref{eqn-H_estimate'}) with $p=2$ we can choose a subsequence of $(\tun )$, still denoted by $(\tun )$, which is   weakly-star convergent in the space ${L}^{2}(\Omegabar ; {L}^{\infty }(0,T;{\rH}))$.
Since by \eqref{eqn-conv_D(0,T;H_w)}  $\Pbar $-a.s.\ $\tun \to \tu $ in $\mathbb{D} ([0,T];{\rH}_{w})$,  we infer that $\tun  \to \tu $  weakly-star  in  ${L}^{2}(\Omegabar ; {L}^{\infty }(0,T;{\rH}))$  and that \eqref{eqn-tu_H_estimate} holds.
\end{proof}

\section{Auxiliary Martingales} \label{sec-existence-II}

In this section, whenever we use Assumption F.\ref{assum-F5},  we specify this explicitly.  \label{assum-F5-used-4} The main results in this section are Lemmata \ref{lem-tM is a martingale} and \ref{lem-tN is a martingale}.

\medskip

Let us recall, see Lemma \ref{lem-Galerkin_existence},  that  for each $n \in \mathbb{N} $,  the process  ${u}_{n}$ is the unique  $\fmath $-adapted,  c\`{a}dl\`{a}g ${\rH}_{n}$-valued solution to the Galerkin approximation equation (\ref{eqn-Galerkin}).

For every $n \in \mathbb{N}$ and $\phi \in\rU$ let us define the following two ${\rH}_n$-,  and respectively, $\mathbb{R
}$-valued processes ${{M}}_{n}$ and   $N_{n,\phi}$, on the  original probability space $\bigl( {\Omega}, \fcal, {\mathbb{P}}, {\mathbb{F}} \bigr)$, by, for $t \in [0,T]$,
\begin{equation}  \label{eqn-M_n}
{{M}}_{n} (t)
   \coloneqq u_{n} (t) - u_n (0)  + \int_{0}^{t} {P}_{n} \acal u_{n} (s) \ds
  + \int_{0}^{t} \rB_n  \bigl( u_{n} (s)  \bigr) \ds
  - \int_{0}^{t} {P}_{n} f(s) \ds
\end{equation}
and
\begin{align}\nonumber
  N_{n,\phi}(t) &\coloneqq\duali{u_{n} (t)}{\phi}{2} - \duali{ u_n(0)}{\phi}{2} \\
  & +2  \int_0^t \duality{{P}_{n} \acal u_{n} u(s) \,+ \rB_{n}(u_{n} (s)){-P_nf(s)}}{\phi}\;\duality{u_{n} (s)}{\phi} \, \ds
 \notag  \\
  &-\int_0^t \int_{Y} \langle {P}_{n}  F(s,u_{n} (s),y),\phi \rangle^2 \, \nun(\ud y) \ds ,  \quad
t \in [0,T].
  \label{eqn-N_n,phi}
\end{align}

It follows from identity \eqref{eqn-Galerkin} that
\begin{equation}  \label{eqn-M_n_2}
{{M}}_{n} (t)=\int_{0}^{t} \int_{Y} {P}_{n} F(s,{u}_{n} ({s}{-}),y)\; \tilde{\eta} (\ud s,\ud y) , \;\;\; \tinOT.
\end{equation}

\begin{lemma}\label{lemma-M_N martingale}
For every $n \in \mathbb{N}$, the process ${{M}}_{n}$ is a square integrable $\mathbb{F}$-martingale and 
\begin{align}
\label{eqn-M_n_3}
\sup_{n \in \mathbb{N}}\mathbb{E} \left[ \sup_{\tinOT}\bigl\vert  {{M}}_{n} (t) \bigr\vert_{\rH}^2 \right]
&\leq C_2 \sup_{n \in \mathbb{N}}
\mathbb{E} \Bigl[   \int_0^{T} \int_{Y} \rvert P_n F(s, u_n(s),y)\vert_{\rH}^2 \, \nun(\ud y) \ds
\Bigr]
\\
&\leq C_2 \sup_{n \in \mathbb{N}} \left( 1+ \mathbb{E} \left[ \int_0^T  \lvert u_n (t) \rvert^2_\rH \dt \right] \right) < \infty. 
\label{eqn-M_n_3_B}
\end{align}
Moreover, if Assumption F.\ref{assum-F5} \label{assum-F5-used-5} is satisfied,  then for every $p\in [2, 4+\gamma] $ and for every $n \in \mathbb{N}$, the process ${{M}}_{n}$ is a $p$-integrable $\mathbb{F}$-martingale
\begin{align}\label{eqn-M_n_3-p}
\sup_{n \in \mathbb{N}}\mathbb{E} \left[ \sup_{\tinOT}\bigl\vert  {{M}}_{n} (t) \bigr\vert_{\rH}^p \right]
&\leq C_p \; \sup_{n \in \mathbb{N}}
\mathbb{E} \Bigl[   \Bigl( \int_0^{T} \int_{Y} \rvert P_n F(s, u_n(s),y)\vert_\rH^2 \, \nun(\ud y) \ds \Bigr)^\frac{p}2
\\
\nonumber
&\hspace{2truecm}+\int_0^{T} \int_{Y} \rvert P_n F(s, u_n(s),y)\vert_\rH^p \, \nun(\ud y) \ds
\Bigr]
\\
\label{eqn-M_n_3-p-B}
&\leq C_p \sup_{n \in \mathbb{N}} \left( 1+ \mathbb{E} \left[ \int_0^T  \lvert u_n (t) \rvert^p_\rH \dt \right] \right) < \infty .
\end{align}
The constants $C_2$, $C_p$ are generic positive constants which can change from line to line but which do not depend on $n$.
\end{lemma}
\begin{proof}[Proof of Lemma \ref{lemma-M_N martingale}. Part I]
Let us first observe that because of \eqref{eqn-M_n_2}, the process $M_n$ is a square-integrable $H$-valued martingale.  Thus inequality 
\eqref{eqn-M_n_3} 
follows from     the BDG inequality, see \cite[Theorem 20.6]{Metivier}.
The first inequality in \eqref{eqn-M_n_3_B} follows by applying  
the linear growth condition \eqref{eqn-F_linear_growth} in Assumption F.\ref{assum-F2}.\label{assum-F2-used2}
The second inequality in \eqref{eqn-M_n_3_B} follows from inequality (\ref{eqn-H_estimate-p>2}) with $p=2$.
\end{proof}

\begin{proof}[Proof of Lemma \ref{lemma-M_N martingale}; part II] 
Let us choose and fix $p \in (2, 4+\gamma]$. 
The first part of inequality \eqref{eqn-M_n_3-p} follows from inequality (2.32) in \cite[Theorem 2.6]{Zhu+Brze+Liu_2019}. Inequality \eqref{eqn-M_n_3-p-B} is a consequence of \eqref{eqn-M_n_3-p} and the linear growth assumption \eqref{eqn-F_linear_growth-gamma} in Assumption \ref{assum-F5}.
The second  part of inequality \eqref{eqn-M_n_3-p-B} follows from the first one and inequality \eqref{eqn-H_estimate-p>2} in Lemma \ref{lem-Galerkin_estimates}.
\end{proof}

For every  $n \in \mathbb{N}$ and $\phi \in\rU$, let us define another pair of   processes ${\Mbar}_{n} $ and $\Nbar_{n,\phi}$ by
\begin{equation}
\label{eqn-tMn}
\begin{aligned}
{\Mbar}_{n} (t)
   &= \tun (t) \,  -   {P}_{n} \tu (0)  + \int_{0}^{t} {P}_{n} \acal \tun (s) \ds
  \\ &+ \int_{0}^{t} {\rB}_{n}  \bigl( \tun (s)  \bigr) \ds
  - \int_{0}^{t} {P}_{n} f(s) \ds, \;\; \tinOT
  \end{aligned}
\end{equation}
and
\begin{equation}
\begin{aligned}
  \Nbar_{n,\phi}(t) &=\duali{\tun (t)}{\phi}{2} - \duali{\overline{u}_n(0)}{\phi}{2} \notag \\
  &+2  \int_0^t \duality{{P}_{n} \acal \tun u(s)+B_{n}(\tun (s))-P_nf(s) }{\phi}\;\duality{\tun (s)}{\phi} \ds\notag  \\
  &-\int_0^t \int_{Y} \langle {P}_{n}  F(s,\tun (s),y),\phi \rangle^2 \,  \nun(\ud y) \ds, \;\; \tinOT.
  \end{aligned}
  \label{eqn-N_n_bar^phi}
\end{equation}

Since $H_n$ is finite dimensional and these processes are c\`adl\`ag in $H$ equipped with the weak topology, it follows that  ${\Mbar}_{n}$ has  trajectories in $\mathbb{D} ([0,T];H_n)$ and therefore also in $\mathbb{D} ([0,T];H) $ and $ \mathbb{D} ([0,T];\rU^\prime)$. Consequently, $\Nbar_{n,\phi}$ has trajectories in  $\mathbb{D} ([0,T];\mathbb{R}) $.

Also, since $\tun$ has trajectories in $\mathbb{D} ([0,T];H_n)$ and $\phi\in U\subset H$, by \eqref{eqn-U'-H_2} we have 
\begin{equation}
\begin{aligned}
  \Nbar_{n,\phi}(t) &=\ip{\tun (t)}{\phi}{H}^2 - \ip{\overline{u}_n(0)}{\phi}{H}^2 \notag \\
  &+2  \int_0^t \duality{{P}_{n} \acal \tun u(s)+B_{n}(\tun (s))}{\phi}\;\duality{\tun (s)}{\phi} \ds\notag  \\
  &-\int_0^t \int_{Y} \langle {P}_{n}  F(s,\tun (s),y),\phi \rangle^2 \,  \nun(\ud y) \ds, \;\; \tinOT.
  \end{aligned}
  \label{eqn-N_n_bar^phi-02}
\end{equation}

\begin{definition}\label{def-filtration_tF_n new}
Let  $\mathscr{N}$ be the family of $\Pbar$-null sets of the $\sigma$-field $\tfcal$. In what follows we will  introduce  a set of auxiliary filtrations. 
For every $n\in \mathbb{N}$,
 let $\mathbb{H}_n=(\mathscr{H}_{n,s}: s\geq 0)$,
 $\mathbb{H}_n^+=(\mathscr{H}_{n,t}^{+}: t\geq 0)$, ${\tF_n}=({\overline{\fcal }}_{n,t}:t\geq 0) $
  be  the filtrations defined by 
\begin{align}
\mathscr{H}_{n,s}&\coloneqq  \sigma \{ \tun (r), \, \, r \le s \}, \;\; s \geq 0, 
\\
\mathscr{H}_{n,t}^{+}&\coloneqq \bigcap_{s>t} \mathscr{H}_{n,s},  \;\; t \geq 0,
\\
{\overline{\fcal }}_{n,t} &\coloneqq \sigma\bigl(  \mathscr{H}_{n,t}^{+} \cup \mathscr{N}\bigr), \;\; t \geq 0,
\label{eqn-filtration_tF_n}
\end{align}
\end{definition}
We have the following result. 
\begin{proposition} \label{prop-filtration_tF_n new}
  The filtration ${\tF_n}$ is the $(\overline{\mathscr{F}},\overline{\mathbb{P}})$-completion,
in the sense of \cite[1.4]{Jacod_Shiryaev}, of the right-continuous filtration $\bigl(\mathscr{H}_{n,t}^{+}\bigr)_{t\geq 0}$ and 
the following equality holds 
\begin{equation}\label{eqn-filtration_limit-2}
 {\overline{\fcal }}_{n,t} = \bigcap_{s>t}  \sigma\bigl(  \mathscr{H}_{n,s} \cup \mathscr{N}\bigr).
\end{equation}
In particular,  the filtration ${\tF_n}$ is the usual augmentation of filtration $\mathbb{H}_n$ and satisfies the usual hypothesis, see   \cite[1.4]{Jacod_Shiryaev}.
    \end{proposition}
    \begin{proof}
        Equality \eqref{eqn-filtration_limit-2} follows by the argument as in the proof of  Lemma \ref{Lem:Indep-Increments-augmentation}, see also \cite[Lemma\ 7.8]{Kallenberg_2002}. 
    \end{proof}

We begin with the following three  more auxiliary results.

\begin{lemma}
\label{lem-continuity on the Skorokhod space}
The function
\[
D([0,T],U') \ni z \mapsto
{\theta }_{n}(|z( \cdot) {|}_{\rU^\prime})
\in D([0,T],\mathbb{R})
\]
is continuous.
\end{lemma}
\begin{proof} Let $z_k \to z$ in $D([0,T],U')$. By the standard characterisation of the Skorokhod topology, see e.g. \cite[Theorem VI.1.14, p. 328]{Jacod_Shiryaev}, we know that there exists a sequence $\lambda_k : [0,T] \to [0.T]$ such that
\begin{equation*}\lim_{k\to + \infty} \sup_{t \in [0,T]} |\lambda_k(t)-t| \to 0, \qquad \lim_{k \to +\infty} \sup_{t \in [0,T]} |z_k(\lambda_k(t))-z(t)|_{U'} \to 0.\end{equation*}
It follows  from the uniform continuity of $\theta_n$ that
\begin{equation*}\lim_{k \to +\infty} \sup_{t \in [0,T]} |\theta_n(|z_k(\lambda_k(t))|_{U'})-\theta_n(|z(t)|_{U'})| \to 0.\end{equation*}
Thus $\theta_n(|z_k|_{U'}) \to \theta_n(|z|_{U'})$ in $D([0,T],\mathbb{R})$.
\end{proof}

\begin{lemma}\label{lem-B.1}
Assume that  $u \in {L}^{2}(0,T;H)$ and   ${({u}_{n} )}_{n} $ is  a bounded sequence in ${L}^{2}(0,T;H)$ such that ${u}_{n} \to u $ in ${L}^{2}(0,T;{H}_{\loc})$. Let 
$s$ be chosen as in \eqref{eqn-number s}.  
Then  for  every  $\psi \in {V}_{s}$,
\begin{equation} \label{eqn-lem-B.1}
  \lim_{n \to \infty } \int_{0}^{T} \vert  \ddual{V_s^\prime}{ B \bigl( {u}_{n} (r) ,{u}_{n} (r) \bigr)  }{ \psi }{V_s}\, -
    \ddual{V_s^\prime}{ B \bigl( u (r),u(r)  \bigr)  }{\psi }{V_s}\, \vert \, dr .
\end{equation}
\end{lemma}

\begin{proof}
See \cite[Lemma B.1]{Brzezniak+Motyl_2013}. But note that our result is stronger although the proof is essentially the same.
\end{proof}

\begin{lemma}
\label{lem-Sk+L1 convergence}
Assume that $x_n \to x $ in  $D([0,T])$ and $u_n \to u$ in $L^1(0,T)$. Then
\[
\int_0^T x_n(t)u_n(t) \dt \to  \int_0^T x(t)u(t) \dt
\]
 \end{lemma}

\begin{proof} There exist positive constant $A>0$ and $C_k>0$ such that
\begin{align}\label{eqn-bounds-x}
A&\coloneqq\sup_{k} \sup_{t \in [0,T]} |x_k(t)| < \infty,
\\
\label{eqn-bounds-u}
C_k&\coloneqq \int_0^T \vert u_k(t)-u(t)\vert \dt \to 0.
\end{align}
Moreover, the set $C$ of discontinuities of function $x$ is at most countable and by \cite[Proposition\ VI.2.1 (b.5)]{Jacod_Shiryaev},
\[
x_k(t) \to x(t) \mbox{ for every } t \in [0,T]\setminus C.
\]
Therefore,
\begin{align}\label{eqn-x_ku_k-xu}
&\left\vert
\int_0^T x_k(t)u_k(t) \dt -  \int_0^T x(t)u(t) \dt \right\vert
\\
\nonumber
& \leq
\int_0^T \vert x_k(t)-x(t) \vert \vert u(t)\vert \dt
+
 \int_0^T  \vert x_k(t))\vert \vert u_k(t)- u(t) \vert \dt
 \\
&\leq \int_0^T \vert x_k(t)-x(t) \vert \vert u(t)\vert \dt
+  AC_k.
\end{align}
By assumptions, $C_k \to 0$ while the first integral on the right-hand side above converges to $0$ by the Lebesgue Dominated Convergence Theorem because the integrand  $\vert x_k(t)-x(t) \vert \vert u(t)\vert$ (i) converges to $ 0$ for almost all $t \in [0,T]$ and (ii) is  bounded by the integrable function $t \mapsto 2 B \vert u(t) \vert$ for some constant $B>0$.
\end{proof}

 \begin{lemma}\label{lem-5.6}
  Let $ \psi \in\rU$ and $n\in \mathbb{N}$ be fixed. Then, the map
\begin{equation}\label{eq-function_J_s_t_psi}
        J^n_{s,t,\psi}\colon \zcal_T \ni z \mapsto \int_{s}^{t}  {\theta }_{n}(|z(r) {|}_{\rU^\prime}) \; b(z(r),z(r),\psi)
        \dr         \in\mathbb{R}
\end{equation}
     is measurable.
 \end{lemma}
Let us point out that in view of inequality \eqref{eqn-b-inequality}, the  assertion in Lemma \ref{lem-5.6} makes sense.

\begin{proof}[Proof of Lemma \ref{lem-5.6}]
Let $ \psi \in U$ and $n\in \mathbb{N}$ be fixed. It is sufficient to prove measurability of $J^n_{s,t,\psi}$ restricted to every subset $\zcal_T^m$ of $\zcal_T$, where
\begin{equation} \label{eqn-Z_cadlag_NS-n}
 \zcal_T^m  \coloneqq \bigl\{ z \in \zcal_T:  \vert z \vert_{L^{2}(0,T;V)} \leq m, \vert z \vert_{L^\infty ([0,T];{\rH})} \leq m\bigr\}
\end{equation}
endowed with the trace topology from $\zcal_T$.
Note that the space $\zcal_T^m$ is metrizable and hence continuity of $J^n_{s,t,\psi}$ restricted to
$\zcal_T^m$ is equivalent to the sequential continuity.

Let us choose and fix a sequence $(x_\ell)_{\ell \in \mathbb{N}}\subset \zcal^m_T$ such that $x_\ell \to x \mbox{  in } \zcal_T^m$.
By the definition of the space $\zcal_T^m$ we see that  the sequence $(x_\ell)_{\ell \in \mathbb{N}}\subset L^2(0,T; H)$ is bounded and $ x_\ell \to x$  in  $L^2(0,T; H_{\loc}).$
Hence, by Lemma \ref{lem-B.1} we infer that  $b(x_\ell, x_\ell, \psi)
\to b(x, x, \psi)$ in $L^1(0,T)$.

On the other hand, by Lemma
\ref{lem-continuity on the Skorokhod space}, 
\[
{\theta }_{n}(|x_\ell( \cdot) {|}_{\rU^\prime}) \to {\theta }_{n}(|x( \cdot) {|}_{\rU^\prime}) \mbox{ in } D([0,T]).
\]
Therefore, by Lemma \ref{lem-Sk+L1 convergence}, we deduce that
\begin{equation*}
 \int_{0}^T
 {\theta }_{n}(|x_\ell(t) {|}_{\rU^\prime}) b(x_\ell{(t)}, x_\ell{(t)}, \psi)
 \dt  \to  \int_{0}^T
{\theta }_{n}(|x(t) {|}_{\rU^\prime})  b(x{(t)}, x{(t)}, \psi)
 \dt .
\end{equation*}
Consequently, the map $J^n_{s,t,\psi}$ is lower semi-continuous, and hence measurable.
\end{proof}

The following results are consequences, via Lemma \ref{lemma-M_N martingale}, of our Assumptions F.\ref{assum-F2}\label{assum-F2-used3}
 and F.\ref{assum-F5}. \label{assum-F5-used-6}

\begin{lemma}\label{lem-tM_n is a martingale}
\begin{enumerate}
\item[(i)] For every $n\in \mathbb{N}$, the process ${\Mbar}_{n} $ is a square integrable $\rU^\prime$-valued ${\tF}_{n}$-martingale. Moreover,
\begin{equation}\label{eqn-2-moment of bar M_n}
 \sup_{n \in \mathbb{N}}  \Ebar \left[ \sup_{\tinOT}\lvert \Mbar_n(t) \rvert^2_{\rU^\prime} \right] <\infty.
\end{equation}
\item[(ii)] Moreover, if Assumption F.\ref{assum-F5} \label{assum-F5-used-7} is satisfied,  then for every $p\in [2, 4+\gamma] $ and for every $n \in \mathbb{N}$, the process
 ${\Mbar}_{n} $ is a $p$-integrable $\rU^\prime$-valued ${\tF}_{n}$-martingale. Moreover,
\begin{equation}\label{eqn-p-moment of bar M_n}
 \sup_{n \in \mathbb{N}} \Ebar \left[ \sup_{\tinOT}\lvert \Mbar_n(t) \rvert^p_{\rU^\prime} \right] <\infty.
\end{equation}
\end{enumerate}
\end{lemma}

\begin{proof}[Proof of Lemma  \ref{lem-tM_n is a martingale}]
We assume here that $p$ is a fixed exponent  belonging to the interval $[2,4+\gamma]$.
For $p=2$ we will use Lemma \ref{lemma-M_N martingale} while  for $p\in (2,4+\gamma]$ we will use the second part of  Lemma \ref{lemma-M_N martingale}.

Following \cite{Jacod_Shiryaev} we denote by $\mathscr{D}_t^0$ the $\sigma$-algebra generated by all maps \begin{equation*}\mathbb{D}([0,T], U^\prime)\ni x \mapsto x(s)\in U^\prime,\, s\le t .\end{equation*} We also set $\mathscr{D}_t=\bigcap_{s>t} \mathscr{D}^0_s,\, t\in [0,T]$. Note that $\mathbf{D}=(\mathscr{D}_t)_{t\in [0,T]}$ is a filtration.
 Let
		\[
		h\colon  \mathbb{D}([0,T];\rU^\prime)  \to \mathbb{R}
		\]
		be a bounded, continuous function and $\mathscr{D}_s$-measurable. Let us also choose and fix $\psi \in\rU$. 

Let us first observe that the process ${M}_{n} $ is a square integrable, $\rU^\prime$-valued $\mathbb{F}$-martingale. Hence, we have
\begin{equation*} \label{eqn-EtMn}
\E \left[ \ddual{\rV^\prime }{ {{M}}_{n} (t)-{{M}}_{n} (s)}{\psi }{\rV}{} \, h\bigl( {u}_{n}  \bigr)\right]
 = 0
\end{equation*}
Moreover, thanks to the linearity of $P_n, \acal$ and Lemma \ref{lem-5.6}, with $\psi$ replaced by $P_n\psi$, we infer  that the function
\begin{equation}\label{eqn-I_st(psi)}
\begin{aligned}
    I_{s,t,\psi}\colon \zcal_T \ni z &\mapsto \Ip{ \int_{s}^{t} {P}_{n} \acal  z(r) \dr
  + \int_{s}^{t} {\rB}_{n}  ( z(r)   ) \dr
  \\& - \int_{s}^{t} {P}_{n} f(r) \dr }{\psi}{}\!{}_{\rH}  \in\mathbb{R}
  \end{aligned}
  \end{equation}
is measurable.

Thus, since by \eqref{eqn-Skorokhod_appl_convergence}, the processes  $\tun$ and $u_n$ have the same laws on $\zcal_T$, the laws of the  $\mathbb{R}$-valued random variables, see \eqref{eqn-M_n},
\[\ddual{\rV^\prime }{ {{M}}_{n} (t)-{{M}}_{n} (s)}{\psi }{\rV}{} \, h \bigl( {u}_{n}  \bigr)
=I_{s,t,\psi}({u}_{n})  h \bigl( {u}_{n} \bigr) 
\]
and, see \eqref{eqn-tMn},
\[\ddual{\rV^\prime }{  \Mbar_n(t)-\Mbar_n (s)}{\psi }{\rV}{} \, h \bigl( \tun  \bigr)
=I_{s,t,\psi}(\tun)  h \bigl( {\tun}\bigr)  
\]
are equal.
Therefore, we infer  that
\begin{equation} \label{eqn-EtMn_tilde}
\mathbb{E} \biggl[\ddual{\rV^\prime }{ {{M}}_{n} (t)-{{M}}_{n} (s)}{\psi }{\rV}{} \, h \bigl( {u}_{n}  \bigr)\biggr]=\overline{\E} \left[ \ddual{\rV^\prime }{  \Mbar_n(t)-\Mbar_n (s)}{\psi }{\rV}{} \,
 h \bigl( \tun  \bigr) \right]
 = 0.
\end{equation}
With this at hand, by using standard Functional Monotone Class argument as in \cite[Proposition IX.1.1]{Jacod_Shiryaev} we see that for any $s,t\in [0,T]$ with $s\le t$ and  $h\colon  \mathbb{D}([0,T];\rU^\prime)  \to \mathbb{R}$ bounded and $\mathscr{D}_s$-measurable we have
\begin{equation*}
\te\bigl[ \duality{ \Mbar_n (t)-\Mbar_n (s) }{\psi }{} \, h \bigl( \tu  \bigr) \bigr]=0,
\end{equation*}
which implies that $\Mbar_n$ is a $\Fbar$-martingale with respect to the filtration generated by $\overline{u}_n$. By Lemma \ref{Lem:Indep-Increments-augmentation} it is a martingale with respect to the filtration $\tF_n$.

 If $p \in [2, 4+\gamma]$, then thanks to the measurability of \eqref{measurable_on_D}, the random variables $\sup_{\tinOT} \lvert M_n(t) \rvert^p_{\rU^\prime}$ and $\sup_{\tinOT} \lvert \Mbar_n(t) \rvert^p_{\rU^\prime}$,  have the same law.
Therefore, by the first part of Lemma \ref{lemma-M_N martingale} in the case when $p=2$,  and by Assumption F.\ref{assum-F5} \label{assum-F5-used-8} and the second part of Lemma \ref{lemma-M_N martingale} in the case when $p \in (2, 4+\gamma]$, we infer that
\begin{equation}\label{Eq:Expect-M_n=tM_n}
\sup_{n \in \mathbb{N}}  \Ebar \left[ \sup_{\tinOT} \lvert \Mbar_n(t) \rvert^{p}_{\rU^\prime} \right]
= \sup_{n \in \mathbb{N}} \E \left[ \sup_{\tinOT} \lvert M_n(t) \rvert^{p}_{\rU^\prime} \right] <\infty.
\end{equation}
Hence, the proof of Lemma  \ref{lem-tM_n is a martingale} is complete.
\end{proof}

The above proof of Lemma \ref{lem-tM_n is a martingale} relies on the fact that the original process $M_n$ is a martingale on the original filtered probability space
$\bigl( {\Omega}, \fcal, {\mathbb{P}}, {\mathbb{F}} \bigr)$, where  ${\mathbb{F}}=(\fcal_t)_{t\geq 0}$.
Similarly, the proof of the following Lemma 
will rely on the fact that the original process $N_{n,\phi}$ defined by \eqref{eqn-N_n,phi} is a martingale on the original filtered probability space, which will be justified using the It\^o formula.

In the following Lemma we use
estimate \eqref{eqn-H_estimate-p>2} 
 with $p=4$, which was proven under  linear growth condition \eqref{eqn-F_linear_growth-p}  in   Assumption F.\ref{assum-F5} with $p=4$.\label{assum-F5-used-9}

\begin{lemma}\label{lem-tN_n is a martingale} Assume that $\phi \in\rU$.
Then the process $\Nbar_{n,\phi}$ is a square integrable $\mathbb{R}$-valued ${\tF}_{n}$-martingale.
\end{lemma}

\begin{proof}
Let $\phi \in\rU$ and $n \in \mathbb{N}$.
Applying the It\^o formula \cite[Theorem II.5.1]{Ikeda+Watanabe_1989}, see also Theorem \ref{thm-Ito-JS-I.4.57},
to the real  semimartingale $X_n= \duality{{u}_{n}}{\phi}$, where the process ${u}_n$ is given by \eqref{eqn-Galerkin}, we get for $t \in [0,T]$
\begin{equation}
\begin{aligned}
\duali{ u_n(t)}{\phi}{2} &= \duali{P_n u_0}{\phi}{2} -2  \int_0^t \duality{P_n \acal  u_n(s)+ \rB_n(u_n(s))}{\phi}\;\duality{ u_n(s)}{\phi} \ds
\\ 
&+\int_0^t \int_{Y} \langle {P}_{n}  F(s,{u}_{n} (s),y),\phi \rangle^2 \, \nun(\ud y) \ds \\
& + 2 \int_0^t \int_{Y} \duality{u_n(s-)}{\phi}\;\duality{P_n F(s,{u}_{n} (s-),y) }{\phi}\, \tilde{\eta}(\ud s,\ud y).
\end{aligned}
\label{eqn-K_6.1.18_A}
\end{equation}
Using the definition \eqref{eqn-N_n,phi} of the process $  N_{n, \phi}$ it follows that equality \eqref{eqn-K_6.1.18_A} can be rewritten in the following way
\begin{align*}
{N}_{n,\phi}(t) =2 \int_0^t \int_{Y} \duality{u_n(s-)}{\phi}\;\duality{P_n F(s,{u}_{n} (s-),y) }{\phi}\, \tilde{\eta}(\ud s,\ud y), \;\;\; t \in [0,T].
\end{align*}
From the linear growth condition \eqref{eqn-F_linear_growth} in Assumption F.\ref{assum-F2}\label{assum-F2-used4} along with the estimate \eqref{eqn-H_estimate-p>2} 
 with $p=4$
  we infer that the process defined on the right hand side of the above equation is a square integrable $\mathbb{F}$-martingale, see \cite[Section\ 3]{Brzezniak_Liu_Zhu}.
This completes the proof of the fact that ${N}_{n,\phi}$ is a square integrable  $\mathbb{F}$-martingale.

With the above result at hand we can closely follow the proof of Lemma \ref{lem-tM_n is a martingale} and show that $\Nbar_{n, \phi}$ is a square integrable  $\tF_n$-martingale.
\end{proof}

Let us recall that according to Remark \ref{rem-process ubar}
the  $ \zcal_T $-valued random variables $\tu$ defined on $\bigl( \Omegabar ,\tfcal ,\Pbar  \bigr) $ induces
a  measurable process (denoted by the same letter)
\[
\tu\colon [0,T]\times \Omegabar \to V \cap {\rH}_{\loc},
\]
whose trajectories belong to the space $\zcal_T$.

Let us introduce the following limiting filtration
\begin{equation}\label{eqn-filtration_limit}
\begin{aligned}
{\tF} &\coloneqq({\overline{\fcal }}_{t})_{\tinOT},  \mbox{ where } {\overline{\fcal }}_{t} = \sigma\bigl(  \mathscr{H}_{t+} \cup \mathscr{N}\bigr)   \\
\mbox{  and } \mathscr{H}_{t+}&\coloneqq \bigcap_{s>t} \mathscr{H}_s \mbox{ with } \mathscr{H}_s\coloneqq  \sigma \{ \tu (r), \, \, r \le s \}.
\end{aligned}\end{equation}

Let us define an  $\rU^\prime$-valued  process ${\Mbar}$  by the following formula (on full $\Omegabar$)
\begin{equation}  \label{eqn-tM}
\begin{aligned}
{\Mbar}  (t)
   &= \tu (t) \,  -    \tu (0)
   - \int_{0}^{t}  f(s) \ds
   \\ &+ \int_{0}^{t}  \acal \tu (s) \ds
  + \int_{0}^{t} {\rB}  \bigl( \tu (s)  \bigr) \ds,  \quad
t \in [0,T].
\end{aligned}
\end{equation}

Also, given $\phi \in\rU$ let us define,  also on full $\Omegabar$, a  $\mathbb{R}$-valued process $\Nbar_\phi$ by

\begin{align}
\begin{aligned}
  \Nbar_\phi(t) &=\duali{\tu (t)}{\phi}{2} - \duali{u_0}{\phi}{2} \\
  &+2  \int_0^t \duality{\acal \tu (s)+\rB(\tu (s))-f(s)}{\phi}\;\duality{\tu (s)}{\phi} \ds
  \\
  &-\int_0^t \int_{Y} \ip{F(s,\tu (s),y)}{\phi}{\rH}^2 \, \nun(\ud y) \ds, \;\; t\in [0,T].
  \end{aligned}
\label{eqn-Nbar^phi}
\end{align}

We have the following important results.

\begin{lemma}\label{lem-tM is a martingale}
The process ${\Mbar} $ is a square integrable $\rU^\prime$-valued ${\tF}$-martingale and all its trajectories belong to the space $\mathbb{D} ([0,T];\rU^\prime)$.
\end{lemma}

\begin{lemma}\label{lem-tN is a martingale}
Let $\phi \in\rU$. Then, the process ${\Nbar_\phi} $ is a square integrable ${\tF}$-martingale.
\end{lemma}
In order to prove Lemmata  \ref{lem-tM is a martingale} and \ref{lem-tN is a martingale} we will need to pass to the limit as $n\to \infty$. For this aim we  will need the following auxiliary results.

\begin{lemma}\label{lem-pointwise_conv}
\begin{enumerate}
\item[(i)] There exists a set $C \subset [0,T]$ such that
\begin{enumerate}
\item[(c1)]\label{item-c1} the complement $ [0,T] \setminus C$ of $C$ is at most countable,
\item[(c2)] $\{0,T\}\subset  C$,
\item[(c3)] for every $t \in C$,
\begin{align}\label{eqn-pointwise_conv-a}
\lim_{n \to \infty } \tun (t)= \tu (t) \mbox{ in } \rU^\prime, \; \mbox{  on $\Omegabar$ }.
\end{align}
\end{enumerate}
\item[(ii)] For all $s,t \in [0,T]$ with $s \le t $ and  every $\psi  \in\rU$, we have on $\Omegabar$,
\begin{align}\label{eqn-pointwise_conv-b}
\lim_{n \to \infty } \int_{s}^{t} \ddual{\rV^\prime }{ \acal_n \tun (\sigma) }{ \psi }{\rV}{} \, \ud\sigma
 =\int_{s}^{t} \ddual{\rV^\prime }{ \acal \tu (\sigma) }{ \psi }{\rV}{}\, \ud\sigma , \quad    \\
 \label{eqn-pointwise_conv-c}
 \lim_{n \to \infty }
 \int_{s}^{t} \ddual{\rV^\prime }{ \rB_n \bigl( \tun (\sigma )  \bigr) }{ \psi }{\rV}{}  \, \ud\sigma
= \int_{s}^{t} \ddual{\rV^\prime }{ \rB\bigl( \tu (\sigma ) \bigr) }{\psi }{\rV}{}  \, \ud\sigma .
\end{align}
\end{enumerate}
\end{lemma}

\begin{proof}[Proof of Lemma \ref{lem-pointwise_conv}]
By assertion \eqref{eqn-conv_D(0,T;U')} we infer that $\tu$ belongs to  $\mathbb{D} ([0,T];\rU^\prime)$ on $\Omegabar$.  Thus,    by \cite[Section \ 13]{Billingsley_1999},
the set
\begin{equation}\label{eqn-C_set}
C =C(\tu) \coloneqq \{ t \in (0,T) : \overline{\mathbb{P}}( \Delta \tu (t) \not= 0) = 0 \} \cup \{ 0,T \},
\end{equation}
satisfies the condition (c1), i.e.\ the set $[0,T]\setminus C$ is at most countable, and (c2).

Let us choose and fix $t \in C$. Then, since $\tun \to \tu $ in $\mathbb{D} ([0,T];\rU^\prime)$  on $\Omegabar$, by \cite[Th.\ 12.5(i)]{Billingsley_1999}  we infer that
\begin{equation*}\label{eqn-weak-Conv-in H}
\tun(t) \to \tu(t) \mbox{  in $\rU^\prime$,  }\mbox{  on $\Omegabar$ }.
\end{equation*}
Hence Claim \eqref{eqn-pointwise_conv-a} holds.

\begin{proof}[We move on to prove assertion \eqref{eqn-pointwise_conv-b}] 
Let us choose  and fix time instants $s,t \in [0,T]$ such that  $s\le t$ and $\psi  \in\rU$.
Next we  observe that  without loss of generality we can assume that $s=0$. Let us now choose and fix  $\omega\in {\Omegabar}$.
Note that by definitions of $\acal_n$ and $\tilde{P}_n$, and equation \eqref{eqn-Acal_ilsk_Dir} we have
\begin{align*}
\ddual{\rV^\prime }{ \acal_n \tun (\sigma) }{ \psi }{\rV}{}
= \ddual{\rV^\prime }{ \acal \tun (\sigma) }{\tilde{P}_{n} \psi }{\rV}{}
= \ip{\nabla \tun (\sigma)}{\nabla \tilde{P}_{n} \psi}{\bL^2}.
\end{align*}
Note also that 
\begin{align*}
\int_{0}^{t} \ddual{\rV^\prime }{ \acal_n \tun (\sigma) }{ \psi }{\rV}{} \, \ud\sigma
&= \int_{0}^{t} \ip{\nabla \tun (\sigma)}{\nabla \tilde{P}_{n} \psi}{\bL^2} \, \ud\sigma  
\\
=\int_{0}^{t} \ip{\nabla \tun (\sigma)}{\nabla  \psi}{\bL^2} \, \ud\sigma  
&+
\int_{0}^{t} \ip{\nabla \tun (\sigma)}{\nabla (\tilde{P}_{n} \psi-   \psi)}{\bL^2} \, \ud\sigma  
\end{align*}
and that by assertion \ref{eqn-conv_L^2_w(0,T;V-dot)}
\begin{align*}
\lim_{n \to \infty }\int_{0}^{t} \ip{\nabla \tun (\sigma)}{\nabla  \psi}{\bL^2} \, \ud\sigma  
=\int_{0}^{t} \ip{\nabla \tu (\sigma)}{\nabla \psi}{\bL^2} \, \ud\sigma
=\int_{0}^{t} \ddual{\rV^\prime }{ \acal \tu (\sigma) }{ \psi }{\rV}{} \, \ud\sigma.
\end{align*}
Moreover,  since  on the one hand, by  Lemma \ref{lem_P_n}, $\tilde{P}_{n} \psi \to \psi $ in $\rU$ and hence,  by the second part of Remark \ref{rem-P_n}, also in $\rV$ 
and,  on the other hand, by inequality \eqref{eqn-V_estimate} in Lemma  \ref{lem-Galerkin_estimates}
 the sequence $(\acal \tun )_{n=1}^\infty$ is bounded in $L^2(0,T;\rV^\prime)$, we infer that 
\begin{align*}
\int_{0}^{t} \ip{\nabla \tun (\sigma)}{\nabla (\tilde{P}_{n} \psi-  \psi)}{\bL^2} \, \ud\sigma &=
\int_{0}^{t}  \ddual{\rV^\prime }{ \acal \tun (\sigma) }{\tilde{P}_{n} \psi - \psi}{\rV}{}  \, \ud\sigma \to 0.  
\end{align*}
Hence the assertion  \eqref{eqn-pointwise_conv-b} follows. 
\end{proof}

\begin{proof}[Now we prove \eqref{eqn-pointwise_conv-c}]
Since by \eqref{eqn-conv_L^2_w(0,T;V)}, as in the proof of \eqref{eqn-pointwise_conv-b},  $\tun \to \tu $
in ${L}^{2}_{w}(0,T;V)$, we infer that  $\tu  \in {L}^{2}(0,T;V)$ on $\overline{\Omega}$ and the sequence
$(\tun (\cdot ,\omega ){)}_{n=1 }^\infty$ is bounded in ${L}^{2}(0,T;V)$ for all $\omega \in \overline{\Omega}$.
Thus for all $\omega \in \overline{\Omega}$. $\tu  \in {L}^{2}(0,T;{\rH})$ and the sequence $(\tun {)}_{n=1 }^\infty$ is bounded in ${L}^{2}(0,T;{\rH})$, as well.

By  \eqref{eqn-boundedness_H}  the sequence $(\tun {)}_{n \ge 1 }$ is  bounded in $B_b([0,T];H)$, i.e.\ we can find  $N=N(\omega) \in \mathbb{N} $ such that
{
\begin{equation}\label{eqn-boundedness_in_U'}
\sup_{n \ge 1} \; \sup_{t \in [0,T]} \;\;\norm{\tun (t) }_{\rH}{} \le N .
\end{equation}
}
Thus, if $n \geq  N $ then  ${\chi }_{n}(\tun ) = \tun $ and consequently
\[
\rB_n\bigl( \tun ,\tun \bigr)=  P_n \rB\bigl( {\chi }_{n}(\tun ),\tun \bigr)
 = P_n \rB\bigl( \tun ,\tun \bigr).
\]
Hence assertion \eqref{eqn-pointwise_conv-c} follows from  \cite[Cor.\ B.2]{Brzezniak+Motyl_2013} because $\tun (\cdot ,\omega ) \to \tu (\cdot ,\omega )$ in ${L}^{2}(0,T,{\rH}_{\loc})$ by
\eqref{eqn-conv_L^2(0,T;H_loc)}   and the sequence $(\tun (\cdot ,\omega ){)}_{n\ge 1 }$ is bounded in ${L}^{2}(0,T;{\rH})$  by \eqref{eqn-boundedness_H}. This completes the proof of  assertion \eqref{eqn-pointwise_conv-c}.
\end{proof}
Since all assertions are proved,   the proof of  Lemma \ref{lem-pointwise_conv} is complete.
\end{proof}

\begin{remark}\label{rem-lem-pointwise_conv}
Let us note that  it follows from equality \eqref{eqn-pointwise_conv-a} in Lemma \ref{lem-pointwise_conv} that for every $t\in C$ and every $\psi \in U $
		\[\lim_{n \to \infty } \ip{\tun (t)}{ \psi }{\rH} = \ip{\tu (t)}{\psi }{\rH},
		 \quad  \Pbar \mbox{ - a.s.}\]
We now claim that the above convergence also holds for all $\psi \in H$.

Indeed, let us fix $\psi \in H$. Since $U$ is dense in H, we infer that for every $\varepsilon >0$ there exists ${\psi }_{\varepsilon } \in U$ such that $|\psi -{\psi }_{\varepsilon }|_H < \varepsilon $.
We have
\[
({\bar{u}}_{n}(t) -\bar{u}(t),\psi )_H 
= ({\bar{u}}_{n}(t) -\bar{u}(t),\psi -{\psi }_{\varepsilon })_H  + ({\bar{u}}_{n}(t) -\bar{u}(t),{\psi }_{\varepsilon })_H 
\]
and hence
\[
|({\bar{u}}_{n}(t) -\bar{u}(t),\psi )_H|  
\le  (|{\bar{u}}_{n}(t)|_H  +|\bar{u}(t)|_H )\cdot |\psi -{\psi }_{\varepsilon }|_H  + |({\bar{u}}_{n}(t) -\bar{u}(t),{\psi }_{\varepsilon })_H|  .
\]
By  \eqref{eqn-boundedness_H}  we can find $K=K(\omega )$ such that 
\[
\sup_{n \in \mathbb{N}} \bigl( \sup_{t\in [0,T]} |{\bar{u}}_{n}(t)|_H + \sup_{t\in [0,T]} |\bar{u}(t)|_H \bigr) \le K .
\]
Thus
\[
|({\bar{u}}_{n}(t) -\bar{u}(t),\psi )_H| \le \varepsilon K + |({\bar{u}}_{n}(t) -\bar{u}(t),{\psi }_{\varepsilon })_H| .
\]
Passing to the upper limit as $n \to \infty $ and using  $\lim_{n\to \infty } ({\bar{u}}_{n}(t) -\bar{u}(t),{\psi }_{\varepsilon })_H =0 $, we obtain
\[
\limsup_{n\to \infty } |({\bar{u}}_{n}(t) -\bar{u}(t),\psi )_H| \le \varepsilon K .
\]
Since $\varepsilon $ is arbitrary, we infer that 
\[
\lim_{n\to \infty } ({\bar{u}}_{n}(t)-\bar{u}(t),\psi )_H =0,
\]
which completes the proof of the claim.
\qed

Let us also emphasize that  the set $C$ is chosen according to the process $\tu$ being considered as $\rU^\prime$-valued c\`adl\`ag process.
		
Let us also observe  that both the left-hand sides  of the equalities \eqref{eqn-pointwise_conv-b} and \eqref{eqn-pointwise_conv-c} in Lemma \ref{lem-pointwise_conv}  make sense as the ${\rH}$-scalar product.
However, the right-hand sides of \eqref{eqn-pointwise_conv-b} and \eqref{eqn-pointwise_conv-c}  make sense only as the $V^\prime-V$ (or $\rU^\prime-U$) duality. Moreover,
in order to prove these two identities  we need to pass to the limit and for this purpose we  use the $V^\prime-V$-duality and we have to consider  $\psi \in \rU$, see Remark
\ref{rem-P_n} for an explanation why taking $\psi \in V$ would not work.\\
This and the previous ones apply as well to Lemma \ref{L:pointwise_conv-B} below.
\end{remark}

We now give the promised proof of Lemma \ref{lem-tM is a martingale}.

\begin{proof}[{Proof of Lemma \ref{lem-tM is a martingale}}]	
Let us first observe that in view of the paths properties of the process $\ubar$,  all trajectories of the process ${\Mbar} $
   belong to  the space  $\mathbb{D} ([0,T];\rU^\prime)$.

		The rest of the proof will be divided into three steps.
		
\begin{proof}[\textbf{Step 1}]   The aim of this step is to prove the square integrability of $\Mbar$. For this purpose let us fix an arbitrary time instant $t\in C$, where $C$ is the set from the first part  of Lemma \ref{lem-pointwise_conv}.
Then, by \eqref{eqn-tMn}  we have
\begin{align*}
\duality{{\Mbar}_{n} (t)}{\psi }{}
= & \duality{\tun (t)}{ \psi } -\duality{P_n \bar{u}(0)}{\psi }
+ \int_{0}^{t} \duality{ {P}_{n} \acal \tun (\sigma) }{ \psi }{} \ud\sigma \\
& +\int_{0}^{t} \duality{ {P}_{n} \rB \bigl( {\chi }_{n}(\tun (\sigma  )),\tun (\sigma ) \bigr)  }{ \psi }{}\ud\sigma \\
& -\int_{0}^{t} \duality{ {P}_{n} f(\sigma) }{ \psi }{} \ud\sigma .
\end{align*}
By   Lemma \ref{lem-pointwise_conv}, we infer that
\begin{equation*} \label{eqn-new-mart_pointwise_conv}
\lim_{n \to \infty }  \duality{{\Mbar}_{n} (t)}{\psi }{} = \duality{\Mbar (t)}{\psi }{},		\quad  \Pbar \mbox{ - a.s.},
\end{equation*}
which clearly implies that $\Pbar$-a.s.
\begin{equation*}
\Mbar_n(t) \to \Mbar(t) \mbox{  weak}^\ast \mbox{  in }\rU^\prime.
\end{equation*}
The latter assertion implies that
		\begin{equation*}
		\lvert \Mbar(t) \rvert_{\rU^\prime}^2\le\liminf_{n \to \infty}  \lvert \Mbar_n(t) \rvert_{\rU^\prime}^2, \;\; \mbox{ $\Pbar$-a.s.}.
		\end{equation*}
		 By applying the Fatou Lemma we see that
		\begin{align*}
		\Ebar \left[ \lvert \Mbar(t) \rvert^2_{\rU^\prime} \right] \le \Ebar \left[ \liminf_{n \to \infty } \lvert \Mbar_n(t) \rvert_{\rU^\prime}^2   \right] \leq \liminf_{n \to \infty } \Ebar  \left[ \lvert \Mbar_n(t) \rvert_{\rU^\prime}^2 \right],
		\end{align*}
		which implies that
		\begin{align}\label{Eq:p-integrability of tM}
		\Ebar \left[ \lvert \Mbar(t) \rvert^2_{\rU^\prime} \right] \le \sup_{n \in \mathbb{N}} \sup_{r\in [0,T]} \Ebar  \left[ \lvert \Mbar_n(r) \rvert_{\rU^\prime}^2 \right]
  = \sup_{n \in \mathbb{N}} \sup_{r\in [0,T]} \E \left[ \lvert M_n(r) \rvert^2_{\rU^\prime} \right] < \infty
		\end{align}
where the last equality is due to \eqref{Eq:Expect-M_n=tM_n} and the finiteness is due to  Lemma \ref{lemma-M_N martingale}.
Thus we infer that
		\begin{equation}\label{eqn-new-Square-Int-on-C}
		\sup_{r\in C }\Ebar \left[ \lvert \Mbar(r) \rvert_{\rU^\prime}^2 \right] \le \sup_{n \in \mathbb{N}}\sup_{s\in [0,T]} \E \left[ \lvert M_n(s) \rvert_{\rU^\prime}^2 \right] <\infty,
		\end{equation}
		which shows that $\Mbar(t)$ is square integrable for all $t\in C$.
		
		Our next goal is to show that  $\Mbar(t)$ is square integrable for every $\tinOT$. In order to achieve this aim we fix an arbitrary $\tinOT$ and consider a $C$-valued sequence $(t_m)_{m \in \mathbb{N}}$ such that $t_m \searrow t$ as $m \to \infty$.
		With this in mind, we deduce from the right-continuity of the paths of $\Mbar$ that $\Pbar$-a.s.
		\begin{equation*}
		\Mbar(t_m) \to \Mbar(t) \mbox{  in }\rU^\prime.
		\end{equation*}
		Thus, similarly to  the proof of assertion \eqref{Eq:p-integrability of tM} we can prove that
		\begin{equation*}
		\Ebar \left[ \lvert \Mbar(t) \rvert^2_{\rU^\prime} \right] \le \sup_{m \in \mathbb{N}} \Ebar \left[ \lvert \Mbar(t_m)\rvert^2_{\rU^\prime} \right].
		\end{equation*}
		With this inequality at hand and \eqref{eqn-new-Square-Int-on-C} we now easily conclude that for all $\tinOT$
		we have \begin{align*} \Ebar \left[ \lvert \Mbar(t) \rvert^2_{\rU^\prime} \right]<\infty.\end{align*}
This completes the proof in Step 1.
\end{proof}

\begin{proof}[\textbf{Step 2}]
Following \cite{Jacod_Shiryaev} we denote by $\mathscr{D}_t^0$ the $\sigma$-algebra generated by all maps \begin{equation*}\mathbb{D}([0,T], U^\prime)\ni x \mapsto x(s)\in U^\prime,\, s\le t .\end{equation*} We also set $\mathscr{D}_t=\bigcap_{s>t} \mathscr{D}^0_s,\, t\in [0,T]$. Note that $\mathbf{D}=(\mathscr{D}_t)_{t\in [0,T]}$ is a filtration.
Let $C$ be the set from the assertion (a) of Lemma \ref{lem-pointwise_conv}.
		Let us choose and fix $s,t \in C$,  with $s \le t$, and  $\psi \in\rU$. {Let us note that the almost sure set considered below is independent of $\psi$.}   Let
		\[
		h\colon  \mathbb{D}([0,T];\rU^\prime)  \to \mathbb{R}
		\]
		be a bounded, continuous function and $\mathscr{D}_s$-measurable.
  The aim of this step is to show that 
\begin{equation}\label{equality_for_martingale_h}
    \te\bigl[ \duality{ \Mbar (t)-\Mbar (s) }{\psi }{} \, h \bigl( \tu  \bigr) \bigr]=0.
\end{equation}
  
For this purpose, we start observing that by \eqref{eqn-tMn}  we have
\begin{align*}
&\duality{{\Mbar}_{n} (t)-{\Mbar}_{n} (s)}{\psi }{} \\
& \qquad=  \duality{\tun (t)}{ \psi } - \duality{\tun (s)}{ \psi} + \int_{s}^{t} \duality{ {P}_{n} \acal \tun (\sigma) }{ \psi }{} \, \ud\sigma \\
& \qquad\quad +\int_{s}^{t} \duality{ {P}_{n} \rB \bigl( {\chi }_{n}(\tun (\sigma  )),\tun (\sigma ) \bigr)  }{ \psi }{}\, \ud\sigma -\int_{s}^{t} \duality{ {P}_{n} f(\sigma) }{ \psi }{} \, \ud\sigma.
\end{align*}
		By   Lemma \ref{lem-pointwise_conv}, we infer that
		\begin{equation} \label{eqn-mart_pointwise_conv}
		\lim_{n \to \infty }  \duality{{\Mbar}_{n} (t)-{\Mbar}_{n} (s)}{\psi }{} = \duality{\Mbar (t)-\Mbar (s)}{\psi }{},
		\quad  \Pbar \mbox{ - a.s.}
		\end{equation}
		Since  $\tun \to \tu $ $\Pbar$-a.s. in $\mathbb{D}([0,T]; U^\prime)$ and $h$ is a continuous function, we infer that 
		\begin{equation}\label{eqn-mart_pointwise_conv-hres}
		\lim_{n \to \infty }h(\tun ) =h( \tu),\;\; \mbox{$\Pbar $-a.s.}.
		\end{equation}
		Let us define a function $f_n\colon  \Omegabar \to \mathbb{R}$ by
		\begin{equation*}
		{f}_{n}(\omega ) \coloneqq \bigl( \duality{{\Mbar}_{n} (t, \omega )}{\psi }{} - \duality{{\Mbar}_{n} (s, \omega )}{\psi }{} \bigr)
		\, h \bigl( \tun \bigr) , \qquad \omega \in \Omegabar .
		\end{equation*}
		We will prove that the sequence  $\{ {f}_{n} {\} }_{n \in \mathbb{N} }$ is  uniformly integrable.
		For this purpose we  claim  that
		\begin{equation} \label{eqn-mart_uniform_int}
		\sup_{n \ge 1}  \Ebar \bigl[ {| {f}_{n} |}^{2} \bigr] < \infty.
		\end{equation}
		Indeed, from the boundedness of the function $h$, the triangle and the Cauchy-Schwarz inequalities  we infer that
		\begin{align*}
		\Ebar \bigl[ {| {f}_{n} |}^{2} \bigr]  \le 2 \lvert h \rvert_{L^\infty}^{2} \Vert \psi \Vert^{2}_{U} \left( \lVert \Mbar_{n}(t) \rVert^{2}_{\rU^\prime} + \lVert \Mbar_{n}(s) \rVert^{2}_{\rU^\prime}  \right),
		\end{align*}
	which along with the inequality \eqref{eqn-2-moment of bar M_n} imply \eqref{eqn-mart_uniform_int}.

The just proven uniform integrability of the sequence $\{ {f}_{n} {\} }_{n \in \mathbb{N} }$ together with the  Vitali  Convergence Theorem, the assertions 		(\ref{eqn-mart_pointwise_conv}) and \eqref{eqn-mart_pointwise_conv-hres}, the martingale property of $\Mbar_{n}$  and the $\mathscr{D}_s$-measurability of $h$  imply that
\begin{align*}
		\lim_{n \to \infty } \te\bigl[ \duality{{\Mbar}_{n} (t)-{\Mbar}_{n} (s) }{\psi }{}\, h \bigl( \tun  \bigr) \bigr]
		= \te\bigl[ \duality{ \Mbar (t)-\Mbar (s) }{\psi }{} \, h \bigl( \tu  \bigr) \bigr]=0.
\end{align*}
Hence, for all $s,t\in C$ with $s\le t$ we have
\begin{equation*}
    \te\bigl[ \duality{ \Mbar (t)-\Mbar (s) }{\psi }{} \, h \bigl( \tu  \bigr) \bigr]=0
\end{equation*}

\end{proof}

\begin{proof}[\textbf{Step 3}] 
Our aim in this step is to prove the identity 
\begin{equation*}\te\bigl[ \duality{ \Mbar (t)-\Mbar (s) }{\psi }{} \, h \bigl( \tu  \bigr) \bigr]=0,\end{equation*} 
for all $t,s \in [0,T]$ with $s\le t$, which will imply that $\Mbar$ is a $\Fbar$-martingale. 
For this purpose we choose and fix  $t,s \in [0,T]$ with $s< t$ and  we consider two $C$-valued  sequences $(t_m)_{m \in \mathbb{N}}, (s_m)_{m \in \mathbb{N}}$ satisfying
		\begin{itemize}
			\item $s<s_m < t <t_m$ for every $m \in \mathbb{N}$,
			\item $s_m \searrow s$ and $t_m \searrow t$ as $m \to \infty$.
		\end{itemize}
  We also choose and fix \[
		h\colon  \mathbb{D}([0,T];\rU^\prime)  \to \mathbb{R}
		\]
	 a bounded, continuous function and $\mathscr{D}_s$-measurable.
  Note that thanks to \eqref{equality_for_martingale_h} we have
  \begin{equation*}
      \te\bigl[ \duality{ \Mbar (t_m)-\Mbar (s_m) }{\psi }{} \, h \bigl( \tu  \bigr) \bigr]=0
  \end{equation*}
Thanks to the right-continuity of the paths of the process $\duality{\Mbar}{\psi}{}$
		we have,  $\Pbar$-a.s.
		\begin{align}
		h(\tu) \duality{\Mbar(s_m)}{\psi}{} & \to h(\tu) \duality{\Mbar(s)}{\psi}{},\label{eqn-mart-convalongD}\\
		h(\tu)\duality{\Mbar(t_m) }{\psi}{} & \to  h(\tu) \duality{\Mbar(t) }{\psi}{}.\label{eqn-mart-tildem-rightcont}
\end{align}

Also, from   \eqref{eqn-new-Square-Int-on-C}, the convergence \eqref{eqn-mart_pointwise_conv} and the boundedness of $h$ we deduce that
\begin{align*}
		\sup_{r \in C }\Ebar \left[ \lvert h(\tu) \duality{\Mbar(r)}{\psi}{}\rvert^2 \right] <\infty,
\end{align*}
which implies that
\begin{align*}
	\sup_{m \in \mathbb{N}  }\Ebar \left[ \lvert h(\tu) \duality{\Mbar(s_m)}{\psi}{}\rvert^2 \right] +
		\sup_{m \in \mathbb{N}  }\Ebar \left[ \lvert h(\tu) \duality{\Mbar(t_m)}{\psi}{}\rvert^2  \right] < \infty.
\end{align*}
Thus, the sequences $\duality{\Mbar(t_m)}{\psi}{} h(\tu)$ and $\duality{\Mbar(s_m)}{\psi}{} h(\tu)$ are uniformly integrable which together with \eqref{eqn-mart-convalongD},  \eqref{eqn-mart-tildem-rightcont}, the Vitali Convergence Theorem and \eqref{equality_for_martingale_h} imply  that
\begin{align}
\te\bigl[ \duality{ \Mbar (t)-\Mbar (s) }{\psi }{} \, h \bigl( \tu  \bigr) \bigr]= \lim_{m\to \infty}	\te\bigl[ \duality{ \Mbar (t_m)-\Mbar (s_m) }{\psi }{} \, h \bigl( \tu  \bigr) \bigr]=0.	
\end{align}
This completes the proof of \textbf{Step 3}.  \end{proof}
Having established Steps 1-3, we  completed the proof of Lemma \ref{lem-tM is a martingale}.
\end{proof}

In order to prove Lemmata  \ref{lem-tN is a martingale} we will need to pass to the limit as $n\to \infty$. For this aim, we  will need the following additional auxiliary result, compare with previous Lemma \ref{lem-pointwise_conv}.

\begin{lemma}\label{L:pointwise_conv-B}
The following assertions hold on $\Omegabar$.
\begin{enumerate}
\item[(i)] If $s,t \in [0,T]$ such that $s \le t $ and   $\phi,\psi  \in\rU$, then
\begin{align}\label{eqn-conv_B_i)}
	 &\lim_{n \to \infty } \int_{s}^{t} \duality{ \acal_n \tun (\sigma) }{ \psi }{}\; \duality{\tun(\sigma)}{\phi}{} \, \ud\sigma
	\\&\hspace{2truecm}=\int_{s}^{t} \duality{ \acal \tu (\sigma) }{ \psi }{}\;\duality{\tu(\sigma)}{\phi}{}\, \ud\sigma.
\nonumber
\end{align}
\item[(ii)] If $\phi \in  {\rH}$, then  in  $L^2([0,T]\times Y,\Leb\otimes \nu)$, 
\begin{align}\label{eqn-F.4_stronger_final}
	&\ip{P_n F(r,\tun(r),y)}{\phi}{H}  \to \ip{F(r,\tu(r),y)}{\phi}{H}.
\end{align}
In particular, for every $\phi \in  U$,
\begin{align}\label{eqn-conv_F_iii)}
	 &\lim_{n \to \infty } \int_s^t \int_{Y} \left(P_n F(r,\tun(r),y),  \phi \right)_H^2 \, \nun(\ud y) \dr 
         \\\nonumber 
	&\hspace{2truecm}
     \int_s^t \int_{Y} \left(  F(r,\tu (r),y),\phi \right)_H^2 \, \nun(\ud y) \dr.
\end{align}

\item[(iii)] Moreover, if $\phi,\psi  \in U$, then
\begin{align}\label{eqn-conv_B_ii)}
			&\lim_{n \to \infty }
	\int_{s}^{t} \duality{ \rB_n \bigl( \tun (\sigma )  \bigr) }{\psi }{} \;\duality{\tun(\sigma)}{\phi}{} \, \ud\sigma
    \\
	&\hspace{2truecm}= \int_{s}^{t} \duality{ \rB\bigl( \tu (\sigma ) \bigr) }{\psi }{}\;\duality{\tu(\sigma)}{\phi}{}  \, \ud\sigma.
\nonumber
\end{align}
\end{enumerate}
\end{lemma}

\noindent
Let us remind that the inner product in the Hilbert space $\rH$ we denote  by   $ \ip{ \cdot}{ \cdot}{\rH}$.

\begin{proof}[Proof of Lemma \ref{L:pointwise_conv-B}] Without loss of generality we can assume that $s=0$. Let us choose and fix $t \in [0,T]$,    $\phi,\psi  \in\rU$ and $\omega\in \Omegabar$.

\begin{proof}[\textbf{Proof of (i)}] We begin with proving
	\begin{equation*}\label{eqn-conv_B_i_01}
	\lim_{n \to \infty} \int_{0}^{t} \ip{\nabla \tun (s)}{\nabla \psi}{\bL^2} \; \duality{\tun(s)}{\phi}{} \ds
	= \int_{0}^{t} \ip{\nabla \tu (s)}{\nabla \psi}{\bL^2}  \;\duality{\tu(s)}{\phi}{}\ds
	\end{equation*}
Indeed, we have
\begin{align}\nonumber
& \int_{0}^{t} \ip{\nabla \tun (s)}{\nabla \psi}{\bL^2}  \; \duality{\tun(s)}{\phi}{} \ds - \int_{0}^{t} \ip{\nabla \tu (s)}{\nabla \psi}{\bL^2} \;\duality{\tu(s)}{\phi}{}\ds \\
& = \int_0^T \ip{\nabla (\tun (s) -\tu(s))}{\mathbf{1}_{[0,t]}(s)\ip{\tu(s)}{\phi}{} \nabla \psi}{\bL^2}  \ds
\\
&+ \int_{0}^{t} \ip{\nabla \tun (s)}{\nabla \psi}{\bL^2} \ip{\tun(s)-\tu(s) }{\phi}{} \ds  =:  I_n(\phi) + I{\!}I_n(\phi)
\label{eqn-conv_B_i_01-additional}
\end{align}
The convergence of the first integral on the right-hand side of equality  \eqref{eqn-conv_B_i_01-additional} to 0 follows from Lemma \ref{lem_conv_L^2_w(0,T;H)} because $\tun  -\tu \to 0$ in $L^2_{w}(0,T;V)$    by \eqref{eqn-conv_L^2_w(0,T;V)}, and the function $[0,T] \ni s \mapsto \mathbf{1}_{[0,t]}(s)  \ip{\tu(s)}{\phi}{} \psi \in V$ belongs to $L^2(0,T;V)$.
The proof that the second integral above converges to $0$ consists of two steps. In the first step  we take $\phi \in \vcal$ and then choose $R\in \mathbb{N}$ such that $\supp (\phi) \subset B_R$. Then, also because $\normb{\phi}{L^2(B_R)}{} \leq \normb{\phi}{\rH}{}$,  we get
\begin{align*}
	I{\!}I_n(\phi) & \leq  \normb{\psi}{\rV}{}\normb{\phi}{L^2(B_R)}{} \left(\int_0^T \normb{\tun(s)}{\rV}{2} \ds\right)^{1/2}    \left(\int_0^T \normb{\tun(s)-\tu(s)}{L^2(B_R)}{2} \ds\right)^{1/2}
\end{align*}
which converges to $0$ as $n\to \infty$ by \eqref{eqn-conv_L^2(0,T;H_loc)} and \eqref{eqn-boundedness_L^2(0,T,V)}.
In the second step we take an arbitrary $\phi \in {\rH}$ and an arbitrary $\eps>0$. By the density of $\vcal$ in ${\rH}$  we find $\phi_\eps  \in \vcal$ such that $\normb{\phi-\phi_\eps}{\rH}{}<\eps (1+2 \sup_n \normb{\tun}{L^2(0,T;V)}{})^{-2}$ and
	then we find $N_\eps\in \mathbb{N}$ such that
	$I{\!}I_n(\phi_\eps)\leq \eps$ for all $n \geq N_\eps$. We thus get, for all $n\geq N_\eps$,
	\[I{\!}I_n(\phi)=I{\!}I_n(\phi_\eps)+I{\!}I_n(\phi_\eps-\psi)<\frac{\eps}{2}+\normb{\phi-\phi_\eps}{\rH}{}\normb{\tun}{L^2(0,T;V)}{2} \leq \eps,
	\]
	what completes the proof that $\lim_{n\to \infty}I{\!}I_n(\phi)=0$ for every $\phi\in {\rH}$.

Next we observe that because $\normb{{P}_{n} \psi- \psi}{\rV}{} \to 0$ and, by \eqref{eqn-boundedness_L^2(0,T,V)},  \begin{equation*}\sup_n  \int_{0}^{T} \normb{ \tun (s) }{\rV}{^2} \ds <\infty,\end{equation*} we have
\begin{equation*}\label{eqn-conv_B_i_02}
	\lim_{n \to \infty} \int_{0}^{t} \ip{\nabla \tun (s)}{\nabla P_n\psi}{\bL^2} \; \duality{\tun(s)}{\phi}{} \ds
	= \int_{0}^{t}  \ip{\nabla \tu (s)}{\nabla \psi}{\bL^2}  \;\duality{\tu(s)}{\phi}{}\ds.
\end{equation*}
We conclude the proof by observing that
\begin{align*}
	& \biggl\vert
	\int_{0}^{t} \duality{ \acal_n \tun (s) }{ \psi }{} \; \duality{\tun(s)}{\phi}{} \ds
	- \int_{0}^{t} \duality{ \acal \tu (s) }{ \psi }{}\;\duality{\tu(s)}{\phi}{}\ds \biggr\vert
	\\
	&= \biggl\vert
	\int_{0}^{t} \duality{ \acal \tun (s) }{{P}_{n} \psi }{} \; \duality{\tun(s)}{\phi}{} \ds
	- \int_{0}^{t} \duality{ \acal \tu (s) }{ \psi }{}\;\duality{\tu(s)}{\phi}{}\ds \biggr\vert
	\\
	&= \biggl\vert
	\int_{0}^{t} \ip{\nabla \tun (s)}{\nabla P_n \psi}{\bL^2} \; \duality{\tun(s)}{\phi}{} \ds
	- \int_{0}^{t} \ip{\nabla \tu (s)}{\nabla \psi}{\bL^2} \;\duality{\tu(s)}{\phi}{}\ds \biggr\vert.
\end{align*}
Hence the proof of claim \eqref{eqn-conv_B_i)} is complete.
\end{proof}

\begin{proof}[\textbf{Proof of (ii)}]
Let us temporarily denote
$Z=[0,T]\times Y$, $m=\Leb\otimes \nu$ and
	\begin{align*}
	\xi_n(r,y)=F(r,\tun(r,\omega),y),\;\;\; & \xi(r,y)=F(r,\tu(r,\omega),y), \;\; (r,y) \in Z.
	\end{align*}
Then, by  \eqref{eqn-conv_L^2(0,T;H_loc)} and Assumption F.\ref{assum-F3} \label{assum-F3-used trully-1} we infer that 	
	\begin{equation}\label{eqn-F.4_weaker}
 \ip{\xi_n}{\phi}{}  \to \ip{\xi}{\phi}{} \mbox{ in  } L^2(Z,m), \mbox{for every } \phi \in  \vcal.
	\end{equation}

Next, by  the linear growth condition \eqref{eqn-F_linear_growth} in Assumption F.\ref{assum-F2}\label{assum-F2-used5}  and \eqref{eqn-conv_L^2_w(0,T;H)}, there exists a constant $C_1=C_1(\omega)$ such that
	\begin{equation}\label{eqn-uniform_xi_n}
 \begin{aligned}
	\int_Z \normb{\xi_n(z)}{\rH}{2} m(\ud z)&= \int_0^T \int_Y \normb{\xi_n(r,y)}{\rH}{2} \,\nu(\ud y)\dr
\\
&\leq C  \int_0^T (1+ \normb{\tun(r)}{\rH}{2})\dr \leq CT(1+C_1), \;\; n\in\mathbb{N}.
\end{aligned}
	\end{equation}
	Therefore, we infer that the following stronger version of \eqref{eqn-F.4_weaker} holds:
	\begin{align}\label{eqn-F.4_stronger}
	\ip{\xi_n}{\phi}{}  \to \ip{\xi}{\phi}{} \mbox{ in  } L^2(Z,m),  \mbox{ for every }\phi \in  H.
	\end{align}
	Next we claim that even stronger than above assertion holds, i.e.
	\begin{align}\label{eqn-F.4_stronger_23}
	\ip{\xi_n}{P_n\phi}{}  \to \ip{\xi}{\phi}{} \mbox{ in  } L^2(Z,m),  \mbox{ for every }\phi \in  H.
	\end{align}
	Indeed, let us choose and fix  $\phi \in  {\rH}$. Then we have
	\begin{align*}
	\Vert  \ip{\xi_n}{P_n\phi}{}  - \ip{\xi}{\phi}{}  \Vert_{L^2(Z,m)}
	&\leq  \Vert  \ip{\xi_n}{\phi}{}  - \ip{\xi}{\phi}{}  \Vert_{L^2(Z,m)}  +\Vert  \ip{\xi_n}{P_n\phi-\phi}{}    \Vert_{L^2(Z,m)} \\
	&\leq \Vert  \ip{\xi_n}{\phi}{}  - \ip{\xi}{\phi}{}  \Vert_{L^2(Z,m)}  + \normb{P_n\phi-\phi}{\rH}{}   \normb{\xi_n}{L^2(Z,m)}{}
	\end{align*}
	The right-hand side converges to $0$ because of \eqref{eqn-F.4_stronger} and \eqref{eqn-uniform_xi_n}.
This completes the proof of assertion \eqref{eqn-F.4_stronger_final}.
Observe that \eqref{eqn-conv_F_iii)} follows from \eqref{eqn-F.4_stronger_23} because it implies that
	\[
	\normb{\ip{\xi_n}{P_n\phi}{}}{L^2(Z,m)}{2} \to  \normb{\ip{\xi}{\phi}{}}{L^2(Z,m)}{2}
	\]
which is equivalent to \eqref{eqn-conv_F_iii)}.
This completes our alternative proof of  claim \eqref{eqn-F.4_stronger_final}.
\end{proof}

\begin{proof}[\textbf{Proof of (iii)}] We will prove first the following, easier than \eqref{eqn-conv_B_ii)}, assertion.
	\begin{align}
	\label{eqn-conv_02}
	&\int_0^t \duality{ \rB(\tun (s))}{\psi}\;\duality{\tun (s)}{\phi} \, \ds \to \int_0^t \duality{ \rB(\tu (s))}{\psi}\;\duality{\tu (s)}{\phi} \, \ds.
	\end{align}
	
	We begin with a case when  $\psi,\phi  \in \vcal$. Then we  choose $R\in \mathbb{N}$ such that $\supp (\phi) \subset B_R$.
	We have the following chain of inequalities, see the proof of \cite[Lemma\ B.1]{Brzezniak+Motyl_2013},
	\begin{align*}
	& \biggl\vert \int_0^t \duality{ \rB(\tun (s))}{\psi}\;\duality{\tun (s)}{\phi} \, \ds - \int_0^t \duality{ \rB(\tu (s))}{\psi}\;\duality{\tu (s)}{\phi} \, \ds \biggr\vert \\
	&\leq \biggl\vert \int_0^t b(\tun (s)-\tu(s),\tun (s), \psi) \ip{\tu (s)}{\phi}{} \, \ds \biggr\vert +
	\biggl\vert \int_0^t b(\tu(s),\tun (s)-\tu(s), \psi) \ip{\tu (s)}{\phi}{} \, \ds \biggr\vert\\
	&+\biggl \vert \int_0^t b(\tun(s),\tun (s), \psi) \ip{\tun (s)-\tu (s)}{\phi}{} \, \ds \biggr\vert\\
	&\leq c \normb{\phi}{}{}\normb{\psi}{V_s}{} \biggl( \int_0^t \normb{\tun (s)-\tu(s)}{H_{B_R}}{} \normb{\tun (s)}{\rH}{}\normb{\tun (s)}{\rH}{}  \, \ds
	\\
	&+ \int_0^t \normb{\tu(s)}{\rH}{}\normb{\tun (s)-\tu(s)}{H_{B_R}}{}\normb{\tun (s)}{\rH}{}  \, \ds \\
	&+\int_0^t \normb{\tun(s)}{\rH}{}\normb{\tun (s)-\tu(s)}{H_{B_R}}{}\normb{\tun-\tu(s) (s)}{H_{B_R}}{}  \, \ds
	\biggr)
	\\
	&\leq c \normb{\phi}{}{}\normb{\psi}{V_s}{}   \biggl( T^{\frac12}\sup_{s \in [0,T]} \normb{\tun (s)}{\rH}{^2} \biggl( \int_0^T \normb{\tun (s)-\tu(s)}{H_{B_R}}{2}   \, \ds \biggr)^{\frac12}
	\\
	&+ T^{\frac12} \sup_{s \in [0,T]} \normb{\tu(s)}{\rH}{}\normb{\tun (s)}{\rH}{}  \biggl(\int_0^t \normb{\tun (s)-\tu(s)}{H_{B_R}}{}  \, \ds  \biggr)^{\frac12} \\
	&+\sup_{s \in [0,T]} \normb{\tun(s)}{\rH}{}  \int_0^t \normb{\tun (s)-\tu(s)}{H_{B_R}}{2}  \, \ds.
	\biggr)
	\end{align*}
Thus the claim \eqref{eqn-conv_02} follows from claims \eqref{eqn-conv_L^2(0,T;H_loc)} and \eqref{eqn-boundedness_H}. Because the space $\vcal $ is dense in $V_s$, we can easily deduce that \eqref{eqn-conv_02}  holds for all  $\psi, \phi \in U$.  The details are omitted. An interested reader can compare with  the last part of the proof of \cite[Lemma\ B.1]{Brzezniak+Motyl_2013}.
Hence the proof of claim \eqref{eqn-conv_B_ii)} is complete.
\end{proof}

All this  completes the proof of the assertion (iii) and also the Lemma \ref{L:pointwise_conv-B}.
\end{proof}

We are now ready to present the proof of Lemma \ref{lem-tN is a martingale}.
\begin{proof}[Proof of Lemma \ref{lem-tN is a martingale}]	
	The proof is very similar to the proof of Lemma \ref{lem-tM is a martingale}, thus we will only outline the  main differences of the two proofs.
	
		\noindent \textbf{Step 1.} The aim in this step is to prove that for each $\tinOT$ the random variable $\Nbar_{\phi}(t)$ is square integrable. In order to do this we  let $C$ be the set defined in  \eqref{eqn-C_set} and we   fix $s,t \in C$,  with $s \le t$, and  $\psi \in\rU$.

	With Lemma \ref{L:pointwise_conv-B} at hand, we can argue as in the proof of Lemma \ref{lem-tM is a martingale} and prove that
	\begin{equation*} \label{eqn-new-mart_pointwise_conv-tN}
	\lim_{n \to \infty }  ({\Nbar}_{n, \phi} (t)-{\Nbar}_{n,\phi} (s)) = {\Nbar_{\phi} (t)-\Nbar_{\phi} (s)},
	\quad  \Pbar \mbox{ - a.s.}
	\end{equation*}

	From the equality of laws of  $u_n$ and $\tu_n$, the fact that both $N_{n, \phi}$ and ${\Nbar}_{n, \phi}  $ are square integrable  martingales,
	the embedding $U\embed {\rH}$, the Cauchy-Schwarz inequality
	and the application of the BDG inequality we infer that  there exists a constant $c_1>0$ independent of $n\in \mathbb{N}$ such that
	\begin{align*}
	\sup_{t \in [0,T]}  \Ebar  \bigl[ \lvert {\Nbar}_{n, \phi} (t) \lvert^2 \bigr]
	= & \sup_{t \in [0,T]}  \mathbb{E}  \bigl[   \lvert N_{n, \phi} (t) \rvert^2 \bigr]
	 \leq  \mathbb{E}  \bigl[ \sup_{t \in [0,T]}   \lvert N_{n, \phi} (t) \rvert^{2}\bigr] \\
	\le &  c_1 \lVert \psi \rVert^2_{U}  \mathbb{E}
	\Bigl[ \Bigl( \int_0^T \int_{Y} \rvert P_n F(s,u_n(s),y)\vert^2_\rH \lvert u_n(s) \rvert^2_\rH  \, \nun(\ud y) \ds \Bigr) ^{\frac{1}{2}} \Bigr].
\nonumber
	\end{align*}
	Since the restriction to ${\rH}$  of the map ${P}_{n} $ is the orthogonal projection from ${\rH}$ onto ${\rH}_{n}$,
	by the linear growth condition \eqref{eqn-F_linear_growth} in Assumption F.\ref{assum-F2}\label{assum-F2-used6}, we see that there exists a constant $c_2>0$ independent of $n\in \mathbb{N}$ such that

	\begin{equation} \label{eqn-new-mart_BDG_est_1-tN}
	\sup_{t \in [0,T]}  \Ebar  \bigl[ \lvert {\Nbar}_{n, \phi} (t) \lvert^2 \bigr]
	\leq  c_2  \lVert \psi \rVert_{U}^2  \left( 1 + \mathbb{E} \left[ \sup_{s\in [0,T]} \lvert u_n(s) \rvert^4 \right] \right).
	\end{equation}
	From the last inequality \eqref{eqn-new-mart_BDG_est_1-tN}   along with the estimate \eqref{eqn-H_estimate-p>2}
 with $p=4$,   we deduce that
	\begin{equation*} \label{eqn-new-mart_uniform_int_2-tN}
	\sup_{t \in [0,T]}  \Ebar  \bigl[ \lvert {\Nbar}_{n, \phi} (t) \lvert^2 \bigr] <\infty.
	\end{equation*}
	Owing to the convergence \eqref{eqn-mart_pointwise_conv-tN}, the Fatou Lemma  and  \eqref{eqn-new-mart_uniform_int_2-tN} we can  argue as in the proof of  \eqref{eqn-new-Square-Int-on-C}  and  show that
	\begin{align}\label{Eq:new-Square-tN-on-C}
	\sup_{r \in C  }\Ebar \left[ \lvert{\Nbar_\phi(r)}\rvert^2 \right] <\infty,
	\end{align}
	which yields the square integrability of $\Nbar_{\phi}(t)$ for each $t\in C$.

 For  general $\tinOT$ we consider a $C$-valued sequence $(t_m)_{m \in \mathbb{N}}$ such that $t_m \searrow t$ as $m \to \infty$.
Note that \eqref{Eq:new-Square-tN-on-C} implies that
	\begin{align}\label{Eq:new-Square-tNm-on-C}
	\sup_{m \in \mathbb{N}  }\Ebar \left[ \lvert{\Nbar_\phi(t_m)}\rvert^2 \right] <\infty.
	\end{align}
	Observe also that $\Nbar_\phi^2(t_m)$ satisfies \eqref{Eq:new-Square-tNm-on-C}.  Owing to the right-continuity of the paths of $\Nbar_{\phi}$ we have
	\begin{equation*} \Nbar^2_\phi(t_m) \to \Nbar^2_\phi(t), \mbox{  $\Pbar$-a.s.} \end{equation*}
This  along with the Fatou Lemma and \eqref{Eq:new-Square-tNm-on-C} imply that $\Nbar_\phi (t)$ is square integrable for each $\tinOT$.
	This completes the first step of the proof.

 \noindent \textbf{Step 2.}  Following the idea of  Lemma \ref{lem-tM is a martingale} we denote by $\mathscr{D}_t^0$ the $\sigma$-algebra generated by all maps \begin{equation*}\mathbb{D}([0,T], U^\prime)\ni x \mapsto x(s)\in U^\prime,\, s\le t .\end{equation*} We also set $\mathscr{D}_t=\bigcap_{s>t} \mathscr{D}^0_s,\, t\in [0,T]$. Note that $\mathbf{D}=(\mathscr{D}_t)_{t\in [0,T]}$ is a filtration.
Let $C$ be the set from the assertion (a) of Lemma \ref{lem-pointwise_conv}. Let us choose and fix $s,t \in C$,  with $s \le t$, and  $\psi \in\rU$. {Let us note that the almost sure set considered below is independent of $\psi$.}   Let
		\[
		h\colon  \mathbb{D}([0,T];\rU^\prime)  \to \mathbb{R}
		\]
		be a bounded, continuous function and $\mathscr{D}_s$-measurable.

  With Lemma \ref{L:pointwise_conv-B} at hand, we can argue as in the proof of Lemma \ref{lem-tM is a martingale} and prove that
	\begin{equation} \label{eqn-mart_pointwise_conv-tN}
	\lim_{n \to \infty }  h(\tun) ({\Nbar}_{n, \phi} (t)-{\Nbar}_{n,\phi} (s)) = h(\tu ) ({\Nbar_{\phi} (t)-\Nbar_{\phi} (s)}),
	\quad  \Pbar \mbox{ - a.s.}
	\end{equation}

	Let us denote
	\begin{equation*}
	{g}_{n}(\omega ) \coloneqq \bigl( \Nbar_{n, \phi} (t, \omega )- {\Nbar}_{n, \phi} (s, \omega ) \bigr)
	\, h \bigl( \tun  \bigr) , \qquad \omega \in \Omegabar .
	\end{equation*}
Thanks to \eqref{Eq:new-Square-tNm-on-C}  we can use the same idea as in the proof of assertion \eqref{eqn-mart_uniform_int} and prove that
	\begin{equation*} \label{eqn-mart_uniform_int-tN}
	\sup_{n \ge 1}  \Ebar \bigl[ {| {g}_{n} |}^{2} \bigr] < \infty,
	\end{equation*}
	which yields the uniform integrability of the sequence $\{ {g}_{n} {\} }_{n \in \mathbb{N} }$.

Now, the convergence
	(\ref{eqn-mart_pointwise_conv-tN}), the uniform integrability of  $\{ {g}_{n} {\} }_{n \in \mathbb{N} }$, the martingale property of $\Nbar_{n,\phi}(t)$ and the Vitali Convergence Theorem yield that for all $s,t\in C$, $s\le t$
	\begin{equation*}
	\lim_{n \to \infty } \Ebar\bigl[ \bigl({\Nbar}_{n, \phi} (t)-{\Nbar}_{n, \phi} (s) \bigr) h \bigl( \tun \bigr) \bigr]
	= \Ebar \bigl[  \bigl(\Nbar_\phi (t)-\Nbar_\phi (s)  \bigr) h \bigl( \tu \bigr) \bigr]=0.
	\end{equation*}
 Hence, for $s,t\in C$ with $s\le t$ we have
\begin{equation*} \Ebar \bigl[  \bigl(\Nbar_\phi (t)-\Nbar_\phi (s)  \bigr) h \bigl( \tu \bigr) \bigr]=0.\end{equation*}

For general $t,s \in [0,T]$ with $s< t$ we consider two $C$-valued sequences $(t_m)_{m \in \mathbb{N}}, (s_m)_{m \in \mathbb{N}}$ satisfying
	\begin{itemize}
		\item $s<s_m$ and $t<t_m$ for every $m \in \mathbb{N}$,
		\item $s_m \le t_m$ for every $m \in \mathbb{N}$,
		\item $s_m \searrow s$ and $t_m \searrow t$ as $m \to \infty$.
	\end{itemize}
 By the right-continuity of $\Nbar_\phi$ we see that $\Pbar$-a.s.
	\begin{equation*} \Nbar_\phi(t_m) \to \Nbar_\phi(t) \mbox{  and } \Nbar_\phi(s_m) \to \Nbar_\phi(s).\end{equation*}
 Notice also that, as above, we can use \eqref{Eq:new-Square-tNm-on-C}	and show that
 the sequences $h(\tu) \Nbar_\phi(t_m)$ and $ h(\tu)\Nbar_\phi(s_m) $ are uniformly integrable. Thus, we can proceed as in Step 2 of the proof of Lemma \ref{lem-tM is a martingale} and prove that
 \begin{equation*} \lim_{m\to \infty}\Ebar \bigl[  \bigl(\Nbar_\phi (t_m)-\Nbar_\phi (s_m)  \bigr) h \bigl( \tu \bigr) \bigr]=\Ebar \bigl[  \bigl(\Nbar_\phi (t)-\Nbar_\phi (s)  \bigr) h \bigl( \tu \bigr) \bigr]= 0.\end{equation*}
This also completes the proof of the Lemma   \ref{lem-tN is a martingale}.

\end{proof}

\section{Existence of solutions} \label{sec-existence-III}

In this section we use the Assumptions F.\ref{assum-F2} and F.\ref{assum-F5}. \label{assum-F5-used-10}  When we use the latter  assumption, we specify this explicitly.
In particular, we assume these assumptions in  our Lemma \ref{lem-purely discontinuous-New}. To be precise we assume the growth condition \eqref{eqn-F_linear_growth-p} in Assumption F.\ref{assum-F5} \label{assum-F5-used-11} with $p=4$  only and  then use, also with $p=4$,  the a'priori estimates \eqref{eqn-H_estimate-p>2} from Lemma \ref{lem-Galerkin_estimates}.

Hereafter, without loss of generality, we can and will assume that  the external force $f$  in equation \eqref{eqn-SNSEs} is equal to $0$.

\subsection{Preliminaries}
\label{subsec-Preliminaries}

In this subsection we will discuss a notion of a purely discontinuous martingale.

Let us begin by  recalling  the following  standard,   well known  fact,  see 
\cite[\S I.4e and Theorem I.4.2]{Jacod_Shiryaev}.

\begin{proposition}\label{prop-predictable quadratic covariation}
    If $M$ and $N$ are  two  locally square integrable $\mathbb{R}$-valued martingales,  then there exists a unique (up to an evanescent set) predictable $\mathbb{R}$-valued process of finite variation $\langle M,N \rangle$ such that the process 
    \[
    MN - \langle M,N\rangle
    \] is a local martingale.
 \end{proposition}
\begin{definition}\label{def-predictable quadratic covariation}
 If $M$ and $N$ are  two  locally square integrable $\mathbb{R}$-valued martingales, then the process 
    $\langle M,N \rangle$ 
    from Proposition \ref{prop-predictable quadratic covariation} is called 
    the \emph{predictable quadratic covariation process} of (the processes) $M$ and $N$.
\end{definition}
Note that while Proposition \ref{prop-predictable quadratic covariation} and Definition \ref{def-predictable quadratic covariation} are stated for real-valed martingales, they  have natural extension to Hilbert space  valued martingales, see below and  \cite[Section 20]{Metivier}.

The following definition presented directly in an infinite dimensional version is taken from 
\cite[Theorem 20.5]{Metivier}. 

\begin{definition}\label{def-quadratic variation}
If $M$ is a  square integrable  martingale taking values in a Hilbert space $X$ with norm $\normb{\cdot}{}{}$, then  
the quadratic variation process of $M$ is an $\mathbb{R}$-valued process $[M,M]$   such that for every $ t\in [0,T]$, $[M,M](t) $  is  the limit in $L^1=L^1(\Omega,\mathscr{F},\mathbb{P})$ over the partitions of $[0,T]$ with mesh converging to $0$, i.e. 
\begin{equation}\label{eqn-quadratic_variation-Metivier}
[M,M](t) = L^1-\lim \sum_{i=1}^n \left\vert M(t_i^{n} \wedge t )-M(t_{i-1}^{n} \wedge t )\right\vert^2.
\end{equation}
\end{definition}
The existence and further properties of the processes $[M,M]$ was proven in \cite[Theorem 20.5]{Metivier}.

According to \cite[Theorem 4.47]{Jacod_Shiryaev}, whose proof can be traced back to \cite[Proposition\ 4.44 and  Theorem\ 4.31]{Jacod_Shiryaev}, one can show
the following variation  of \eqref{eqn-quadratic_variation-Metivier}, namely that the convergence is  in probability, uniformly on compact intervals, i.e.\ for every $T>0$,
\begin{equation}\label{eqn-quadratic_variation-Jacod Shiryaev}
\sup_{ t\in [0,T]} \bigg\vert \sum_{i=1}^n \left(M(t_i^{n} \wedge t )-M(t_{i-1}^{n} \wedge t )\right)^2 -[M,M]_{t}\bigg\vert
\stackrel{\mathclap{\mathbb{P}}}{\to} 0.
\end{equation}

 For two square integrable $\mathbb{R}$-valued martingales $M$ and $N$, the bracket $[M,N]$ is  defined by the polarisation formula, i.e. 
\begin{equation}\label{eqn-bracket}
 [M,N]\coloneqq \frac14 [M+N,M+N]-\frac14 [M-N,M-N].
\end{equation}
 
The following result, see  \cite[Theorem \ I.4.52]{Jacod_Shiryaev}, \cite[Theorem 3.4.2]{Kallianpur_Xiong} and \cite[Theorem 20.5]{Metivier},   relates these two brackets. We will use notation \eqref{eqn-f-jumps}.
\begin{proposition}\label{prop-brackets}
If  $M$ and $N$ are  two square integrable  martingales taking values in Hilbert space $(X, \langle .,.\rangle_{X})$, then 
\begin{equation}\label{eqn-K-3.4.2}
[M,N]_t=\langle M^{\mathrm{c}},N^{\mathrm{c}}\rangle_t+ \sum_{s\in [0,t]} \langle \Delta M(s), \Delta N(s)\rangle_{X},\;\;\mathbb{P}\mbox{-a.s., for every } t\in [0,T].
\end{equation}
\end{proposition}

For the following definition, see  \cite[Definition\ 3.4.1]{Kallianpur_Xiong} or \cite[Definition\ I.4.11]{Jacod_Shiryaev}, and \cite{Marinelli+Rockner_2014} or \cite[Theorem 20.5]{Metivier} for the infinite dimensional case. 

\begin{definition}\label{def-purely discontinuous martingale}
A square integrable $\mathbb{R}$-valued (local) martingale $M$ such that $M(0)=0$  is purely discontinuous if and only if  for any square integrable real continuous martingale $N$ the product $MN$ is a (local) martingale.
\end{definition}

The following result is taken from \cite[Theorem I.4.18]{Jacod_Shiryaev}.

\begin{proposition}\label{prop-decomposition of a martingale}
Every $\mathbb{R}$-valued local martingale $M$ can be uniquely, up to indistinguishability,   represented as a sum of a continuous local martingale $M^c$ and purely discontinuous local martingale $M^d$, i.e. 
\[
M=M^c+M^d.
\]
The  continuous part $X^c$ of a semimartingale $X$   is defined as the continuous part of the martingale part of $X$, see \cite[Proposition\ I.4.27]{Jacod_Shiryaev}.
    \end{proposition}

As a consequence of Proposition \ref{prop-brackets}, one deduces the following result, see \cite[Th.\ 3.4.2]{Kallianpur_Xiong}.
\begin{corollary}\label{cor-purely discontinuous}
A  square integrable   $\mathbb{R}$-valued martingale $M$  is   purely discontinuous  if and only if  for every $t \in [0,T]$, 
\begin{equation}\label{eqn-purely_discontinuous_martingale}
[M,M]_{t} = \sum_{s\leq t} \vert \Delta M(s) \vert^2, \;\;\mathbb{P}\mbox{-a.s.}.
\end{equation}
\end{corollary}
The assertion of the Corollary above is often used as a definition of a purely discontinuous martingale, see e.g.  \cite[page 156]{Yaroslavtsev_2021}.
For vector-valued martingales we follow the  recent papers  \cite[Section\ 6]{Yaroslavtsev_2020} or \cite[p. 156]{Yaroslavtsev_2021} by Yaroslavtsev.

\begin{definition}\label{def-martingael purely discontinuous}
Let $(X, \langle .,.\rangle_X)$ be a Hilbert space.
A X-valued  martingale $M$ is purely discontinuous if for every $h\in {X}$ the $\mathbb{R}$-valued  martingale $\langle M,h \rangle_X$ is purely discontinuous.
\end{definition}

To define  the predictable quadratic covariations of Hilbert space valued martingales we follow  \cite[Theorem\ 18.6]{Metivier}.

For two square integrable  martingales $M$  and $N$ taking values in a Hilbert space $X$ the  quadratic covariations are defined as in \cite[Theorem\ 23.4]{Metivier}. 

\begin{definition}\label{def-predictable quadratic covariation-H}
If $M$ and $N$ are two square integrable martingales taking values in a Hilbert space $(X,\langle \cdot , \cdot \rangle_{X})$, then     
the predictable quadratic covariation  $\langle M,N \rangle$ of processes $M$  and $N$ is the  unique predictable real process of finite variation such that the process
\[\langle M, N \rangle_{X} - \langle M, N \rangle\] is a martingale.
\end{definition}

\subsection{Identification of the quadratic variation}
\label{subsec_identification}

In this, as well in all subsequent subsections, we assume that $T>0$ is given and fixed  and we consider

\begin{itemize}
\item  the probability space $\bigl( \Omegabar ,\tfcal ,\Pbar  \bigr) $  introduced at the beginning of subsection \ref{subsec-SJ}, see formula \eqref{eqn-new_probability_space};
\item $ \zcal_T $-valued random variables $\tu$ defined on $\bigl( \Omegabar ,\tfcal ,\Pbar  \bigr) $  such that \eqref{eqn-Skorokhod_appl_convergence} holds;
\item the filtration ${\tF} \coloneqq({\overline{\fcal }}_{t})_{\tinOT}$,  defined in equality  \eqref{eqn-filtration_limit};
\item the process $\Mbar$ defined in formula \eqref{eqn-tM}.
\end{itemize}

Let us recall that according to Remark \ref{rem-process ubar}
the  $ \zcal_T $-valued random variables $\tu$ defined on $\bigl( \Omegabar ,\tfcal ,\Pbar  \bigr) $ induces
a  measurable process (see Remark \ref{rem-process ubar})
\[
\tu\colon [0,T]\times \Omegabar \to V \cap {\rH}_{\loc},
\]
whose trajectories belong to the space
$\zcal_T$.

Let us also  choose  and fix an arbitrary element $\phi \in\rU$
and  let us define a $\mathbb{R}$-valued process
$\Mbarphi{}$ by

\begin{align}\label{definition_M_bar-phi}
\Mbarphi{}(t) &\coloneqq \duality{\Mbar(t)}{\phi}
\\&=
\duality{\tu(t)}{\phi} - \duality{\bar{u}(0)}{\phi} +  \int_0^t \duality{\acal \tu(s)}{\phi} \ds \notag  \\
&\quad +\int_0^t \duality{\rB(\tu(s))}{\phi} \ds,
\;\; \tinOT,
\nonumber
\end{align}
 where $\Mbar$ is the process defined in formula \eqref{eqn-tM}.

Let us immediately observe that in view of Lemma \ref{lem-tM is a martingale}
the process $\Mbarphi{}$ is a square integrable $\mathbb{R}$-valued ${\tF}$-martingale and
all its trajectories belong to  $\mathbb{D} ([0,T])$.

We begin with the analysis of the first process appearing on the right-hand side of the above identity \eqref{definition_M_bar-phi}.

\begin{lemma}\label{lem-u-bar is a semimartingale}
The process $\duality{\tu}{\phi}\coloneqq \bigl\{\duality{\tu(t)}{\phi}:\, \tinOT\bigr\}$
is a $\mathbb{R}$-valued semimartingale.
Moreover, its quadratic variation process is equal to the quadratic variation process of the process $\Mbarphi{}$, i.e.\ $\Pbar$-a.s.
\begin{equation}\label{eqn-K_6.1.17}
[\duality{\tu}{\phi},\duality{\tu}{\phi}](t) =[\Mbarphi{},\Mbarphi{}](t), \;\; \tinOT.
\end{equation}
\end{lemma}

\begin{proof}[Proof of Lemma \ref{lem-u-bar is a semimartingale}]
Let us begin by an observation that according to formula \eqref{eqn-tM} and Lemma \ref{lem-tM is a martingale}, it follows from definition of a semimartingale, \cite[Definition\  I.4.21]{Jacod_Shiryaev}, that the process $\duality{\tu}{\phi}$ is a real semimartingale.
Moreover,
because its finite variation part is a continuous process, identity \eqref{eqn-K_6.1.17}  follow immediately from \cite[Proposition\  I.4.49(d)]{Jacod_Shiryaev}.\label{ZB-additional explanation-01}

The proof of Lemma \ref{lem-u-bar is a semimartingale} is complete.
\end{proof}

\begin{proposition}
\label{prop_angle_bracket}
The  predictable quadratic variation of the martingale $\Mbarphi{}$ satisfies $\Pbar$-a.s.
\begin{equation}
\label{angle_brackef_M_phi_bar}
\langle \Mbarphi{} \rangle (t)
= \int_0^t \int_{Y} \ip{F(s,\tu(s),y)}{\phi}{\rH}^2\, \nun(\ud y) \ds, \;\; t \in [0,T].
\end{equation}
\end{proposition}

\begin{proof}[Proof of Proposition \ref{prop_angle_bracket}]
We follow the lines of the proof of   \cite[Lemma\ 6.1.10]{Kallianpur_Xiong}.
Let us choose and fix an arbitrary element $\phi \in\rU$. Let us note  that the process defined by the right hand side of equality \eqref{angle_brackef_M_phi_bar} is continuous and therefore is ${\tF}$-predictable. Thus it is enough to show that the difference
\begin{equation*}(\Mbarphi{})^2(t) - \int_0^t \int_{Y} \ip{F(s,\tu(s),y)}{\phi}{\rH}^2\, \nun(\ud y) \ds.\;\; t \in [0,T],\end{equation*}
is a martingale.

\textbf{Step 1.}
Applying the It\^o formula \eqref{eqn-Ito_JS_I.4.58_b} to the $\mathbb{R}$-valued semimartingale $\xi=\bigl\{ \xi(t)= \duality{\tu(t)}{\phi}, \; \tinOT\bigr\}$, see \eqref{definition_M_bar-phi} we get
\begin{align}
  \begin{aligned}
\duali{\tu(t)}{\phi}{2} &= \duali{\tu(0)}{\phi}{2} -2  \int_0^t \duality{\acal \tu(s)+\rB(\tu(s))}{\phi}\;\duality{\tu(s)}{\phi} \, \ds
\\
&\quad + 2\int_0^t \duality{\tu(s-)}{\phi} \, \ud \Mbarphi{} (s)+[\duality{\tu}{\phi},\duality{\tu}{\phi}](t), \;\; \tinOT.
\end{aligned}
\label{eqn-K_6.1.18}
\end{align}

Recalling  the definition of the $\mathbb{R}$-valued process $\Nbar_\phi$, see  formula \eqref{eqn-Nbar^phi}, we find that the above equality \eqref{eqn-K_6.1.18} can be rewritten as
\begin{align*}\begin{aligned}
  \Nbar_\phi(t) &=2\int_0^t \duality{\tu(s-)}{\phi} \, \ud \Mbarphi{} (s)+[\duality{\tu}{\phi},\duality{\tu}{\phi}](t)\\
  &\qquad -\int_0^t \int_{Y} \ip{F(s,\tu(s),y)}{\phi}{\rH}^2\, \nun(\ud y) \ds, \;\;  \;\; \tinOT.
  \end{aligned}
\end{align*}

From Lemma \ref{lem-u-bar is a semimartingale}
we infer that $\Pbar$-a.s.
\begin{align*}
\begin{aligned}
    \Nbar_\phi(t) &=2\int_0^t \duality{\tu(s-)}{\phi} \, \ud \Mbarphi{} (s)+[\Mbarphi{},\Mbarphi{}](t)
    \\
    &\quad -\int_0^t \int_{Y} \ip{F(s,\tu(s),y)}{\phi}{\rH}^2\, \nun(\ud y) \ds, \;\;  \;\; \tinOT.
    \end{aligned}
\end{align*}
The last equality can be rearranged as
\begin{align}\label{eqn-K_6.1.19}
\begin{aligned}
  (\Mbarphi{}(t))^2 & -\int_0^t \int_{Y} \ip{F(s,\tu(s),y)}{\phi}{\rH}^2\, \nun(\ud y) \ds\\
  &=\bigl(  (\Mbarphi{}(t) )^2 -[\Mbarphi{},\Mbarphi{}](t) \bigr)+\Nbar_\phi(t)-2\int_0^t \duality{\tu(s-)}{\phi} \, \ud \Mbarphi{} (s)
  , \;\;  \;\; \tinOT.
  \end{aligned}
\end{align}

\textbf{Step 2.} Let us note that each of the three processes on the right hand side of equality \eqref{eqn-K_6.1.19} is a martingale. Indeed, $\Nbar_\phi$ is a martingale by Lemma \ref{lem-tN is a martingale}, the process expressed as  integral is a local martingale by \cite[Theorem\ I.4.40a]{Jacod_Shiryaev}
and the process $( \Mbarphi{} )^2 -[\Mbarphi{},\Mbarphi{}]$ is a martingale by \cite[Corollary\ 2 p.\ 125]{Metivier}.

Thus the proof of Proposition \ref{prop_angle_bracket} is complete.
\end{proof}

\subsection{Pure discontinuity of the martingale \texorpdfstring{$M_\phi$}{Mp}}
\label{subsec_pure}

In a forthcoming Proposition \ref{prop-purely discontinuous} we  will prove that for  every $\phi \in\rU$,
the martingale $\Mbarphi{}$  defined  in  \eqref{definition_M_bar-phi},  is purely discontinuous. {This result is an analogue of \cite[Theorem\  6.1.3]{Kallianpur_Xiong}.}

We begin  with formulating  the main result of this subsection. After this has been achieved  we will  present three  lemmata which will be used in the  proof of our main result. Let us emphasize that these two results are in the framework of the original probability space while the main result is in the framework of the new probability space. At the end of this section we will turn back to the proof of the main result.

\begin{proposition}
\label{prop-purely discontinuous}
Assume that $\phi \in\rU$.
The martingale $\Mbarphi{}$,  defined  in  formula \eqref{definition_M_bar-phi},  is purely discontinuous.
\end{proposition}

\begin{condition}\label{condition-g function}
Assume that we are given a function $g:\R \to \R_+$ that is  Lipschitz and even function such that for some $a>0$ 
\begin{align}\label{condition-g function-ineq}
0 \leq g(r) \leq r^2  \; \; \mbox{ for }r \in [-a,a].
\end{align}
\end{condition}

Note that under Condition \ref{condition-g function}, the function $g$ satisfies the following linear  growth condition
\begin{align}\label{eqn-pure linear growth condition}
g(r) &\leq c \vert r \vert , \;\; r \in \mathbb{R},
\end{align}
 as this is used in one of the proofs.

\begin{lemma}\label{lem-measurability}
Assume that a function  $g \colon  \mathbb{R} \to \mathbb{R}_+$  satisfies Condition \ref{condition-g function}.
 If  $ {z}_k \to {z}$ in $\zcal_T$,  then
\begin{equation*}
\int_0^t \int_Y g(\ip{P_k F(s,z_k(s),y)}{\phi}{\rH})\, \nu(\ud y) \ds
\to \int_0^t \int_Y g(\ip{F(s,z(s),y)}{\phi}{\rH})\, \nu(\ud y) \ds.
\end{equation*}
\end{lemma}

\begin{proof}
The proof is based on the proof of \cite[Lemma\ 6.1.7]{Kallianpur_Xiong}.

Let us fix an arbitrary  time instance $t \in [0,T]$ and a $\zcal_T$-valued sequence $ ({z}_k)_{k \in \mathbb{N}}$ such that $ {z}_k \to {z}$ in $\zcal_T$.
We will show that
\begin{equation}
\label{by_F.4-appendix'}
\int_0^t \int_{Y} \ip{P_k F(s,{z}_k(s),y)-F(s,{z}(s),y)}{\phi}{\rH}^2 \, \nu(\ud y) \ds \to 0.
\end{equation}
We have
\begin{align}
\label{by_triangle_ineq}
&\hspace{-1truecm}\int_0^t \int_{Y} \ip{P_k F(s,{z}_k(s),y)-F(s,{z}(s),y)}{\phi}{\rH}^2 \,  \nu(\ud y) \ds  \notag \\
&\leq 2 \int_0^t \int_{Y} \ip{ P_k F(s,{z}_k(s),y)-F(s,{z}_k(s),y)}{\phi}{\rH}^2 \, \nu(\ud y) \ds  \notag \\
& \quad + 2 \int_0^t \int_{Y} \ip{ F(s,{z}_k(s),y)-F(s,{z}(s),y)}{\phi}{\rH}^2 \, \nu(\ud y) \ds.
\end{align}
Since by assumptions  ${z}_k \to {z}$ in the space $\zcal_T$,  
 by Assumption F.\ref{assum-F3} \label{assum-F3-used trully-2} we infer  that
\begin{equation}
\label{by_F.4-appendix}
\int_0^t \int_{Y} \ip{ F(s,{z}_k(s),y)-F(s,{z}(s),y)}{ \phi}{\rH}^2 \, \nu(\ud y) \ds \to 0.
\end{equation}
Secondly, by the Vitali Theorem, we have
\begin{equation}
\label{by_Vitali}
\begin{aligned}
&\hspace{-1truecm} \int_0^t \int_{Y} \ip{P_k F(s,{z}_k(s),y)-F(s,{z}_k(s),y)}{\phi}{\rH}^2 \, \nu(\ud y) \ds \\
&= \int_0^t \int_{Y} \ip{ F(s,{z}_k(s),y) }{(P_k -\Id)\phi} {\rH}^2 \, \nu(\ud y) \ds
\to 0.
\end{aligned}
\end{equation}

Let us note that the  Vitali Convergence Theorem can be applied here because  by  inequality \eqref{eqn-F_linear_growth-p} 
and assertion \eqref{eqn-boundedness_H} we have
\begin{align}
&\hspace{-2truecm}\sup_{k \in \N} \int_0^t \int_{Y} \ip{F(s,{z}_k(s),y)}{(P_k -\Id)\phi}{ \rH}^4\, \nu(\ud y) \ds \notag \\&\hspace{-1truecm}\leq 2  \normb{\phi}{\rH}{4}  \sup_{k \in \N} \int_0^t \int_Y \normb{F(s,{z}_k(s),y)}{\rH}{4} \, \nu(\ud y) \ds \notag \\
&\leq 2C_4 \normb{\phi}{\rH}{4}  \sup_{k \in \N} \int_0^t ( 1+ \normb{{z}_k(s)}{\rH}{4}) \ds  < \infty.
\label{integrability}
\end{align}

Now, consider an arbitrary subsequence $(n_k)$.
Since, by \eqref{by_F.4-appendix'},
\begin{equation*}
\ilsk{{P}_{n_k} F(s,{z}_{n_k}(s);y)}{\phi }{H} \to \ilsk{F(s,z(s);y)}{\phi }{H} 
\mbox{  in } L^2([0,t]\times Y; \Leb \otimes \nu), 
\end{equation*}
we can choose a further  subsequence, which for simplicity of notation is denoted by  $(k)$, 
  such that
\begin{equation*}
\ilsk{P_k F(s,z_k(s);y)}{\phi }{H}  \; \to \; \ilsk{F(s,z(s);y) }{\phi }{H}
\;\;\; \Leb \otimes \nu \mbox{-almost everywhere}.
\end{equation*}
By the continuity of the function $g$ we infer  that
\begin{equation*}
g(\ip{F(s,{z}_{k}(s),y)}{ \phi}{ \rH}) \to g(\ip{F(s,z(s),y)}{\phi}{\rH}),
\;\;\;\Leb \otimes \nu\mbox{-almost everywhere}.
\end{equation*}
Note that \eqref{by_F.4-appendix'} implies the convergence of the $L^2$-norms:
\begin{equation*}
\int_0^t \int_Y \ip{P_k F(s,{z}_k(s),y)}{\phi}{ \rH}^2\, \nu(\ud y) \ds
\to \int_0^t \int_Y \ip{F(s,{z}(s),y)}{\phi}{ \rH}^2\, \nu(\ud y) \ds.
\end{equation*}
Moreover, since $\Leb \otimes \nu$-almost everywhere
\begin{equation*}
g\left( \ip{P_k F(s,{z}_k(s),y)}{\phi}{ \rH} \right)
\leq c \ip{P_k F(s,{z}_k(s),y)}{\phi}{ \rH}^2,
\end{equation*}
and  ${z}$ belongs to $L^2(0,T;V)$ we infer that
\begin{align*}
\begin{aligned}
\int_0^t \int_Y \ip{ F(s,{z}(s),y)}{\phi}{ \rH}^2\, \nu(\ud y) \ds&\leq C_2 \abs{\phi}_{\rH }^{2} \int_0^T (1 + \normb{z(s)}{\rH}{2} )\ds \\
&\leq C_2 \abs{\phi}_{\rH }^{2} \int_0^T (1 + \norm{z(s)}_\rV^2 )\ds
< \infty.
\end{aligned}
\end{align*}
Thus, by \cite[Theorem\ 2.8.8]{Bogachev_I} we infer that
\begin{equation*}
\label{finite_g_integral_for_conv_lemma}
\int_0^t \int_Y g(\ip{P_k F(s,z_k(s),y)}{\phi}{\rH})\, \nu(\ud y) \ds
\to \int_0^t \int_Y g(\ip{F(s,z(s),y)}{\phi}{\rH})\, \nu(\ud y) \ds.
\end{equation*}
Since this holds for a subsequence of an arbitrary subsequence, it follows that this convergence holds for the original sequence.
This finishes the proof of the claimed  assertion. Therefore, the proof of Lemma \ref{lem-measurability} is complete.
\end{proof}

Let us recall that in the following result $u_n$ is the unique solution to \eqref{eqn-Galerkin}. In particular, $u_n$ is an $U^\prime$-valued c\`adl\`ag process.

\begin{lemma}
\label{lem-purely discontinuous-New}
Assume that $\phi \in\rU$ and a function   $g\colon \R \to \R_+$
satisfies Condition \ref{condition-g function}.
Then, for every $n\in \mathbb{N}$ and every  $t \in [0,T]$, the following identities hold $\mathbb{P}$-almost surely,
\begin{align}
\sum_{s\in [0,t]} g(\Delta \langle {u}_n(s), \phi \rangle)
&=\int_0^t \int_{Y} g\bigl(  \ip{P_n F(s,{u}_n(s-),y)}{\phi}{\rH}\, \bigr) {\eta}(\ud s,\ud y),
\label{eqn-uncompensated_integral_1_New}
	\end{align}
and
\begin{align}
\label{eqn-compensated_integral_1_New}
\int_0^t \int_{Y} g\bigl(  \ip{P_n F(s,{u}_n(s-),y)}{\phi}{\rH}\, \bigr) \tilde{\eta}(\ud s,\ud y) &=	\sum_{s\in [0,t]} g(\Delta \langle {u}_n(s), \phi \rangle)
\\ \nonumber
&\hspace{-3truecm}-\int_0^t \int_{Y} g\bigl(  \ip{P_n F(s,{u}_n(s),y)}{\phi}{\rH}\, \bigr) \ud s\, \nu(\ud y).
\end{align}
\end{lemma}

\begin{proof}[Proof of Lemma \ref{lem-purely discontinuous-New}]
Let us choose and fix an arbitrary $\tinOT$.
Let us begin with an observation that the integrals in \eqref{eqn-uncompensated_integral_1_New} and \eqref{eqn-compensated_integral_1_New} are well defined and almost surely finite. This is true provided  the integrand belongs to $\mathbf{F}_p^1 \cap \mathbf{F}_p^2$, see \cite[Section\ 3]{Brzezniak_Liu_Zhu}.
Below we will show that this is the case.  For this purpose, we first apply, \eqref{eqn-pure linear growth condition}, the linear growth condition \eqref{eqn-F_linear_growth} in Assumption F.\ref{assum-F2} and the a'priori estimates (\ref{eqn-H_estimate-p>2}) with $p=2$ from Lemma \ref{lem-Galerkin_estimates} to infer the following chain of inequalities.
\begin{align}
 & \hspace{-2truecm} \E \left[ \int_0^t \int_{Y}  \vert g\bigl(  \ip{P_n F(s,{u}_n(s),y)}{\phi}{\rH}\, \bigr) \vert^2 \ud s\, \nu(\ud y) \right]
 \label{eqn-4.15}
  \\& \leq c^2 {\E} \left[ \int_0^t \int_{Y} \vert \ip{P_n F(s,{u}_n(s),y)}{\phi}{\rH}\vert^{2} \, \nun(\ud y) \ds \right]  \nonumber \\
  &= c^2 {\E} \left[ \int_0^t \int_{Y} \ip{F(s,{u}_n(s),y)}{P_n \phi}{\rH}^{2} \,  \nun(\ud y) \ds \right]
   \nonumber \\
 & \leq c^2 \vert   P_n \phi\vert_{\rH}^{2} {\E} \left[ \int_0^t \int_{Y}  \vert   F(s,{u}_n(s),y) \vert^{2} \,  \nun(\ud y) \ds \right]
  \nonumber\\
&  \leq c^2 \vert   \phi\vert_{\rH}^{2} {\E} \left[ \int_0^t \int_{Y}  \vert   F(s,{u}_n(s),y) \vert^{2} \,  \nun(\ud y) \ds \right]
 \nonumber \\ &\leq c^2 C_2 \normb{\phi }{\rH}{2}  {\E} \left[ \int_0^t (1+ \normb{{u}_n(s) }{\rH}{2})  \ds \right]
\nonumber\\&\leq c^2 C_2\normb{\phi }{\rH}{2} \Bigl(T+  {\E} \left[ \int_0^T  \normb{{u}_n(s) }{\rH}{2})  \ds  \right] \Bigr)<\infty.
 \nonumber
\end{align}

Secondly, since $g$  satisfies condition \eqref{condition-g function-ineq}
\begin{align*}
 \E  \bigg[  \int_0^t & \int_{Y}  \vert g\bigl(  \ip{P_n F(s,{u}_n(s),y)}{\phi}{\rH}\, \bigr) \vert \,\ud s\, \nu(\ud y) \bigg] \\
 &\leq c {\E} \bigg[ \int_0^t \int_{Y} \vert \ip{P_n F(s,{u}_n(s),y)}{\phi}{\rH}\vert^2 \, \nun(\ud y) \ds \bigg]
<\infty.
\nonumber
\end{align*}

Let $\phi \in\rU$ and $g\colon \R \to \R_+$ be as in the statement of the Lemma. We will prove \eqref{eqn-uncompensated_integral_1_New}.
It follows from \eqref{eqn-Galerkin} that $\Pbar$-a.s.

\[
\begin{aligned}
\Delta  \ddual{U'}{u_n (s)}{\phi }{U} \; = \; \Delta \ilsk{u_n (s)}{\phi }{H}
\; &= \; \Delta  \int_{0}^{s} \int_{Y} \ilsk{P_n F(r,u_n (r-),y}{\phi }{H} \tilde{\eta } (\ud r,\ud y).
\end{aligned}  
\]
Since the integrand belongs to  ${\mathbf{F}}_{p}^{1}$,
\[ 
\begin{aligned}
&\int_{0}^{s} \int_{Y} \ilsk{P_n F(r,u_n (r-),y}{\phi }{H} \tilde{\eta } (\ud y,\ud r) \\
\; &= \; 
\int_{0}^{s} \int_{Y} \ilsk{P_n F(r,u_n (r-),y}{\phi }{H} \eta  (\ud r,\ud y)
- \underset{=\int_{0}^{s} \int_{Y} \ilsk{P_n F(r,u_n (r),y}{\phi }{H} \nu (\ud y) \ud r }{\underbrace{\int_{0}^{s} \int_{Y} \ilsk{P_n F(r,u_n (r-),y}{\phi }{H} \nu (\ud y) \ud r }} .
\end{aligned}
\]
Since moreover,
\[
\Delta \int_{0}^{s} \int_{Y} \ilsk{P_n F(r,u_n (r),y}{\phi }{H} \nu (\ud y) \ud r  \; = \; 0 ,
\]
we have
\[ 
\Delta \int_{0}^{s} \int_{Y} \ilsk{P_n F(r,u_n (r-),y}{\phi }{H} \tilde{\eta }(\ud r,\ud y) 
\; = \; \Delta  \int_{0}^{s} \int_{Y} \ilsk{P_n F(r,u_n (r-),y}{\phi }{H} \eta  (\ud r,\ud y).
\]
Thus
\[
\begin{aligned}
\Delta  \ddual{U'}{u_n (s)}{\phi }{U} 
\; &= \; \Delta  \int_{0}^{s} \int_{Y} \ilsk{P_n F(r,u_n (r-),y}{\phi }{H} \eta  (\ud r,\ud y). 
\end{aligned}  
\]
Note that the processes ${(\int_{0}^{s} \int_{Y} \ilsk{P_n F(r,u_n (r-),y}{\phi }{H} \tilde{\eta } (\ud r,\ud y) )}_{s \in [0,T]}$ and 

${(\int_{0}^{s} \int_{Y} \ilsk{P_n F(r,u_n (r-),y}{\phi }{H} \eta  (\ud y,\ud r) )}_{s \in [0,T]}$ are  c\`{a}dl\`{a}g.
We have
\[
\begin{aligned}
&\Delta  \int_{0}^{s} \int_{Y} \ilsk{P_n F(r,u_n (r-),y}{\phi }{H} \eta  (\ud r,\ud y)
\\
\; &= \;  \iint_{(0,s] \times Y} \ilsk{P_n F(r,u_n (r-),y}{\phi }{H} \eta  (\ud r,\ud y)
- \lim_{m \to \infty } \iint_{(0,s_m] \times Y} \ilsk{P_n F(r,u_n (r-),y}{\phi }{H} \eta  (\ud r,\ud y)
\end{aligned}
\]
where ${(t_m)}_{m \in \mathbb{N} } \subset (0,s) $ is a sequence such that $s_m \nearrow s$.
We will use the definition of the integral with respect of the Poisson random measure.
Let $p$ be the $\mathbb{F}$-optional $Y$-valued point process corresponding to the Poisson random measure $\eta $ and let ${D}_{p}$ denote its domain
\[
\begin{aligned}
&\iint_{(0,s] \times Y} \ilsk{P_n F(r,u_n (r-),y}{\phi }{H} \eta  (\ud r,\ud y)
-  \iint_{(0,s_m] \times Y} \ilsk{P_n F(r,u_n (r-),y}{\phi }{H} \eta  (\ud r,\ud y)
\\
\; &= \; \sum_{\stackrel{s_m < r \le s}{r \in {D}_{p}}} \ilsk{P_n F(r,u_n (r-),p(r)}{\phi }{H} \ind{{D}_{p}}(r) 
\end{aligned}
\]
Passing to the limit as $m \to \infty $ we infer that  

\begin{align*}
& \iint_{(0,s] \times Y} \ilsk{P_n F(r,u_n (r-),y}{\phi }{H} \eta (\ud r,\ud y)\\
& \hspace{2truecm}- \lim_{m\to \infty }  \iint_{(0,s_m] \times Y} \ilsk{P_n F(r,u_n (r-),y}{\phi }{H} \eta  (\ud r,\ud y) 
\\
 &= \;  \ilsk{P_n F(s,u_n (s-),p(s)}{\phi }{H} \ind{{D}_{p}}(s) .
\end{align*}

Thus we get 
\[
\Delta  \ddual{U'}{u_n (s)}{\phi }{U} \; = \;  \ilsk{P_n F(s,u_n (s-),p(s)}{\phi }{H} \ind{{D}_{p}}(s) .
\] 
Let us next observe that because $g(0) = 0$ we have, for every $(\omega ,s) \in  \Omega \times [0,T]$,
\[
g\bigl( \ilsk{P_n F(s,u_n (s-),p(s)}{\phi }{H} \ind{{D}_{p}}(s) \bigr)
\; = \; g\bigl( \ilsk{P_n F(s,u_n (s-),p(s)}{\phi }{H} \bigr) \ind{{D}_{p}}(s)
\]
Therefore, we deduce that for every $(\omega ,t) \in \Omega \times [0,T]$,

\begin{align*}
&\sum_{s\in [0,t]} g\bigl( \ilsk{P_n F(s,u_n (s-),p(s)}{\phi }{H} \ind{{D}_{p}}(s) \bigr)
\\&\hspace{2truecm}  = \; \sum_{s\in [0,t]} g\bigl( \ilsk{P_n F(s,u_n (s-),p(s)}{\phi }{H}  \ind{{D}_{p}}(s) \bigr)
\\
 & \hspace{2truecm}= \; \sum_{s\in [0,t]} g\bigl( \ilsk{P_n F(s,u_n (s-),p(s)}{\phi }{H} \bigr) \ind{{D}_{p}}(s)
\\
&= \; \int_{0}^{t} \int_{Y} g\bigl( \ilsk{P_n F(s,u_n (s-),y}{\phi }{H} \bigr) \, \eta (\ud s,\ud y).
\end{align*}

Thus the proof of assertion \eqref{eqn-uncompensated_integral_1_New} is complete.
	
Next let us  observe that for each $n$,  the definition of the integral with respect to a compensated Poisson random measure, see  \cite[equality  (3.4.1)]{Kallianpur_Xiong} and/or formula \eqref{eqn-H.6}, we have $\Pbar$-a.s.
\begin{align*}
&\int_0^t \int_{Y} g\bigl(  \ip{P_n F(s,{u}_n(s-),y)}{\phi}{\rH}\, \bigr) \tilde{\eta}(\ud s,\ud y) \\ 
& \hspace{2truecm}  =	 \int_0^t \int_{Y} g\bigl(  \ip{P_n F(s,{u}_n(s-),y)}{\phi}{\rH}\, \bigr) {\eta}(\ud s,\ud y)
\\ & \hspace{3truecm}  -\int_0^t \int_{Y} g\bigl(  \ip{P_n F(s,{u}_n(s),y)}{\phi}{\rH}\, \bigr) \, \nu(\ud y) \ud s.
\end{align*}
By applying now equality \eqref{eqn-uncompensated_integral_1_New} we deduce that equality \eqref{eqn-compensated_integral_1_New} holds.
\end{proof}

The previous results were  formulated on the original probability space.
The next result is formulated for the new filtered probability space as listed at the beginning of Subsection \ref{subsec_identification}.
Let us  recall that $(\tun)_{n}$ is a sequence of $ \zcal_T$-valued random variables defined on the probability space $\bigl( \Omegabar ,\tfcal ,\Pbar  \bigr)$, see \eqref{eqn-new_probability_space}, such that
   \begin{align*}
  & \tun  \to \tu  \mbox{  in } \zcal_T
    \quad  \mbox{ on  }\Omegabar.
\end{align*}

The last assertion implies assertions \eqref{eqn-conv_L^2_w(0,T;V)}-\eqref{eqn-conv_D(0,T;U')}.

For a topological space $X$, we define an evaluation map
\begin{equation*}\pi_t: \mathbb{D}([0,T];X) \ni z \mapsto  z(t) \in X.\end{equation*}

\begin{lemma}\label{Lem:Compensator}
	Let us assume that a function   $g\colon \R \to \R_+$
satisfies Condition \ref{condition-g function}.
	Assume that  $\phi \in\rU$. For $n\in \mathbb{N}$ let us consider the maps
\[\Psi_n,  \Psi: \zcal_T \to \mathbb{D}([0,T]; \mathbb{R})\] defined by for $u \in \zcal_T$ and $\tinOT$
	\begin{align}\label{eqn-Psi_n}
	\Psi_n(u)(t)&=\sum_{s\in [0,t]} g(\Delta \langle u(s), \phi\rangle)-\int_0^t\int_{Y} g \left(\langle P_n F(s, u(s),y) , \phi \rangle\right)\, \nu(\ud y) \ds,\\
\label{eqn-Phi}
		\Psi(u)(t)&=\sum_{s\in [0,t]} g(\Delta \langle u(s), \phi\rangle)-\int_0^t\int_{Y} g\left(\langle  F(s, u(s),y) , \phi \rangle\right)\, \nu(\ud y) \ds.
	\end{align}
Then, for every $\tinOT$, the maps $\pi_t\circ \Psi_n:\zcal_T \to \mathbb{R}$ and $\pi_t\circ \Psi:\zcal_T \to \mathbb{R}$ are measurable and, if
$t_0\in C$, where the set $C=C(\tu)$ has been defined in \eqref{eqn-C_set}, the following equality holds
	\begin{equation}\label{eqn-6.33}
	\lim_{n \to \infty} \Psi_n(\tun)(t_0) = \Psi(\tu)(t_0) 
 \mbox{ on  }\Omegabar. 
	\end{equation}
\end{lemma}

\begin{proof}
Let us choose and fix: an element $\phi\in\rU$, a function $g\colon \R \to \R_+$.  Let  the set  $C$ be as in the statement of the lemma, see also \eqref{eqn-C_set}.
Let us first observe that by \cite[Corollary\  2.4.2]{Kallianpur_Xiong} the map
\begin{equation*}
G: \mathbb{D}([0,T];\rU^\prime) \to \mathbb{D}([0,T];\R), \qquad G(z)(t)= \sum_{s \in [0,t]} g(\Delta z(s)), \; t\in [0,T],     
\end{equation*}
is continuous. Combining this with the continuity of
\begin{equation*}\mathbb{D}([0,T];\rU^\prime) \ni z \mapsto \duality{z}{\phi} \in \mathbb{D}([0,T];\R)\end{equation*}
and the measurability of the evaluation map $\pi_t$,
see \cite[Theorem\ 12.5(ii)]{Billingsley_1999} establishes the measurability of
\[\zcal_T \ni u \mapsto \sum_{s\in [0,t]} g(\Delta \langle u(s), \phi\rangle) \in \mathbb{R}
\]

For the measurability of the maps $\pi_t \circ \Psi_n$ and $\pi_t \circ \Psi$, for $\tinOT$,
it is sufficient to consider the case of $\pi_t\circ \Psi$ as the measurability of $\pi_t \circ \Psi_n$ is obtained through the identity $\ip{ P_n F(s,{u}(s),y)}{ \phi}{\rH} =
\ip{ F(s,{u}(s),y)}{ P_n \phi}{ \rH}$. To finish this part of the  proof it is sufficient to prove that the map
\[\zcal_T \ni u \mapsto \int_0^t \int_{Y} g(\ip{ F(s,{u}(s),y)}{ \phi}{ \rH})\, \nu(\ud y) \ds \in \mathbb{R}
\]
is continuous.
This can be done by repeating the argument in Lemma \ref{lem-measurability}. 

We show \eqref{eqn-6.33}.
Choose an arbitrary element $t_0\in C$.
Since by \eqref{eqn-conv_D(0,T;U')}

\begin{equation*}
\lim_{n \to \infty} \langle \tun,\phi \rangle = \duality{\tu}{\phi} \mbox{ in  } \mathbb{D}([0,T]) \;\; {\mbox{ on } \Omegabar},
\end{equation*}
it follows from \cite[Corollary\ 2.4.2]{Kallianpur_Xiong} that
\begin{equation*}
\label{conv_g_jumps-new}
\lim_{n \to \infty} \sum_{s\in [0,\cdot]} g(\Delta \langle \tun(s), \phi \rangle) = \sum_{s\in [0,\cdot]} g(\Delta \duality{\tu(s)}{\phi})
\;\;\mbox{in  $\mathbb{D}([0,T])$} {\mbox{ on }  \Omegabar},
\end{equation*}
i.e. the following sequence of functions
\begin{equation*}
[0,T] \ni t \mapsto \sum_{s\in [0,t]} g(\Delta \langle \tun(s), \phi \rangle)    
\end{equation*}
converges in $\mathbb{D}([0,T])$ to a function  $[0,T] \ni t \mapsto \sum_{s\in [0,t]} g(\Delta \duality{\tu(s)}{\phi})$ {$\mbox{ on }  \Omegabar$}.

Since $t_0 \in C$ we infer that the process 
\begin{equation*}
    t \mapsto \sum_{s \leq t} g(\Delta \duality{\tu(s)}{\phi}),
\end{equation*}    
    with probability $1$ has no jump at $t_0$.
By the continuous mapping theorem, see \cite[Theorem\ 2.3]{van_der_Vaart} and \cite[Theorem\ 12.5(i)]{Billingsley_1999},

\begin{equation}
\label{conv_g_jumps-new-2}
\sum_{s\in [0,t_0]} g(\Delta \langle \tun(s), \phi \rangle) \to \sum_{s\in [0,t_0]} g(\Delta \duality{\tu(s)}{\phi}) \qquad \Pbar\mbox{ -a.s.}
\end{equation}
On the other hand, by Lemma \ref{lem-measurability} we infer  that $\Pbar$-almost surely,
\begin{equation*}
\int_0^{t_0} \int_{Y} g(\ip{P_n F(s,\tun(s),y)}{\phi}{\rH})\, \nun(\ud y) \ds \to \int_0^{t_0} \int_{Y} g(\ip{F(s,\tu(s),y)}{\phi}{\rH}\,)\, \nun(\ud y) \ds.
\end{equation*}
By combining \eqref{eqn-Psi_n} and  \eqref{eqn-Phi} with \eqref{conv_g_jumps-new-2} and the last equality we complete the proof of the Lemma.
\end{proof}

\begin{lemma}
\label{lem-purely discontinuous-from the proof of main result}
Assume that $\phi \in\rU$ and a function   $g\colon \R \to \R_+$
satisfies Condition \ref{condition-g function}. Let $C=C(\tu)$  be  the set defined by 	equality \eqref{eqn-C_set}.
Then for every $t \in C$,
\begin{equation}
\label{eqn-expectation_equals_0}
\Ebar \left[\sum_{s\in [0,t]} g(\Delta \langle \tu(s), \phi\rangle)-\int_0^t\int_{Y} g\left( \ip{F(s,\tu(s),y)}{\phi}{\rH}\right)\, \nu(\ud y) \ds \right]=0.
\end{equation}

\end{lemma}
\begin{proof}[Proof of Lemma \ref{lem-purely discontinuous-from the proof of main result}]
Let us choose and fix  $\phi\in\rU$ and, for the time being, an arbitrary   $t\in [0,T]$. We will prove
equality \eqref{eqn-expectation_equals_0}.
 For this purpose we consider  the maps $\Psi_n$, $n\in \mathbb{N}$, and $\Psi\colon \zcal_T \to  \mathbb{D}([0,T]; \mathbb{R})$ defined by
\eqref{eqn-Psi_n} and \eqref{eqn-Phi} respectively. Since $g$ satisfies Condition \ref{condition-g function},
from Lemma \ref{lem-purely discontinuous-New} we infer that
	\begin{align}
	\label{compensated_integral-New}
	\bigl(\Psi_n(u_n)\bigr)(t)
	&= \int_0^t \int_{Y} g(\ip{P_n F(s,{u}_n(s-),y)}{\phi}{\rH}\,)) \, \tilde{{\eta}}(\ud y, \ud s), \;\;\mathbb{P}\mbox{-a.s}.
	\end{align}
	On the other hand, by Lemma \ref{lem-measurability} the maps $\pi_t\circ \Psi_n\colon \zcal_T \to \mathbb{R}$ and $\pi_t\circ \Psi\colon \zcal_T \to \mathbb{R}$ are measurable.  Hence, owing to the equality of laws of $\tun$ and $u_n$ and \eqref{compensated_integral-New}  we see that for every $n \in \mathbb{N}$ we have
	\begin{equation*}\label{Eq:Conv-in L1}
	\mathbb{E} \left[ \Psi_n(u_n)(t) \right] = \overline{\E} \left[ \Psi_n(\tun)(t) \right]=0
	\end{equation*}
and
\begin{equation*}\label{eqn-6.43}
\begin{aligned}
\overline{\E} \left[ \vert\Psi_n(\tun)(t)\vert^2 \right] &= \mathbb{E} \left[ \vert\Psi_n(u_n)(t)
\vert^2 \right] \\ &= {\E} \left[ \left( \int_0^t \int_{Y} g(\ip{P_n F(s,u_n(s-)},y){\phi}{\rH}\,) \tilde{\eta}(\ud y) \ds \right)^2 \right].
\end{aligned}
  \end{equation*}

From the last equality along with the It\^o isometry for $\mathbb{R}$-valued processes and the same argument as in \eqref{eqn-4.15} we infer that there exist a constant $C_0>0$ such that for all $n\in \mathbb{N}$,
\begin{align*}
&\overline{\E} \left[ (\Psi_n(\tun)(t))^2 \right] =  {\E} \left[ \int_0^t \int_{Y} \vert g(\ip{P_n F(s,{u}_n(s),y)}{\phi}{\rH})\, \vert^2 \,  \nun(\ud y) \ds \right] \leq C_0.
\end{align*}

From now on we  assume that $t\in C$.
Thus,   the family  $\{\Psi_n(\tun)(t): n \in\mathbb{N}\}$ is uniformly integrable  and therefore by the Vitali Theorem and  assertion \eqref{eqn-6.33} in  Lemma \ref{Lem:Compensator} we deduce  that 
\begin{equation*}
\label{convergence_of_Psi_n_to_Psi}
\te[\Psi_n(\tun)(t)] \to \Ebar \left[ \Psi(\tu)(t) \right].
\end{equation*}
Therefore we infer that
\begin{equation*}\overline{\E} \left[\Psi(\tu)(t)\right]=0 .\end{equation*}
Thus, we have  proved  equality  \eqref{eqn-expectation_equals_0} as claimed.
\end{proof}

\begin{proof}[Proof of Proposition \ref{prop-purely discontinuous}]
Let us choose and fix $\phi\in\rU$. Let $C$ be the set defined by equality \eqref{eqn-C_set}.
We divide the proof into two steps.

\textbf{Step 1.} We will  prove that the martingale $\Mbarphi{}$ satisfies  equality \eqref{eqn-purely_discontinuous_martingale} for every  $t\in C$.  For this purpose let us choose and fix $t\in C$.

Let us consider a function $g\colon \mathbb{R} \ni r \mapsto r^2  \in [0,\infty)$.
Then we can find  a sequence  $(g_m)_{m=1}^\infty $  of functions such that each $g_m$ satisfies Condition \ref{condition-g function} with $a=m$ and $g_m(r) \toup  g(r)$ for all $r \in \R$.
Since $t\in C$, we can apply Lemma \ref{lem-purely discontinuous-from the proof of main result} and see that the identity \eqref{eqn-expectation_equals_0} is satisfied  with $g$ replaced by $g_m$.  Thus, by the Monotone Convergence Theorem we deduce that
\begin{equation*}
\lim_{m\to \infty}\Ebar \left[\sum_{s\in [0,t]} g_m(\Delta \langle \tu(s), \phi\rangle)-\int_0^t\int_{Y} g_m\left( \ip{F(s,\tu(s),y)}{\phi}{\rH}\right)\, \nu(\ud y) \ds \right]=0.
\end{equation*}
That is, we have
\begin{equation}
\label{expectation_of_squared_jumps-New}
\overline{\E} \left[ \sum_{s\in [0,t]} (\Delta \;\duality{\tu(t)}{\phi})^2 \right]
= \overline{\E} \left[ \int_0^t \int_{Y} \ip{F(s,\tu(s),y)}{\phi}{\rH}^2 \, \nun(\ud y) \ds \right].
\end{equation}
Note that finiteness of both terms above follows form Assumption F.\ref{assum-F2}\label{assum-F2-used0} and  Lemma  \ref{lem-Galerkin_estimates'}.

Next let $\Mbarphi{}$ be the martingale defined in \eqref{definition_M_bar-phi}. It follows from the proof of  \cite[Proposition\ I.4.50b)]{Jacod_Shiryaev}  that  for every $t \in [0,T]$, 
\begin{equation*}\overline{\E} \left[ \langle \Mbarphi{},\Mbarphi{} \rangle(t) \right] = \overline{\E} \left[ [\Mbarphi{},\Mbarphi{}](t) \right].
\end{equation*}
From this equality along with 
$\Pbar$-a.s. for $s\in [0,T]$,    \begin{equation*}\Delta \left( \duality{\tu}{\phi} \right)(s) = \Delta \Mbarphi{} (s)  \mbox{ and }   \langle \Mbarphi{\mathrm{c}},\Mbarphi{\mathrm{c}}\rangle (s) = \langle \dualityc{\tu}{\phi},
  \dualityc{\tu}{\phi}\rangle (s)\end{equation*} 
  by
 \eqref{expectation_of_squared_jumps-New} and Proposition \ref{prop_angle_bracket} we infer  the following chain of equations
\begin{align}
\label{quadratic_var_sum_squared_jumps-New}
&\overline{\E} \left[ \sum_{s\in [0,t]} (\Delta \, \duality{\Mbar(s)}{\phi})^2 \right]
= \overline{\E} \left[ \sum_{s\in [0,t]} (\Delta \,\duality{\tu(s)}{\phi})^2 \right]  \\
&= \overline{\E} \left[ \int_0^t \int_{Y} \ip{F(s,\tu(s),y)}{\phi}{\rH}^2 \, \nun(\ud y) \ds \right]
= \overline{\E} \left[ \langle \Mbarphi{},\Mbarphi{} \rangle(t) \right]  = \overline{\E} \left[ [\Mbarphi{},\Mbarphi{}](t) \right].
\notag
\end{align}

Let us recall that by \cite[Theorem\  3.4.2]{Kallianpur_Xiong}, see also \cite[Lemma\ 4.51 and Theorem\ 4.52]{Jacod_Shiryaev},  we have $\Pbar$-a.s.,
\begin{equation}
\label{two_brackets_and_jumps-New}
[\Mbarphi{},\Mbarphi{}](t)- \sum_{s\leq t} \bigl( \Delta \Mbarphi{}(s)\bigr)^2
= \big\langle \Mbarphi{\mathrm{c}},\Mbarphi{\mathrm{c}}\big\rangle(t), \;\; t \in [0,T].
\end{equation}
Thus the random variable defined by \eqref{two_brackets_and_jumps-New} is,  on the one hand non-negative, because the right-hand side of \eqref{two_brackets_and_jumps-New} is non-negative,  and  on the other hand, by \eqref{quadratic_var_sum_squared_jumps-New}, has expectation equal to $0$. Therefore this random variable is  almost surely equal to $0$, i.e.
$\Pbar$-a.s
\begin{equation*}
  [\Mbarphi{},\Mbarphi{}](t)- \sum_{s\in  [0,t]} \bigl( \Delta \Mbarphi{}(s)\bigr)^2
= 0, \;\; t \in C.
\end{equation*}
 Therefore the proof in Step 1 is complete.

\textbf{Step 2.}  In Step 1 we proved that
the martingale $ \Mbarphi{} $ satisfies  equality \eqref{eqn-purely_discontinuous_martingale} for every  $t\in C$. Now we will show that  it satisfies equality \eqref{eqn-purely_discontinuous_martingale} for every  $\tinOT$.

By applying  \cite[Lemma\ 4.51 and Theorem\ 4.52]{Jacod_Shiryaev}, see also formula \eqref{two_brackets_and_jumps-New} above, we infer
that the predictable quadratic variation of the process $\Mbarphi{\mathrm{c}}$ is  equal to $0$ for every $t \in C$.
Since the  process $\Mbarphi{\mathrm{c}}$   is a continuous martingale and since by \cite[Proposition\  17.2]{Metivier} the predictable quadratic variation of a continuous martingale is continuous,   by the density of the set $C$ in the interval $[0,T]$, we infer
the predictable quadratic variation of the process $\Mbarphi{\mathrm{c}}$ is  equal to $0$ for every $t \in [0,T]$.
Thus, by \cite[Lemma\ 4.51 and Theorem\ 4.52]{Jacod_Shiryaev}, see also formula \eqref{two_brackets_and_jumps-New} above, we conclude
that the martingale $\Mbarphi{}$ satisfies  equality
\eqref{eqn-purely_discontinuous_martingale} for every  $\tinOT$.
This finishes the proof of Step 2.

From Step 1 and Step 2 we  conclude that the square integrable martingale $\Mbarphi{}$ is a purely discontinuous martingale.
\end{proof}

\subsection{Identifying the compensator \texorpdfstring{$M_\phi$}{Mp}}

In this subsection we work in the framework of  the new filtered probability space as listed at the beginning of Section \ref{subsec_identification}. {\color{black} The main result of this section, i.e. Proposition \ref{prop_compensated_martingale},  states that for any Borel set $A$, separated from $0$, the process $\overline{N}_A$ defined in \eqref{eqn-compensated_martingale_2} is a martingale. The proof consists of several preparatory lemmas. Firstly, we define some continuous functions $a_k$ approximating an indicator function $\1_A$ and prove certain properties of those. Then we use them to construct two families of martingales $\overline{N}_{a_k,n}$ and $\overline{N}_{a_k,n}$ approximating $\overline{N}_A$. We then show that $\overline{N}_A$ is martingale for every set which is closed, bounded and separated from $0$. Finally, the martingale property follows for any Borel set separated from $0$.}

Note that in this whole sub-section we consider the processes and their jumps without any  duality pairing. Let us recall that
according to Lemma \ref{lem-tM is a martingale} the process ${\Mbar}$ is a square integrable $\rU^\prime$-valued  ${\tF}$-martingale.
Moreover, by the same lemma, all trajectories of the martingale ${\Mbar}$ are $\rU^\prime$-valued c\`adl\`ag. Note  that the $\rU^\prime$-valued jump process $\Delta \Mbar$ is well-defined.
Finally, by Proposition \ref{prop-purely discontinuous},  the martingale  ${\Mbar} $ is purely discontinuous. Thus it  follows from \cite[Corollary\ 2.4.2]{Kallianpur_Xiong} that for every continuous
function $\tilde{g}\colon \rU^\prime \to \mathbb{R}$ vanishing in a ball around the origin $0\in\rU^\prime $, the real  process $G(\Mbar)$ defined by
\begin{align}\label{eqn-KX_2.4.10}
(G(\Mbar))(t)\coloneqq\sum\limits_{0<s\leq t} \tilde{g}(\Delta \Mbar(s)), \;\;\; \tinOT,
  \end{align}
{is c\`adl\`ag and all its trajectories belong to the Skorokhod space $\mathbb{D}([0,T])$.}

 The process $\Delta \Mbar$ can also be seen  as an  ${\rH}$-valued process  provided  the limit in formula \eqref{eqn-f-jumps} is understood in the weak topology of $\rH$.
Moreover,  since by Lemma \ref{lem-tM is a martingale} the process $\Mbar$ is ${\tF}$-adapted  and $\rU^\prime$-valued c\`adl\`ag, we infer from \cite[Corollary\ I.1.25]{Jacod_Shiryaev}  and  the separability of $U^\prime$ that  both the processes  $\left(\Mbar(t-)\right)_{t\in [0,T]}$ and $\Delta \Mbar$ (resp. $G(\Mbar)$) are  $\rU^\prime$-valued (resp. $\mathbb{R}$-valued)  ${\tF}$-optional.

\begin{remark}
\label{rem_compansated_martingale}
Because ${\Mbar} $ is  a   $\rU^\prime$-valued c\`adl\`ag process,  according to an infinite-dimensional version of 
\cite[Proposition\ II.1.16, p. 69]{Jacod_Shiryaev}, there exists an  $\overline{\mathbb{N}}$-valued random measure $\mu$ on $[0,T]\times\rU^\prime$ such that for every $\tomega\in \Omegabar$,
\begin{equation*}\label{eqn-JS_II.1.16}
  \mu(\tomega;(s,t]\times A)=\sum_{r \in (s,t]}\1_{\{\Delta\, \Mbar(\tomega,s)\not=0\}}\1_A(\Delta \Mbar(\tomega,s))
\end{equation*}
for  all pairs $(s,t)\in [0,T]^2_{+}$ and every  set $A \in \mathscr{B}(\rU^\prime)$.
Hence,
\[
\sum\limits_{0<s\leq t} \1_A(\Delta \Mbar(\tomega,s))= \mu(\tomega;(0,t]\times A), \;\; \tomega\in \Omegabar.
\]
Let us note that the left-hand side in the identity above is the
first term on the right-hand side of  the definition \eqref{eqn-compensated_martingale_2} below. \\
Note also that by \cite[Definition\ II.1.13, p.\ 68]{Jacod_Shiryaev}, the random measure  $\mu$  is $\tF$-optional and, by \cite[Definition\ II.1.4 and Prop.\ I.1.24]{Jacod_Shiryaev}, the left-continuous and   $\tF$-adapted function
$\1_{\Omegabar\times (s,t]\times A}$ is  $\tF$-optional as well.
Therefore, by  \cite[Definition\ II.1.6, p.\ 66]{Jacod_Shiryaev} we deduce that
the process
\[\Omegabar \times [0,T] \ni (\tomega,t) \mapsto  \sum\limits_{0<s\leq t} \1_A(\Delta \Mbar(\tomega,s)) \in \mathbb{R}
\]
is $\tF$-optional.  Hence by \cite[Proposition\ I.1.21 a)]{Jacod_Shiryaev},  we infer that it is $\tF$-adapted.
In particular,  for every $\tinOT$, the random variable
\[\Omegabar  \ni \tomega \mapsto  \sum\limits_{0<s\leq t} \1_A(\Delta \Mbar(s)) \in \mathbb{R}
\]
is $\bar{\fcal}_t$-measurable.
\end{remark}

\begin{definition}\label{def-class A}
\begin{enumerate}
    \item We denote by  $\mathscr{A}_0$ the set of Borel sets $A \subset\rU^\prime \setminus\{0\}$ such that
{\color{black}\begin{equation}\label{ass-compensated martingale}
\Ebar \left[ \sum\limits_{0<s\leq T} \1_A(\Delta \Mbar(s)) \right] < \infty.
\end{equation}}

\item Let $(\alpha,\beta)  \in (0, \infty)^{2}_{+}$. We denote 
\begin{align*}\label{eqn-A-S_alpha-beta}
 S_{\alpha,\beta} &\coloneqq \{ u \in\rU^\prime : \alpha \leq \abs{u}_{\rU^\prime} \leq \beta \},
 \\
 S_\alpha&\coloneqq\{u\in U^\prime: \alpha \le \abs{u}_{U^\prime}\}.
\end{align*}

By $\Aoneab$, respectively $\Atwoa$, we will denote  the set of all closed subsets $A$ of $U^\prime \setminus\{0\}$ such that 
\begin{equation*}\label{eqn-A-S_alpha-beta}
  A \subset S_{\alpha,\beta} \mbox{ or  respectively, }A\subset S_\alpha .
\end{equation*}
\item  We put
\[
\mathscr{A}_1\coloneqq\bigcup_{(\alpha,\beta)  \in (0, \infty)^{2}_{+}} \Aoneab.
\]
\end{enumerate}
\end{definition}

We will now present the main result of the subsection, {\color{black}which is our version of \cite[Lemma\ 6.1.11]{Kallianpur_Xiong}},  a proof of which will be the focus of the remainder of this subsection.

\begin{proposition}
\label{prop_compensated_martingale}
Let $A \in \mathscr{A}_0$. Then, the  process $\Nbar_A$ defined by
\begin{equation}\label{eqn-compensated_martingale_2}
\Nbar_A(t) \coloneqq \sum\limits_{0<s\leq t} \1_A(\Delta \Mbar(s)) - \int_0^t \int_{Y} \1_A(F(s,\tu(s),y))\, \nun(\ud y) \ds, \;\;  \tinOT,
\end{equation}
is an   ${\tF}$-martingale.
\end{proposition}

Let us choose and fix  $\beta>\alpha >0$. We also fix a set $A\in \Aoneab$. 

Consider a sequence $(a_\emm)_{\emm = 2}^\infty $ of continuous and bounded functions on $\rU^\prime$ defined by
\begin{equation}\label{eqn-a_n}
a_\emm(u) \coloneqq f \left( \frac{k^2 \rho(u,A)^2}{\alpha^2} \right), \;\;u\in \rU^\prime,
\end{equation}
where 
$f:[0,\infty) \to [0,1])$ is a $ C^1$-class  decreasing function such that
\[
\1_{\{0\}} \leq f \leq \1_{[0,\frac34)}.
\]
and $\rho(u,A)$ is the $\rU^\prime$-distance of the vector $u \in\rU^\prime$ from the set $A$.

Note that $a_\emm(u) \in [0,1]$ for every $u\in \rU^\prime$ and  $a_\emm(u)=0$ if $\rho(u,A)\geq \frac{\sqrt{3}\alpha}{2k}$. 

Let us show  few more properties of the functions $a_k$, which will be used later.

\begin{lemma}\label{lem-properties-of-a_k}
Functions $a_k$ defined in \eqref{eqn-a_n} satisfy
\begin{enumerate}
\item[(i)] 
$a_k \todown \1_A \mbox{  as } k \to \infty$.
\item[(ii)] The support of $a_k$ is contained in $S_{\frac{\alpha}{2}}$.
\item[(iii)] For future reference we also show  that
\begin{equation}\label{linear_growth_of_a_k}
a_k(u) \leq \left(\frac{2}{\alpha} \abs{u}_{\rU^\prime}\right)^{\ell} \qquad \mbox{  for all } u \in\rU^\prime, k \geq 2, \ell\ge 1.
\end{equation}
\item[(iv)] For $u_1,u_2 \in \rU^\prime$
\begin{equation}\label{eqn-a_k_Lipschitz}
\vert a_\emm(u_2)-a_\emm(u_1) \vert
\leq  \frac{k^2}{a^2}\norm{f}_{\mathrm{Lip}} \abs{u_1-u_2}_{\rU^\prime} \bigl( \rho(u_1,A)+ \rho(u_2,A)\bigr),
\end{equation}
where  $\norm{f}_{\mathrm{Lip}}$ is the Lipschitz constant of the function $f$.
\end{enumerate}
\end{lemma}
Let us note that inequality \eqref{linear_growth_of_a_k} with $\ell =2 $ is used in the proof of Lemma \ref{lem-Claim 5.1}.

\begin{proof}
Points (i) and (ii) are obvious. We show (iii). Indeed, if $u$ is such that $\abs{u}_{\rU^\prime}\geq \frac{\alpha}{2}$, then inequality \eqref{linear_growth_of_a_k} follows from
\[a_k(u) \leq 1 \leq \left(\frac{2}{\alpha} \abs{u}_{\rU^\prime}\right)^{\ell} \mbox{  for all } u \in\rU^\prime, k \geq 2, \ell\ge 1.
\]
On the other hand, if $u$ is such that $\abs{u}_{\rU^\prime} < \frac{\alpha}{2}$, then $\rho(u,A) \geq \frac{\alpha}{2}$ because $A \subset S_{\alpha}$. This implies that $\frac{k^2 \rho(u,A^2)}{\alpha^2} \geq 1$ for all $k\geq 2$ and thus $a_k(u) = 0$.

We show (iv). Let us observe that because of the following property of the distance function $\rho(\cdot, A)$
	\begin{equation}
	\label{distance_is_Lipschitz}
	\abs{\rho(x,A)-\rho(y,A)}
	\leq \abs{x-y}_{\rU^\prime},\; x,y\in\rU^\prime,
	\end{equation}
we may estimate
\begin{equation*}
\begin{aligned}
\vert a_\emm(u_2)-a_\emm(u_1) \vert
 &\leq  \frac{k^2}{a^2}\norm{f}_{\mathrm{Lip}}   \abs{ \rho(u_2,A)^2 - \rho(u_1,A)^2} \\ &\leq  \frac{k^2}{a^2}\norm{f}_{\mathrm{Lip}}
\abs{u_1-u_2}_{\rU^\prime} \bigl( \rho(u_1,A)+ \rho(u_2,A)\bigr).
\end{aligned}
\end{equation*}
\end{proof}

\begin{lemma}\label{lem-Step 1}
Let $(a_k)_{k\ge 2}$ be the sequence defined in \eqref{eqn-a_n}.
Let $z \in \zcal_T$ and $(z_n)$ be a $\zcal_T$-valued sequence  such that
\begin{align}\label{eqn-Conv-L2Hloc}
z_n \to z \mbox{  in } \zcal_T  \mbox{ as $n \to \infty$}.
\end{align}

Then, for every  $m\in \N$ and every integer $k> 2$ we have
\begin{equation}\label{Eq:cont-phi-2}
\begin{aligned}
    &\lim_{n\to \infty} \int_0^t \int_{Y} a_\emm(P_m F(s,z_n(s),y))\, \nu (\ud y) \ds \\
&  \hspace{2truecm}= \int_0^t \int_{Y} a_\emm(P_m F(s,z(s),y))\, \nu (\ud y) \ds
\end{aligned}
\end{equation}
and 	
	\begin{equation}\label{Eq:cont-phi}
	\begin{aligned}
    &
    \lim_{n\to \infty} \int_0^t \int_{Y} a_\emm(P_n F(s,z_n(s),y))\, \nu (\ud y) \ds 
    \\ &  \hspace{2truecm}
    = \int_0^t \int_{Y} a_\emm(F(s,z(s),y)) \, \nu (\ud y) \ds.
	\end{aligned}
    \end{equation}
\end{lemma}

\begin{proof}[Proof of Lemma \ref{lem-Step 1}]
The proof of  convergence \eqref{Eq:cont-phi-2} is similar, in fact easier than, to the proof  convergence \eqref{Eq:cont-phi} and hence we will only
provide a detailed proof of the latter. For this purpose, we shall closely follow the proof of \cite[Lemma\ 6.1.7]{Kallianpur_Xiong}.  Let us first observe  since
$z_n \to z $ in $\mathscr{Z}_T$, then also $z_n \to z$ in $\mathbb{D}([0,T]; H_w)$ and hence by \cite[Lemma 2]{Brzezniak+Hornung+Manna}  $z_n$ and $z$  satisfy the following a priori estimates 
\begin{equation} \label{eqn-z_H_estimate}
 \sup_{n \in \mathbb{N}}\sup_{\sinOT } \bigl| z_n (s){\bigr| }_{\rH}^{2} +\sup_{\sinOT } \bigl| z (s){\bigr| }_{\rH}^{2}  < \infty.
\end{equation}
Let
	\begin{align}
\label{Good-set}
\begin{aligned}
	D_n &\coloneqq \left\{ (t,y) \in [0,T] \times Y : \rho(P_n F(t,z_n(t),y)),A) \leq \frac{\alpha}{2} \right\},  \\
	D &\coloneqq \left\{ (t,y) \in [0,T] \times Y : \rho(F(t,z(t),y), A) \leq \frac{\alpha}{2} \right\}.
\end{aligned}
	\end{align}
It follows from the definition \eqref{eqn-a_n} of the functions $a_\emm$  that for every $k \ge 2 $
	\begin{equation*}
	\label{integrand_zero1}
	(s,y) \in D_n^c \implies a_\emm(P_n F(s,z_n(s),y)) = 0
	\end{equation*}
	and
	\begin{equation}
	\label{integrand_zero2}
	(s,y) \in D^c \implies a_\emm(F(s,z(s),y)) = 0.
	\end{equation}
Therefore we can estimate by \eqref{eqn-a_k_Lipschitz}:
	\begin{align}\label{integral_to_estimate}
	& \abs{\int_0^t \int_Y a_\emm(P_n F(s,z_n(s),y)) - a_\emm(F(s,z(s),y))\, \nu (\ud y) \ds} \notag \\
	&\leq \frac{k^2}{a^2} \norm{f}_{\mathrm{Lip}} \int_0^t \int_Y \abs{ \rho(P_n F(s,z_n(s),y)),A)^2 - \rho(F(s,z(s),y),A)^2} \1_{D_n \cap D}(s,y) \, \nu(\ud y) \ds \notag \\
	&\quad + \int_0^t \int_Y a_\emm(P_n F(s,z_n(s),y)) \1_{D_n \cap D^c}(s,y)\, \nu (\ud y) \ds \notag \\
	&\quad + \int_0^t \int_Y  a_\emm(F(s,z(s),y)) \1_{D_n^c \cap D}(s,y)\, \nu (\ud y) \ds
\\
	&\quad=: \norm{f}_{\mathrm{Lip}} \frac{k^2}{a^2} I_n^1 + I_n^2 + I_n^3 \notag
	\end{align}

Note that for all $u\in U'$
\begin{equation}
\label{trick_with_balls_2}
\mbox{  if } \rho(u,A) \leq \frac \alpha{2}, \mbox{  then } \rho(u,A) \leq \abs{u}_{U'}.
\end{equation}
Indeed, suppose by contradiction that  
there exists $u_0 \in U'$ such that $\rho (u_0,A) \le \frac{\alpha }{2}$  and  ${|u_0|}_{U'} < \rho (u_0,A)$. Since $A \subset S_{\alpha}$, we have that $\rho(0,A)\geq \alpha$. By applying the triangle inequality we infer that 
\begin{align}
\alpha \le  \rho (0,A) \le  \rho (0,u_0) + \rho (u_0,A)
 =  {|u_0|}_{U'} + \rho (u_0,A) <  2 \rho (u_0,A)  \le  \alpha ,
\nonumber\end{align}
which is a contradiction thanks to the strict inequality.

Note that by the definition \eqref{Good-set} of the set $D_n$ and by \eqref{trick_with_balls_2} we have
\begin{equation}\label{trick_with_balls}
\rho(P_n F(s,z_n(s),y),A) \leq \abs{P_n F(s,z_n(s),y)}_{U'}\quad \mbox{ for $(s,y) \in D_n$}.
\end{equation}
Also, by the inequalities \eqref{distance_is_Lipschitz} and \eqref{trick_with_balls} we infer  that for $(s,y) \in D_n \cap D$
\begin{align*}
&\hspace{1.5truecm} \abs{\rho(P_n F(s,z_n(s),y),A)^2 - \rho(F(s,z(s),y),A)^2}  \\
&= \abs{\rho(P_n F(s,z_n(s),y),A) - \rho(F(s,z(s),y),A)} \left( \rho(P_n F(s,z_n(s),y),A) + \rho(F(s,z(s),y),A) \right) \notag \\
&\leq  \normb{P_n F(s,z_n(s),y) - F(s,z(s),y)}{\rU^\prime}{} \left( \normb{P_n F(s,z_n(s),y)}{\rU^\prime}{}  + \normb{F(s,z(s),y)}{\rU^\prime}{}  \right).
\notag
\end{align*}
Using this together with the H\"older inequality we can estimate the term $I_n^1$ as follows:
\begin{align*}
I_n^1 &\leq \left( \int_0^t \int_{Y} \normb{P_n F(s,z_n(s),y) - F(s,z(s),y)}{\rU^\prime}{2} \1_{D_n\cap D}(s,y) \, \nu(\ud y) \ds  \right)^{1/2} \\
& \times \biggl( \left(\int_0^t \int_{Y} \normb{P_n F(s,z_n(s),y) }{\rU^\prime}{2} \1_{D_n\cap D}(s,y) \, \nu(\ud y) \ds  \right)^{1/2}
\\&\hspace{2truecm}+\left(\int_0^t \int_{Y} \normb{ F(s,z(s),y)}{\rU^\prime}{2} \1_{D_n\cap D}(s,y) \, \nu(\ud y) \ds  \right)^{1/2}
\biggr)
\end{align*}

Since  the linear map  $P_n$ is a contraction in $\rH$, see  \eqref{eqn-P_n contraction on H},  by the linear growth condition \eqref{eqn-F_linear_growth} in  Assumption F.\ref{assum-F2}\label{assum-F2-used7},
by the continuity of the embedding  ${\rH} \embed\rU^{\prime }$, see \eqref{eqn-embeddings}, 
and by \cite[Lemma 2]{Brzezniak+Hornung+Manna},  there exists a constants $c>0$ and $C_2>0$ such that for every $n \in \mathbb{N}$,
\begin{align*}
&\int_0^t \int_{Y} \normb{P_n F(s,z_n(s),y) }{\rU^\prime}{2} \1_{D_n\cap D}(s,y) \, \nu(\ud y) \ds  \\
&\leq c \int_0^t \int_{Y} \normb{P_n F(s,z_n(s),y) }{\rH}{2} \, \nu(\ud y) \ds \leq \int_0^T \int_{Y} \normb{F(s,z_n(s),y) }{\rH}{2} \, \nu(\ud y) \ds
\\
&\leq C_2\int_0^T  (1+\normb{z_n(s) }{\rH}{2} ) \ds \leq C_2(T+\sup_{s\in [0,T]}\normb{z_n(s) }{\rH}{2} ).
\end{align*}
A similar inequality holds for the limit process $z$, i.e.
\begin{align*}
&\int_0^t \int_{Y} \normb{ F(s,z(s),y)}{\rU^\prime}{2} \1_{D_n\cap D}(s,y) \, \nu(\ud y) \ds  \leq c\int_0^t \int_{Y} \normb{ F(s,z(s),y) }{\rH}{2} \, \nu(\ud y) \ds
\\&\leq C_2\int_0^T  (1+\normb{z(s) }{\rH}{2} ) \ds \leq C_2(T+ \sup_{s\in [0,T]}\normb{z(s) }{\rH}{2} ).
 \end{align*}
Moreover, by \eqref{eqn-z_H_estimate} we deduce  the following estimate
\begin{equation*}
\label{Eq:F-Uprime}
I_n^1
\leq C \left( \int_0^t \int_{Y} \normb{P_n F(s,z_n(s),y) - F(s,z(s),y)}{\rU^\prime}{2} \1_{D_n}(s,y) \1_D(s,y) \, \nu(\ud y) \ds  \right)^{1/2}.
\end{equation*}
By the triangle inequality
\begin{equation*}\label{Eq:F-Uprime-1}
\begin{aligned}
I_n^1
\leq 2C \biggl( \int_0^t \int_{Y} \biggl(& \normb{P_n F(s,z_n(s),y) - P_n F(s,z(s),y)}{\rU^\prime}{2} \\ & \quad + \normb{P_n F(s,z(s),y) -  F(s,z(s),y)}{\rU^\prime}{2}  \biggr) \1_{D_n\cap D}(s,y) \, \nu(\ud y) \ds  \biggr)^{1/2}.
\end{aligned}
\end{equation*}
Next, by Lemma \ref{lem_P_n}(iii)  we find that $\sup_{n \in \N} \norm{P_n}_{\mathscr{L}(\rU^\prime,\rU^\prime)} \leq 1$. Therefore, we have
\begin{equation}\label{eqn-F-Uprime-2}
\begin{aligned}
I_n^1
\leq & 2C  \biggl( \int_0^t \int_{Y} \normb{ F(s,z_n(s),y) -  F(s,z(s),y)}{\rU^\prime}{2} \,  \nu(\ud y) \ds  \\  & \quad +  \int_0^t \int_{Y} \normb{P_n F(s,z(s),y) -  F(s,z(s),y)}{\rU^\prime}{2} \,  \nu(\ud y) \ds  \biggr)^{1/2}.
\end{aligned}
\end{equation}

By using  Assumption F.\ref{assum-F4} \label{assum-F4-used-2}
 and \eqref{eqn-Conv-L2Hloc} we see that the first term on the right hands side of \eqref{eqn-F-Uprime-2}  converges to 0 as $n \to \infty$. Regarding  the second term,  we argue as follows.
Since the embedding $\rH \embed \rU^\prime$ is continuous, we have
\begin{align}\label{eqn-a01}
\begin{aligned}
&\int_0^t \int_{Y} \normb{P_n F(s,z(s),y) -  F(s,z(s),y)}{\rU^\prime}{2} \,  \nu(\ud y) \ds
\\
&\leq C \int_0^t \int_{Y} \normb{P_n F(s,z(s),y) -  F(s,z(s),y)}{\rH}{2} \,  \nu(\ud y) \ds.
\end{aligned}
 \end{align}
Next, we have
\begin{align*}
 \normb{P_n F(s,z(s),y) -  F(s,z(s),y)}{\rH}{2} &=
 \normb{(I-P_n) F(s,z(s),y)}{\rH}{2}, \\
\normb{(I-P_n) F(s,z(s),y)}{\rH}{2} &\todown 0 \mbox{  as } n\to \infty.
\end{align*}
Finally, by \eqref{eqn-z_H_estimate} 
\begin{align*}
\int_0^t \int_{Y} \normb{(I-P_n) F(s,z(s),y)}{\rH}{2} \,  \nu(\ud y) \ds &\leq
\int_0^t \int_{Y} \normb{F(s,z(s),y)}{\rH}{2} \,  \nu(\ud y) \ds \\
 &\hspace{-2truecm}\leq C_2 \int_0^t (1+
\normb{z(s))}{\rH}{2} ) \ds <\infty.
\end{align*}
Hence, by the Lebesgue Dominated Convergence Theorem we infer that
\begin{equation*}
\int_0^t \int_{Y} \normb{P_n F(s,z(s),y) -  F(s,z(s),y)}{\rU^\prime}{2} \,  \nu(\ud y) \ds  \to 0 \mbox{ as } n\to \infty.
\end{equation*}
Thus, we proved that the term $I_n^1$ introduced in equality \eqref{integral_to_estimate} satisfies
\begin{equation}\label{Eq:I1tozero}
I_n^1 \to 0 \mbox{ as }n \to \infty.
\end{equation}
	
Next, similarly to \eqref{integrand_zero2}  we see that for $k \geq 2$
	\begin{equation*}
	\rho(P_nF(s,z_n(s),y),A) > \frac{\alpha}{\emm} \implies a_\emm(P_nF(s,z_n(s),y)) = 0.
	\end{equation*}
	Thus,
	\begin{align*}
	I_n^2
	&\leq \int_0^t \nu \left( y \in Y : (s,y) \in D_n \cap D^c, \rho (P_nF(s,z_n(s),y),A) \leq \frac{\alpha}{k} \right) \ds \notag \\
	&\leq  \int_0^t \nu \left( y \in Y : \abs{\rho(P_n F(s,z_n(s),y),A) - \rho(F(s,z(s),y),A)} \geq \left( \frac12 - \frac{1}{\emm} \right) \alpha \right) \ds \notag \\
	&\leq 
    \frac{4\emm^2}{\alpha^2(\emm-2)^2}
    \int_0^t \int_Y \abs{\rho(P_n F(s,z_n(s),y),A) - \rho(F(s,z(s),y),A)}^2 \, \nu(\ud y) \ds
	\end{align*}
	By \eqref{distance_is_Lipschitz}
	\begin{align*}
	I_n^2
	&\leq 
    \frac{4\emm^2}{\alpha^2(\emm-2)^2}
    \int_0^t \int_Y \normb{P_n F(s,z_n(s),y) - F(s,z(s),y)}{\rU^\prime}{2} \, \nu(\ud y) \ds .
	\end{align*}
With this estimate in mind we can follow the same reasoning as in the proof of assertion \eqref{Eq:I1tozero} and show that as $n \to \infty$
\begin{equation*}
I_n^2= \int_0^t \int_Y a_\emm(P_n F(s,z_n(s),y)) \1_{D_n \cap D^c}(s,y)\, \nu (\ud y) \ds \to 0.
\end{equation*}
Similarly, we also prove  that  as $n \to \infty$
\begin{equation*}
I_n^3= \int_0^t \int_Y a_\emm( F(s,z(s),y)) \1_{D_n^c \cap D}(s,y)\, \nu (\ud y) \ds \to 0.
\end{equation*}
Plugging the last two assertions   and \eqref{Eq:I1tozero} into \eqref{integral_to_estimate} implies \eqref{Eq:cont-phi}.
 This completes the proof of Lemma \ref{lem-Step 1}.
\end{proof}

We now proceed with the statement and the proof of the following result.
\begin{lemma}\label{lem-Claim 5.1}
Let $z\in \zcal_T$, $\alpha>0$ and $A\in \Atwoa$. {\color{black}Let $a_k$ be defined by \eqref{eqn-a_n}.}
Then, for every pair $(s,t)\in [0,T]^2_{+}$,
\begin{equation}\label{eqn-Conv-Int-a_k}
\int_s^t\int_Y a_k(F(r, z(r), y) )\, \nu(\ud y) \ud r \to \int_s^t \int_Y \1_A (F(r, z(r), y))  \,  \nu(\ud y) \ud r.
\end{equation}
\end{lemma}

\begin{proof}[Proof of  Lemma \ref{lem-Claim 5.1}] 
It is sufficient to assume that $s=0$.
Let us choose and fix    $z\in \zcal_T$, $\beta>\alpha>0$ and
$A\in \Aoneab$.
Then, thanks to Lemma \ref{lem-properties-of-a_k}(i) we infer  that
\begin{equation}\label{eqn-a.e.-conv-a_k}
 a_k(F(s, z(s), y) ) \todown  \1_A (F(s, z(s), y)) \mbox{ $\Leb \otimes \nu  $-a.e.} .
\end{equation}
Furthermore, since $z\in \zcal_T$ by using the linear growth condition \eqref{eqn-F_linear_growth} from Assumption F.\ref{assum-F2}\label{assum-F2-used8}
and the property  \eqref{linear_growth_of_a_k} with $\ell =2 $ of the function $a_2$, we infer  that
\begin{align*}
\int_0^t \int_{Y} a_2(F(s,z(s), y))\, \nu (\ud y) \ds \le C_2 \int_0^t \int_Y \lVert F(s, z(s), y) \rVert^2_{\rU^\prime} \, \nu(\ud y) \ds \nonumber  \\
\le C_2^\prime \int_0^T (1+ \lVert z(s) \rVert^2_{\rH}) \ds<\infty.
\end{align*}
 Thus, the convergence \eqref{eqn-a.e.-conv-a_k} and the Lebesgue Dominated  Convergence Theorem imply  \eqref{eqn-Conv-Int-a_k} which completes the proof of the Lemma.
\end{proof}

\begin{lemma}\label{lem-Claim 5.2} Let us assume that $z\in \D([0,T];\rU^\prime )$, $\alpha>0$
and $A \in \Atwoa$.
Let $(a_k)_{k\ge 2}$ be a sequence of functions {satisfying assumptions of Lemma \ref{lem-Claim 5.1}.}

Then for every pair $(s,t)\in [0,T]^2_{+}$, the following convergence holds
\begin{equation}\label{eqn-conv_sums}
 \sum_{s< r\le t} a_k(\Delta z(r)) \to \sum_{ s< r\le t} \1_A(\Delta z(r)).
\end{equation}
\end{lemma}

\begin{proof}[Proof of Lemma \ref{lem-Claim 5.2}] Let us choose and fix
$z\in \D([0,T];\rU^\prime )$, $\alpha>0$
and $A \in \Atwoa$.
Since $z$ is c\`adl\`ag there exists a finite set $ r_1<\cdots<r_N \in [s, t]$ at which $z$ has jumps of size no smaller than $\frac{\alpha}{2}$. Hence, if $r\in [s, t]$ is such that  $a_k(\Delta z(r))\not=0$, then $ r \in \{r_1,\ldots,r_N\} $.
Therefore, from part (i) of Lemma \ref{lem-properties-of-a_k} 
we infer that
\begin{align*}
 \sum_{s< r\le t} a_k(\Delta z(r)) &=   \sum_{i=1}^N a_k(\Delta z(r_i))
  \todown \sum_{i=1}^N \1_A (\Delta z(r_i)) =\sum_{ s< r\le t} \1_A(\Delta z(r)).
\end{align*}
The proof of Lemma \ref{lem-Claim 5.2} is complete.
\end{proof}

\begin{lemma}\label{lem-Step 2}
Let $(a_k)_{k\ge 2}$ be the sequence defined in \eqref{eqn-a_n}.
Then, for every $\emm \in \mathbb{N}$, the  real    process  $\Nbar_{a_\emm,n}$ defined by
\begin{equation}
\label{eqn-N_A,k,n_bar}
\Nbar_{a_\emm,n}(t) = \sum\limits_{s\leq t} a_\emm(\Delta \overline{M}_{n}(s)) - \int_0^t \int_{Y} a_\emm( P_n F(s,\tu_n(s),y) )\, \nu (\ud y) \ds, \;\; t \in [0,T],
\end{equation}
is an   $\tF_n$-martingale, where the filtration $\tF_n$ was introduced in Definition \ref{def-filtration_tF_n new}.
\end{lemma}

\begin{proof}[Proof of Lemma \ref{lem-Step 2}]
Let us choose and fix  $\emm \in \mathbb{N}$ with $k\ge 2$.
Let us define an auxiliary real   process ${N}_{a_\emm,n}$ by
\begin{equation}
\label{eqn-N_A,k,n}
{N}_{a_\emm,n}(t) = \sum\limits_{s\leq t} a_\emm(\Delta {M}_{n}(s)) - \int_0^t \int_{Y} a_\emm(P_n F(s, u_n(s),y))\, \nu (\ud y) \ds, \;\; t \in [0,T].
\end{equation}
Notice that by Remark \ref{rem_compansated_martingale} and by \cite[Corollary\ 2.4.2]{Kallianpur_Xiong}, the first terms on the right-hand sides of identities \eqref{eqn-N_A,k,n} and \eqref{eqn-N_A,k,n_bar} are $\mathbb{R}$-valued, ${\F}$-, respectively ${\tF_n}$-optional, c\`adl\`ag processes.

To prove that the process $\Nbar_{a_\emm,n}$ is a $\tF_n$-martingale we argue as follows.
We first notice that because the jumps of the processes $\overline{M}_{n} $, resp.\ ${M}_{n} $,  are the same as the jumps of the processes $\overline{u}_n$, resp.\ $u_n$, see equality \eqref{eqn-tMn},  we have that on $\Omegabar$
\begin{equation*}
\Nbar_{a_\emm,n}(t) = \sum\limits_{s\leq t} a_\emm(\Delta  \overline{u}_n(s)) - \int_0^t \int_{Y} a_\emm( P_n F(s,\tu_n(s),y) )\, \nu (\ud y) \ds, \;\; t \in [0,T],
\end{equation*}
 and on $\Omega$,
\begin{equation*}
{N}_{a_\emm,n}(t) = \sum\limits_{s\leq t} a_\emm(\Delta  {u}_n(s)  ) - \int_0^t \int_{Y} a_\emm( P_n F(s,  u_n(s),y))\, \nu (\ud y) \ds, \;\; t \in [0,T].
\end{equation*}
Next, by identity \eqref{eqn-compensated_integral_1_New} in Lemma \ref{lem-purely discontinuous-New} applied to the function $g=a_\emm$ we infer that
 every  $t \in [0,T]$, the following identity hold $\mathbb{P}$-almost surely,
\begin{align}
{N}_{a_\emm,n}(t)  &=\int_0^t \int_{Y} a_\emm \bigl(   P_n F(s, {u}_n(s),y)  \bigr) \tilde{{\eta}}(\ud s,\ud y), \;\; t\geq 0.
\label{eqn-uncompensated_integral-3-new}
\end{align}
Thus, after noting that the  integrand belongs to the space $\mathbf{F}_p^2$ with $p$ equal to the Poisson point process corresponding to the Poisson random measure $\eta$, we infer that $N_{a_\emm,n}$ is an $\F$-martingale, see \cite[Section\ 3]{Brzezniak_Liu_Zhu}.

Now let us consider  the map $\Psi_{a_\emm,n}\colon  \zcal_T \to \mathbb{D}([0,T]; \mathbb{R})$ defined  by
\begin{equation}\label{eqn-Psi_nk}
\left(\Psi_{a_\emm,n}(u)\right)(t)
= \sum_{s\in [0,t]} a_\emm (\Delta u(s)) -\int_0^t\int_{Y} a_\emm \left( P_n F(s, u(s),y) \right) \, \nu(\ud y) \ds
\end{equation}
for $u \in \zcal_T$ and $\tinOT$.

Thanks to assertion \eqref{Eq:cont-phi-2} in Lemma \ref{lem-Step 1} and \cite[Corollary\ 2.4.2]{Kallianpur_Xiong}, we infer  that the map $\Psi_{a_k,n}$ is measurable. This, along with the measurability of the evaluation map $\pi_t\colon  \mathbb{D}([0,T]; \mathbb{R})\to \mathbb{R} $, see proof of   \cite[Lemma\ 4]{Lindvall_1973} and a related result in \cite[Theorem\ 12.5(ii)]{Billingsley_1999}, implies  that  the function $\pi_t\circ \Psi_{a_\emm,n}$ is measurable for all $\tinOT$. Thus, by the martingale property of $N_{a_\emm,n}$ and the equality of laws of $\tu_n$ and $u_n$ on $\zcal_T$ we deduce that for every pair $(s,t) \in [0,T]^2_{+}$ and
 every continuous and bounded function $h\colon  \mathbb{D}([0,s];\rU^\prime)  \to \mathbb{R}$
the following equality holds
\begin{equation}\label{eqn-4.84}
\mathbb{E} \left[ h(u_{n}{\lvert _{[0,s]}}) \left( N_{a_\emm,n} (t) - N_{a_\emm,n}(s)\right) \right] = 	\overline{\mathbb{E}} \left[ h(\tu_n{\lvert _{[0,s]}}) \left( \Nbar_{a_\emm,n} (t) - \Nbar_{a_\emm,n}(s)\right) \right] = 0.
\end{equation}
Hence, since $h$, $s$ and $t$  are  arbitrary, we infer that the  process $\Nbar_{a_\emm,n}$ is a martingale with respect to the filtration generated by $\overline{u}_n$. 
 Since by Proposition \ref{prop-filtration_tF_n new}, the filtration $\tF_n$ is equal to the usual augmentation of the 
 filtration $\mathbb{H}_n$ generated by $\overline{u}_n$ (see Definition \ref{def-filtration_tF_n new}),  in view of  Lemma \ref{Lem:Indep-Increments-augmentation},  the proof of Lemma \ref{lem-Step 2} is complete.
\end{proof}

Note that according to Remark \ref{rem_compansated_martingale} the process $\Nbar_A$ is well defined.

According to our  next result,  whose proof will be given at the end of this subsection,  in order to prove Proposition \ref{prop_compensated_martingale} it is sufficient to prove it for a certain smaller  class  of sets $A$.

\begin{lemma}\label{lem-reduced}
If the process $\Nbar_A$ defined by \eqref{eqn-compensated_martingale_2} is an   ${\tF}$-martingale for every  set  $A\in  \sigma\bigl(\Aoneab\bigr)$ for some $\beta>\alpha>0$,
 then  the process $\Nbar_A$ is an   ${\tF}$-martingale for every
set $A \in \mathscr{A}_0$.

\end{lemma}
The proof of this lemma requires several lemmata and is postponed to page \pageref{Proof-of-Lemma-reduced}.
We will also prove the following companion result.

\begin{lemma}\label{lem-propostion for A_2}
Let $A\in \mathscr{A}_1$, i.e.\ $A\in \Aoneab$ for some $\beta>\alpha>0$.
Then, the process $\Nbar_A$ defined by
\eqref{eqn-compensated_martingale_2}
is an   ${\tF}$-martingale.
\end{lemma}
Lemma \ref{lem-propostion for A_2} will also be proved later, but for now let us  formulate a corollary to it. Note that 
$\mathscr{B}(S_{\alpha,\beta})=\sigma\bigl(\Aoneab\bigr)$.

\begin{corollary} \label{cor-propostion for A_2}
If  $A\in  \mathscr{B}(S_{\alpha,\beta})$ for some $\beta>\alpha>0$,
then the process $\Nbar_A$ defined by
\eqref{eqn-compensated_martingale_2}
is an   ${\tF}$-martingale.
\end{corollary}

Now, we will formulate the following lemma where therein the filtration $\mathbb{H}= (\mathscr{H}_t)_{t\in [0,T]}$ is the filtration generated by $\tu$, see \eqref{eqn-filtration_limit}.
\begin{lemma}\label{lem-Step 3}
Let $C$ be the set defined in  \eqref{eqn-C_set}. Let  $\Nbar_{a_\emm}=\left( \Nbar_{a_\emm}(t): t\geq 0\right)$ be the $\mathbb{R}$-valued process defined by
\begin{equation}\label{eqn-Nbar_A,k}
\Nbar_{a_\emm}(t) \coloneqq \sum\limits_{s\leq t} a_\emm(\Delta \overline{M} (s)) - \int_0^t \int_{Y} a_\emm(F(s,\tu(s),y))\, \nu (\ud y) \ds,\;\; t\geq 0.
\end{equation}
Then for every pair  $(s,t)\in C^2_{+}$ 
\begin{equation}\label{eqn-NANK_Mart-Prop-setC-0}
\Ebar \bigl[  \Nbar_{a_\emm}  (t)\lvert \mathscr{H}_{s} \bigr]= \Nbar_{a_\emm}(s).
\end{equation}
Moreover,
\begin{align}\label{eqn-Square-NAN-on-C-2}
\sup_{k \ge 2}\;\;\sup_{t \in C  }\; \Ebar \left[ \lvert{\Nbar_{a_\emm}(t)}\rvert^{2} \right] <\infty.
\end{align}
\end{lemma}

\begin{proof}[Proof of Lemma \ref{lem-Step 3}] First we will deal with the proof of assertion \eqref{eqn-NANK_Mart-Prop-setC-0}. For this aim we closely follow \cite{Jacod_Shiryaev} and denote by $\mathscr{D}_t^0$ the $\sigma$-algebra generated by all maps \begin{equation*}\mathbb{D}([0,T], U^\prime)\ni x \mapsto x(s)\in U^\prime,\, s\le t .\end{equation*} We also set $\mathscr{D}_t=\bigcap_{s>t} \mathscr{D}^0_s,\, t\in [0,T]$. Note that $\mathbf{D}=(\mathscr{D}_t)_{t\in [0,T]}$ is a filtration. Next, let us choose and fix a pair  $(s,t)\in C^2_{+}$ and a continuous and  bounded function $h\colon  \mathbb{D}([0,T];\rU^\prime)  \to \mathbb{R}$ such that $h$ is $\mathscr{D}_s$-measurable. By using standard Functional Monotone Class argument as in \cite[Proposition IX.1.1]{Jacod_Shiryaev}, to prove \eqref{eqn-NANK_Mart-Prop-setC-0} it is enough to prove that 
\begin{equation}\label{eqn-NANK_Mart-Prop-setC-2}
\Ebar \bigl[ (\Nbar_{a_\emm}(t)-\Nbar_{a_\emm}(s)) h(\overline{u} ) \bigr] = 0.
\end{equation}

Let us define, compare with \eqref{eqn-Psi_nk},  a map $\Psi_{a_\emm}\colon  \zcal_T \to \mathbb{D}([0,T]; \mathbb{R})$  by
\begin{equation*}\label{eqn-Psi_k}
\left(\Psi_{a_\emm}(u)\right)(t)
=\sum_{s\in [0,t]} a_\emm (\Delta u(s))-\int_0^t\int_{Y} a_\emm \left( F(s, u(s),y) \right) \, \nu(\ud y) \ds.
\end{equation*}
We first claim that
\begin{equation*}\label{Eq:NANK-ASCONV-C}
\lim_{n \to \infty }  \Psi_{a_\emm,n}(\tu_n)(t) -\Psi_{a_\emm}(\tu)(t)=0 \quad \mbox{  on $\bar{\Omega}$. }
\end{equation*}
 We will separately show the convergence of the first and the second terms defining the functions $\Psi_{a_\emm,n}(\tu_n)$ and $\Psi_{a_\emm}(\tu)$.

Since the evaluation map $\pi_t\colon  \mathbb{D}([0,T];\rU^\prime) \ni x \mapsto x(t)\in\rU^\prime $ is continuous if and only if $x$ is continuous at $t$, the convergence of the second terms follows from the definition of the set $C$, the convergence \eqref{eqn-Skorokhod_appl_convergence}
according to which  $\tun \to \tu$ in $\zcal_T$ on $\Omegabar$ and \cite[Corollary \ 2.4.2]{Kallianpur_Xiong}.
The convergence of the first terms follows from  Lemma \ref{lem-Step 1}.

Let us also notice that  $\lim_{n \to \infty }h(\tun ) =h( \tu )$ on $\Omegabar$.  Hence,  we also  have on $\Omegabar$
\begin{equation}\label{eqn-NANK-ASCONV}
\lim_{n \to \infty} \left( \Nbar_{a_\emm,n}(t)-\Nbar_{a_\emm, n}(s) \right) h(\overline{u}_n ) = \left( \Nbar_{a_\emm}(t)-\Nbar_{a_\emm}(s) \right) h(\overline{u} ).
\end{equation}
We now prove the uniform integrability of the random variables
\begin{equation*}\left(\Nbar_{a_\emm,n}(t)-\Nbar_{a_\emm,n}(s)\right) h(\overline{u}_n ), \quad n \in \N.\end{equation*}

Obviously it is sufficient to prove the following stronger property 
\begin{align}\label{eqn-Square-NANK-on-C}
C_4\coloneqq\sup_{n \in \mathbb{N}} \sup_{r \in [0,T] }\Ebar \left[ \lvert{\Nbar_{a_\emm,n}(r)}\rvert^\power \right]<\infty.
\end{align}
In order to prove \eqref{eqn-Square-NANK-on-C}
we observe first that by the measurability of $\pi_t\circ \Psi_{a_\emm,n} $ and  the equality of laws of $\tun$ and $u_n$ on $\zcal_T$
\begin{align*}
& \sup_{r \in [0,T]  }\Ebar \left[ \lvert{\Nbar_{a_\emm,n}(r)}\rvert^\power \right] = \sup_{r \in [0,T]  }\E  \left[\lvert{ N_{a_\emm,n}(r)}\rvert^\power \right] \le \E \left[ \sup_{r \in [0,T]  } \lvert{ N_{a_\emm,n}(r)}\rvert^\power \right].
\end{align*}
Next, by the Burkholder-Davis-Gundy inequality applied to the real-valued process ${N}_{a_\emm,n}$ satisfying equality \eqref{eqn-uncompensated_integral-3-new}, see \cite[Theorem\ 26.12]{Kallenberg_2002}, 
there exists a constant $c_3>0$
such that for all $n \in \mathbb{N}$ and  $k \ge 2$
\begin{align*}
\E \left[ \sup_{r \in [0,T]} \lvert{N_{a_\emm,n}(r)}\rvert^\power \right] &\le c_3\E \left[  \int_0^T \int_Y \left(a_\emm( F_n(s, u_n(s),y))\right)^2  \, \nu (\ud y) \ds\right] \\
&\leq c_3 \frac{2^2}{\alpha^2}
 \E \left[ \int_0^T \int_Y \norm{P_n F(s, u_n(s),y)}_{\rU^\prime}^2 \,  \nu (\ud y) \ds \right]
\\
&\leq c_3^\prime \E \left[ \int_0^T \int_Y \norm{P_n F(s, u_n(s),y)}_{\rH}^2 \, \nu (\ud y) \ds \right],
\end{align*}
where the second inequality is a consequence of the condition \eqref{linear_growth_of_a_k} with $\ell=1$ and the last
inequality is a consequence of the  boundedness of the natural embedding $\rH  \embed\rU^\prime  $.

Since   $P_n\colon  \rH \to \rH$,  $n\in \mathbb{N}$, is an orthogonal  projection, see \eqref{eqn-P_n in H}, by the linear growth condition \eqref{eqn-F_linear_growth} in Assumption F.\ref{assum-F2}\label{assum-F2-used9}
and the estimate (\ref{eqn-H_estimate-p>2}) with $p=2$ we deduce that there exists a constant $c_4>0$ such that for all $n \in \mathbb{N}$ and all $k \ge 2$

\begin{align}\nonumber
 \E \left[    \int_0^T \int_Y \norm{P_n F(s, u_n(s),y)}_{\rH}^2 \, \nu (\ud y) \ds  \right]
&\le
\E \left[  c_4   \int_0^T \int_Y \norm{F(s, u_n(s),y)}_{\rH}^2 \, \nu (\ud y) \ds
\right]
 \\
& \le  c_4    \E \left[ \int_0^T \left( 1 + \norm{u_n(s)}_\rH^\power \right)  \ud s \right] <\infty.
\label{sup_over_0_T}
\end{align}
All the previous inequalities  together with the a' priori estimate (\ref{eqn-H_estimate-p>2}) with $p=2$ from  Lemma \ref{lem-Galerkin_estimates} imply the sought estimate \eqref{eqn-Square-NANK-on-C}.

Now, by the  Vitali Convergence Theorem applied to the almost sure convergence \eqref{eqn-NANK-ASCONV} we infer that \eqref{eqn-NANK_Mart-Prop-setC-2}, and hence \eqref{eqn-NANK_Mart-Prop-setC-0}, holds.

Finally, we will deal with the proof of inequality \eqref{eqn-Square-NAN-on-C-2}.  Let us choose and fix $k\in \mathbb{N}\setminus\{0,1\}$ and $t \in C$. Then by \eqref{eqn-NANK-ASCONV}, the Fatou Lemma and inequality \eqref{eqn-Square-NANK-on-C}  we have the following chain of inequalities
\begin{align*}
\Ebar \left[ \lvert{\Nbar_{a_\emm}(t)}\rvert^\power \right] &= \Ebar \left[ \lvert{ \lim_{n \to \infty} \Nbar_{a_\emm,n}(t) }\rvert^\power \right] = \Ebar \left[ \lim_{n \to \infty} \lvert{ \Nbar_{a_\emm,n}(t) }\rvert^\power \right] \\
& \leq \liminf_{n \to \infty} \Ebar \left[  \lvert{ \Nbar_{a_\emm,n}(t) }\rvert^\power \right] \leq C_4.
\end{align*}

Note that the finiteness follows from Assumption Assumption F.\ref{assum-F2}\label{assum-F2-used10}.
This obviously implies inequality \eqref{eqn-Square-NAN-on-C-2} and  hence, the proof of Lemma \ref{lem-Step 3} is complete.
\end{proof}

\begin{lemma}\label{lem-Step 4}
In the framework of Lemma \ref{lem-Step 3}, the process
$\Nbar_{a_\emm}$ is an $\tF$-martingale.
\end{lemma}

\begin{proof}[Proof of Lemma \ref{lem-Step 4}]
Let us  choose and fix a pair  $(s,t)\in [0,T]^2_{+}$ such that $s<t$.  We consider an  $C^2_{+}$-valued sequence $(s_m,t_m)_{m \in \mathbb{N}}$  satisfying
\begin{itemize}
	\item $s<s_m < t<t_m$ for every $m \in \mathbb{N}$,
	\item $s_m \searrow s$ and $t_m \searrow t$ as $m \to \infty$.
\end{itemize}
Let $\mathbf{D}=(\mathscr{D}_t)_{t\in [0,T]}$ be the filtration defined in the proof of Lemma \ref{lem-Step 3}. We also fix an arbitrary $h: \mathbb{D}([0,T]; U^\prime) \to \mathbb{R} $ bounded, continuous and $\mathscr{D}_s$-measurable function. 
Observe that by Lemma \ref{lem-Step 3} and its proof, see equality \eqref{eqn-NANK_Mart-Prop-setC-2} and inequality \eqref{eqn-Square-NAN-on-C-2},  we have

\begin{align}
\label{Eq:NANK_Mart-Prop-setC-1}
&\Ebar\bigl[  \left(\Nbar_{a_\emm}  (t_m)-\Nbar_{a_\emm}(s_m)\right)h(\tu)\bigr]=0,\qquad \mbox{ for all } m \in \mathbb{N},
\\
\label{Eq:Square-NAN-on-C-3}
&\sup_{k \ge 2} \sup_{m \in \mathbb{N}  }\Ebar \left[ \lvert{\Nbar_{a_\emm}(t_m)}\rvert^\power \right] <\infty.
\end{align}
The second property above and the boundedness of $h$  imply that the sequences $(\Nbar_{a_\emm}(t_m))_{m\in \mathbb{N}}$,  $(\Nbar_{a_\emm}(s_m))_{m\in \mathbb{N}}$  and $\left( \bigl[\Nbar_{a_\emm}(t_m)-\Nbar_{a_\emm}(s_m)\bigr]\right)_{m \in \mathbb{N}}$ are uniformly integrable. Thus, we can proceed as in Step 3 of the proof of Lemma \ref{lem-tM is a martingale}  and prove that
\begin{align}
 &\Nbar_{a_\emm} (s_m)\to  \Nbar_{a_\emm}(s) \mbox{  in } L^1(\Omegabar),\label{Eq:Conv-NAK-L3}\\
& \Ebar\bigl[  \left(\Nbar_{a_\emm}  (t)-\Nbar_{a_\emm}(s)\right)h(\tu)\bigr] =  \lim_{m\to \infty}\Ebar\bigl[  \left(\Nbar_{a_\emm}  (t_m)-\Nbar_{a_\emm}(s_m)\right)h(\tu)\bigr]=0.\nonumber
\end{align}
The second set of equations  along with a standard Functional Monotone Class argument  imply that 
\begin{equation*}
    \Ebar \bigl[  \Nbar_{a_\emm}  (t)\lvert \mathscr{H}_{s} \bigr]= \Nbar_{a_\emm}(s).
\end{equation*}
Since the filtration $\Fbar=(\tfcal_t)_{t\in [0,T]}$ is the usual augmentation of the filtration $(\mathscr{H}_{t})_{t\in [0,T]}$ generated by $\tu$, we infer from Lemma \ref{Lem:Indep-Increments-augmentation} that $\Nbar_{a_\emm}$ is an $\bar{\mathbb{F}}$-martingale. This completes of the proof of Lemma \ref{lem-Step 4}.
\end{proof}

We now give the promised proof of Lemma \ref{lem-propostion for A_2}.
\begin{proof}[Proof of Lemma \ref{lem-propostion for A_2}] Let us choose and fix $A \in \Atwoa$ for some $\alpha>0$. Let us point out that this set satisfies assumptions of Lemmata   \ref{lem-Claim 5.1} and \ref{lem-Claim 5.2}.
Next,  let us choose and fix a pair  $(s,t)\in [0,T]^2_{+}$. Let us recall that the process ${\Nbar }_A$ has been defined in equality \eqref{eqn-compensated_martingale_2}.

Let us observe that by equality \eqref{eqn-Conv-Int-a_k} from Lemma \ref{lem-Claim 5.1} and by equality \eqref{eqn-conv_sums} from Lemma \ref{lem-Claim 5.2}
we have respectively the following two assertions $\Pbar$-a.s.
 \begin{align*}
& \int_s^t\int_Y a_k(F(r, \tu(r), y) )\, \nu(\ud y) \, \ud r \to \int_s^t \int_Y 1_A (F(r, \tu(r), y))  \,  \nu(\ud y) \dr, \\
&  \sum_{s\le r\le t} a_k(\Delta \Mbar(r)) \to \sum_{s\le r\le t} \1_A(\Delta \Mbar(r)) \mbox{  as }  k\to \infty.
 \end{align*}
 Therefore, in view of definition  \eqref{eqn-compensated_martingale_2} of the process $\Nbar_A$ and definition \eqref{eqn-Nbar_A,k}  of the process  $\Nbar_{a_\emm}$ we infer that
  $\Pbar$-a.s.
 \begin{equation*}
 \Nbar_{a_\emm} (t)- \Nbar_{a_\emm} (s) \to \Nbar_{A}(t)-\Nbar_{A}(s) \mbox{  as }  k\to \infty .
 \end{equation*}

 Let us observe that  by a similar argument we used to prove inequality \eqref{eqn-Square-NAN-on-C-2} in Lemma \ref{lem-Step 3}, and by using
 \eqref{eqn-Square-NAN-on-C-2}, \eqref{Eq:Square-NAN-on-C-3}  and \eqref{Eq:Conv-NAK-L3} we can deduce that
 \begin{equation*}
 \sup_{k \ge 2}  \Ebar \Bigl[\lvert \Nbar_{a_\emm} (s)\rvert^\power + \Ebar \lvert \Nbar_{a_\emm} (t)\rvert^\power  \Bigr] <\infty.
 \end{equation*}
 This  implies that the sequence $\rvert \Nbar_{a_\emm}(t)\rvert$, $k\in \mathbb{N}$,  is uniformly integrable and therefore by the Vitali Convergence Theorem we infer that
 \begin{equation*}
\label{convergence_N_A_k_in_L^2}
 \Nbar_{a_\emm}(t) \to \Nbar_A(t) \mbox{ in $L^1(\Omegabar)$ as } k\to \infty.
 \end{equation*}
 This along with Lemma \ref{lem-Step 4} according to which
\begin{equation*}
\te\left(\left[ \Nbar_{a_\emm}(t) - \Nbar_{a_\emm}(s)\right]\bigg \vert \bar{\mathscr{F}}_s \right)=0,
\end{equation*}
we deduce  that
\begin{equation*}
\begin{aligned}
\te\left(\left[\Nbar_{A}(t) - \Nbar_{A}(s) \right] \biggl\vert \bar{\mathscr{F}}_s\right)= \lim_{k \to \infty} \te\left(\left[\Nbar_{a_\emm}(t) - \Nbar_{a_\emm}(s) \right] \biggl\vert \bar{\mathscr{F}}_s\right)
= 0 \mbox{ in $L^1(\Omegabar)$}.
\end{aligned}
\end{equation*}
This concludes the proof of Lemma \ref{lem-propostion for A_2}.
\end{proof}

The following proof provides some details to the 5 lines long proof of  \cite[Lemma\ 6.1.11]{Kallianpur_Xiong}.
{Since we have not been able to follow that proof we have decided to provide a detailed proof as below. Our proof refers to Lemma \ref{lem-Step 3} which was formulated and proved earlier on. }

\begin{proof}[Proof of Corollary \ref{cor-propostion for A_2}]
 Let us choose and fix $\beta>\alpha >0$ and  $t\geq s\geq  0$. Let $A\in \sigma(\mathscr{A}_1(\alpha, \beta))$. We shall prove that
 \begin{equation*}
     \overline{\mathbb{E}} \left[ \Nbar_A(t)\lvert \bar{\mathscr{F}_s} \right]= \Nbar_A(s).
 \end{equation*}
 For this purpose let us consider the family
 \begin{equation*}
     \mathscr{B}_{\alpha, \beta}= \{ B \in \mathscr{B}(S_{\alpha, \beta}): \overline{\mathbb{E}}\left[ \Nbar_B(t)\lvert \bar{\mathscr{F}_s} \right] = \Nbar_B(s),\, \mbox{ $\tP$-a.s.}  \} .
 \end{equation*}
 Note that if  $B\in \mathscr{B}(S_{\alpha, \beta})$ is closed, then it follows from the proof of  Lemma \ref{lem-propostion for A_2}  that $\overline{\mathbb{E}}\left[\Nbar_B(t)\lvert \bar{\mathscr{F}_s}\right]$ is well-defined. If  $B\in \mathscr{B}(S_{\alpha, \beta})$ is open,  then there exists a closed (the trace of a Borel closed set in $U^\prime$ on $S_{\alpha, \beta}$) $C\in \mathscr{B}_{\alpha, \beta}$ such that $B= S_{\alpha, \beta}\setminus C$ and, by the proof of  Lemma \ref{lem-propostion for A_2}, we see that $\overline{\mathbb{E}} \left[ \Nbar_{B} \right]= \overline{\mathbb{E}}\left[\Nbar_{S_{\alpha, \beta}}\right] -\overline{\mathbb{E}}\left[\Nbar_C\right] <\infty$, which implies that $\overline{\mathbb{E}}\left[\Nbar_B(t)\lvert \bar{\mathscr{F}_s} \right]$ is well-defined.

 Next, it is not difficult to see that $\mathscr{A}_1(\alpha, \beta)$ is a $\pi$-system, i.e., it is  closed under the formation of finite intersections. We also have the following:
 \begin{enumerate}
     \item $S_{\alpha, \beta}\in \mathscr{B}_{\alpha, \beta}$. In fact this follows from the fact that $S_{\alpha, \beta} \in \mathscr{A}_{\alpha, \beta}$ and Lemma \ref{lem-propostion for A_2}.
     \item If $A,B\in \mathscr{B}_{\alpha, \beta}$ with $A\subset B$, then $B\setminus A\in \mathscr{B}_{\alpha, \beta}$. In fact, since $A\subset B$, we have $1_{B\setminus A }=\1_{B}- \1_A$ and $\tP$-a.s.
     \begin{align*}
     \Nbar_{B\setminus A}= & \Nbar_B - \Nbar_A, \\
         \overline{\mathbb{E}}[\Nbar_{B\setminus A}(t)\lvert \bar{\mathscr{F}_s}]=&  \overline{\mathbb{E}}[(\Nbar_{B}(t)-\Nbar_A(t)) \lvert \bar{\mathscr{F}_s}]= \overline{\mathbb{E}}[\Nbar_{B}(t)\lvert \bar{\mathscr{F}_s}]- \overline{\mathbb{E}}[\Nbar_{A}(t)\lvert \bar{\mathscr{F}_s}]\\
        = & \Nbar_{B}(s) -\Nbar_A(s) =\Nbar_{B\setminus A}(s).
     \end{align*}
     \item If $(B_n)_{n \in \mathbb{N}} \subset \mathscr{B}_{\alpha,\beta}, \mbox{  with } B_n \subset B_{n+1} \mbox{  for all } n\in \mathbb{N}$,
 then $B\coloneqq\bigcup_{n\ge 1} B_n \in \mathscr{B}_{\alpha,\beta}$. In fact, it follows from the Monotone Convergence Theorem that $\displaystyle \Nbar_{B_n} (r)\nearrow \Nbar_{B}(r)$ on $\overline{\Omega}$ for any $r\in [0, T]$ and $\tP$-a.s.
 \begin{equation}
 \begin{aligned}
\overline{\mathbb{E}}[\Nbar_B(t) \lvert \bar{\mathscr{F}}_s ] &= \overline{\mathbb{E}}[ \lim_{n\to \infty} \Nbar_{B_n}(t) \lvert \bar{\mathscr{F}}_s ]= \lim_{n\to \infty} \overline{\mathbb{E}} [\Nbar_{
B_n}(t)\lvert \bar{\mathscr{F}}_s]
\\
&= \lim_{n\to \infty} \Nbar_{B_n}(s)= \Nbar_B(s).
 \end{aligned}
 \end{equation}
 \end{enumerate}
From the observations (i)-(iii) we infer that  $\mathscr{B}_{\alpha, \beta}$ is a $\lambda$-system, which along with  \cite[Theorem \ I.3.2]{Billingsley_2012} implies that $\sigma\bigl( \Aoneab\bigr) \subset \mathscr{B}_{\alpha,\beta}$. Thus, from this last result and the definition of $\mathscr{B}_{\alpha, \beta}$, we easily complete the proof of the Corollary \ref{cor-propostion for A_2}.
\end{proof}

\begin{proof}[Proof of Lemma \ref{lem-reduced}]\label{Proof-of-Lemma-reduced}

Let us choose and fix  an arbitrary $A\in \mathscr{A}_0$, i.e., a  Borel set $A \subset\rU^\prime\setminus \{0\}$ for which \eqref{ass-compensated martingale} holds. Let us define a sequence $(A_n)_{n=1}^\infty$ of Borel sets by
\begin{equation*}
\label{eqn_A_n}
A_n\coloneqq A \cap S_{1/n,n}.
\end{equation*}
Let us choose and fix $s<t \in [0,T]$.  Because $A_n$ is a Borel subset of $S_{1/n,n}$, it follows from
 Corollary \ref{cor-propostion for A_2} that for every $n \in \mathbb{N}$,  $\Pbar$-almost surely,
 \begin{equation}\label{eqn-cond at s-1}
 \Ebar[\Nbar_{A_n}(t) \vert \bar{\mathscr{F}}_s] =    \Nbar_{A_n}(s).
 \end{equation}
Observe that  $\displaystyle \Nbar_{A_n} (r)\nearrow \Nbar_{A}(r)$ on $\overline{\Omega}$ for any $r\in [0, T]$. Thus, from \eqref{eqn-cond at s-1} and the Monotone Convergence Theorem we deduce that $\overline{\mathbb{P}}$-almost surely
 \begin{equation}\label{eqn-cond at s}
 \Ebar[\Nbar_{A}(t) \vert \bar{\mathscr{F}}_s] =    \Nbar_{A}(s).
 \end{equation}
Because the pair $(s,t)\in [0,T]^2_+$ is arbitrary, this completes the proof of Lemma \ref{lem-reduced}.
\end{proof}

\begin{proof}[\textbf{End of the proof of Proposition \ref{prop_compensated_martingale}}]
As explained earlier,
 Proposition \ref{prop_compensated_martingale} follows from Lemma \ref{lem-reduced}.
    \end{proof}

\subsection{Applying   the Martingale Representation Theorem}
\label{subsec_application}

Now, we  are ready to complete the proof of the  main result of our paper, see Theorem \ref{thm-main-existence}.

\begin{proof}[Proof of Theorem \ref{thm-main-existence}] 
We will now apply Theorem \ref{thm-martingale_rep} about the  representation theorem of  purely discontinuous martingales in terms of a Poisson random measure. For this aim we will verify all assumptions
of that theorem.  We begin with fixing  the notation. We put
\begin{itemize}
\item[-] $\bX =\rU^\prime$, where $\rU$ is the Hilbert space introduced in \eqref{eqn-U_comp_V_s};
\item[-] as the filtered probability space $\mathfrak{A}$ we consider
   the probability space $\bigl( \Omegabar ,\tfcal ,\Pbar  \bigr)$  introduced at the beginning of subsection \ref{subsec-SJ}, endowed with  the filtration ${\tF} \coloneqq({\overline{\fcal }}_{t})$ defined in equality \eqref{eqn-filtration_limit};
\item[-] as the martingale $M$ we consider the process $\Mbar$ defined in equation \eqref{eqn-tM};
\item[-] as the space $(\rY,\ycal)$ we take the space $(\rY,\ycal)$ from Assumption F.\ref{assum-F1};
\item[-] as the map $\theta$ we consider the following function
\begin{equation}
\label{eqn-theta map}
\theta\colon   [0,\infty)\times Y \times \Omegabar \ni (t, y,\tomega) \mapsto   F(t,\tu(t-,\tomega),y) \in\rU^\prime,
\end{equation}
where the function $F$ is from Assumption F.\ref{assum-F2} and $\tu $ is the $ \zcal_T $-valued random variable introduced at the beginning of subsection \ref{subsec-SJ}.
\item[-] As the measure $m$ we take the measure $\nu$ from Assumption \ref{item-P_Y_is_standard}.
\end{itemize}

Note that with this choice of function  $\theta$, assumption \eqref{eqn-integrability_in_representation_theorem}
 is satisfied. Indeed,
\begin{equation}
\label{interability2}
\Ebar \left[ \int_0^T \int_Y \norm{F(t,\tu(t-),y)}_{\rU^\prime}^2 \, \nu(\ud y) \dt \right]
\leq C_2 \Ebar \left[ \int_0^T \bigl(1 + \norm{\tu (t)}_\rH^2\bigr) \dt \right]
< \infty
\end{equation}
by Assumption F.\ref{assum-F2} and assertion \eqref{eqn-boundedness_H}.

Obviously, our space $\bX$ is a separable Banach space. According to Assumption F.\ref{assum-F1}, $(Y,\ycal)$ is a standard measurable space.

According to Lemma \ref{lem-tM is a martingale} the process ${\Mbar} $ is a square integrable, $\rU^\prime$-valued c\`adl\`ag ${\tF}$-martingale.

Moreover, the martingale  ${\Mbar} $ is purely discontinuous, see Proposition \ref{prop-purely discontinuous}.

Since $\tu $ is a  $ \zcal_T $-valued random variable, all the trajectories of the corresponding process $\tu(t)$, $\tinOT$, are $\rU^\prime$-valued c\`adl\`ag. Hence, since the process $\tu$ is
also $\tF$-adapted, its left limit version $\tu(t-)$,  $\tinOT$, is a $\rU^\prime$-valued $\tF$-predictable process. Therefore, by the measurability property of the function $F$ from Assumption F.\ref{assum-F1},\label{assum-F1-used}  the generalized process 
$F(t,\tu(t-,\tomega),y)$,  $(t,\tomega,y)\in [0,T]\times \Omegabar\times \rY$, is  $\tF$-predictable, see \cite[Definition II.3.3]{Ikeda+Watanabe_1989}.
This implies that the function $\theta$ is $\tF$-predictable.

Next, let $\eta_{\Mbar}$ on $\rU^\prime \setminus \{0\}$ be the integer-valued random measure defined by
\begin{equation*}
\eta_{\Mbar}(t,A;\tomega) = \sum_{s\in (0,t]} \1_A(\Delta \Mbar(s;\tomega)), \;\; t>0, \; A\in \mathscr{B}(\rU^\prime \setminus \{0\}), \tomega \in \Omegabar.
\end{equation*}
We will show that $\eta_{\Mbar}$ is a random measure of class (QL), see Definition \ref{def-class-QL}.

For this purpose let us choose and fix a set $A\in \mathscr{B}(\rU^\prime \setminus \{0\})$ satisfying \eqref{ass-compensated martingale}. Then we claim  that by Proposition  \ref{prop_compensated_martingale}
the process  $\tilde{\eta}_{\Mbar}(t,A)$, $t\geq 0$,
defined by
\begin{align}\label{eqn-tilde_eta_M}
\tilde{\eta}_{\Mbar}(t,A)
&\coloneqq \eta_{\Mbar}(t,A) - \int_0^t \int_Y \1_A(F(s,\tu(s-),y)) \, \nu(\ud y) \ds \\
&= \sum_{s\in (0,t]} \1_A(\Delta \Mbar(s)) - \int_0^t \int_Y \1_A(F(s,\tu(s),y)) \, \nu(\ud y) \ds, \;\; t\geq 0,
\nonumber
\end{align}
is an ${\tF}$-martingale. We conclude that the compensator $\hat{\eta}_{\overline{M}}$ of $\eta_{\overline{M}}$ is given by
\begin{equation*}\hat{\eta}_{\Mbar}(t, A) = \int_0^t \nu \left( y \in Y : F(t,\tu(s),y) \in A \right) \ds.\end{equation*}

Inside equality  \eqref{eqn-tilde_eta_M} we can freely replace $\tu(s-)$ with $\tu(s)$ because the functions $(0,t] \ni s \mapsto \tu(s-) \in\rU^\prime$ and $(0,t] \ni s \mapsto \tu(s) \in\rU^\prime$ differ at a most countable set, as the latter function is c\`adl\`ag, so that their Lebesgue integrals are equal.
Therefore we  conclude that the random measure  $\eta_{\Mbar}$ is a of class (QL) and that assumption \eqref{eqn-N_M_hat} holds.

We conclude by applying Theorem \ref{thm-martingale_rep} according to which  there exists an extension $(\underline{\Omega},\underline{\fcal}, \underline{\mathbb{P}}, \underline{\mathbb{F}})$ of  $(\Omegabar,\Fbar, \Pbar, \overline{\mathbb{F}})$ and a Poisson random measure $\underline{\eta}$ with characteristic measure $\nu$ such that for every $t\in \mathbb{R}_+$,

\begin{equation}
\label{by_martingale_representation}
\begin{aligned}
    \Mbar(t) = \int_0^t \int_Y F(s,\tu(s-, \underline{\omega}),y) \tilde{\underline{\eta}}(\ud s,\ud y) \mbox{ in }\rU^\prime, \underline{\mathbb{P}}\mbox{-almost surely}.
\end{aligned}
\end{equation}
Note that $\tu$ trivially extends to a process $\underline{u}$ defined on $\underline{\Omega}$.
Substituting \eqref{by_martingale_representation} into \eqref{eqn-tM} we see that $(\underline{\Omega},\underline{\fcal}, \underline{\mathbb{P}}, \underline{\fcal}_t, \underline{u},\underline{\eta})$ is a solution of the Navier-Stokes equation.

The proof of Theorem \ref{thm-main-existence} is thus complete.
\end{proof}

\appendix

\section{The It\^{o} formula}

It\^o formula play an important role in this paper. In this short section we recall the following special case of the It\^o Theorem, see  \cite[Theorem\ I.4.57]{Jacod_Shiryaev}.

\begin{theorem}\label{thm-Ito-JS-I.4.57}
Suppose that $\xi$ is  real-valued  semimartingale. Then for every $t\geq 0$, almost surely
\begin{align}\label{eqn-Ito_JS_I.4.58}
  \xi(t)^2 &=\xi(0)^2+ \int_0^t 2 \xi(s-)\, \ud \xi(s) + \langle \xi^{\mathrm{c}}\rangle(t) + \sum_{s\leq t}  (\Delta \xi(s))^2, \\
\label{eqn-Ito_JS_I.4.58_b}
&=\xi(0)^2+ \int_0^t 2 \xi(s-) \, \ud \xi(s)+ [\xi](t).
\end{align}
\end{theorem}

\begin{proof}
The theorem is a direct corollary from \cite[Theorem\ I.4.57]{Jacod_Shiryaev}.  The fact that \eqref{eqn-Ito_JS_I.4.58} and \eqref{eqn-Ito_JS_I.4.58_b} are equal follows from \cite[Theorem\ I.4.52]{Jacod_Shiryaev}.
\end{proof}

This special case of the Ito formula is used in the proof of Proposition \ref{prop_angle_bracket}.

\section{An extension of a probability space}

We begin with recalling   \cite[Definition\ II.7.1, p.\ 89]{Ikeda+Watanabe_1989}. We recommend the reader to read also Definition II.7.2 from the same source.
\begin{remark}\label{rem-notation-integral wrt point process}
Let us note that when we write $\int_0^t \int_{\mathbb{U}} \cdot \eta(\ud s, \ud u)$, then the value of the point process $p$ corresponding to $\eta$ at $0$ does not play any role, whereas its value at $t$ does influence the integral. For this reason some authors, see for instance \cite{Ikeda+Watanabe_1989}, write $\int_0^{t+} \int_{\mathbb{U}} \cdot \eta(\ud s, \ud u)$ or $\int_{(0,t]} \int_{\mathbb{U}} \cdot \eta(\ud s, \ud u)$.
\end{remark}

\begin{definition}\label{def-extension of probability space}
We say that a system $
\bigl(\tilde{\mathfrak{A}}, \pi  \bigr)$,
where
$
\tilde{\mathfrak{A}}=\bigl( \tilde{\Omega}, \tilde{\fcal}, \tilde{\mathbb{P}}, \tilde{\mathbb{F}}  \bigr)
$
is a filtered probability space,  with  $\tilde{\mathbb{F}}=(\tilde{\fcal}_t)_{t\geq 0}$ , and $
\pi\colon  \tilde{\Omega} \to \Omega
$
is a map,  is an extension of a
filtered probability space
$
\bigl( {\Omega}, {\fcal}, {\mathbb{P}}, {\mathbb{F}} \bigr), \mbox{ where } {\mathbb{F}}=({\fcal}_t)_{t\geq 0} ,
$
if and only if  the following four conditions are satisfied
\begin{trivlist}
\item[(i)] $\pi$ is $\tilde{\fcal}/{\fcal}$-measurable;
\item[(ii)] $\pi$ is $\tilde{\fcal}_t/{\fcal}_t$-measurable for every $t\geq 0$,
\item[(iii)] $\mathbb{P}=\tilde{\mathbb{P}} \circ \pi^{-1}$,
\item[(iv)] for every $ {\fcal}$-measurable and $ {\mathbb{P}}$-a.s.\ bounded, function  $X\colon  \Omega \to \mathbb{R}$, and every $t\geq 0$, the following equality holds
\[
\tilde{\mathbb{E}}\bigl[ \tilde{X}\vert \tilde{\fcal}_t \bigr]=
\mathbb{E}\bigl[ {X}\vert {\fcal}_t \bigr] \circ \pi, \;\; \tilde{\mathbb{P}} \mbox{-a.s.}
\]
where
\[
\tilde{X}\coloneqq X \circ \pi \;\;\; :\tilde{\Omega} \to \mathbb{R}.
\]
\end{trivlist}
\end{definition}

The notion of predictability used in the following obvious result was defined in   \cite[Definition \ II.3.3]{Ikeda+Watanabe_1989}, see also \cite[Definition\ 3.4.4]{Kallianpur_Xiong}.

\begin{lemma}\label{lem-extension of probability space}
Let  $(\mathbf{X},\mathscr{B}_\mathbf{X})$ and $(\mathbf{U}, \mathcal{U})$ be two measurable spaces.
If
\[
\theta\colon  (0,\infty) \times \mathbf{U}\times \Omega \to \mathbf{X}
\]
is a generalized predictable  process, then
the process $\tilde{\theta}$ defined by
\begin{equation}\label{eqn-tilde_theta}
\tilde{\theta}\colon  (0,\infty) \times \mathbf{U}\times \tilde{\Omega} \ni (t,u,\tilde{\omega}) \mapsto \theta\bigl(t,u,\pi(\tilde{\omega})\bigr)\in  {\color{teal}\mathbf{X}}
\end{equation}
is a generalized predictable  process, see \cite[Definition II.3.3]{Ikeda+Watanabe_1989}.
\end{lemma}

\begin{proof}[Proof of Lemma \ref{lem-extension of probability space}]

Let $\tilde{\mathscr{P}}_{\mathbf{U}}$ denote the predictable $\sigma$-field on $(0,\infty) \times \mathbf{U}\times \tilde{\Omega}$. It is sufficient to show that the map
\[
(\id_{(0,\infty)},\id_{\mathbf{U}},\pi)\colon  (0,\infty) \times \mathbf{U}\times \tilde{\Omega} \ni (t,u,\tilde{\omega}) \mapsto \bigl(t,u,\pi(\tilde{\omega})\bigr) \in
(0,\infty) \times \mathbf{U}\times \Omega
\]
is $\tilde{\mathscr{P}}_{\mathbf{U}}/\mathscr{P}_{\mathbf{U}}$  measurable.
For this purpose, we  identify the set  $(0,\infty)\times \mathbf{U} \times \Omega$ with $ (0,\infty) \times\Omega \times   \mathbf{U}$. With this in mind we see that
 $\mathscr{P}_{\mathbf{U}}=\mathscr{P} \otimes  \mathcal{U}$ and $\tilde{\mathscr{P}}_{\mathbf{U}}=\tilde{\mathscr{P}} \otimes  \mathcal{U}$, where $\tilde{\mathscr{P}}$ (resp. $\mathscr{P}$) is the predictable $\sigma$-algebra on $(0,\infty) \times\tilde{\Omega} $ (resp. $(0,\infty) \times\Omega $).
 Thus, to prove the lemma it is enough to show that $(\id_{(0,\infty)},\pi)$ is $\tilde{\mathscr{P}}/\mathscr{P}$-measurable. So, let $s\le t$, $F \in \mathscr{F}_s$ and $F_0\in \mathscr{F}_0$. Since $\pi$ is $\tilde{\mathscr{F}}_r/\mathscr{F}_r$-measurable for any $r\ge 0$, we have $(\id_{(0,\infty)},\pi)^{-1}((s,t]\times {F})\in \tilde{\mathscr{P}}$ and $(\id_{(0,\infty)},\pi)^{-1}(\{0\}\times {F}_0)\in \tilde{\mathscr{P}}$.
 Since the $\sigma$-algebra ${\mathscr{P}}$ is generated by the sets $(s,t]\times {F}$ and $\{0\}\times {F}_0$ with $s\le t$ and ${F}\in {\mathscr{F}}_s$ and $F_0\in {\mathscr{F}}_0$, see \cite[Theorem\ 3.3]{Metivier}, we have proved that $(\id_{(0,\infty)},\pi)$ is $\tilde{\mathscr{P}}/\mathscr{P}$-measurable. This completes the proof of the lemma.
\end{proof}

We also have the following, which shows the importance of property (iv) in Definition \ref{def-extension of probability space}.

\begin{proposition}\label{prop-martingale on extension of probability space}
Assume that $M=(M_t)_{t\geq 0}$ is a $p$-integrable, for some $p\in (1,\infty)$,  martingale on a
filtered probability space
$
\mathfrak{A}=\bigl( {\Omega}, {\fcal}, {\mathbb{P}}, {\mathbb{F}} \bigr), \mbox{ where } {\mathbb{F}}=({\fcal}_t)_{t\geq 0}$.
Assume that a filtered probability space
$
\tilde{\mathfrak{A}}=\bigl( \tilde{\Omega}, \tilde{\fcal}, \tilde{\mathbb{P}}, \tilde{\mathbb{F}} \bigr)$,  where  $\tilde{\mathbb{F}}=(\tilde{\fcal}_t)_{t\geq 0}$,
is an extension of the former.  Define a process  $\tilde{M}=(\tilde{M}(t))_{t\geq 0}$ by
\begin{equation}\label{eqn-tilde_M_t}
\tilde{M}(t) \coloneqq M(t) \circ \pi \;\;\; (:\tilde{\Omega} \to \mathbb{R}),
\end{equation}
where $
\pi\colon  \tilde{\Omega} \to \Omega
$ is a map as in Definition  \ref{def-extension of probability space}. Then
$\tilde{M}$ is a martingale on $\tilde{\mathfrak{A}}$.
\end{proposition}
\begin{proof}
The integrability condition is obvious. Indeed, by property (iii) in Definition \ref{def-extension of probability space} the change of measure theorem we have
\begin{align*}
 \tilde{\mathbb{E}} \left[ \vert \tilde{M}(t) \vert^p  \right]  &=
 \tilde{\mathbb{E}} \left[ \vert {M}(t) \circ \pi  \vert^p \right] =\mathbb{E} \left[ \vert {M}(t)\vert^p \right].
\end{align*}
The adaptedness condition follows from (ii) in Definition \ref{def-extension of probability space}.
The proof is complete by observing that in view of property (iv), if $t>s$, then $\tilde{\mathbb{P}}$-a.s.\ the following chain of equalities hold
\begin{equation*}
\tilde{\mathbb{E}}\bigl[ \tilde{M}(t)\vert \tilde{\fcal}_s \bigr]
= \mathbb{E}\bigl[ {M}(t)\vert {\fcal}_s \bigr] \circ \pi
= M(s) \circ \pi
= \tilde{M}(s).
\end{equation*}
 \end{proof}

The main result of this section is the following.

\begin{proposition}\label{prop-purely discontinuous on extension of probability space}
Assume that $M=(M_t)_{t\geq 0}$ is a square integrable purely discontinuous $\mathbb{R}$-valued martingale on a filtered probability space
$
\mathfrak{A}=\bigl( {\Omega}, {\fcal}, {\mathbb{P}}, {\mathbb{F}} \bigr)$, { where } ${\mathbb{F}}=({\fcal}_t)_{t\geq 0}$.
Assume that a filtered probability space
$
\tilde{\mathfrak{A}}=\bigl( \tilde{\Omega}, \tilde{\fcal}, \tilde{\mathbb{P}}, \tilde{\mathbb{F}} \bigr)$,  where  $\tilde{\mathbb{F}}=(\tilde{\fcal}_t)_{t\geq 0}$,
is an extension of the former.  Define a process  $\tilde{M}=(\tilde{M}_t)_{t\geq 0}$ by \eqref{eqn-tilde_M_t},
where $
\pi\colon  \tilde{\Omega} \to \Omega
$ is a map as in Definition  \ref{def-extension of probability space}. Then
$\tilde{M}$ is a purely discontinuous square integrable martingale on $\tilde{\mathfrak{A}}$.
\end{proposition}

\begin{proof} In view of the characterisation \eqref{eqn-purely_discontinuous_martingale} in Section \ref{subsec-Preliminaries},$\tilde{M}$ is a  purely discontinuous (square integrable) $\mathbb{R}$-valued martingale if and only if for every $t>0$, almost surely
\begin{equation}\label{eqn-purely_discontinuous_martingale_tilde_M}
[\tilde{M}](t) = \sum_{s\leq t} (\Delta \tilde{M}(s))^2.
\end{equation}
To prove the above, we choose and fix $t>0$  and  observe that in view of   the equality \eqref{eqn-quadratic_variation-Metivier}  we have,
in $L^1$ over the partitions of $[0,t]$ with mesh converging to $0$,
\begin{align*}
[\tilde{M}](t)&=
 \lim_{n\to \infty} \sum_{i=1}^n \left(\tilde{M}(t_i^n\wedge t)-\tilde{M}(t_{i-1}^n\wedge t)\right)^2\\
&=\lim \sum_{i=1}^n \left({M}(t_i^n\wedge t)\circ \pi -{M}(t_{i-1}^n\wedge t)\circ \pi\right)^2\\
 &=\lim \sum_{i=1}^n \left({M}(t_i^n\wedge t)  -{M}(t_{i-1}^n\wedge t) \right)^2 \circ \pi=[{M}](t)\circ \pi.
 \end{align*}
Hence, because ${M}$ is a  purely discontinuous (square integrable) real martingale we infer that
\begin{align*}
[\tilde{M}](t)&=\Bigl( \sum_{s\leq t} (\Delta {M}(s))^2 \Bigr) \circ \pi\\
&=\Bigl( \sum_{s\leq t} (\Delta {M}(s))^2  \circ \pi \Bigr)
= \sum_{s\leq t} (\Delta {M}(s)\circ \pi )^2  \\
&= \sum_{s\leq t} (\Delta \tilde{M}(s) )^2.
\end{align*}
This proves \eqref{eqn-purely_discontinuous_martingale_tilde_M} and therefore, the proof is complete.
\end{proof}

\section{Point processes}
\label{sec-Point processes}

Our presentation of the point processes will follow    \cite[Section 9 of Chapter I, p. 43]{Ikeda+Watanabe_1989}.

\begin{definition}\label{def-point function}
Suppose that $X $ is a non-empty set. A point function on $X $ is a partial function, \textit{i.e.},\ a functional binary relation,
$$ p \subset (0,\infty) \times X .$$
which is  at most a countable set.\\
If $p$ is a point function on $X $ then we define
\[\begin{aligned}\dom(p)&=\mathrm{D}_p:=\bigl\{ t \in (0,\infty): p \cap (\{t\} \times X) \not= \emptyset\}
\\&
=\bigl\{ t \in (0,\infty):  \mbox{ there exists an element } x\in X: (t,x) \in  p \}.
\end{aligned}
\]
 \\
If $t\in \dom(p)$, then we will denote by $p(t)$ the value of $p$ at $t$, i.e.\ the unique element $x\in X$ such that $ (t,x) \in  p $.\\
Moreover, when we will write $p(t)$ we will implicitly understand   that  $t\in \dom(p)$. 

The family of all point functions on $X$ will be denoted by $\Pi_X$.
\end{definition}

\begin{remark}\label{rem-point function alternative}
Often a point function is defined differently, see e.g.\ \cite{Revuz+Yor_2005}. One introduces an auxiliary set $X^\ast:=X \cup \{ \partial\}$ with $ \partial$ an element not belonging to $X$. Then a point function on $X$ is defined
to be a function $p:(0,\infty) \to X^\ast$ such that the set $\bigl\{ t \in (0,\infty): p(t) \in X \bigr\}$, denoted by $D_p$, is at most countable. Obviously, these two definitions are equivalent.
\end{remark}

\begin{example}\label{example-point function from a cadlag function}
Assume that $X$ is a Banach space and $f:[0,\infty) \to X$ is a \cadlag \, function. Then, see e.g.\ \cite[Lemma 3.1.24]{Zhu-PhD} the set
\[
J_p:=\bigl\{ t\in [0,\infty): \Delta f(t) \not= 0 \bigr\}
\]
is at most countable. Let us define a function $p \subset (0,\infty)  \times X$ by putting
\begin{equation*}
\begin{aligned}
\dom(p)&=J_p,\\
p(t)&= \Delta f(t), \;\; \mbox{ if } t \in J_p.
\end{aligned}
\end{equation*}
Then $p$ is a point function (on both $X$ and $X\setminus\{0\}$.).
\end{example}

\begin{definition}\label{def-counting measure}
Suppose that $(X, \mathscr{B}_X) $ is a measure space and  $p$ is  a point function on $X $. We put
\begin{equation}\label{eqn-IW_I.9.1}
  \eta_p((0,t]\times\rU):= \# \bigl\{ s \in \dom(p)\cap (0,t] : p(s) \in\rU \bigr\}, \;\; t>0,\;\;U \in \mathscr{B}_X.
\end{equation}
Following \cite[page 43]{Ikeda+Watanabe_1989} we have the following result. 

\begin{proposition}\label{prop-counting measure}
Suppose that $(X, \mathscr{B}_X) $ is a non-empty measure space and  $p$ is a point function on $X $.
The function $\eta_p$ defined above in \eqref{eqn-IW_I.9.1} has a unique extension to a $\sigma$-additive measure
\[
\mathscr{B}(0,\infty)\otimes \mathscr{B}_X \to \bar{\mathbb{N}}:=\mathbb{N}\cup \{\infty\},
\]
where $\mathscr{B}(\mathbb{R}_+)$ is the Borel $\sigma$-field on $\mathbb{R}_+$. Moreover, this  unique extension is of the following form
\[
\sum_{(t,x) \in p} \delta_{(t,x)}.
\]
In other words, this extension is the  counting measure on $(0,\infty) \times X$  associated with $p$ and it will be denoted by   $\eta_p$.
\end{proposition}

The following exercises can be found useful by some attentive readers.
\begin{exercise}\label{exercise-counting measure}
In the framework of Defintion \ref{def-counting measure}, if $U \in \mathscr{B}_X$, then the function
\[
(0,\infty) \ni t \mapsto \eta_p((0,t]\times\rU) \in \bar{\mathbb{N}}
\]
is weakly-increasing and right-continuous.\\
Moreover, a number $t_0\in (0,\infty)$ is a discontinuity point of it if and only if $t_0 \in \dom(p)$ and $p(t_0) \in\rU$.
\end{exercise}
\begin{exercise}\label{exercise-counting measure-b}
Show that formula \eqref{eqn-IW_I.9.1} is equivalent to
\begin{equation}\label{eqn-IW_I.9.1_b}
  \eta_p((0,t]\times\rU)= \sum_{s>0} \1_{(0,t]\times\rU}\bigl(s,p(s)\bigr),   \;\; t>0,\;\;U \in \mathscr{B}_X,
\end{equation}
with an obvious understanding that $\1_{(0,t]\times\rU}\bigl(s,p(s)\bigr)=1$ if and only if $\bigl(s,p(s)\bigr)\in (0,t]\times\rU$ if and only if $s \in (0,t] \cap \dom(p)$ and $p(s) \in\rU$. \\
Deduce that
\begin{equation}\label{eqn-IW_I.9.1_c}
 \iint \1_{(0,t]\times\rU} \eta_p(\ud s,\ud u)= \sum_{s>0} \1_{(0,t]\times\rU}\bigl(s,p(s)\bigr),   \;\; t>0,\;\;U \in \mathscr{B}_X,
\end{equation}
or
\begin{equation}\label{eqn-IW_I.9.1_d}
 \iint_{(0,t]\times\rU}  \eta_p(\ud s,\ud u)= \sum_{s>0} \1_{(0,t]\times\rU}\bigl(s,p(s)\bigr),   \;\; t>0,\;\;U \in \mathscr{B}_X,
\end{equation}
\end{exercise}

It is obvious that if $p$ and $q$ are two point functions  on $X$ with $p=q$, then the measures $\eta_p$ and $\eta_q$ are equal. We have also a converse assertion.

\begin{proposition}\label{prop-injectivity of map p to eta_p}
Suppose that singletons are $\mathscr{B}_X$-measurable. If $p$ and $q$ are two point functions on $X$ such that  the measures $\eta_p$ and $\eta_q$ are equal, i.e.\ for every $U \in \mathscr{B}_X$ and every $t>0$,
\begin{equation}\label{eqn-eta_p=eta_q}
\# \bigl\{ s \in \dom(p)\cap (0,t] : p(s) \in\rU \bigr\}= \# \bigl\{ s \in \dom(q)\cap (0,t] : q(s) \in\rU \bigr\},
\end{equation}
or, equivalently,
\begin{equation}\label{eqn-eta_=eta_q_2}
\sum_{s \in (0,t]} \1_{\dom(p) \times\rU}(s,p(s)) = \sum_{s \in (0,t]} \1_{\dom(q) \times\rU}(s,q(s))
\end{equation}
then \[p=q.\]
\end{proposition}
\begin{proof}[Proof of Proposition \ref{prop-injectivity of map p to eta_p}]
Let us choose and fix $t>0$. Then, since \eqref{eqn-eta_p=eta_q} holds  for every $U \in \mathscr{B}_X$ we infer that
\[
 \bigl\{ p(s): s \in \dom(p)\cap (0,t]  \bigr\}=  \bigl\{ q(s): s \in \dom(q)\cap (0,t] \bigr\}.
\]
Since the above equality holds also for every $r<t$ we infer that
\[
 \bigl\{ p(s): s \in \dom(p)\cap [t,t]  \bigr\}=  \bigl\{ q(s): s \in \dom(q)\cap [t,t] \bigr\}.
\]
This obviously implies that
\[
t \in \dom(p) \iff  t \in \dom(p) \mbox{ and if so } p(t)=q(t).
\]
The proof is complete.
\end{proof}

In the framework of Example \ref{example-point function from a cadlag function} Proposition \ref{prop-injectivity of map p to eta_p} yields the following result.

\begin{corollary}\label{cor-injectivity of map p to eta_p}
Assume that $X$ is a Banach space and $f,g:[0,\infty) \times \Omega \to X$ are two  \cadlag \, processes such that for  $\mathbb{P}$-almost every  $\omega\in \Omega$
\begin{equation}\label{eqn-eta_=eta_q_3}
\begin{aligned}
\sum_{s >0} \1_{(0,t] \times A}(s,\Delta f(s,\omega)) &= \sum_{s >0}  \1_{ (0,t] \times A}(s, \Delta g(s))\\
&\qquad \mbox{ for all  }  t\in [0,\infty), \;\; A\in \mathscr{B}(\mathbf{X}\setminus\{0\}).
\end{aligned}
 \end{equation}
Then,
for  $\mathbb{P}$-almost every  $\omega\in \Omega$,
\begin{equation}\label{eqn-eta_=eta_q_4}
\begin{aligned}
\Delta f(t,\omega)&=\Delta g(t,\omega), \quad \mbox{ for every  }  t\in [0,\infty).
\end{aligned}
 \end{equation}
\end{corollary}

By $\mathscr{B}(\Pi_X)$ we denote the $\sigma$-field on $\Pi_X$ generated by the family of maps
\[
\Pi_X \ni p \mapsto  \eta_p((0,t]\times\rU) \in \bar{\mathbb{N}},\;\; t>0, \;\;U \in \mathscr{B}_X.
\]
Equivalently, $\mathscr{B}(\Pi_X)$ is the smallest  $\sigma$-field on $\Pi_X$ such that the maps above are measurable.
\end{definition}
\begin{remark}\label{rem-counting measure} Compare with  \cite[Definition II.1.13 ]{Jacod_Shiryaev}.
Suppose that $(X, \mathscr{B}_X) $ is a non-empty measure space and  $p$ is  point function on $X $. \\
Then the  function $\eta_p$ satisfies the following condition
\[
\eta_p(\{t\} \times X) \leq 1.
\]
Indeed, by Definition \ref{def-counting measure} and Proposition  \ref{prop-counting measure} we have
\[
 \eta_p( \{t\}\times X)= \# \bigl\{ s \in \dom(p)\cap \{t \} : p(s) \in X \bigr\}=\begin{cases} 1 , &\mbox{ if } t \in \dom(p), \\
 0 , &\mbox{ if } t \in \dom(p),
 \end{cases} .
\]

\end{remark}

\begin{definition}\label{def-point process}
Suppose that $(X, \mathscr{B}_X) $ is a non-empty measure space  and
$\bigl( {\Omega}, {\fcal}, {\mathbb{P}} \bigr)$ is a probability space. A {random} point process (on $X$ and on $\bigl( {\Omega}, {\fcal}, {\mathbb{P}}  \bigr)$) is a map \[
p: \Omega \to \Pi_X
\]
which is ${\fcal}/\mathscr{B}(\Pi_X)$-measurable.
Hereafter, when there is no ambiguity we just say a point process instead of random point process.

If $p: \Omega \to \Pi_X$ is a point process, then
 we will write $p(t,\omega)$ for the value of $p(\omega)$ at $t \in D(p(\omega))$.
Moreover, see Definition \ref{def-point function}, when we will write $p(t,\omega)$ we will implicitly imply that $t\in \dom(p(\omega))$.

\end{definition}
\begin{remark}\label{rem-point process alternative}
Using the approach from \cite{Revuz+Yor_2005}, a point process (on $X$ and on $\bigl( {\Omega}, {\fcal}, {\mathbb{P}}  \bigr)$) is a map
\[
p: (0,\infty) \times \Omega \to X^\ast
\]
 such that for every $\omega \in \Omega$ the set
 \[\bigl\{ t \in (0,\infty): p(t,\omega) \in X \bigr\}=:\mathrm{D}_{p(\omega)}\]  is at most countable and, for all $t>0$ and $U \in \mathscr{B}_X$, the map
 \[
 \Omega \ni \omega \mapsto \# \bigl\{ s \in  (0,t] : p(s,\omega) \in\rU \bigr\},
 \]
is ${\fcal}/\mathscr{B}_X$-measurable.
\end{remark}

The following example is a random version of Example \ref{example-point function from a cadlag function} and is related to \cite[Proposition I.1.32]{Jacod_Shiryaev}
\begin{example}\label{example-point process from a cadlag process}
Assume that $X$ is a Banach space and $\xi:[0,\infty) \times \Omega \to X$ is a \cadlag \, process. Since  the set, see e.g.\ \cite[Lemma 3.1.24]{Zhu-PhD},
\[
A(\omega):=\bigl\{ t\in [0,\infty): \Delta \xi(t,\omega) \not= 0 \bigr\}, \;\; \omega \in \Omega,
\]
is at most countable, we can define a random function $p:(0,\infty) \times \Omega  \to X$ by
\[
p(t,\omega)= \begin{cases} \Delta \xi(t,\omega), \;\; &\mbox{ if } t \in A(\omega),\\
\partial, \;\; &\mbox{ if } t \in (0,\infty) \setminus A(\omega).
\end{cases}
\]
Then $p$ is a point process, in the sense of \cite{Revuz+Yor_2005},  (on both $X$ and $X\setminus\{0\}$.).\\
Let us point out, see \cite[Proposition I.1.3]{Jacod_Shiryaev}, that if additionally $\xi$ is an $\mathbb{F}$-adapted process, then the random set $A$ is thin, i.e., there exists a sequence $(\tau_n)$ of $\mathbb{F}$-stopping times such that $[[\tau_n]] \cap  [[\tau_m]]=\emptyset$ for $n\not= m$ and
\[
A=\bigcup_{n=1}^\infty  [[\tau_n]].
\]
\end{example}

We can prove the following result.

\begin{proposition}\label{prop-point process}
Suppose that $(X, \mathscr{B}_X) $ is a non-empty measure space  and
$\bigl( {\Omega}, {\fcal}, {\mathbb{P}}  \bigr)$ is a probability space.
Then a map
\[
p: \Omega \to \Pi_X
\]
is a point process if and only if for all $t>0$ and $U \in \mathscr{B}_X$ the map
\[
\Omega \ni \omega \mapsto \eta_{p(\omega)}((0,t]\times\rU) \in \bar{\mathbb{N}}
\]
is $\mathscr{F}$-measurable.
\end{proposition}

 We now recall some definitions from \cite[Definition 3.4.2]{Kallianpur_Xiong} and \cite[Definitions II.1.3 and II.1.13]{Jacod_Shiryaev}.

\begin{definition}\label{def-QL calss random measure}
Suppose that
\[
\mathfrak{A}=\bigl( \Omega, \mathscr{F}, \mathbb{P}, \mathbb{F}  \bigr), \mbox{ with } \mathbb{F}=(\mathscr{F}_t)_{t\geq 0} ,
\]
is a filtered probability space
 and $(E, \mathscr{E}) $ is a non-empty measurable space.
\begin{itemize}
\item A  \emph{random measure} is a mapping \[\eta : \Omega \times \bigl( \mathscr{B}(\R_+) \otimes \mathscr{E}\bigr) \to \R\] such that the following two conditions are satisfied.
\begin{trivlist}
\item[(i)]  for every $\omega \in \Omega$, the function  \[\eta(\omega, \cdot): \mathscr{B}(\R_+) \otimes \mathscr{E} \to \R\]
is a measure such that $\eta(\omega;\{0\}\times E)=0$ for all $\omega\in \Omega$; 
\item[(ii)]  for every $B \in \mathscr{B}(\R_+) \otimes \mathscr{E}$, the function  $\eta(\cdot,B):\Omega \to \mathbb{R}$ is measurable.
\end{trivlist}
\end{itemize}
Suppose from now on that  $\eta : \Omega \times \mathscr{B}(\R_+) \otimes \mathscr{E} \to \R$  is a random measure and put
\begin{equation}\label{eqn-Gamma_eta}
\Gamma_\eta := \{ A \in \mathscr{E} : \E \,\abs{\eta([0,t] \times A)} < \infty \mbox{ for all } t \geq 0 \}  .
\end{equation}
The  random measure  $\eta$ is called
\begin{itemize}
\item  $\sigma$\emph{-finite} if and only if there exists a sequence  $(U_n)$ of elements of $\Gamma_\eta$ such that 
$\mathbb{P}$-almost surely $\eta(U_n)<\infty$ and 
$U_n \nearrow E$; 
\item \emph{integer-valued} if and only if
\begin{trivlist}
\item[(i)] for all $(\omega,t)\in \Omega\times [0,\infty)$,
$
\eta(\omega;\{t\}\times E) \leq 1;
$
\item[(ii)] for every $A\in \mathscr{B}(\R_+) \otimes \mathscr{E}$ the random variable $\eta(\cdot;A)$ takes values in $\bar{\mathbb{N}}$;
\item[(iii)] $\eta$ is optional, i.e.,\ for every $A\in  \mathscr{E}$, the process $$\Omega\times \mathbb{R}_+\ni (\omega,t) \mapsto \eta(\omega;[0,t]\times A) \in \mathbb{R},$$ is
$\mathbb{F}$-optional;
\item[(iv)] $\eta$ is $\sigma$-finite.
\end{trivlist}
\item   $\mathbb{F}$-\emph{adapted} if and only if for every $A \in \Gamma_\eta$ the process
\[\eta_A:\Omega \times \mathscr{B}(\R_+) \ni (\omega,t) \mapsto \eta(\omega;[0,t]\times A)\] is   $\mathbb{F}$-adapted;
\item  a \emph{martingale random measure} if and only if for every $A \in \Gamma_\eta $, the process  $\eta_A$ is an $\mathbb{F}$-martingale;
\end{itemize}
\end{definition}

\begin{proposition}\label{prop-random counting measure}
Suppose that $(X, \mathscr{B}_X) $ is a non-empty measure space,
$\bigl( {\Omega}, {\fcal}, {\mathbb{P}}  \bigr)$ is a probability space
and  $p$ is a point process on $X$ and on $\bigl( {\Omega}, {\fcal}, {\mathbb{P}}  \bigr)$.\\
Then the function   defined by
\[
\Omega \ni \omega \mapsto \eta_{p(\omega)}((0,t]\times\rU) \in \bar{\mathbb{N}}
\]
for $t>0$ and $U \in \mathscr{B}_X$, extends to an "integer-valued random measure" in the sense of {the above definition, see also }
  \cite[Definition II.1.13]{Jacod_Shiryaev}.
 
\end{proposition}

\begin{definition}\label{def-shift of point function}
Suppose that $X $ is a non-empty set and $p$ is  a point function on $X$.  Given $t \geq 0$ we define $\theta_tp$ by
\begin{equation}\label{eqn-shift_of_point_function}
\begin{aligned}
\dom(\theta_tp)&:=\bigl\{s \in (0,\infty): s+t \in \dom(p) \bigr\}=(\dom(p)-t)\cap (0,\infty), \\
\theta_tp(s)&:= p(s+t), \;\; s \in \dom(\theta_tp).
\end{aligned}
\end{equation}
In other words, if $\dom(p)=\{ s_i\}$, then $\dom(\theta_tp)=\{ s_i-t: s_i>t\}$ and $\theta_tp(s_i-t)=p(s_i)$.
\end{definition}

We have the following two results.

\begin{lemma}\label{lem-shift of point function}
Suppose that $X $ is a non-empty set, $p$ is  a point function on $X$ and  $t \geq 0$. Then
\[
\eta_{\theta_t p}\bigl( (0,s]\times\rU\bigr)=\eta_{p}\bigl( (0,t+s]\times\rU\bigr)-\eta_{p}\bigl( (0,t]\times\rU\bigr), \;\; s>0,\rU \in \mathscr{B}_X.
\]
\end{lemma}
\begin{proof}[Proof of Lemma \ref{lem-shift of point function}] We omit the proof because the assertion of the lemma follows from straightforward calculations.

\end{proof}

\begin{proposition}\label{prop-shift of point process}
Suppose that $(X, \mathscr{B}_X) $ is a non-empty measure space  and
$\bigl( {\Omega}, {\fcal}, {\mathbb{P}}  \bigr)$ is a probability space.
If a map $p: \Omega \to \Pi_X$
is a point process, then for every $t\geq 0$,  the map
\[
\theta_p: \Omega \ni \omega  \mapsto  \theta_{p(\omega)} \in  \Pi_X
\]
 is also a point process.
\end{proposition}
\begin{proof}[Proof of Proposition \ref{prop-shift of point process}] Follows from  Lemma \ref{lem-shift of point function}.
\end{proof}

\begin{proposition}\label{prop-intensity measure of point process}
Suppose that $(X, \mathscr{B}_X) $ is a non-empty measure space  and
$\bigl( {\Omega}, {\fcal}, {\mathbb{P}}  \bigr)$ is a probability space.
If a map $p: \Omega \to \Pi_X$ is a point process, we put
\[
  n_p((0,t]\times\rU)= \mathbb{E} \bigl[  \eta_p((0,t]\times\rU) \bigr],  \;\; t>0,\;\;U \in \mathscr{B}_X.
\]
Then the function $n_p$ has a unique extension to a $\sigma$-additive $\bar{\mathbb{N}}:=\mathbb{N}\cup \{\infty\}$-valued function on the set $\mathscr{B}(0,\infty)\otimes \mathscr{B}_X$. \\
It is called the intensity measure of the point process $p$.
\end{proposition}
The  nomenclature "intensity measure" is extracted from \cite[Definition I.9.1]{Ikeda+Watanabe_1989}. A related and similar name "intensity measure" of an extended Poisson measure  is used in \cite[Definition II.1.20]{Jacod_Shiryaev}. The notation used there is different.

\begin{definition}
\label{def-stationary point process}
Suppose that $(X, \mathscr{B}_X) $ is a non-empty measure space  and $\bigl( {\Omega}, {\fcal}, {\mathbb{P}}  \bigr)$ is a probability space. A point process
 (on $X$ and on $\bigl( {\Omega}, {\fcal}, {\mathbb{P}}  \bigr)$) $p: \Omega \to \Pi_X$ is called a stationary point process
if and only if for every $t>0$,  the random variable $\theta_tp:\Omega \to \Pi_X$ has the same law (on $\Pi_X$) as the random variable $p: \Omega \to \Pi_X$.
\end{definition}

Recall that a Poisson random measure and a Poisson random measure with respect to a filtration $\mathbb{F}$ were defined in Definitions \ref{def_PRM} and \ref{def_PRM_filtration}, respectively.

\begin{definition}
\label{def-Poisson point process}
Suppose that $(X, \mathscr{B}_X) $ is a non-empty measure space  and
$\bigl( {\Omega}, {\fcal}, {\mathbb{P}}  \bigr)$ is a probability space. A point process
 (on $X$ and on $\bigl( {\Omega}, {\fcal}, {\mathbb{P}}  \bigr)$) $p: \Omega \to \Pi_X$ is called a Poisson point process
if and only if its corresponding random measure $\eta_p$ is a Poisson random measure.
\end{definition}

We have the following well-known classical results.

\begin{proposition}\label{prop-Poisson point process stationary}
Suppose that $(X, \mathscr{B}_X) $ is a non-empty measure space and
$\bigl( {\Omega}, {\fcal}, {\mathbb{P}}  \bigr)$ is a probability space. A  Poisson point process $p: \Omega \to \Pi_X$
 (on $X$ and on $\bigl( {\Omega}, {\fcal}, {\mathbb{P}}  \bigr)$) is  stationary if and only if there exists a non-negative measure $n$ on $(X, \mathscr{B}_X)$
   such that
 \begin{equation}\label{eqn-IK_I.9.2}
 n_p((0,t]\times\rU)=t n(U), \;\; \mbox{ for all } t>0, \; u\in \mathscr{B}_X.
  \end{equation}
\end{proposition}

\begin{definition}\label{def-characteristic measure Poisson point process stationary}
In the frameword of Proposition \ref{prop-Poisson point process stationary},
the non-negative measure $n$ on $(X, \mathscr{B}_X)$    such that condition \eqref{eqn-IK_I.9.2} is satisfied
  is called  the \textbf{characteristic measure} of the Poisson point process $p$.
\end{definition}

Let us observe that obviously if such a measure $n$  exists, then  it is unique and it makes sense to speak about a stationary point process $p$ with
the characteristic measure $n$.

\begin{proposition}\label{prop-Poisson point process stationary-2}
Suppose that $(X, \mathscr{B}_X) $ is a non-empty measure space  and
$\bigl( {\Omega}, {\fcal}, {\mathbb{P}}  \bigr)$ is a probability space. If
$n$ is a non-negative measure  on $(X, \mathscr{B}_X)$ then a  point process $p: \Omega \to \Pi_X$  is a  stationary Poisson point process with the characteristic measure $n$ iff
for all $t>s \geq 0$, every $\ell \in \mathbb{N}$  and all pair-wise disjoint sets $U_1, \ldots,\rU_\ell \in \mathscr{B}_X$ the following identity holds
\begin{align}
\label{eqn-Poisson_point_process_stationary_2}
\mathbb{E} \Bigl[ e^{-\sum_{j=1}^\ell \lambda_j \eta_p((s,t] \times\rU_j) }\Big\vert \tilde{\fcal}_s \Bigr]
&=e^{(t-s)\sum_{j=1}^\ell (e^{-\lambda_j t}-1)n(U_j)}
\end{align}
where $\tilde{\fcal}_s$ is the $\sigma$-field generated by the following family  of random variables
\[
\eta_p((0,r]\times\rU): \;\; r\in (0,s], \;\; u \in \mathscr{B}_X.
\]
\end{proposition}

The following exposition is based on the Section II.3, p. 59,  of the monograph \cite{Ikeda+Watanabe_1989}.

\begin{definition}\label{def-adapted point process}
Suppose that $(X, \mathscr{B}_X) $ is a non-empty measure space  and that
\[
\mathfrak{A}=\bigl( \Omega, \mathscr{F}, \mathbb{P}, \mathbb{F}  \bigr), \mbox{ with } \mathbb{F}=(\mathscr{F}_t)_{t\geq 0} ,
\]
is a filtered probability space. We say that a point process $p: \Omega \to \Pi_X$ (on $X$ defined on $\Omega$) is $\mathbb{F}$-adapted if and only if for every $t>0$ and every $Y \in \mathscr{B}_X$ the random variable
\begin{equation}
\label{PRM_associated_to_point_process}
\eta_p((0,t] \times\rU):=\sum_{ s \in \dom(p)\cap (0,t]} \mathds{1}_U(p(s))
\end{equation}
is $\mathscr{F}_t$-measurable.
\end{definition}

\begin{definition}\label{def-sigma finite point process}
A point process $p: \Omega \to \Pi_X$  is $\sigma$-finite if and only if the associated  random measure is $\sigma$-finite.
If  a point process $p: \Omega \to \Pi_X$ (on $X$ defined on $\Omega$) is $\sigma$-finite  and $\mathbb{F}$-adapted, then we put
\begin{equation}
\label{eqn-Gamma_p}
\mathbf{\Gamma}_p:=\Bigl\{\rU\in \mathscr{B}_X: \;
\mathbb{E} \eta_p((0,t]\times\rU)< \infty, \;\; t>0\Bigr\}.
\end{equation}
\end{definition}

\begin{exercise}\label{exercise-sigma finite point process}
In the framework of Definition \ref{def-sigma finite point process}, if $U \in \mathbf{\Gamma}_p$, then  $\mathbb{P}$-a.s
\[
\eta_p((0,t]\times\rU)< \infty \mbox{ for all } t>0.
\]
\end{exercise}

In what follows we   use the usual  shorthand \[\eta(t,U):= \eta((0,t] \times\rU).\]

\begin{proposition}\label{prop-compensator}
Suppose that $(X, \mathscr{B}_X) $ is a non-empty measure space  and that
\[
\tilde{\mathfrak{A}}=\bigl( \tilde{\Omega}, \tilde{\fcal}, \tilde{\mathbb{P}}, \tilde{\mathbb{F}}  \bigr), \mbox{ with } \tilde{\mathbb{F}}=(\tilde{\fcal}_t)_{t\geq 0} ,
\]
is a filtered probability space and  $p: \Omega \to \Pi_X$ is a $\sigma$-finite  and $\mathbb{F}$-adapted   point process  (on $X$ defined on $\Omega$). Then the following holds.
\begin{trivlist}
\item[(i)] If $U \in \Gamma_p$, then
\[
(0,\infty) \ni t \mapsto  \eta_p((0,t]\times\rU)
\]
is an $\mathbb{F}$-adapted, integrable and increasing process.
\item[(ii)] If $U \in \Gamma_p$, then there exists a natural  $\mathbb{F}$-adapted integrable and increasing process  $\widehat{\eta}_p(t,U)$, $t\in (0,\infty)$, such that the process
$\tilde{\eta}_p(t,U)$, $t\in (0,\infty)$, defined by
\begin{equation}
\label{eqn-tilde_N_p}
\tilde{\eta}_p(t,U):= \eta_p(t,\rU)-\widehat{\eta}_p(t,U), \;\; t>0,
\end{equation}
is an $\mathbb{F}$-martingale.
\end{trivlist}
\end{proposition}
\begin{proof}[Proof of Proposition \ref{prop-compensator}] 
Part (i) is obvious. Part (ii) follows from \cite[Theorem II.1.8]{Jacod_Shiryaev}.
\end{proof}

\begin{remark}In the framework of Proposition \ref{prop-compensator}, 
\begin{trivlist}
\item[(i)]
 the process \[t\mapsto \widehat{\eta}_p(t,U)\] is not continuous in general, i.e.\ there exists such examples; 
 \item[(ii)]
 the function $\Gamma_p \ni\rU \mapsto  \widehat{\eta}_p(t,U)$ is in general not $\sigma$-additive, i.e.\ there exists such examples.
 \end{trivlist}
 However, it is known, see 
 \cite[Theorem I.3.1]{Ikeda+Watanabe_1989},
 that if $X$ satsisfies some additional assumptions, e.g.\ $X$ is standard measurable space, see   \cite[Definition I.3.3]{Ikeda+Watanabe_1989}, e.g.\ a Borel subset of a Polish space,  then there exists a modification of $\widehat{\eta}_p$ such that for all $t>0$ and for almost all $\omega \in \Omega$, the map
 \[
 \Gamma_p \ni\rU \mapsto  \widehat{\eta}_p(t,U) \mbox{ is  $\sigma$-additive}.
 \]
\end{remark}

The following definition is copied,  word by word,  from \cite[Definition II.3.1]{Ikeda+Watanabe_1989}.
\begin{definition}\label{def-point process of class QL}
Let us assume that  that $(X, \mathscr{B}_X) $ is a non-empty measure space  and that
\[
\tilde{\mathfrak{A}}=\bigl( \tilde{\Omega}, \tilde{\fcal}, \tilde{\mathbb{P}}, \tilde{\mathbb{F}}  \bigr), \mbox{ with } \tilde{\mathbb{F}}=(\tilde{\fcal}_t)_{t\geq 0} ,
\]
is a filtered probability space. 
An $\tilde{\mathbb{F}}$-adapted and $\sigma$-finite  point process \[p: \tilde{\Omega} \to \Pi_X,  \]
 is of \textbf{(QL)} class   if and only if there exists a family
\[
\widehat{\eta}_p=\Bigl\{ \widehat{\eta}_p(t,U): \, t>0, \;\rU\in \Gamma_p \Bigr\}
\]
such that the following conditions are satisfied.
\begin{trivlist}
\item[(i)] For every $U\in \Gamma_p$,
\[(0,\infty) \ni t \mapsto \widehat{\eta}_p(t,U)
\]
is a continuous $\tilde{\mathbb{F}}$-adapted  and increasing process .
\item[(ii)] For every $t>0$,  $\tilde{\omega} \in \tilde{\Omega}$,
\[\mathscr{B}_X \ni\rU \mapsto \widehat{\eta}_p(t,U;\tilde{\omega}) \]
is, $\tilde{\mathbb{P}}$-a.s.,  a $\sigma$-finite measure on $(X,\mathscr{B}_X)$.
\item[(iii)] If $U \in \Gamma_p$,  the process
\begin{equation}
\label{eqn-tilde_N_p-2}
\tilde{\eta}_p(t,U):=\eta_p((0,t]\times\rU)-\widehat{\eta}_p(t,U), \;\; t>0,
\end{equation}
is an $\tilde{\mathbb{F}}$-martingale.
\end{trivlist}
The family $\widehat{\eta}_p(t,U)$ is called the \textbf{compensator} of the point process $p$, or the \textbf{compensator} of the integer valued random measure $\eta_p$.
\end{definition}

Let us now make some comments.

\begin{remark}\label{rem-comparison of compensators}
If a random measure $\eta$ is generated by a point process $p$,  then the 4th bullet point in Definition \ref{def-QL calss random measure} and Definition \ref{def-point process of class QL} are very similar but not identical. We don't know if  they are equivalent.  First of all, according to Definition \ref{def-sigma finite point process}, 
point process $p: \Omega \to \Pi_X$  is $\sigma$-finite if and only if  the associated  random measure is $\sigma$-finite. Secondly,
if  $\widehat{\eta}$  a $\sigma$-finite random measure such that conditions (ii) and (iii) of the former are satisfied, then we can define a family $\widehat{\eta}_p=\Bigl\{ \widehat{\eta}_p(t,U): \, t>0, \;\rU\in \Gamma_p \Bigr\}$ by
\[
\widehat{\eta}_p(t,U):= \widehat{\eta}\bigl( (0,t] \times\rU\bigr), t>0, \;\rU\in \Gamma_p
\]
and prove that this family satisfies conditions (i), (ii) and (iii) of the latter.  
We believe  the following equivalence:
$p$ is point process of class \textbf{(QL)} according to Definition \ref{rem-comparison of compensators} if and only if the random measure $\eta_p$ is of class \textbf{(QL)} according to Definition \ref{def-QL calss random measure}.
The implication `$\Leftarrow$' is obvious.
For `$\Rightarrow$' we'd need to show that $\hat{\eta}((0,t] \times\rU):= \hat{\eta}_p(t,U)$ extends to a $\sigma$-finite random measure on $\mathscr{B}(\R_+)\otimes \mathscr{B}_X$.
\end{remark}

\begin{remark}\label{rem-notation section 9 of Chapter I}
In  \cite[Section I.9]{Ikeda+Watanabe_1989} the authors use the notation $\mathscr{B}(X)$  instead of an earlier notation $\mathscr{B}_X$.  Moreover, they use a symbol  $\eta_p(t,U)$ instead of $\eta_p((0,t]\times\rU)$. We will continue to use the former notation.
\end{remark}

\begin{definition}
\label{def-Poisson F-point process}
Suppose that $(X, \mathscr{B}_X) $ is a non-empty measure space  and
\[
\mathfrak{A}=\bigl( \Omega, \mathscr{F}, \mathbb{P}, \mathbb{F}  \bigr), \mbox{ with } \mathbb{F}=(\mathscr{F}_t)_{t\geq 0} ,
\]
is a filtered probability space.
 A  point process $p: \Omega \to \Pi_X$
 (on $X$ defined on $\Omega$) is called an $\mathbb{F}$-Poisson   point process if and only if  it is
 $\mathbb{F}$-adapted, $\sigma$-finite Poisson   point process and, for every $t>0$,  the random variables
 \[\eta_p((0,t+h] \times\rU)-\eta_p((0,t] \times\rU), \;\; h>0, \;U\in \mathscr{B}_X
 \]
 are independent of $\mathscr{F}_t$.
\end{definition}

\begin{remark}\label{rem-Poisson F-point process}
A point process $p$ is an $\mathbb{F}$-Poisson point process if the associated random measure $\eta$ is an $\mathbb{F}$-Poisson random measure, see Definition \ref{def_PRM_filtration}.
\end{remark}

The following result is stated on page 60 of \cite{Ikeda+Watanabe_1989}.

\begin{proposition}
\label{prop-F Poisson pp class QL}
Suppose that $(X, \mathscr{B}_X) $ is a non-empty measure space and 
\[
\mathfrak{A} = \bigl( \Omega, \mathscr{F}, \mathbb{P}, \mathbb{F} \bigr), \mbox{ with } \mathbb{F} = (\mathscr{F}_t)_{t\geq 0} ,
\]
is a filtered probability space. Assume that a  point process $p: \Omega \to \Pi_X$ (on $X$ defined on $\Omega$) is  an $\mathbb{F}$-Poisson   point process. Then the following two conditions are equivalent.
 \begin{trivlist}
 \item[(a)] The point process $p: \Omega \to \Pi_X$  is of class \textbf{(QL)}  (with respect to filtration $\mathbb{F}$).
 \item[(b)] For every $U\in \Gamma_p$, the function
\[(0,\infty) \ni t \mapsto \mathbb{E}\bigl[ {\eta}_p((0,t] \times\rU)\bigr]
\]
is continuous.
\end{trivlist}
\end{proposition}

\begin{corollary}
\label{cor-F Poisson pp class QL}
Suppose that $(X, \mathscr{B}_X) $ is a non-empty measure space  and 
\[
\mathfrak{A} = \bigl( \Omega, \mathscr{F}, \mathbb{P}, \mathbb{F} \bigr), \mbox{ with } \mathbb{F}=(\mathscr{F}_t)_{t\geq 0} ,
\]
is a filtered probability space. Assume that
  $p: \Omega \to \Pi_X$ is a  stationary $\mathbb{F}$-Poisson   point process
 (on $X$ defined on $\Omega$). Then
 the point process $p: \Omega \to \Pi_X$  is of class \textbf{(QL)}  (with respect to filtration $\mathbb{F}$) with compensator
 \[
 \widehat{\eta}_p(t,U):=tn(U),
 \]
where $n$ is the characteristic measure of $p$, see
Definition \ref{def-characteristic measure Poisson point process stationary} and
Proposition  \ref{prop-Poisson point process stationary-2}.
\end{corollary}

A converse to the above result is given in    \cite[Theorem I.6.2]{Ikeda+Watanabe_1989}.

\section{Representation Theorem for purely discontinuous martingales}\label{App:Appendix-RepThm}

In this section we state and prove a Representation Theorem for purely discontinuous martingales from Kallianpur and Xiong \cite[Theorem\  3.4.7]{Kallianpur_Xiong} adapted to our needs (for Hilbert-space valued processes rather than processes in a conuclear space).
The proof relies on  \cite[Theorem\ 3.4.6]{Kallianpur_Xiong}, which is taken from \cite[Theorem\ II.7.4]{Ikeda+Watanabe_1989}. They write that the result goes back to the following papers: Grigelionis \cite{Grigelionis} (in Russian), Karoui-Lepeltier \cite{El_Karoui_Lepeltier} (in Russian) and Tanaka \cite{Tanaka} (in English).

Before we proceed we will formulate the following clear generalization  of \cite[Theorem\ I.3.3.]{Ikeda+Watanabe_1989} because the aforementioned Lemma II.7.1 from therein follows easily from it.

\begin{theorem}\label{thm-regular conditional expectation with parameters}
Assume that    $(\mathbf{Z},\mathscr{B}_\mathbf{Z})$, $(\mathbf{X},\mathscr{B}_\mathbf{X})$ and $(\mathbf{Y},\mathscr{B}_\mathbf{Y})$ are three  measurable spaces.
Assume further that   $(\mathbf{Z},\mathscr{B}_\mathbf{Z})$ is a standard  measurable space and that $m$ is a non-negative measure on it. \\
Assume that a map
\begin{align*}
\theta\colon & \mathbf{Y} \times \mathbf{Z} \to \mathbf{X}
\end{align*}
is $\mathscr{B}_\mathbf{Y}\otimes \mathscr{B}_\mathbf{Z}/\mathscr{B}_\mathbf{X}$-measurable. For each  $y\in \mathbf{Y}$ let us put \footnote{Let us observe that $\theta_y$ is $ \mathscr{B}_\mathbf{Z}/\mathscr{B}_\mathbf{X}$-measurable.} $\theta_y\coloneqq\theta(y,\cdot)\colon \mathbf{Z} \to \mathbf{X}$ and let us assume that  $q(y,\cdot)$ is equal to the image of the measure $m$ via $\theta_y$:
\begin{equation}
\label{image_of_m_under_theta_y}
q(y,B) = m (\theta_y^{-1}(B)), \;\; B \in \mathbf{X}.
\end{equation}
We finally assume that for every $B \in \mathbf{X}$ the map
\begin{align}\label{eqn-q(y,cdot)}
   \mathbf{Y}  \ni y \mapsto q(y,B) \in [0,\infty]
\end{align}
is $\mathscr{B}_\mathbf{Y}$-measurable.

Then, there exists a function
\begin{align*}
Q\colon & \mathbf{Y} \times \mathbf{X} \times \mathscr{B}_{\mathbf{Z}} \to [0,\infty]
\end{align*}
such that the following three conditions are satisfied.
\begin{trivlist}
\item[(i)] for all $y \in \mathbf{Y}$ and $x \in \mathbf{X}$, the map
\begin{align*}
Q(y,x,\cdot)\colon & \mathscr{B}_Z \to [0,\infty]
\end{align*}
is a measure;
 \item[(ii)] for every $A\in \mathscr{B}_{\mathbf{Z}}$ the map
\begin{align*}
Q(\cdot,\cdot,A)\colon & \mathbf{Y} \times \mathbf{X}  \to [0,\infty]
\end{align*}
 is $\mathscr{B}_{\mathbf{x}} \otimes \mathscr{B}_{\mathbf{Y}}$-measurable,
 \item[(iii)] for all $y \in \mathbf{Y}$, $A\in \mathscr{B}_{\mathbf{Z}}$ and $B\in \mathscr{B}_{\mathbf{X}}$, the following equality holds:
\begin{align}\label{eqn-IW_II.7.28_a}
m (A \cap \theta_y^{-1}(B))
=\int_{\mathbf{X}}  \1_{B}(x) Q(y,x,A)\, q(y,\ud x).
\end{align}
\end{trivlist}
\end{theorem}

\begin{remark}\label{rem-regular conditional expectation with parameters}
Let us observe that in view of {\color{teal}\eqref{image_of_m_under_theta_y}} and the change of measure Theorem  the equality \eqref{eqn-IW_II.7.28_a} takes the following equivalent form.
\begin{align}\label{eqn-IW_II.7.28_b}
m (A \cap \theta_y^{-1}(B))&= \int_{\mathbf{Z}}  \1_{B}(\theta(y,z)) Q(y,\theta(y,z),A)\, m(\ud y).
\end{align}

On the other hand, the formula \cite[(II.7.28) from Lemma II.7.1]{Ikeda+Watanabe_1989} is as follows. If
$f\colon  \mathbf{Y} \times \mathbf{X}  \to [0,\infty]$
 is $\mathscr{B}_{\mathbf{Z}} \otimes \mathscr{B}_{\mathbf{Z}}$-measurable, then for every $y \in \mathbf{Y}$,
\begin{align}\label{eqn-IW_II.7.28}
 \int_{\mathbf{Z}} f(\theta(y,z),z) m(\ud z)=   \int_{\mathbf{X}} \Bigl[\int_{\mathbf{Z}} f(x,z) Q(y,\theta(y,z),\ud z) \Bigr]\, q(y,\ud x).
 \end{align}
Taking a function $f$ defined by $f(x,z)=\1_B(x) \1_A(z)$ we can prove that the last equality \eqref{eqn-IW_II.7.28} reduces to
\eqref{eqn-IW_II.7.28_a}.
\end{remark}

We will be using the following version of  \cite[Theorem\  II.7.4]{Ikeda+Watanabe_1989} rewritten in terms of a  random measure in place of  a point process.

\begin{theorem}\label{thm_IW-Theorem II.7.4-PRM}
Assume that
\[
\mathfrak{A}=\bigl( {\Omega}, {\fcal}, {\mathbb{P}}, {\mathbb{F}} \bigr), \mbox{ where } {\mathbb{F}}=({\fcal}_t)_{t\geq 0} ,
\]
is a  filtered probability space. Assume that $(\mathbf{X},\mathscr{B}_\mathbf{X})$ is a measurable space and
$\eta$ is a  random measure on $\mathfrak{A}$ and $(\mathbf{X},\mathscr{B}_\mathbf{X})$.
Assume that $\eta$ is of class \textbf{(QL)} and there exists
an $\mathbb{F}$-predictable generalized process
\begin{equation}
 \label{eqn-N_q}
q\colon (0,\infty)\times \mathscr{B}_\mathbf{X} \times \Omega \to [0,\infty]
\end{equation}
such that  the compensator $\widehat{\eta}$  of $\eta$ satisfies, for  $\mathbb{P}$-almost every  $\omega\in \Omega$
\begin{equation}
 \label{eqn-N_compensator}
\begin{aligned}
\widehat{\eta}\bigl((0,t]\times A;\omega\bigr)&=\int_{(0,t]}\int_A q(s,\ud x,\omega) \ds, \mbox{ for all $t\in (0,\infty)$, $A\in \mathscr{B}_\mathbf{X}$}.
\end{aligned}
\end{equation}
Suppose  that $(\rY,\ycal)$ is a standard measurable space and  $m$ is a  $\sigma$-finite non-negative measure on it. Suppose  also that
\[
\theta\colon  (0,\infty) \times \rY\times \Omega \to \mathbf{X}^\ast\coloneqq\mathbf{X} \cup \{ \partial\}
\]
is another  $\mathbb{F}$-predictable generalized process
such that for  $\mathbb{P}$-almost every  $\omega\in \Omega$
\begin{equation}\label{eqn-IW_II.7.26_2}
  m(\{ u \in \rY: \theta(t,u,\omega) \in A\}) = q(t,A,\omega),  \mbox{ for all } t\in (0,\infty), \;\; A\in \mathscr{B}_\mathbf{X}.
\end{equation}
Then there exists an extension $\bigl(\tilde{\mathfrak{A}}, \pi  \bigr)$ of  $\mathfrak{A}$,
where
\[
\tilde{\mathfrak{A}}=\bigl( \tilde{\Omega}, \tilde{\fcal}, \tilde{\mathbb{P}}, \tilde{\mathbb{F}}  \bigr), \mbox{ with } \tilde{\mathbb{F}}=(\tilde{\fcal}_t)_{t\geq 0},
\]
 and there exists an  $\rY$-valued stationary Poisson point process $\tilde{p}$ on $\tilde{\mathfrak{A}}$,
with  the  corresponding  Poisson random measure  denoted by ${\eta}_{\tilde{p}}$,  such that
\begin{trivlist}
\item[(i)] the characteristic measure of $\tilde{p}$ is equal to the measure $m$,
\item[(ii)] and the following equality holds for 
{$\widetilde{\mathbb{P}}$}-almost every  
{$\tilde{\omega}\in \widetilde{\Omega}$}
 \end{trivlist}
 \begin{equation}\label{eqn-IW_II.7.27_a}
 \begin{aligned}
 \eta\bigl((0,t]\times A, \pi(\tilde{\omega})\bigr)&=
 \int_0^t \int_{\rY} (\1_A\circ \theta)(s,u,\pi(\tilde{\omega}))\, {\eta}_{\tilde{p}}(\ud s, \ud u)
 \\&= \# \bigl\{ s \in \dom({\tilde{p}(\tilde{\omega})}) \cap (0,t]: \theta(s,{\tilde{p}}(s,\tilde{\omega}),\pi(\tilde{\omega})) \in A \bigr\},
 \\&\mbox{for all  $t\in [0,\infty)$, $A\in \mathscr{B}_\mathbf{X}$}.
\end{aligned}
\end{equation}
\end{theorem}

\begin{proof}
Theorem \ref{thm_IW-Theorem II.7.4-PRM} follows from Theorem \cite[Theorem\ II.7.4]{Ikeda+Watanabe_1989}.
\end{proof}

\begin{remark}
\label{rem_compensator}
Let us observe that in view of assumption \eqref{eqn-IW_II.7.26_2} by the change of measure Theorem, identity  \eqref{eqn-N_compensator} can be written in the following equivalent form:
\begin{equation}
\label{eqn-N_compensator_b}
\begin{aligned}
\widehat{\eta}\bigl((0,t]\times A;\omega\bigr)&=\int_0^t \int_{\rY} (\1_{A} \circ \theta) (s,u,\omega)\, m(\ud u) \ds \quad \mbox{ for all } t\in (0,\infty),\,A\in \mathscr{B}_\mathbf{X}.
\end{aligned}
\end{equation}
\end{remark}

\begin{remark}
\label{rem_predictability}
Let us also observe that \eqref{eqn-IW_II.7.26_2} implies
\begin{align*}
q(t,A,\omega) &= \int_{\rY} (\1_A \circ   \theta)(t,u,\omega) \, m(\ud u), \mbox{ for all }t\in (0,\infty) \mbox{ and } A\in \mathscr{B}_\mathbf{X}.
\end{align*}
This implies that we do not have to assume the existence of the generalized process $q$. Moreover, if the generalized process $\theta$ is $\mathbb{F}$-predictable, by   the Fubini Theorem we infer  the $\mathbb{F}$-predictability of the generalized process $q$.

This observation is reflected by the way we formulate our next result, i.e.\ Theorem \ref{thm-martingale_rep}.
\end{remark}
%}

Now we will formulate and prove a small modification of   \cite[Theorem\ 3.4.7]{Kallianpur_Xiong}. Note that by \cite[Prop.\ 2.2.2]{Veraar-PhD}, see also \cite[Proposition\ 2.2.1.]{Yaroslavtsev-PhD}, every  Banach space valued  martingale has a c\`adl\`ag version, hence, without loss of generality, we can assume below that the martingale $M$ is c\`adl\`ag.
The notation in the theorem below differs slightly from notation use in \cite{Kallianpur_Xiong}. In particular,
we use notation $(\rY,\ycal)$ instead of  $(\mathbf{U},\mathscr{B}_\mathbf{U})$.

\begin{theorem}\label{thm-martingale_rep}
Assume that
\[
\mathfrak{A}=\bigl( {\Omega}, {\fcal}, {\mathbb{P}}, {\mathbb{F}} \bigr), \mbox{ with } {\mathbb{F}}=({\fcal}_t)_{t\geq 0} ,
\]
is a  filtered probability space.
Assume that $\bX$ is a separable Banach space. Denote the Borel $\sigma$-field on $\bX$ (resp. on $\bX\setminus\{0\}$) by $\mathscr{B}({\bX})$ (resp. $\mathscr{B}(\bX\setminus\{0\})$).
Assume that $M=\bigl(M(t):\; t\geq 0\bigr)$ is a square integrable $\bX$-valued c\`adl\`ag,
purely discontinuous martingale defined on $\mathfrak{A}$.

Let $\eta_M$ be an  integer valued random measure  on  $(\mathbb{R}_+ \times \bX\setminus\{0\},\mathscr{B}(\mathbb{R}_+) \otimes \mathscr{B}(\bX\setminus\{0\}))$ and  $\mathfrak{A}$ associated with $M$,  defined for $\mathbb{P}$-almost every $\omega\in \Omega$ by
 \begin{equation}
 \label{eqn-N_M}
\eta_M\bigl((0,t]\times A;\omega\bigr)
\coloneqq\sum_{s\in (0,t]}\1_A(\Delta M(s,\omega))
\end{equation}
for all $t\in [0,\infty)$, $A\in \mathscr{B}(\bX\setminus\{0\})$.
\footnote{Note that $\eta_M\bigl((0,t]\times A;\omega\bigr) = \#\bigl\{s\in (0,t]: \Delta M(s,\omega) \in A \bigr\} = \sum_{s>0}\1_{(0,t] \times A}(s,\Delta M(s,\omega))$.}
Assume that  $\eta_M$ is of class \textbf{(QL)}.\\
Assume also that  $(\rY,\ycal)$ is  a standard measurable space, see \cite[Definition\ I.3.3]{Ikeda+Watanabe_1989},
 $m$ is a $\sigma$-finite non-negative measure defined  on it and
\[
\theta\colon  [0,\infty) \times \rY\times \Omega \to (\bX\setminus\{0\})^\ast\coloneqq(\bX\setminus\{0\}) \cup \{ \partial\}
\]
is a generalized predictable process such that
\begin{equation}
\label{eqn-integrability_in_representation_theorem}
\mathbb{E} \left[ \int_0^t \int_{\rY} \vert \theta(s,u) \vert^2 m(\ud u)\ds \right] < \infty.
\end{equation}
Assume finally that the compensator  $\widehat{\eta}_M$ of $\eta_M$   satisfies
\begin{equation}
 \label{eqn-N_M_hat}
\begin{aligned}
\widehat{\eta}_M\bigl((0,t]\times A;\omega\bigr)&=\int_0^t  m(\{ u \in \rY: \theta(s,u,\omega) \in A\}) \ds   \mbox{ for all } t>0,\,A\in \mathscr{B}(\bX\setminus\{0\}).
\end{aligned}
\end{equation}
Then there exists an extension
$\bigl(\tilde{\mathfrak{A}}, \pi  \bigr)$ of  $\mathfrak{A}$, where
\[
\tilde{\mathfrak{A}}=\bigl( \tilde{\Omega}, \tilde{\fcal}, \tilde{\mathbb{P}}, \tilde{\mathbb{F}}  \bigr), \mbox{ with } \tilde{\mathbb{F}}=(\tilde{\fcal}_t)_{t\geq 0} ,
\]
 and there exists a stationary $\rY$-valued  Poisson point   process $\tilde{p}$ on $\tilde{\mathfrak{A}}$,
with  the  corresponding the Poisson random measure  denoted by ${\eta}_{\tilde{p}}$,  such that
\begin{trivlist}
\item[(i)] the characteristic measure of the stationary Poisson point process $\tilde{p}$ is equal to the measure $m$, and
\item[(ii)] the following equality holds for $\tilde{\mathbb{P}}$-almost every  
$\tilde{\omega}\in \tilde{\Omega}$
 \end{trivlist}
 \begin{equation}\label{eqn-KX_3.4.6}
 \begin{aligned}
M(t, \pi(\tilde{\omega}))&=
 \int_0^t \int_{\rY}  \theta(s,u,\pi(\tilde{\omega}))\, \tilde{{\eta}}_{\tilde{p}}(\ud s,\ud u), \;\; \mbox{ for every }t\in [0,\infty).
 \end{aligned}
 \end{equation}
\end{theorem}

\begin{proof}[Proof of Theorem \ref{thm-martingale_rep}]

\textbf{Step 1.} Let  an object $\eta_M$ be defined in \eqref{eqn-N_M}.
According to \cite[Proposition\ II.1.16]{Jacod_Shiryaev}, where a proof is given only  for $\mathbb{R}$-valued processes,   $\eta_M$ is an integer-valued random measure on $\mathbb{R}_+ \times \mathbf{X}$.

The assumptions imposed on the random measure $\eta_M$ are the same as in Theorem \ref{thm_IW-Theorem II.7.4-PRM}.
Hence by Theorem \ref{thm_IW-Theorem II.7.4-PRM} there exists  an extension
$\bigl(\tilde{\mathfrak{A}}, \pi  \bigr)$,
where $\tilde{\mathfrak{A}}=\bigl( \tilde{\Omega}, \tilde{\fcal}, \tilde{\mathbb{P}}, \tilde{\mathbb{F}}  \bigr)$,  with  $\tilde{\mathbb{F}}=(\tilde{\fcal}_t)_{t\geq 0}$,
of  $\mathfrak{A}$ and there exists a stationary $\rY$-valued stationary Poisson point process $\tilde{p}$ on $\tilde{\mathfrak{A}}$,
with  the  corresponding Poisson random measure  denoted by ${\eta}_{\tilde{p}}$,  such that
\begin{trivlist}
\item[(i)] the characteristic measure of $\tilde{p}$ is equal to  $m$, and
\item[(ii)] the following equality holds for every for  $\tilde{\mathbb{P}}$-almost every $\tilde{\omega}\in \tilde{\Omega}$
 \end{trivlist}
\begin{equation}\label{eqn-IW-II.7.27_b}
\begin{aligned}
\eta_M\bigl((0,t]\times A, \pi(\tilde{\omega})\bigr)&=
\int_0^t \int_{\rY} (\1_A\circ \theta)(s,u,\pi(\tilde{\omega}))\, {\eta}_{\tilde{p}}(\ud s,\ud u) \\
&= \sum_{s >0} \1_{(0,t] \times A}(s,\theta(s,{\tilde{p}}(s,\tilde{\omega}),\pi(\tilde{\omega})))
\end{aligned}
\end{equation}
for all $t\in [0,\infty)$, $A\in \mathscr{B}_\mathbf{X}$

\textbf{Step 2.}  Note that the preimage by $\pi$ of a subset $\Omega$ of measure zero has measure zero in $(\tilde{\Omega},\tilde{\fcal},\tilde{\mathbb{P}})$. From \eqref{eqn-N_M} we infer that
\begin{equation}
\label{eqn-N_M_tilde}
\begin{aligned}
\eta_M\bigl((0,t]\times A;\pi(\tilde{\omega})\bigr)
&=\sum_{s>0}\1_{(0,t] \times A}(s,\Delta M(s,\pi(\tilde{\omega})))
\end{aligned}
\end{equation}
for all $t\in [0,\infty)$, $A\in \mathscr{B}(\mathbf{X}\setminus\{0\})$.
By applying Corollary \ref{cor-injectivity of map p to eta_p} 
we infer that for $\tilde{\mathbb{P}}$-almost every $\tilde{\omega}\in \tilde{\Omega}$
  \[
  \Delta M(t,\pi(\tilde{\omega}))=\theta(t,{\tilde{p}}(t, \tilde{\omega}),\pi(\tilde{\omega})), \;\; \mbox{ for every } t \in (0,\infty).
  \]
Let us next define a process $\tilde{M}$ by formula \eqref{eqn-tilde_M_t}, which by Proposition \ref{prop-purely discontinuous on extension of probability space}, is a square-integrable purely discontinuous martingale on $\bigl(\tilde{\mathfrak{A}}, \pi  \bigr)$. Moreover, it satisfies for  $\mathbb{P}$-almost every  $\omega\in \Omega$
  \[
  \Delta \tilde{M}(t,\tilde{\omega})=\tilde{\theta}(t,{\tilde{p}}(t,\tilde{\omega}),\tilde{\omega}), \;\; \mbox{ for every } t \in (0,\infty),
  \]
where the process $\tilde{\theta}\colon [0,\infty) \times \rY\times \tilde{\Omega} \to X\setminus\{0\} $ is defined by \eqref{eqn-tilde_theta}.

\textbf{Step 3.} Define an auxiliary process $\xi$ on $\bigl(\tilde{\mathfrak{A}}, \pi  \bigr)$, by
\[
\begin{aligned}
\xi(t,\tilde{\omega})&\coloneqq \int_0^t \int_{\rY} \theta\bigl(s,u,\pi(\tilde{\omega})\bigr)\, \tilde{\eta}_{\tilde{p}}(\ud s, \ud u)=\int_0^t \int_{\rY} \tilde{\theta}\bigl(s,u,\tilde{\omega}\bigr)\, \tilde{\eta}_{\tilde{p}}(\ud s,\ud u).
\end{aligned}
\]
By the 2nd part of Lemma \ref{lem-extension of probability space} the process
$\tilde{\theta}$ is a generalized predictable  process. Moreover,  by property (ii) in Definition \ref{def-extension of probability space} the change of measure theorem we have
\begin{align*}
  \tilde{\mathbb{E}} \left[ \int_0^t \int_{\rY} \vert \tilde{\theta}(s,u) \vert^2 m(\ud u) \ds \right]
 &=  \mathbb{E} \left[ \int_0^t \int_{\rY} \vert \theta(s,u) \vert^2 m(\ud u)\ds \right].
\end{align*}
which is finite by assumptions.

Therefore, by \cite[Theorem\  3.4.5]{Kallianpur_Xiong} we infer that $\xi$ is a square integrable purely discontinuous martingale  on  probability space $\tilde{\mathfrak{A}}$ and that for  $\mathbb{P}$-almost every $\omega\in \Omega$
\[
\Delta \xi(t,\tilde{\omega})=\tilde{\theta}\bigl(t,{\tilde{p}}(t,\tilde{\omega}),\tilde{\omega}\bigr), \;\; \mbox{ for every } t \in (0,\infty).
\]

We conclude that for $\mathbb{P}$-almost every  $\omega\in \Omega$
\begin{equation}
\label{M_and_xi_have_equal_jumps}
\Delta \tilde{M}(t,\tilde{\omega})=
\Delta \xi(t,\tilde{\omega})
, \;\; \mbox{ for every } t \in (0,\infty).
\end{equation}

\textbf{Step 4.} From \eqref{M_and_xi_have_equal_jumps} we deduce that the process $\tilde{M}-\xi$ is square integrable continuous  martingale  on  probability space $\tilde{\mathfrak{A}}$.
On the other hand, since both processes $\tilde{M}$ and $\xi$ are  square-integrable purely discontinuous martingales and the set of all  square-integrable purely discontinuous martingales is a vector space we infer that the process $\tilde{M}-\xi$ is also a is a square-integrable purely discontinuous martingale.  We conclude that $\tilde{M}-\xi$ is equal to $0$. Hence we proved that for $\mathbb{P}$-almost every $\omega\in \Omega$
\[
\begin{aligned}
\tilde{M}(t,\tilde{\omega})=\int_0^t \int_{\rY} \tilde{\theta}\bigl(s,u,\tilde{\omega}\bigr)\, \tilde{\eta}_{\tilde{p}}(\ud s, \ud u), \;\; \;\; \mbox{ for every } t \in (0,\infty).
\end{aligned}
\]

This completes the proof of the theorem.
\end{proof}

\section{Filtration enlargement and martingales}
\label{app_martingales}

\begin{definition}
\label{def_martingale}
A process  $\rM\colon  \Omega  \times [0,\infty) \to \mathbb{R}$ is a martingale with respect to a filtration $\mathbb{H}=(\mathscr{H}_t)_{t\geq 0}$ if and only if it satisfies the following properties
\begin{enumerate}
\item[(i)] For every   $t \geq 0$ the function $M_t\coloneqq M(\cdot,t)$ is $\mathscr{H}_t$-measurable,
\item[(ii)] For every   $t \geq 0$ the function $M_t$ is integrable,
\item[(iii)] For all pairs $(s,t) \in [0,\infty)^{2}_{+}$,
\begin{equation}\label{eqn-martingale definition (iii)}
\mathbb{E}\left[ M_t|\mathscr{H}_s \right] = M_s,
\end{equation}
i.e.\ for every $A\in \mathscr{H}_s$
\[
\int_A{M_t}\, \ud \mathbb{P}=\int_A{M_s}\, \ud \mathbb{P}.
\]
or
\[
\mathbb{E} \bigl[ \1_A(M_t-M_s) \bigr]=0.
\]

\end{enumerate}
\end{definition}

\begin{lemma}\label{Lem:Indep-Increments-augmentation}
	Assume that  $(\Omega, \mathscr{F}, \mathbb{P})$ is a probability space 
 with the family  of all $\mathbb{P}$-null sets of $\mathscr{F}$ denoted by $\mathscr{N}$, i.e. 
	\begin{equation}\label{eqn-N-null sets}
\mathscr{N}\coloneqq\bigl\{ A \subset \Omega: \exists B \in \mathscr{F}: A\subset B \mbox{  and } \mathbb{P}(B)=0 \bigr\}.
\end{equation}
 Assume that $\mathbb{H}$ is the filtration $\mathbb{H}=(\mathscr{H}_t)_{t\ge 0}$ and  $\mathbb{G}=(\mathscr{G}_t)_{t\ge 0}$ is the usual augmentation of $\mathbb{H}$, i.e.
	\begin{equation}\label{eqn-G_t}
	\mathscr{G}_t=\bigcap_{s>t} \tilde{\mathscr{H}_s} ,\, t\ge 0,
	\end{equation}
	where $\tilde{\mathscr{H}_s}$ is  the $(\mathscr{F},\mathbb{P})$-completion of the
$\sigma$-field $\mathscr{H}_s$, i.e. 
	\begin{equation}\label{eqn-H_t}
 \tilde{\mathscr{H}_s}\coloneqq\sigma(\mathscr{H}_s \cup  \mathscr{N}),\, s\ge 0.
	\end{equation}
If a process 	$\rM=\left(M_t: t\geq 0\right)$ is a c\`adl\`ag $\mathbb{H}$-martingale on $(\Omega, \mathscr{F}, \mathbb{P})$,		Then, it is also a $\mathbb{G}$-martingale.
\end{lemma}

\begin{proof}[Proof of Lemma \ref{Lem:Indep-Increments-augmentation} ]
Let us choose and fix a probability space  $(\Omega, \mathscr{P}, \mathbb{P})$, a filtration $\mathbb{H}=(\mathscr{H}_t)_{t\ge 0}$,
its augmentation $\mathbb{G}=(\mathscr{G}_t)_{t\ge 0}$
and an $\mathbb{H}$-martingale $\rM$.

First of all let us observe that in view of \cite[Lemma\ 7.8]{Kallenberg_2002}, for every $t\geq 0$, the $\sigma$-field $\mathscr{G}_t$ is equal to the $(\mathscr{F},\mathbb{P})$-completion $\tilde{\mathscr{H}_{t+}}$ of the
the $\sigma$-field $\mathscr{H}_{t+}\coloneqq\bigcap_{s>t} \mathscr{H}_s$, i.e.
	\begin{equation}\label{eqn-G_t-2}
	\mathscr{G}_t=  \sigma\Bigl( \bigl( \bigcap_{s>t} \mathscr{H}_s \bigr) \cup  \mathscr{N} \Bigr)=\tilde{\mathscr{H}_{t+}}.
	\end{equation}

Thus we infer that the filtration $(\mathscr{G}_t)_{t\ge 0}$ satisfies the usual conditions, especially it is right-continuous and complete.

We first notice that since $\rM$ is an $\mathbb{H}$-martingale, it satisfies  the property  (ii) of Definition \ref{def_martingale}.  Thus,   we only need to check that $\rM$ satisfies   conditions (i) and (iii) of that definition  for the enlarged filtration $\mathbb{G}$.

	We first check the condition (i) of Definition \ref{def_martingale}.  In order to do this let $t\in [0,\infty)$ be arbitrary, but fixed. We observe that since  $M_t$  is $\mathscr{H}_t$-measurable, it   is also $\mathscr{H}_r$-measurable for any $r> t$ and hence, by \eqref{eqn-G_t},  $\mathscr{G}_t$-measurable.
	
	We now check the condition (iii) of Definition  \ref{def_martingale}. For this purpose we  choose and fix an arbitrary pair $(s,t)\in [0,\infty)^2_{+}$ with $s<t$.
We first prove that
\begin{equation}\label{eqn-martingale w.r.t. right continuous extension}
  \mathbb{E}[M_t|\mathscr{H}_{s+}]=M_s.
\end{equation}
For this purpose let us consider a sequence $(s_n)_{n\in \mathbb{N}}$ such that
\begin{equation}\label{eqn-s_n}
s_n \in (s,t)  \mbox{ for every $n\in \mathbb{R}$  and } s_n \searrow s.
\end{equation}
Then we have
\[
\mathscr{H}_{s+}=\bigcap_{n\in \mathbb{N}} \mathscr{H}_{s_n}.
\]
Therefore, by the conditioning limits Theorem, see \cite[Theorem\ 7.23]{Kallenberg_2002},
\begin{align*}
 \mathbb{E}[M_t|\mathscr{H}_{s+}]=\lim_{n\to \infty}  \mathbb{E}[M_t|\mathscr{H}_{s_n}] \mbox{ a.s.\ and in } \mathbb{L}^1.
\end{align*}
Since by assumptions, $M$ is  an  $\mathbb{H}$-martingale  we infer that  for every $n\in \mathbb{N}$,
\begin{align*}
  \mathbb{E}[M_t|\mathscr{H}_{s_n}]= M_{s_n} \mbox{ a.s.}.
\end{align*}
Moreover, since by assumptions, $M$ is a c\`adl\`ag process we infer that
\begin{align*}
\lim_{n\to \infty} M_{s_n}=M_s  \mbox{ a.s.}
\end{align*}
Combining the last three equalities  we conclude the proof of  identity
\eqref{eqn-martingale w.r.t. right continuous extension}.

To conclude the proof of Lemma \ref{Lem:Indep-Increments-augmentation}, in view of equality \eqref{eqn-G_t-2},  it is sufficient to observe that
if $X$ is $\mathscr{F}$-measurable integrable function and $\mathscr{H}$ is a sub-$\sigma$-field of $\mathscr{F}$,   then

\begin{equation}\label{eqn-expectation_wrt_extension}
  \mathbb{E}[X|\mathscr{H}]=\mathbb{E}[X|\overline{\mathscr{H}}].
\end{equation}
Obviously, the random variable $Y\coloneqq\mathbb{E}[X|\mathscr{H}]$ being $\mathscr{H}$-measurable is also
$\overline{\mathscr{H}}$-measurable. We only need to prove, with $Z=X-Y$,  that
\[
\mathbb{E} \bigl[] \1_AZ \bigr]=0 \mbox{ for every } A \in \overline{\mathscr{H}}.
\]
Let us choose and fix a set $A \in \overline{\mathscr{H}}$. Then by \cite[Lemma\ 7.25]{Kallenberg_2002} there exist two sets $A_{\pm}\in \mathscr{H}$ with
\begin{align*}
A_{-} \subset A \subset A_{+} \mbox{ and } \mathbb{P}(A_{+}\setminus A_{-})=0.
\end{align*}
Therefore   $\mathbb{E} \bigl[ \1_{A+}Z \bigr] = 0 $,
$\mathbb{E} \bigl[ \1_{A-}Z \bigr] = 0 $ and because $ \1_{A\setminus A_{-}} Z = 0$ a.s.\ and $\1_{A_{+}\setminus A} = 0$ a.s.
\begin{equation*}\mathbb{E} \bigl[ \1_{A\setminus A_{-}} Z  \bigr] = 0 \mbox{  and } \mathbb{E} \bigl[  \1_{A_{+}\setminus A} Z  \bigr]=0.\end{equation*}
Since $Z$ is integrable, we infer that
\begin{align*}
   \mathbb{E} \bigl[ \1_A Z\bigr]   &= \frac12 \bigl\{ \mathbb{E} \bigl[ \1_{A_{+}} Z\bigr]+ \mathbb{E} \bigl[ \1_{A_{-}} Z\bigr]    \\
   &- \mathbb{E} \bigl[ \1_{A_{+}\setminus A} Z\bigr]   +\mathbb{E} \bigl[ \1_{A\setminus A_{-}} Z\bigr] \bigr\}=\frac12 \times 0=0.
\end{align*}
\end{proof}

\section{A small strengthening of Lemma B.1 from paper  \texorpdfstring{\cite{Brzezniak+Motyl_2013}}{[11, Lemma B.1]}}
\label{sec-B.1}

We will formulate and prove  a slightly stronger version of \cite[Lemma B.1]{Brzezniak+Motyl_2013}.

\begin{lemma}
Assume that  $u \in {L}^{2}(0,T;H)$ and   ${({u}_{n} )}_{n} $ is  a bounded sequence in ${L}^{2}(0,T;H)$ such that ${u}_{n} \to u $ in ${L}^{2}(0,T;{H}_{\loc})$. If $\gamma > \frac{d}{2}+1$, 
then  for  every  $\psi \in {V}_{\gamma }$,
\begin{equation} \label{eqn-lem-B.12}
  \lim_{n \to \infty } \int_{0}^{T} \vert  \dualf{ B \bigl( {u}_{n} (s),{u}_{n}(s)  \bigr)  }{\psi }{{V}_{\gamma }}{{V}_{\gamma }^{\prime }}\, -
    \dualf{ B \bigl( u (s),u(s)  \bigr)  }{\psi }{{V}_{\gamma }}{{V}_{\gamma }^{\prime }}\, \vert \ds .
\end{equation}
\end{lemma}

\begin{proof}
 Assume first that $\psi \in \vcal $.  Then there exists $R>0$ such that $\supp \psi $ is a compact subset of ${\ocal }_{R}$.
Then, by inequality \eqref{eqn-b_R} we infer that for every $v ,w \in H $
\begin{equation} \label{eqn-lem-B.1-2}
 | \dual{B(v,w)}{\psi }{} | \leq c
 \normb{v}{\HR}{} \normb{w}{\HR}{} \normb{\psi}{{\rV}_{s}}{}
\end{equation}
Since,
$ B({u}_{n} , {u}_{n} ) - B(u,u) =  B({u}_{n} -u , {u}_{n} ) + B(u,{u}_{n} -u) $.
by  using the estimate \eqref{eqn-lem-B.1-2} and the H\H{o}lder inequality, we infer that there exists $C>0$ such that
\begin{eqnarray*}
& & \int_{0}^{t} \Bigl| \dual{ B \bigl( {u}_{n} (s) ,{u}_{n} (s) \bigr) }{\psi }{}
-  \dual{B \bigl( u(s), u(s)\bigr)  }{\psi }{} \Bigr| \ds
\\
& &\leq \int_{0}^{t} \Bigl|  \dual{ B \bigl( {u}_{n} (s) - u(s) ,{u}_{n} (s) \bigr) \ }{\psi }{} \Bigr| \ds
  +  \int_{0}^{t} \Bigl| \dual{ B \bigl( u (s) ,{u}_{n} (s) - u(s) \bigr)  }{\psi }{} \Bigr| \ds \\
& &\leq \Bigl( \int_{0}^{t}   |{u}_{n} (s) -u(s) {|}_{H_R} |{u}_{n} (s){|}_{H_R} \ds
   + \int_{0}^{t}   |u(s) {|}_{H_R}  |{u}_{n} (s) - u(s){|}_{H_R} \ds  \Bigr)
 \cdot \normb{\psi}{{\rV}_{s}}{} \\
& & \leq C
\Vert {u}_{n} - u \Vert_{L^2(0,T;H_R)} \bigl( \Vert u_n\Vert_{L^2(0,T;H_R)}
+\Vert   u \Vert_{L^2(0,T;H_R)} \bigr)
\normb{\psi}{{\rV}_{s}}{} ,
\end{eqnarray*}
Since ${u}_{n} \to u $ in ${L}^{2}(0,T;{H}_{\loc})$, we infer that \eqref{eqn-lem-B.12} holds for every  $\psi \in \vcal $.

\bigskip  \noindent
If $\psi \in {V}_{\gamma }$ then for every  $\eps > 0 $ there exists ${\psi }_{\eps } \in \vcal $ such that  $\normb{\psi-{\psi }_{\eps }}{{\rV}_{s}}{} \leq \eps $.
Then, for every $s \in [0,T]$, with $u_n={u}_{n} (s )$ and $u=u (s )$, we have
\begin{align*}
& \bigl| \dual{ B\bigl( {u}_{n} ,u_n \bigr) - B\bigl( u,u \bigr) }{\psi }{} \bigr| \\
 &  \le  \bigl| \dual{ B\bigl( u_n ,u_n (s )\bigr) - B\bigl( u,u \bigr) }{
\psi -{\psi }_{\eps }}{} \bigr|
 + \bigl| \dual{ B\bigl( u_n ,u_n (s )\bigr) - B\bigl( u,u (s )\bigr)  }{{\psi }_{\eps } }{} \bigr| \\
&  \le  \bigl( \bigl| B\bigl( u_n ,u_n \bigr){\bigr| }_{{V}_{\gamma }^{\prime }}
  + \bigl| B\bigl( u ,u \bigr) {\bigr| }_{{V}_{\gamma }^{\prime }} \bigr)
   \normb{\psi -\psi _\eps }{{\rV}_{s}}{}
   \\
 &   + \bigl| \dual{ B\bigl( u_n ,u_n (s )\bigr) - B\bigl( u,u (s )\bigr)  }{{\psi }_{\eps } }{} \bigr| \\
 &  \le \eps  \bigl( |u_n (s ){|}_{H}^{2}
   +|u (s ){|}_{H}^{2}\bigr)
 + \bigl| \dual{ B\bigl( u_n ,u_n (s )\bigr)
   - B\bigl( u,u (s )\bigr)  }{{\psi }_{\eps } }{} \bigr| .
\end{align*}
Hence
\begin{align*}
&  \int_{0}^{T} \bigl| \dual{ B\bigl( u_n (s ) ,u_n (s ) \bigr) - B\bigl( u(s ) ,u (s ) \bigr) }{\psi }{} \bigr|  \ds
\\
&  \le  \eps \int_{0}^{T}\bigl(  |u_n (s ){|}_{H}^{2} +  |u (s ){|}_{H}^{2}\bigr) \ds
 + \bigl| \int_{0}^{T}\dual{ B\bigl( u_n (s ) ,u_n (s )\bigr) - B\bigl( u(s ) ,u (s )\bigr)  }{{\psi }_{\eps } }{} \ds \bigr| \\
&  \le \eps \cdot \bigl( \sup_{n\ge 1}\Vert u_n \Vert^2_{L^2(0,T;H)} +\Vert u \Vert^2_{L^2(0,T;H)}  \bigr)
+ \int_{0}^{T} \bigl|  \dual{ B\bigl( u_n (s ) ,u_n (s )\bigr) - B\bigl( u(s ) ,u (s )\bigr)  }{{\psi }_{\eps } }{} \bigr| \ds
.
\end{align*}
Passing to the upper limit as $n \to \infty $, we obtain
\begin{equation*}
 \limsup_{n \to \infty }
 \int_{0}^{T} \bigl| \dual{ B\bigl( u_n (s ) ,u_n (s ) \bigr) - B\bigl( u(s ) ,u (s ) \bigr) }{\psi }{} \bigr| \ds
 \le   M \eps ,
\end{equation*}
where $M\coloneqq\sup_{n\ge 1}\Vert u_n
\Vert^2_{L^2(0,T;H)} + \Vert u \Vert^2_{L^2(0,T;H)}<\infty $.
Since $\eps >0$ is arbitrary, we infer that
\[
 \lim_{n \to \infty }
 \int_{0}^{T}  \bigl| \dual{ B\bigl( u_n (s) ,u_n (s ) \bigr) - B\bigl( u(s ) ,u (s ) \bigr) }{\psi }{} \bigr| \ds =0.
\]
This completes the proof.
\end{proof}

\bigskip  \noindent

\bigskip

\section{Proof of  Lemma \ref{lem-Galerkin_estimates}} \label{app-proof_of_apriori_est}
Let us choose and fix $n \in \mathbb{N}$. Let  ${u}_{n}=\bigl( {u}_{n}(t) {\bigr) }_{t \in [0,T]} $  be the the processes from Lemma \ref{lem-Galerkin_existence}. 
For every  all $R>0$ let us define
\begin{equation} \label{eq:stopping_time}
{{\tau }_{n}}_{}(R):= \inf \{ t \ge 0: \, \, |{u}_{n}(t){|}_{H} > R  \} \wedge T. 
\end{equation}
Since the process ${u}_{n}$ is $\mathbb{F}$-adapted and right-continuous, 
and the filtration $\mathbb{F} $ is right-continuous,
by Proposition 2.1.6 from \cite{Ethier_Kurtz_1986}, we infer that
${\tau }_{n}(R)$ is an $\mathbb{F}$-stopping time. Moreover, since  it is c\`{a}dl\`{a}g on $[0,T]$,  the trajectories $[0,T] \ni t \mapsto {u}_{n}(t)$, $\mathbb{P} $-a.s. Thus
${\tau }_{n} (R)\toup T$, $\mathbb{P} $-a.s., as $R\to \infty $. In fact, $\mathbb{P} $-a.s. there exits (a random)  $R_0>0$ such that 
${\tau }_{n} (R)=T$ for all $R\geq R_0$.

Assume that $p =2$ or $p=4+\gamma $.
Using the It\^{o} formula, 
see \cite[Th. II.5.1]{Ikeda+Watanabe_1989},
to the function  
\[\phi(x):=|x{|}^{p}:= |x{|}_{H}^{p}, \;\;  x \in H,
\] 
we obtain for all $t \in [0,T]$
\begin{equation*}
\begin{aligned}
&|{u}_{n}  ({t\wedge {\tau }_{n} (R)}) {|}^{p} =  |P_n {u}_{0}{|}^{p} \\
&+ \int_{0}^{t\wedge {\tau }_{n} (R)} \bigl\{ p | u_n (s){| }^{p-2} \dual{u_n (s)}{- {\acal }_{n} u_n (s) - {B}_{n}(u_n (s)) + P_n f (s)}{} \bigr\}  \ds \\
& + \int_{0}^{t\wedge {\tau }_{n} (R)} \int_{{Y}} \big\{ \phi \bigl( u_n (s-)  + P_n F(s,u_n (s-),y) \bigr) -
   \phi \bigl( u_n (s-) \bigr)  \bigr\} \, \tilde{\eta}(\ud s,\ud y)   \\
& + \int_{0}^{t\wedge {\tau }_{n} (R)} \int_{{Y}} \big\{ \phi \bigl( u_n (s))  + P_n F(s,u_n (s),y) \bigr) -
 \phi \bigl( u_n (s) \bigr)  \\
&  - \dual{\phi '(u_n (s) ) }{P_n F(s,u_n (s),y)}{} \bigr\} \, \nu (\ud s,\ud y) .
\end{aligned}
\end{equation*} 
By \eqref{eqn-A_n_duality} and \eqref{eqn-B_n_b} we obtain for all $t \in [0,T]$
\begin{equation*}
\begin{aligned} 
& |{u}_{n}  ({t\wedge {\tau }_{n} (R)}) {|}^{p} 
=   |P_n {u}_{0}{|}^{p} \\
&+ \int_{0}^{t\wedge {\tau }_{n} (R)}  \bigl\{ -p | u_n (s){| }^{p-2} {\| u_n (s)\| }^{2} 
  + p | u_n (s){| }^{p-2} \dual{u_n (s)}{f (s)}{} \bigr\}  \ds   \\  
 & + \int_{0}^{t\wedge {\tau }_{n} (R)} 
   \int_{{Y}} \big\{ | u_n (s-)  + P_n F(s,u_n (s-),y) {|}^{p} -
    | u_n ({s}^{-}) {|}^{p}  \bigr\} \, \tilde{\eta}(\ud s,\ud y)    \\
& + \int_{0}^{t\wedge {\tau }_{n} (R)} \int_{{Y}} 
\big\{  | u_n (s)  + P_n F(s,u_n (s),y) {| }^{p}- | u_n (s) {|}^{p}   \\
 &  - p | u_n (s) {|}^{p-2} \dual{u_n (s)  }{P_n F(s,u_n (s),y)}{} \bigr\} \, 
 \nu (\ud y) \ds .
\end{aligned}
\end{equation*}
By the Cauchy-Schwarz inequality,  for $f \in \rV^\prime$ and $u_n \in \rV$, we have 
\begin{equation*}
\begin{aligned}
&  \dual{f}{u_n }{} \le  \, {|f|}_{\rV^\prime} \cdot {\| u_n \| }_{V}
\le \frac{1}{2} \bigl\{ {|f|}_{\rV^\prime}^{2} + |u_n {|}_{H}^{2} + {\| u_n \| }^{2} \bigr\} .
\end{aligned} 
\end{equation*}
Thus we infer that 
\begin{equation*}
\begin{aligned}
p {|u_n |}^{p-2} \dual{u_n }{f }{} 
& \le \frac{p}{2} {|u_n |}^{p-2} {\| u_n \| }^{2} + \frac{p}{2} {|u_n |}^{p} 
+ \frac{p}{2} {|f|}_{\rV^\prime}^{2} {|u_n |}^{p-2}  .
\end{aligned}
\end{equation*}
Since by the Young inequality  
\[
\begin{aligned}
\frac{p}{2} {|f|}_{\rV^\prime}^{2} {|u_n |}^{p-2}  
& \le  \frac{p}{2} \Bigl\{ \frac{2}{p} {|f|}_{\rV^\prime}^{p} + \Bigl( 1-\frac{2}{p}  \Bigr) {|u_n |}^{p}   \Bigr\} \\
& =  {|f|}_{\rV^\prime}^{p} + \Bigl( \frac{p}{2}-1  \Bigr) {|u_n |}^{p},
\end{aligned}
\]
we obtain
\begin{equation*}
\begin{aligned}
p {|u_n |}^{p-2} \dual{u_n }{f }{} 
& \le  \frac{p}{2} {|u_n |}^{p-2}{\| u_n \| }^{2} 
+ (p-1)  {|u_n |}^{p}  +  {|f|}_{\rV^\prime}^{p} .
\end{aligned}
\end{equation*}
Thus, 
applying the above for $f=f(s)$ and $u_n=u_n(s)$, by Assumption A\ref{assum-A.1}, \eqref{eqn-norm_V},  we obtain for all $s \in [0,T\wedge {\tau }_{n} (R)] $, with $c_n= |P_n u_0 {|}^{p} +  \int_{0}^{T} {|f(s)|}_{\rV^\prime}^{p} \, ds$ and $c_1= p-1$, we have 
\begin{equation}
\begin{aligned} 
& |{u}_{n}  ({t \wedge  {\tau }_{n} (R)}){|}^{p}
+ \frac{p}{2} \int_{0}^{t\wedge {\tau }_{n} (R)} |u_n (s){|}^{p-2} {\| u_n (s)\| }^{2}  \ds  
\\
 &\le c_n+ c_1   \int_{0}^{t\wedge  {\tau }_{n} (R)} | u_n (s){|}^{p}  \ds   \\
&  + \int_{0}^{t\wedge {\tau }_{n} (R)} 
   \int_{{Y}} \big\{ | u_n (s-)  +P_n F(s,u_n (s-),y) {|}^{p} -
   | u_n (s-) {|}^{p}  \bigr\} \, \tilde{\eta}(\ud s,\ud y)    \\
&  + \int_{0}^{t\wedge {\tau }_{n} (R)} \int_{Y} 
   \big\{ |u_n (s)  + P_n F(s,u_n (s),y) {|}^{p}-
   | u_n (s) {|}^{p}    \\
&   - p | u_n (s) {|}^{p-2} \ilsk{u_n (s)  }{P_n F(s,u_n (s),y)}{H} \bigr\} \, 
\nu (\ud y) \ds  . 
\end{aligned}
\label{eq:Ito_formula}
\end{equation}

From the Taylor formula, it follows that there exists a positive constant ${c}_{p}>0$ such that  the following inequality holds
\begin{equation}
\begin{aligned}
    \vert \phi(x+h)-\phi(x)-\phi^\prime(x)h \vert&= \bigl| |x+h{|}_{H}^{p} -|x{|}_{H}^{p}- p |x{|}_{H}^{p-2} \ilsk{x}{h}{H} \bigr| \\
&\leq {c}_{p} ( {|x|}_{H}^{p-2} +  {|h|}_{H}^{p-2} ) \, |h{|}_{H}^{2}, \;\; x,h \in H.  \label{eq:Taylor_II}
\end{aligned}
\end{equation}
Let us define a (local) process  ${M}_{n }$ and a process  ${I}_{n }$ by  respectively  
\begin{equation}\label{eqn-M_n-local}
    \begin{aligned}
            {M}_{n}(t)&:=  \int_{0}^{t} \int_{{Y}} 
\big\{ |u_n (s-)+ P_n F(s,u_n (s-),y) {|}^{p} - |u_n (s-) {|}^{p}  \bigr\} \, \tilde{\eta}(\ud s,\ud y) , 
 \\ 
 & \hspace{2 truecm} t\in [0,T\wedge {\tau }_{n} (R)],
     \end{aligned}
 \end{equation}
and 
\begin{equation}  
\begin{aligned}
& {I}_{n}(t):= \int_{0}^{t} \int_{{Y}} 
   \big\{  | u_n (s)  + P_n F(s,u_n (s),y) {| }^{p} -| u_n (s) {|}^{p}    \\
&  - p | u_n (s) {|}^{p-2} {(u_n (s) ,P_n F(s,u_n (s),y))}_{H} \bigr\} \, 
\nu (\ud y) \ds  , \quad t\in [0,T].  
\end{aligned}
\label{eq:I_n}
\end{equation}

By \eqref{eq:Taylor_II},  \eqref{eqn-F_linear_growth-p} and \eqref{eq:stopping_time}, we deduce that the stopped process 
$\bigl( {M}_{n}(t\wedge {\tau }_{n} (R)) {\bigr) }_{t \in [0,T]}$,
is a square integrable martingale. Hence we infer that 
\begin{equation}
\mathbb{E} [{M}_{n}(t\wedge {\tau }_{n}(R))] = 0  \mbox{ for all }t\in [0,T].
\end{equation}
Let us define another auxiliary process
By \eqref{eq:Taylor_II} and \eqref{eqn-F_linear_growth-p} we we infer that 
for some constant ${\tilde{c}}_{p}>0$,  the following inequalities hold, 
\begin{equation*}
\begin{aligned}
& |{I}_{n}(t)|  \\
& \le  {c}_{p}\int_{0}^{t} \int_{Y} | P_n F(s,u_n (s),y) {|}_{H}^{2}
 \bigl\{ | u_n (s) {|}_{H}^{p-2} + | P_n F(s,u_n (s),y) {|}_{H}^{p-2} \bigr\}
  \nu (\ud y) \ds \\
& \le  {c}_{p}\int_{0}^{t} \bigl\{ 
{C}_{2} | u_n (s) {|}_{H}^{p-2} \bigl( 1+ | u_n (s) {| }_{H}^{2}\bigr) 
 + {C}_{p}\bigl( 1+ |u_n (s) {|}_{H}^{p}\bigr)   \bigr\} \ds  \\
& \le  {\tilde{c}}_{p} \int_{0}^{t} \bigl\{ 1+ |u_n (s) {|}_{H}^{p} \bigr\} \ds
= {\tilde{c}}_{p}t + {\tilde{c}}_{p}\int_{0}^{t} |u_n (s){|}_{H}^{p} \ds , \qquad t \in [0,T]. 
\end{aligned}
\end{equation*}
Thus by the Fubini Theorem, we obtain the following inequality
\begin{equation} \label{eq:E_I_n(t)}
\mathbb{E} \bigl[ | {I}_{n}(t)| \bigr] \le  
 {\tilde{c}}_{p}t + {\tilde{c}}_{p} \int_{0}^{t} \mathbb{E} \bigl[|u_n (s) {|}_{H}^{p} \bigr] \ds , \qquad t \in [0,T]. 
\end{equation}

By \eqref{eq:Ito_formula} and \eqref{eq:E_I_n(t)}, we have for all $t\in [0,T]$
\begin{equation} 
\begin{aligned}
& \mathbb{E} \bigl[  \bigl| {u}_{n}  ({t\wedge {\tau }_{n} (R)}) {\bigr| }_{H}^{p} \bigr]  
 + \frac{p}{2} 
 \mathbb{E} \Bigl[ \int_{0}^{T\wedge {\tau }_{n} (R)} | u_n (s){|}_{H}^{p-2} {\| u_n (s)\}  }^{2} \ds \Bigr] 
 \\
\qquad & \le  c+ {\tilde{c}}_{p}T 
+ (c_1 +{\tilde{c}}_{p} ) \int_{0}^{t\wedge {\tau }_{n} (R)} \mathbb{E} \bigl[|u_n (s) {|}_{H}^{p} \bigr] \ds .
\end{aligned}
\label{eq:Ito_formula_est}
\end{equation} 
In particular, for $t \in [0,T]$, 
\[
\mathbb{E} \bigl[ | {u}_{n}  ({t\wedge {\tau }_{n} (R)}) {| }_{H}^{p} \bigr]
\le   c+ {\tilde{c}}_{p}T + (c_1 +{\tilde{c}}_{p} )
\int_{0}^{t\wedge {\tau }_{n} (R)}  \mathbb{E} \bigl[ |u_n (s \wedge {\tau }_{n} (R) ){|}_{H}^{p} \bigr] \ds .
\]
By the Gronwall Lemma we infer that there exists a constant ${\tilde{\tilde{C_p}}}$ independent of  $R>0$ and $n \in \mathbb{N} $,
\[
\mathbb{E} \bigl[| {u}_{n}  ({t\wedge {\tau }_{n} (R)}) {|}^{p} \bigr]
 \le   {\tilde{\tilde{C_p}}}, \;\; t \in [0,T].  
 \]
 Therefore, 
\[ \label{eq:E(un(t))_R_est}
\sup_{n \ge 1} \sup_{t\in [0,T]} \mathbb{E} \bigl[ |{u}_{n} (t\wedge {\tau }_{n} (R)) {|}_{H}^{p} \bigr]
 \le   {\tilde{\tilde{C_p}}}
\]
and  in particular, for some constant $\tilde{{C}_{p}}>0$, 
\[ \label{eq:E(int_un(t))_R_est}
\sup_{n \ge 1} \mathbb{E} \Bigl[ \int_{0}^{{T\wedge {\tau }_{n} (R)}} | {u}_{n}(s) {|}_{H}^{p} \ds \Bigr] \le \tilde{{C}_{p}}. 
\]
Passing to the limit as $R \to \infty $, by the Fatou Lemma we infer that
\begin{equation} \label{eq:E(int_un(t))_est}
\sup_{n \ge 1} \mathbb{E} \Bigl[ \int_{0}^{T}|{u}_{n} (s) {|}_{H}^{p} \, ds \Bigr]  
 \le   \tilde{{C}_{p}} .
\end{equation} 
By \eqref{eq:Ito_formula_est} and \eqref{eq:E(int_un(t))_est}, we infer that, for some positive constant ${C}_{p}$, 
\begin{equation}  \label{eq:HV_R_estimate}
\sup_{n \ge 1} \mathbb{E} \Bigl[  \int_{0}^{T\wedge {\tau }_{n} (R)}| u_n (s){|}_{H}^{p-2} {\| u_n (s)\| }^{2} \ds \Bigr]  \le {C}_{p}.
\end{equation}

Passing to the limit as $R \to \infty $ and using again the Fatou Lemma we infer that
\begin{equation}  \label{eq:HV_estimate}
\sup_{n \ge 1} \mathbb{E} \Bigl[  \int_{0}^{T} |u_n (s){|}_{H}^{p-2} {\| u_n (s)\| }^{2} \ds \Bigr] 
\le {C}_{p} .
\end{equation} 
In particular, putting $p:=2$ by \eqref{eqn-norm_V} \eqref{eq:HV_estimate} and \eqref{eq:E(int_un(t))_est}   we obtain assertion  \eqref{eqn-V_estimate}.

Let us move to the proof of inequality \eqref{eqn-H_estimate-p>2}.
From now we assume that $p:=4+\gamma $.
By the Burkholder-Davis-Gundy inequality we obtain
\begin{align}
\label{eq:BDG_R_new}
& \mathbb{E} \Bigl[ \sup_{r\in [0,t ]} |{M}_{n}(r\wedge {\tau }_{n} (R))| \Bigr]  \\
& =\mathbb{E} \Bigl[ \sup_{r\in [0,t ]} 
 \Bigl| \int_{0}^{r\wedge {\tau }_{n} (R)} 
 \int_{Y}  \big\{  \bigl| u_n ({s}^{-}\! )+P_n F(s,u_n ({s}^{-});y) {\bigr| }_{H}^{p} -
\bigl| u_n ({s}^{-}\! ) {\bigr| }_{H}^{p}  \bigr\}  \tilde{\eta}(\ud s,\ud y) \Bigr|  \Bigr]  
\notag\\
& \le {\tilde{K}}_{2,1} \mathbb{E} \Bigl[
\Bigl(  \int_{0}^{t\wedge {\tau }_{n} (R)} 
\int_{Y} \bigl(  \bigl| u_n ({s}^{}\! )+P_n F(s,u_n ({s}^{}),y) {\bigr| }_{H}^{p} 
- \bigl| u_n ({s}^{}) {\bigr|}_{H}^{p}  {\bigr) }^{2}  \nu (\ud y) \ds {\Bigr) }^{\frac{1}{2}} 
 \Bigr] 
 \notag
\end{align}
for some constant ${\tilde{K}}_{2,1}>0$.
Using the Mean Value Theorem we infer that there exists a constant ${\tilde{c}}_p>0$ such that  for all $x,h \in \rH$, 
\begin{equation}
\begin{aligned}
\bigl| |x+h{|}_{H}^{p} -|x{|}_{H}^{p} {\bigr| }^{2}
\; &\le \; {\tilde{c}}_{p} ( {|x|}_{H}^{p-1} {|h|}_{H} +  {|h|}_{H}^{p} {)}^{2} 
\\
\; &= \;  {\tilde{c}}_{p} ({|x|}_{H}^{2p-2}|h{|}_{H}^{2} + 2 {|x|}_{H}^{p-1} {|h|}_{H}^{p+1} + {|h|}_{H}^{2p} ) .
\end{aligned}
\label{eq:Taylor_1'}
\end{equation} 
Hence by inequality \eqref{eqn-F_linear_growth-p} we obtain for all $s \in [0,T] $

\begin{align}
\label{eq:BDG_R_Taylor_1''_new}  
& \int_{Y}\bigl( {u_n (s) + P_n F(s,u_n (s),y) }_{H}^{p} -
{u_n (s)}_{H}^{p}  {\bigr) }^{2} \nu (\ud y) \\
& \le   {\tilde{c}}_{p} {u_n (s)}_{H}^{2p-2} 
\int_{Y} {F(s,u_n (s),y)}_{H}^{2} \, \nu (\ud y)  
\notag \\ 
&  + {\tilde{c}}_{p} {u_n (s)}_{H}^{p-1} 
\int_{Y} { F(s,u_n (s),y)}_{H}^{p+1} \, \nu (\ud y)
 + {\tilde{c}}_{p}\int_{Y} {F(s,u_n (s),y)}_{H}^{2p} \nu (dy) \notag \\
& \le   {\tilde{c}}_{p} \Bigl\{ {C}_{2} {u_n (s)}_{H}^{2p-2} (1+{u_n (s)}_{H}^{2} )
  + {C}_{5+\gamma } {u_n (s)}_{H}^{p-1} (1+ {u_n (s)}_{H}^{p+1} ) 
  + {C}_{2(4+\gamma )} (1+{u_n (s) }_{H}^{2p})  \Bigr\}. 
\notag \end{align}

By \eqref{eq:BDG_R_Taylor_1''_new} and the Young inequality we infer that \[
\int_{Y}\bigl( |u_n (s)+ P_n F(s,u_n (s);y) {|}_{H}^{p} -|u_n (s) {|}_{H}^{p}  {\bigr) }^{2} \nu (\ud y) 
\le {K}_{1} + {K}_{2} | u_n (s) {|}_{H}^{2p}
\]     
for some positive constants ${K}_{1}$ and ${K}_{2}$. Thus
\begin{eqnarray}
& & 
 \Bigl( \int_{0}^{t\wedge {\tau }_{n} (R)} 
\int_{Y} \bigl( |u_n (s)  + P_n F(s,u_n (s),y) {|}_{H}^{p} -
|u_n (s) {|}_{H}^{p}  {\bigr) }^{2}  \, \nu (\ud y) \ds {\Bigr) }^{\frac{1}{2}} 
   \nonumber  \\
 & & \le \sqrt{T{K}_{1}} + \sqrt{{K}_{2}}  \Bigl( 
  \int_{0}^{t\wedge {\tau }_{n} (R)} |u_n (s) {|}_{H}^{2p} \ds
{\Bigr) }^{\frac{1}{2}}  . \label{eq:BDG_R_Taylor_1'_new} 
\end{eqnarray}
By \eqref{eq:BDG_R_new}, \eqref{eq:BDG_R_Taylor_1'_new} and \eqref{eq:E(int_un(t))_est}
we obtain the following inequalities
\begin{eqnarray} 
& &\mathbb{E} \Bigl[ \sup_{r\in [0,t]} |{M}_{n}(r \wedge {\tau }_{n} (R))| \Bigr] \nonumber \\
& &\le {\tilde{K}}_{p}\sqrt{T{K}_{1}} + {\tilde{K}}_{p}\sqrt{{K}_{2}} \mathbb{E} \Bigl[ \Bigl( 
  \int_{0}^{t\wedge {\tau }_{n} (R)} | u_n (s) {|}_{H}^{2p} \ds
{\Bigr) }^{\frac{1}{2}} \Bigr]  \nonumber \\
& &\le {\tilde{K}}_{p}\sqrt{T{K}_{1}} + {\tilde{K}}_{p}\sqrt{{K}_{2}} \mathbb{E} \Bigl[ 
 \Bigl( \sup_{s\in [0,t]} | u_n ({s\wedge {\tau }_{n} (R)}^{}) {|}_{H}^{p} {\Bigr) }^{\frac{1}{2}}
 \Bigl( \int_{0}^{t\wedge {\tau }_{n} (R)} |u_n (s) {|}_{H}^{p} \ds 
{\Bigr) }^{\frac{1}{2}} \Bigr]   \nonumber \\
& &\le {\tilde{K}}_{p}\sqrt{T{K}_{1}} + \frac{1}{4} \mathbb{E} \Bigl[ 
  \sup_{s\in [0,t]}  | u_n (s\wedge {\tau }_{n} (R)) {|}_{H}^{p} \Bigr]
 + {\tilde{K}}_{p}^{2} {K}_{2} \mathbb{E} \Bigl[ \int_{0}^{t\wedge {\tau }_{n} (R)}| u_n (s) {|}_{H}^{p} \ds   \Bigr] \nonumber \\
& & \le \frac{1}{4} \mathbb{E} \Bigl[ 
  \sup_{s\in [0,t]}  |u_n ({s\wedge {\tau }_{n} (R)}^{}) {|}_{H}^{p} \Bigr] + \tilde{K} ,
 \label{eq:BDG_R_cont._new}
\end{eqnarray}
where $\tilde{K} = {\tilde{K}}_{p}\sqrt{T{K}_{1}} + {\tilde{K}}_{p}^{2} {K}_{2} {\tilde{C}}_{p}$, with the constant ${\tilde{C}}_{p}$ is the same as in \eqref{eq:E(int_un(t))_est}.
Therefore, by \eqref{eq:Ito_formula} , 
\begin{eqnarray} 
& &| {u}_{n}({t\wedge {\tau }_{n} (R)}){|}^{p}
\le  c + {c}_{1}\int_{0}^{T} | u_n (s){|}^{p}  \ds 
+ \sup_{r \in [0,T] } |{M}_{n}(r\wedge {\tau }_{n} (R))| \nonumber  \\
& &  + |{I}_{n}(T\wedge {\tau }_{n} (R))|, \mbox{ for all $t\in [0,T]$},
\label{eq:Ito_BDG_R_new} 
\end{eqnarray} 
where ${I}_{n}$ is defined by \eqref{eq:I_n}. Since inequality \eqref{eq:Ito_BDG_R_new} holds for all $t\in [0,T]$ and the right-hand side of \eqref{eq:Ito_BDG_R_new} in independent of $t$, we infer that
\begin{eqnarray} 
& &\mathbb{E} \Bigl[ \sup_{t \in [0,T]} |{u}_{n} (t\wedge {\tau }_{n} (R)) {|}^{p} \Bigr]
\le  c + {c}_{1}\mathbb{E} \Bigl[ \int_{0}^{T} |u_n (s){|}^{p}  \ds  \Bigr] \nonumber \\
& & + \mathbb{E} \Bigl[ \sup_{r \in [0,T] } |{M}_{n}(r\wedge {\tau }_{n} (R))| \Bigr] 
+\mathbb{E} \Bigl[ |{I}_{n}(T\wedge {\tau }_{n} (R)) |  \Bigr]  . 
 \label{eq:Ito_BDG_R_E_new}    
\end{eqnarray}
Using inequalities \eqref{eq:E(int_un(t))_est}, \eqref{eq:BDG_R_cont._new} and \eqref{eq:E_I_n(t)}  in \eqref{eq:Ito_BDG_R_E_new}  we infer that
\[ \label{eq:H_est_R}
 \mathbb{E} \Bigl[ \sup_{t \in T}|{u}_{n}(t\wedge {\tau }_{n} (R)) {|}^{p} \Bigr]
 \le {C}_{1}
\]
for some constant ${C}_{1}$ independent of $n \in \mathbb{N} $ and $R>0$. Passing to the limit as $R \to \infty $, we obtain inequality \eqref{eqn-H_estimate-p>2} for $p=4+\gamma  $.
To obtain inequality \eqref{eqn-H_estimate-p>2} for $p \in [1,4+\gamma )$ it suffices to apply the Jensen inequality.
The proof of  Lemma \ref{lem-Galerkin_estimates} is thus complete. \qed

\section{Acknowledgments}

We would like to thank Erika Hausenblas for stimulating discussions, to Adam Jakubowski for important discussion about the Skorohod space 
and to Martin Ondrej\'at  who inspired us to write Appendix \ref{app_martingales}.

\bibliography{References-BKR}%
\bibliographystyle{plain}

\end{document}